\documentclass[12pt]{amsart}
\usepackage{tikz}
\usetikzlibrary{decorations.markings,calc,matrix}
\usepackage{mathtools}
\usepackage{amsfonts,amssymb,mathtools,tabmac}
\usepackage[margin=1.1in]{geometry}
\usepackage{hyperref}
\usepackage{color}

\makeatletter

\newcommand{\dashedrightarrow}{\dashrightarrow}

\newcommand{\longdashedrightarrow}[1][2.5pt]{%
  \settowidth{\@tempdima}{$\longrightarrow$}\longrightarrow
  \makebox[-\@tempdima]{\hskip-0.5ex\color{white}\rule[0.5ex]{#1}{1pt}}
    \phantom{\longrightarrow}
  \makebox[-\@tempdima]{\hskip-2.8ex\color{white}\rule[0.5ex]{#1}{1pt}}
  \phantom{\longrightarrow}
}
\makeatother
\newcommand{\sslash}{\mathbin{/\mkern-6mu/}}

\newcommand{\blackdot}[1]{\filldraw[black] #1 circle (0.1cm);}

\newcommand{\whitedot}[1]{\filldraw[white] #1 circle (0.1cm); \draw #1 circle (0.1cm);}

\def\cl{{\rm cl}}
\def\rank{{\rm rank}}
\def\SU{{\rm SU}}
\def\T{{\mathcal{T}}}
\def\Tr{{\rm Tr}}
\def\Frac{{\rm Frac}}
\def\SL{{\rm SL}}
\def\reg{{\rm reg}}
\def\pt{{\rm pt}}
\def\tS{{\tilde S}}
\def\hGr{{\rm  \hat Gr}}
\def\hPi{{\hat \Pi}}
\def\hX{{\hat X}}

\def\sl{{\mathfrak {sl}}}

\def\tf{{\tilde f}}
\def\diag{{\rm diag}}
\def\Spec{{\rm Spec}}
\newcommand{\oX}{\mathring{X}}
\newcommand{\oPi}{\mathring{\Pi}}
\def\Fl{{\rm Fl}}
\def\sign{{\rm sign}}
\def\PP{{\mathcal{P}}}
\def\FF{{\mathfrak F}}
\def\tF{{\tilde F}}
\def\Tr{{\rm Tr}}
\def\wt{{\rm wt}}
\def\Gr{{\rm Gr}}
\def\GL{{\rm GL}}
\def\Glue{{\rm Glue}}
\def\Rel{{\rm Rel}}
\def\F{{\mathcal{F}}}
\def\R{{\mathbb R}}
\def\C{{\mathbb C}}
\def\CC{{\mathcal{C}}}
\def\CP{{\mathbb {CP}}}
\def\P{{\mathbb P}}
\def\L{{\mathcal{L}}}
\def\O{{\mathcal{O}}}
\def\B{{\mathcal{B}}}
\def\oM{{\bar{ \mathcal{M}}}}
\def\M{{\mathcal{M}}}
\def\A{{\mathbb{A}}}
\def\AA{{\mathcal{A}}}

\def\G{{\mathcal{G}}}
\def\sort{{\rm sort}}
\def\I{{\mathcal{I}}}
\def\II{{\mathfrak{I}}}
\def\J{{\mathcal{J}}}
\def\Z{{\mathbb Z}}
\def\S{{\mathcal{S}}}

\def\dlog{{\rm dlog}\,}
\def\Res{{\rm Res}}
\def\P{{\mathbb P}}
\def\Mat{{\rm Mat}}
\def\X{{\mathcal{X}}}
\def\Y{{\mathcal{Y}}}
\def\id{{\rm id}}
\def\inv{{\rm inv}}
\def\Bound{\mathcal{B}}
\def\hBound{\hat{\mathcal{B}}}
\def\codim{{\rm codim}}

\def\spn{{\rm span}}
\def\rowspan{{\rm rowspan}}
\newcommand{\defn}[1]{{\it #1}}

\newtheorem{theorem}{Theorem}

\newtheorem{thm}[theorem]{Theorem}

\newtheorem{proposition}[theorem]{Proposition}

\newtheorem{problem}[theorem]{Problem}
\newtheorem{corollary}[theorem]{Corollary}

\newtheorem{conj}[theorem]{Conjecture}
\newtheorem{conjecture}[theorem]{Conjecture}
\newtheorem{lemma}[theorem]{Lemma}

\newtheorem{remark}[theorem]{Remark}
\numberwithin{theorem}{section}
\newtheorem{example}[theorem]{Example}
\newtheorem{definition}[theorem]{Definition}

\newcommand\tabll[1]{{\tableau[pY]{#1}}}
\begin{document}
\title[Totally nonnegative Grassmannian and Grassmann polytopes]{Totally nonnegative Grassmannian \\ and \\ Grassmann polytopes}
\
\author{Thomas Lam}
\address{Department of Mathematics, University of Michigan,
2074 East Hall, 530 Church Street, Ann Arbor, MI 48109-1043, USA}
\email{tfylam@umich.edu}\thanks{T.L. was supported by NSF grant DMS-1160726.}
\begin{abstract}
These are lecture notes intended to supplement my second lecture at the Current Developments in Mathematics conference in 2014.  In the first half of article, we give an introduction to the totally nonnegative Grassmannian together with a survey of some more recent work.  In the second half of the article, we give a definition of a Grassmann polytope motivated by work of physicists on the amplituhedron.  We propose to use Schubert calculus and canonical bases to replace linear algebra and convexity in the theory of polytopes.
\end{abstract}

\maketitle
\setcounter{tocdepth}{1}
\tableofcontents

This work is split into two halves.

\medskip

The first part is an introduction to the totally nonnegative Grassmannian.  Most of it should be accessible to graduate students with some background in algebra and combinatorics.  Indeed, Sections \ref{sec:TP}--\ref{sec:positroids} are an expansion of lecture notes \cite{Lamnotes} that I used in part of a graduate course in total positivity.  My approach differs from Postnikov's seminal work \cite{Pos}, in that I build the theory from scratch using the enumeration of perfect matchings (or dimer configurations) as a starting point.  Sections \ref{sec:Schub}--\ref{sec:coord} are a survey of results mostly from \cite{KLS} and \cite{Lam+}, and requires a bit more background in algebraic geometry.  Sections \ref{sec:form}--\ref{sec:relspace} contain some material that is likely somewhat familiar to experts, but the details of which have not been written down as far as I know.  

\medskip

The second part studies a notion of a Grassmann polytope, motivated by Arkani-Hamed and Trnka's definition of an amplituhedron \cite{AT}.  Our aim is to explain some phenomena in this theory via examples and counterexamples.  In Section \ref{sec:matroids}, we propose a related definition of a (realizable) Grassmann matroid.  Sections \ref{sec:trunc}--\ref{sec:sphericoid} work through techniques that allow us to compute Grassmann matroids.  In Sections \ref{sec:facets}--\ref{sec:triangulations}, we explore the face structure and the notion of triangulation for Grassmann polytopes.  In Section \ref{sec:amplitudes} we give an informal explanation of the relation to scattering amplitudes.

\section{Introduction}\label{sec:intro}

\subsection{Total positivity}
A real matrix $g \in \GL(n,\R)$ is \defn{totally nonnegative} if all of its minors are nonnegative.  This notion goes back to the works of Schoenberg \cite{Sch} and Gantmacher and Krein \cite{GaKr}, who noticed that such matrices possess remarkable spectral properties and a variation-diminishing property.

For the purpose of this article, the first main result of the totally nonnegative part $\GL(n)_{\geq 0}$ is that it has a trio of descriptions (see Section \ref{sec:TP} for details).
\begin{theorem}\label{thm:TPfirst}
Let $g \in \GL(n,\R)$.  Then the following are equivalent:
\begin{enumerate} \item $g$ is totally nonnegative;
\item $g$ is in the semigroup generated by positive Chevalley generators and positive diagonal matrices;
\item $g$ is representable by a planar network.
\end{enumerate}
\end{theorem}
Thus $\GL(n)_{\geq 0}$ can be described by inequalities, as a semigroup with specified generators, and as matrices obtained by a combinatorial construction.  It is the interplay between these structures that give rise to a rich theory.  Another important feature of $\GL(n)_{\geq 0}$ is that it has a natural cell decomposition (Theorem \ref{thm:Lusstrata}).

Lusztig \cite{LusTP} used the semigroup description to generalize $\GL(n)_{\geq 0}$ to other split real reductive groups.

\subsection{The totally nonnegative Grassmannian and the dimer model}
Let $\Gr(k,n)$ denote the complex Grassmannian of $k$-planes in $\C^n$.  We review some basic facts concerning the Grassmannian in Section \ref{sec:Gr}.
Postnikov \cite{Pos} defines the \defn{totally nonnegative Grassmannian} $\Gr(k,n)_{\geq 0}$ to be the locus in the Grassmannian with nonnegative Pl\"ucker coordinates.  Lusztig \cite{LusTP} defined the totally nonnegative part of any generalized partial flag variety $G/P$ that turns out to agree (Theorem \ref{thm:same} and Remark \ref{rem:same}) with Postnikov's in this case.

The immediate goal of Sections \ref{sec:dimer}--\ref{sec:representable} is to prove an analogue of Theorem \ref{thm:TPfirst} for the totally nonnegative Grassmannian $\Gr(k,n)_{\geq 0}$.  To construct points $X \in \Gr(k,n)_{\geq 0}$ we study \defn{almost perfect matchings} (or \defn{dimer configurations}) $\Pi$ in a planar bipartite network $N$.  Theorem \ref{thm:matchingplucker} states that counting perfect matchings with particular boundary conditions one obtains quantities $\{\Delta_I(N) \mid I \in \binom{[n]}{k}\}$ that are the Pl\"ucker coordinates of a point $X(N)$ in the Grassmannian.  
The main argument for this result goes back to work of Kuo \cite{Kuo}, and we have formulated it in the more algebraic language of \defn{Temperley-Lieb invariants} \cite{Lamweb}.  

In Theorem \ref{thm:main}, we show (the much harder direction) that every point in $\Gr(k,n)_{\geq 0}$ is of the form $X(N)$ for some planar bipartite network $N$.
The argument we use is close in spirit to Whitney's result \cite{Whi} for totally nonnegative matrices.  Namely, we reduce a point $X \in \Gr(k,n)_{\geq 0}$ by repeatedly applying column operations until we obtain one of the torus-fixed points of $\Gr(k,n)$.  On the combinatorial side, the reduction procedure corresponds to adding or removing ``bridges" or ``lollipops" from a network $N$.  These bridges are the network analogue of the semigroup generators in Theorem \ref{thm:TPfirst}(2).

Section \ref{sec:flows} describes how to obtain points in $\Gr(k,n)_{\geq 0}$ from Postnikov's \defn{plabic networks} that are more general than the planar bipartite networks that we use.  We do not review Postnikov's original construction but summarize Talaska's approach \cite{Tal} via flows.  The connection to the enumeration of matchings was observed by Postnikov, Speyer, and Williams \cite{PSW}.  

Section \ref{sec:relspace} describes yet another way to obtain points in $\Gr(k,n)$ from graphs, this time using linear algebra instead of combinatorics.  The construction goes under the name of \defn{on-shell diagram} in the scattering amplitudes literature, and we call it the ``relation space" $\Rel(N) \in \Gr(k,n)$ of a network $N$.  I could not find in the literature a comparison of this construction with any of the combinatorial approaches (matchings, flows, or paths), so I have formulated and proved some basic results.   A key feature is that the relation space construction typically produces points in $\Gr(k,n)$ that are not totally nonnegative.

The connection between the totally nonnegative Grassmannian and planar networks has also found applications in certain integrable systems, see \cite{KoWi,GSV}.

\subsection{Stratification of $\Gr(k,n)_{\geq 0}$}
Each point $X \in \Gr(k,n)$ has a matroid $\M_X$, defined in \eqref{eq:matroid}.  A \defn{positroid} is the matroid of a point $X \in \Gr(k,n)_{\geq 0}$ in the totally nonnegative Grassmannian.  An important byproduct of the construction $N \mapsto X(N)$ is that it gives a stratification \cite{Pos}
\begin{equation}\label{eq:Grstratum}
\Gr(k,n)_{\geq 0} = \bigsqcup_{f \in \Bound(k,n)} \Pi_{f, >0}
\end{equation}
of $\Gr(k,n)_{\geq 0}$ into \defn{positroid cells} $\Pi_{f,>0}$, with the properties (Theorem \ref{thm:main}) that

\noindent
(a)
each stratum $\Pi_{f,>0}$ is homeomorphic to $\R_{>0}^d$ for some $d \geq 0$,

\noindent
(b)
the positroid $\M_X$ is constant on each stratum $\Pi_{f,>0}$, and distinct strata have distinct positroids,

\noindent
(c)
for each stratum $\Pi_{f,>0}$ there exists a  planar bipartite graph $G$ so that every point $X \in \Pi_{f,>0}$ is equal to $X(N)$ for a network $N$ obtained by placing edge weights on $G$.

Postnikov \cite{Pos} gave (very!) many ways to index these strata.  Our preferred indexing set is the set $\Bound(k,n)$ of $(k,n)$-bounded affine permutations $f$ (essentially equivalent to Postnikov's decorated permutations).  In \cite{KLS} (Theorem \ref{thm:partialorder} here), it is shown that the closure partial order of positroid cells is (dual to) the well-studied Bruhat order of the affine symmetric group.  In Section \ref{sec:perms}, we also describe the Grassmann necklaces and cyclic rank matrices that can be used to index the strata.

The most important result about positroids is Oh's theorem \cite{Oh} stating that a positroid $\M$ is the intersection of cyclically rotated Schubert matroids.  Our proof of this in Theorem \ref{thm:oh} appears to be new.  In Section \ref{sec:positroids}, we also state two other characterizations of positroids: (a) together with Postnikov \cite{LPmembrane}, we showed that positroids are exactly the sort-closed matroids; (b) Ardila, Rincon, and Williams \cite{ARW} characterize positroids as exactly the underlying matroids of positively orientable matroids.

The elegant combinatorics of the stratification \eqref{eq:Grstratum} is a reflection of topological properties, some proved \cite{Lusintro,PSW,RW}, and some conjectural.

\subsection{Positroid varieties}
The remarkable stratification \eqref{eq:Grstratum} of the totally nonnegative Grassmannian is the intersection of $\Gr(k,n)_{\geq 0}$ with an equally remarkable stratification of the complex Grassmannian into the \defn{positroid varieties} $\Pi_f$.  This stratification was introduced by Lusztig \cite{Luspartial} for a generalized partial flag manifold, and systematically studied in the Grassmannian case in our work with Knutson and Speyer \cite{KLS}.  

We define positroid varieties and summarize some of their geometric properties in Section \ref{sec:Schub}.  For our purposes, the most important fact is that $\Pi_f$ is an irreducible, projectively normal subvariety of the Grassmannian whose ideal is linearly generated (Proposition \ref{prop:irred} and Theorems \ref{thm:prime} and \ref{thm:normal}).

The positroid stratification has been of interest in a number of directions.  We list some not mentioned in the main text here. 

(a) In \cite{KLS2}, it is shown that positroid varieties are exactly the compatibly Frobenius split subvarieties of the Grassmannian, with respect to the standard Frobenius splitting.  

(b) Goodearl and Yakimov \cite{GY} showed that (open) positroid varieties are exactly the torus orbits of symplectic leaves of the Grassmannian as a Poisson manifold.  

(c) There is an analogous classification of torus-invariant primes in quantum Schubert cell algebras by M\'eriaux and Cauchon \cite{MC} and Yakimov \cite{Y}.  See also \cite{LL} for a survey of the relations between total nonnegativity, quantum matrices, and Poisson geometry.  

(d) The cluster algebra structure of the coordinate ring $\C[\oPi_f]$ of open positroid varieties has attracted much recent attention \cite{Lec,MS,LY}.

(e) Positroid varieties play a role in the study of period integrals on the flag variety, where they are candidates to be large complex structure limit points, as discussed by Huang, Lian, and Zhu \cite{HLZ}.

\medskip

In Section \ref{sec:cohom}, we review some standard facts about the cohomology ring $H^*(\Gr(k,n))$ of the Grassmannian, and in Theorem \ref{thm:KLS} formulate the result from \cite{KLS} that the cohomology class of the positroid variety $[\Pi_f]$ is equal to the \defn{affine Stanley symmetric function} $\tF_f$ of \cite{LamAffineStanley}.  This result will play an important role in the applications to Grassmann matroids.  

There are many explicit relations to quantum and affine Schubert calculus that we will not pursue here.  For example, certain positroid varieties turn out to be projections of two-point Gromov-Witten varieties for the Grassmannian, see \cite{BCMP,KLS}.  In addition, there is a mysterious relation \cite{KLS,HL,Sni} between the positroid stratification and the geometry of affine flag varieties.  We have already noted that the closure partial order for positroid cells agrees with the affine Bruhat order.  The affine Stanley symmetric functions $\tF_f$ themselves appear in the study of affine Schubert caculus \cite{book,LamJAMS}.  See \cite{Knu} for another application of positroid varieties.

\subsection{Cyclicity, promotion, and canonical bases}
The cyclic group $\Z/n\Z$ acts on all the objects we have discussed so far: the stratifications $\Gr(k,n)_{\geq 0} = \bigsqcup_{f \in \Bound(k,n)} \Pi_{f,>0}$ and $\Gr(k,n) = \bigsqcup_{f \in \Bound(k,n)} \oPi_{f}$ are invariant under the cyclic group action, and the indexing set $\Bound(k,n)$ of bounded affine permutations has a natural action by the cyclic group.

The same cyclic group arises in another place in algebraic combinatorics: the set of rectangular semistandard Young tableaux with entries in $[n] \coloneqq \{1,2,\ldots,n\}$ have an action of the cyclic group $\Z/n\Z$ via \defn{promotion}.  In \cite{Lam+}, the connection with the positroid stratification is made by describing the homogeneous coordinate ring $\C[\hPi_f]$ and the homogeneous ideal $I(\Pi_f)$ of a positroid variety in terms of the canonical basis \cite{Luscan,Kascan}.  We survey these results in Sections \ref{sec:tableaux}--\ref{sec:coord}.

It is a classical theme in representation theory to consider the space of sections of line bundles on flag varieties.  Let $\O(1)$ denote the line bundle on $\Gr(k,n)$ pulled back from the Pl\"ucker embedding, and let $\O(d)$ denote its $d$-th tensor power.  By the classical Borel-Weil theory, the space of sections $\Gamma(\Gr(k,n), \O(d))$ on a Grassmannian can be identified with the (dual of the) irreducible representation $V(d\omega_k)$ of $\GL(n)$ indexed by the $k \times d$ rectangular shape.  The space of sections $\Gamma(X_I,\O(d))$ on a Schubert subvariety $X_I \subset \GL(n)$ is then identified with a Demazure submodule $V_I(d\omega_k) \subset V(d\omega_k)$.  The space of sections $\Gamma(\Pi_f,\O(d))$ on a positroid variety can then be identified with the intersection of cyclically rotated Demazure modules, which we call the \defn{cyclic Demazure module} $V_f(d\omega_k)$.  The cyclic Demazure submodule is spanned by canonical basis elements with remarkable positivity and cyclicity properties (Theorem \ref{thm:canbasis} and Theorem \ref{thm:Posideal}).  We remark that Lakshmibai and Littelmann \cite{LaLi} have also constructed a standard monomial basis for the vector space $\Gamma(\Pi_f,\O(d))$.

One of the new perspectives that we hope to advertise is that the crystal graph (the natural indexing set for the canonical basis) of $\Gamma(\Gr(k,n), \O(d))$ is a higher degree analogue of the uniform matroid of rank $k$ on $[n]$.  Indeed, when $d=1$, the crystal graph of $\Gamma(\Gr(k,n), \O(1))$ can be identified with the set of $k$ element subsets of $[n]$.  We have the following analogies:
\begin{center}
\begin{tabular}{|c|c|c|c|}
\hline
Geometry & Representation theory & $d=1$ Combinatorics & $d>1$ Combinatorics \\
\hline
Grassmannian & Rectangular irred. & uniform matroid & crystal on rect. tableaux \\
\hline
Schubert & Demazure submod. & Schubert matroid & Demazure crystal \\
\hline
Positroid & cyclic Demazure & positroid & cyclic Demazure crystal \\
\hline
\end{tabular}
\end{center}

\subsection{Scattering amplitudes and the canonical form}
In Section \ref{sec:form}, we study the \defn{canonical form} $\omega_f$ of a positroid variety, a distinguished mermorphic top form on $\Pi_f$ with simple poles exactly along the boundary $\partial \Pi_f = \bigcup_{g < f} \Pi_f$.  The canonical form $\omega_{\Gr(k,n)}$ generates the positroid stratification, as follows.  The positroid divisors $\Pi_1,\Pi_2,\ldots,\Pi_n$ are the poles of $\omega_{\Gr(k,n)}$; the canonical form of a positroid divisor $\Pi_r$ is exactly the residue $\Res_{\Pi_r} \omega_{\Gr(k,n)}$.  Repeating, we can produce all the positroid varieties and their canonical forms.

These forms were considered implicitly in \cite{KLS2}, where it arises from the standard Frobenius splitting of the Grassmannian.  Quite unexpectedly to me, these differential forms also appears in two seemingly unrelated contexts, where it is part of an integrand to be integrated along certain real cycles: (a) in the study of Whittaker functions \cite{LamWhittaker}, and (b) in the study of scattering amplitudes, to be discussed in Section \ref{sec:amplitudes} and briefly presently. 

\defn{Scattering amplitudes} \cite{EH,HP} are quantities studied in quantum field theory used to compute the probabilities that certain particle scattering experiments occur.  One of the remarkable recent developments in the theory is that scattering amplitudes in maximally supersymmetric Yang-Mills theory can be computed (at tree level) as an integral over the Grassmannian:
\begin{equation}\label{eq:scat}
\text{amplitude} = \int_{\text{some contour}} \text{(some delta function) }\; \omega_{\Gr(k,n)}. 
\end{equation}
Here, the delta function amounts to considering the integral over a subGrassmannian of $\Gr(k,n)$ that depends on the momenta of the particles being considered.  

The equation \eqref{eq:scat} was made more combinatorial in \cite{ABCGPT} where a formula 
\begin{equation}\label{eq:posscat}
\text{amplitude} = \sum_{\text{$f \in \CC(k,n) \subset \Bound(k,n)$}} \int \text{(some delta function) }\; \omega_{\Pi_f} 
\end{equation}
was given.  Here, the delta function has the same dimension as $\Pi_f$, so the integral amounts to formally evaluating $\omega_{\Pi_f}$ at certain (possibly complex) points of $\Pi_f$.  Only the delta functions, and not the subset $\CC(k,n) \subset \Bound(k,n)$, depend on the momenta of the particles involved.  Furthermore, the subset $\CC(k,n) \subset \Bound(k,n)$ is generated by a ``BCFW recursion", and many such subsets would give the same answer.  

\subsection{The amplituhedron and Grassmann polytopes}
In \cite{AT}, it was suggested that \eqref{eq:posscat} could be considered an expression for the volume of some space as the sum over the simplices of a triangulation of that space (see also \cite{ABCHT}).  Arkani-Hamed and Trnka called this space the \defn{amplituhedron}.

When $k = 1$, the Grassmannian $\Gr(1,n)$ is the projective space $\P^{n-1}$.  The totally nonnegative Grassmannian $\Gr(1,n)_{\geq 0}$ can be identified with the $(n-1)$-dimensional simplex sitting inside $\P^{n-1}$.  Let $Z$ be a real $n \times r$ matrix with $r \leq n$.  The matrix $Z$ can be considered a linear map $\R^n \to \R^r$ and induces a rational map $Z:\P^{n-1} \dashedrightarrow \P^{r-1}$, and more generally a rational map $Z_\Gr: \Gr(k,n) \dashedrightarrow \Gr(k,r)$.  The image $Z(\Gr(1,n)_{\geq 0})$ can then be identified with the polytope equal to the convex hull of the row vectors of $Z$.

Call $Z$ \emph{positive} if its maximal ($r \times r$) minors are strictly positive.  The \defn{amplituhedron} is the image of $\Gr(k,n)_{\geq 0}$ under the map $Z_\Gr$ for a positive $Z$, with the case of physical importance being $r = k+4$.  When $k=1$, the amplituhedron is a cyclic polytope.

In Section \ref{sec:poly}, we define a \defn{Grassmann polytope} to be the image $Z(\Pi_{f,\geq 0})$, under the following condition:
$$
\mbox{there exists a $r \times k$ real matrix $M$ such that $Z \cdot M$ has positive $k \times k$ minors.}
$$
The analogous condition for $k = 1$ appears in variants of Farkas' Lemma, and in linear programming.  In the current analogy, the totally nonnegative Grassmannian $\Gr(k,n)_{\geq 0}$ is an analogue of the simplex, the positroid stratification is an analogue of the face stratification of the simplex, and the poset $\B(k,n)$ is an analogue of the boolean lattice.  We explain in Section \ref{sec:poly} some behavior of Grassmann polytopes that may be considered unusual from the perspective of the classical theory.


\subsection{Grassmann matroids} Convexity may be thought of as the study of positive linear combinations.  Before we can study convexity, it is natural to first study linearity.

In Section \ref{sec:matroids}, we study the ``linear" behavior of the rational map $Z_\Gr:\Gr(k,n) \dashedrightarrow \Gr(k,r)$, by defining a notion of the \defn{Grassmann matroid} $\G_Z$ of $Z$.  As a replacement for the notion of ``linear span of the vectors $z_{i_1},z_{i_2},\ldots,z_{i_s}$", we consider the Zariski-closure $Z(\Pi_f) \coloneqq \overline{Z(\Pi_{f,\geq 0})}$ of the image $Z(\Pi_{f,\geq 0})$.  For example, $f \in \Bound(k,n)$ is called \defn{independent} if $Z(\Pi_{f})$ has the same dimension as $\Pi_{f}$.  Also, $f, g \in \Bound(k,n)$ \defn{belong to the same flat} if $Z(\Pi_f) =Z(\Pi_g)$.

We do not attempt to axiomatize Grassmann matroids here, but in Sections \ref{sec:trunc}--\ref{sec:sphericoid} we discuss some techniques from \cite{Lamtrunc,Lam+} that can be used to compute Grassmann matroids.

When $k = 1$, the rank function of the matroid $\M_Z$ of $Z$ is simply the function $r_Z: f \mapsto \dim(Z(\Pi_{f}))+1$.  When $k > 1$, we propose that the invariant $\dim(Z(\Pi_{f}))$ should be upgraded to the cohomology class $[Z(\Pi_f)] \in H^*(\Gr(k,r))$,to give the \defn{class function} $c_Z: f \mapsto [Z(\Pi_f)]$.  The calculation of $[Z(\Pi_f)]$ is a Schubert calculus problem.  To compute this cohomology class is equivalent to computing the number of intersection points $\#(Z(\Pi_f) \cap Y_J)$ where $Y_J \subset \Gr(k,r)$ is a Schubert variety of complementary dimension in general position with respect to $Z(\Pi_f)$.  Thus, for Grassmann matroids, linear algebra is replaced by Schubert calculus.

In our earlier work \cite{Lamtrunc}, we gave a formula for the cohomology class $[Y_f] \in H^*(\Gr(k,r))$ of an \defn{amplituhedron variety} $Y_f \subset \Gr(k,r)$, defined to be $Y_f \coloneqq Z(\Pi_f)$ when $Z$ is a generic matrix.  The main result (Theorem \ref{thm:trunc}) states that $[Y_f]$ is the \emph{truncation} of the affine Stanley symmetric function mentioned previously.  In the context of Grassmann matroids, this result is then a formula for the class function of the \defn{uniform Grassmann matroid}.

\subsection{Amplituhedron varieties and sphericoid varieties}
In Sections \ref{sec:ampliideal} and \ref{sec:sphericoid}, we explain results from \cite{Lam+} concerning the homogeneous ideals $I(Y_f)$ of amplituhedron varieties.  It is necessary to compute these ideals to understand flats of Grassmann matroids.  We illustrate in examples that these ideals may not be linearly generated (so flats of Grassmann matroids are cut out by higher degree equations).

Since $Y_f$ depends on the matrix $Z$, to describe the ideal $I(Y_f)$, we consider the universal amplituhedron variety $\Y_f \to \Mat(n,r)$ whose fibers over generic $Z \in \Mat(n,r)$ are the amplituhedron varieties $Y_f$.  Some geometry and invariant theory related to $\Y_f$ is described in Section \ref{sec:ampliideal}.

Let $\ell = n - r$.  There is a direct sum rational morphism 
\begin{align*}
\bigoplus: \Gr(k,n) \times \Gr(\ell,n) &\longdashedrightarrow \Gr(k+\ell,n)\\
(X, K) &\longmapsto X + K
\end{align*}
where $X+K$ is simply the linear span of the $k$-plane $X$ and the $\ell$-plane $K$.  For $f \in \Bound(k,n)$ and $f' \in \Bound(\ell,n)$, we define the \defn{sphericoid variety} $$\Pi_{f,f'} \coloneqq \overline{\bigoplus(\Pi_f,\Pi_{f'})} \subseteq \Gr(k+\ell,n).$$  Proposition \ref{prop:YPi} states that computing the ideal $I(\Pi_{f,\id})$ is equivalent to computing $I(Y_f)$.  Here $\id \in \Bound(\ell,n)$ is the bounded affine permutation such that $\Pi_{\id} = \Gr(\ell,n)$.  The $\ell$-plane $K$ should be identified with the kernel $\ker(Z)$.

The advantage of considering the map $\bigoplus$ is that the corresponding map on homogeneous coordinate rings has a familiar representation theoretic description.  It is induced from the unique (up to scalar) non-trivial $\GL(n)$-homomorphism
$$
V(d\omega_k) \otimes V(d\omega_\ell) \longrightarrow V(d\omega_{k+\ell})
$$
where as before $V(d\omega_k)$ denotes the highest weight $\GL(n)$-representation indexed by a $k \times d$ rectangle.  Combining with the results from Section \ref{sec:coord}, we obtain a representation theoretic description of $I(\Pi_{f,f'})$ in Theorem \ref{thm:sphericoid}.  We work through some examples in Section \ref{sec:sphericoid}.  We also state in Theorems \ref{thm:sphebic} and \ref{thm:equatorial} a construction of points $X(N) \in \Pi_{f,f'}$ by counting matchings on a \defn{spherical bipartite network}, explaining the nomenclature.

\subsection{Facets and triangulations of Grassmann polytopes}
In Section \ref{sec:facets}, we discuss facets of Grassmann polytopes.  We do not study a complete definition of faces here, but instead we illustrate some phenomena in examples.  In particular, we show how to analyze some faces of the amplituhedron within our framework, and illustrate the feature that geometric facets of Grassmann polytopes are typically unions of smaller Grassmann polytopes.  

In Section \ref{sec:form2}, we define the canonical form $\omega_{Z(\Pi_f)}$ on the varieties $Z(\Pi_f)$.  These differential forms are defined to be the pushforward, or trace, of the canonical form $\omega_f$ on positroid varieties $\Pi_f$.  We formulate a conjecture (Conjecture \ref{conj:polesandzeroes}) on the divisor of poles and zeroes of $\omega_f$.

In Section \ref{sec:triangulations}, we make contact with the work of Arkani-Hamed and Trnka \cite{AT} by formulating an informal conjecture (Conjecture \ref{conj:form}) that every Grassmann polytope $P \subset \Gr(k,r)$ itself has a canonical form $\omega_P$ with remarkable properties.  In particular, this form $\omega_P$ should be the sum of the canonical forms $\omega_{Z(\Pi_f)}$ over a triangulation $P = \bigcup_{f \in \T} Z(\Pi_{f,\geq 0})$, reminiscient of equation \eqref{eq:posscat}.
When $P$ is the amplituhedron, this form should be the (tree) amplitude form $\omega_{SYM}$ of super Yang-Mills theory as studied in \cite{AT}.  Conjecture \ref{conj:form} does hold in the case that $P$ is a usual polytope, and we give a brief construction of this form $\omega_P$, which will be further studied in joint work with Arkani-Hamed and Bai \cite{ABL}.

\medskip
\noindent
{\bf Acknowledgements.}
It is a pleasure to thank Nima Arkani-Hamed, Yuntao Bai, Pierre Baumann, Allen Knutson, Alex Postnikov, Mark Shimozono, David Speyer, Jara Trnka, and Lauren Williams for helpful discussions and comments.  I thank Jacob Bourjaily and Willem de Graaf for their help with computer computations.

\part{The totally nonnegative Grassmannian}\label{part:one}
Let $\GL(n)$ denote the complex general linear group, and $\GL(n,\R)$ denote the real general linear group.
We use the notation $[n]\coloneqq \{1,2,\ldots,n\}$ and $\binom{S}{k}$ denotes the set of $k$-element subsets of a finite set $S$.

\section{Total positivity}\label{sec:TP}
In this section we set the stage by giving a brief introduction to some classical and some more recent results in total positivity.  We make no attempt to be comprehensive, and point the reader to \cite{FZ:tests} for an accessible introduction and many references.
\subsection{Total nonnegativity in $\GL(n)$}
Let $M$ be a matrix 
with real entries.  We say that $M$ is \defn{totally nonnegative} (TNN for short) if the determinant of any finite submatrix of $M$ is nonnegative.  We say that $M$ is \defn{totally positive} (TP for short) if the determinant of any finite submatrix of $M$ is positive.  We write $\GL(n)_{\geq 0}$ (resp. $\GL(n)_{>0}$) for the subset of TNN elements (resp. TP elements) in $\GL(n,\R)$.  

Let $X$ be a $p \times q$ matrix and $Y$ a $q \times p$ matrix.  Then the \defn{Cauchy-Binet formula} states that
\begin{equation}\label{eq:CB}
\det(XY) = \sum_{I \in \binom{[q]}{p}} \det(X_{[p],I})\det(Y_{I,[p]})
\end{equation}
where $X_{A,B}$ denotes the submatrix of $X$ with rows indexed by the set $A$ and columns indexed by the set $B$.  Note that the summation on the right hand side is empty if $p > q$, which agrees with the fact that the determinant on the left hand side is 0.  

\begin{corollary}
The totally nonnegative part $\GL(n)_{\geq 0}$ is a submonoid of $\GL(n)$.  The totally positive part $\GL(n)_{>0}$ is a subsemigroup of $\GL(n)$.
\end{corollary}

\subsection{Semigroup generators}
We now describe the semigroup generators of $\GL(n)_{\geq 0}$.  For $(i,j) \in [n]^2$, let $e_{i,j}$ denote the matrix which has a $1$ in the $i$-th row and $j$-th column and $0$-s elsewhere.  For $a \in \C$, and an integer $i \in [n-1]$, define $x_i(a)\coloneqq I_n+a\,e_{i,i+1}$ and $y_i(a) \coloneqq I_n+a\,e_{i+1,i}$, where $I_n$ denotes $n \times n$ identity matrix.  It is easy to check that $x_i(a), y_i(a) \in \GL(n)_{\geq 0}$ when $a \in \R_{\geq 0}$.  For example, for $n = 4$, we would have
$$
x_2(a) = \left[\begin{array}{cccc} 1 & && \\&1&a&\\&&1& \\&&&1\end{array}\right] 
\qquad \text{and} \qquad
y_{3}(b) = \left[\begin{array}{cccc} 1 &&& \\&1&&\\&&1& \\&&b&1\end{array}\right] 
$$

\begin{theorem}[Loewner-Whitney theorem \cite{Loe,Whi}]\label{thm:LW}
$\GL(n)_{\geq 0}$ is the subsemigroup of $\GL(n,\R)$ generated by the elements $\{x_i(a) \mid a \in \R_{>0}\}$, $\{y_i(a) \mid a \in \R_{>0}\}$, and diagonal matrices with positive real entries.
\end{theorem}

Let us briefly sketch the main idea of the proof of Theorem \ref{thm:LW}.  Multiplication by the matrices $x_i(a)$ and $y_i(a)$ act as row operations.  The idea is to start with an arbitrary $g \in \GL(n)_{\geq 0}$, and to ``reduce" it to a diagonal matrix by row operations.  The key step is to find $i \in [n-1]$ and $a \in \R_{>0}$ so that $g' = x_i(-a)g$ (or $g' = y_i(-a)g$) has more zero entries than $g$, but $g'$ is still totally nonnegative.  We shall apply the same philosophy to prove the harder Theorem \ref{thm:main}.  

In 1994, Lusztig \cite{LusTP} turned Theorem \ref{thm:LW} around to define the totally nonnegative part of any real reductive group as a semigroup generated by distinguished elements.

\subsection{Lindstr\"om-Gessel-Viennot}\label{sec:LGV}
Suppose we have a directed acyclic planar network $N$ with sources labeled by $[n]$ and sinks labeled by $[n]'$, with all positive real edge weights, as illustrated below:
\begin{center}
\begin{tikzpicture}
\draw[thick] (0,0)--(4,0);
\draw[thick] (0,1)--(4,1);
\draw[thick] (0,2)--(4,2);
\draw[thick] (1,0)--(2,1);
\draw[thick] (2,1)--(3,2);
\draw[thick] (1,2)--(2,1);
\blackdot{(0,0)};
\blackdot{(1,0)};
\blackdot{(4,0)};
\blackdot{(0,1)};
\blackdot{(2,1)};
\blackdot{(4,1)};
\blackdot{(0,2)};
\blackdot{(1,2)};
\blackdot{(3,2)};
\blackdot{(4,2)};
\node at (1.3,1.5) {$a$};
\node at (1.3,0.5) {$b$};
\node at (2.7,1.5) {$c$};
\node at (-1.5,1) {$N=$};
\node at (-0.5,2) {$1$};
\node at (-0.5,1) {$2$};
\node at (-0.5,0) {$3$};
\node at (4.5,2) {$1'$};
\node at (4.5,1) {$2'$};
\node at (4.5,0) {$3'$};
\node at (4,-1) {All edges are directed to the right.  Unlabeled edges have weight 1.} ;
\node at (7,1) {$M(N) = $};
\matrix (A) [matrix of math nodes,%
             left delimiter  = \lbrack,%
             right delimiter = \rbrack] at (9.5,1)
{%
  1 + ac & a &  0  \\
  c & 1 & 0 \\
  bc & b & 1\\
};
\end{tikzpicture}
\end{center}

Define a $n \times n$ matrix $M(N)$ with entries $(m_{ij})$ where $m_{ij}$ is the weight generating function of directed paths from source $i$ to sink $j'$.  Here, we define the weight of a path to be the product of the weights of edges on the path.

\begin{theorem}\label{thm:LGV}
Suppose $g \in \GL(n)$.  Then $g$ is totally nonnegative if and only if $g = M(N)$ for some directed acyclic planar network $N$ with positive real edge weights.
\end{theorem}
The ``if" part of Theorem \ref{thm:LGV} follows from the Lindstr\"om Lemma \cite{Lin}, sometimes also called the Gessel-Viennot method: each minor $\det(M(N)_{I,J})$ of $M(N)$ has a combinatorial interpretation as the weight generating function of \emph{non-intersecting} families of paths in $N$ with source set $I$ and sink set $J$.  The ``only if" part of Theorem \ref{thm:LGV} follows from Theorem \ref{thm:LW} and the observation that $M(N) \cdot M(N') = M(N \ast N')$ where $N \ast N'$ denotes the concatenation of $N$ and $N'$.  It is then enough to construct a network representing each of the generators $x_i(a),y_i(a)$ and positive diagonal matrices.

\subsection{Stratification}
Let $B \subset \GL(n)$ (resp. $B_- \subset \GL(n)$) denote the subgroup of upper (resp. lower) triangular matrices.  For a permutation $w \in S_n$, we also use $w$ to denote the corresponding permutation matrix, and we let $\ell(w)$ denote the length of $w$.   We have the Bruhat decompositions $\GL(n) = \bigcup_{w \in S_n} BwB = \bigcup_{v \in S_n} B_- v B_-$.  Define $$\GL(n)_{\geq 0}^{w,v} \coloneqq \GL(n)_{\geq 0} \cap BwB \cap B_- v B_-.$$  Then we have $\GL(n)_{\geq 0} = \bigsqcup \GL(n)_{\geq 0}^{w,v}$.
 
\begin{thm}[\cite{LusTP}] \label{thm:Lusstrata}
The topological space $\GL(n)_{\geq 0}^{w,v}$ is homeomorphic to $\R_{>0}^{n+\ell(w)+\ell(v)}$.  
\end{thm}
The homeomorphism is given explicitly by a map of the form 
\begin{equation}\label{eq:TNNfactor}
(a_1,\ldots,a_r,t_1,\ldots,t_n,b_1,\ldots,b_s)  \mapsto x_{i_1}(a_1)\cdots x_{i_r}(a_r) \diag(t_1,\ldots,t_n) y_{i_1}(b_1) \cdots y_{i_s}(b_s),
\end{equation}
where $w = s_{i_1} \cdots s_{i_r}$ and $v = s_{i_1} \cdots y_{i_s}$ are reduced factorizations, and $a_i,t_j,b_k \in \R_{>0}$.  Theorem \ref{thm:Lusstrata} is proven by analyzing the relationship between the Bruhat decomposition and the reduction procedure used in the proof of Theorem \ref{thm:LW}.

Denote by $w' \leq w$ the Bruhat order on the symmetric group $S_n$.
\begin{thm}[\cite{LusTP}]\label{thm:closure}
We have $\overline{\GL(n)_{\geq 0}^{w,v}} = \bigsqcup_{w' \leq w, v' \leq v} \GL(n)_{\geq 0}^{w',v'}$.
\end{thm}

The $\supseteq$ inclusion of Theorem \ref{thm:closure} is obtained by sending some of the parameters $a_i$ and $b_k$ in \eqref{eq:TNNfactor} to 0, and using the characterization of Bruhat order via subwords of reduced words.  The $\subseteq$ inclusion of Theorem \ref{thm:closure} is obtained by the geometric characterization of Bruhat order: $\overline{BwB} = \bigsqcup_{w' \leq w} Bw'B$.  The topological structure of the decomposition $\GL(n)_{\geq 0} = \bigsqcup \GL(n)_{\geq 0}^{w,v}$ and of similar stratified spaces has drawn quite a bit of recent interest, see for example \cite{Hersh}.

\section{The Grassmannian}\label{sec:Gr}
\subsection{Real and complex Grassmannians}
Let $k \leq n$ be positive integers.  The Grassmannian $\Gr(k,n)$ is the space of $k$-dimensional subspaces of the complex vector space $\C^n$.  The space $\Gr(k,n)$ can be given the structure of a smooth complex projective variety, as follows.  Let $X \subset \C^n$ be a $k$-dimensional subspace.  Then $X$ has a basis $\{x_1,x_2,\ldots,x_k\} \subset \C^n$, so we have
$$
X = \rowspan \begin{bmatrix} \;\;\rule[3pt]{2cm}{0.4pt} \;\; x_1 \;\;\rule[3pt]{2cm}{0.4pt} \;\; \\
\;\;\rule[3pt]{2cm}{0.4pt} \;\; x_2 \;\;\rule[3pt]{2cm}{0.4pt} \;\;\\
\;\;\rule[3pt]{2cm}{0.4pt} \hspace{1.3pt}\;\;\; \vdots \hspace{1.3pt}\;\;\;\rule[3pt]{2cm}{0.4pt} \;\; \\
\;\;\rule[3pt]{2cm}{0.4pt} \;\; x_n \;\;\rule[3pt]{2cm}{0.4pt} \;\;
\end{bmatrix}.
$$ 
Every full rank $k \times n $ matrix $M$ represents a point in $\Gr(k,n)$.  Two $k \times n$ matrices $M,M'$ represent the same point $X \in \Gr(k,n)$ if we have $M' = g \cdot M$ for $g \in \GL(k)$.  We will often abuse notation by identifying $X \in \Gr(k,n)$ with a matrix $M$ representing it.

For $I \in \binom{[n]}{k}$, let $\Delta_I(X) \coloneqq \Delta_I(M)$ denote the $k \times k$ minor of $M$ with columns indexed by the elements of $I$.  Since $\Delta_I(gM) = \det(g) \Delta_I(M)$, the collection of \defn{Pl\"{u}cker coordinates} $\{\Delta_I(X) \mid I \in \binom{[n]}{k}\}$ are well-defined up to a common scalar.  Thus we have a map
\begin{align*}
\Gr(k,n) &\longrightarrow \P^{\binom{n}{k}-1} \\
X &\longmapsto (\Delta_I(X))_{I \in \binom{[n]}{k}}
\end{align*}
mapping the Grassmannian to the projective space with homogeneous coordinates labeled by $\binom{[n]}{k}$.  This map is an injection called the \defn{Pl\"ucker embedding}, and endows $\Gr(k,n)$ with the structure of a smooth irreducible projective variety of (complex) dimension $k(n-k)$.  If $X \in \Gr(k,n)$, then the Pl\"ucker coordinates $\Delta_I(X)$ satisfy quadratic relations known as the Pl\"ucker relations.

In the following, Pl\"ucker coordinates will also be indexed by $k$-tuples $(i_1,i_2,\ldots,i_k) \in [n]^k$, with the convention that the coordinates are anti-symmetric in the indices.  So for example $\Delta_{1,3} = -\Delta_{3,1}$.  The following standard result can be found in \cite{Ful}.

\begin{proposition}
The Pl\"ucker coordinates $\Delta_I(X)$ satisfy the relations
$$
\Delta_{i_1,\ldots,i_k} \Delta_{j_1,\ldots,j_k} - \sum \Delta_{i'_1,\ldots,i'_k} \Delta_{j'_1,\ldots,j'_k} = 0,
$$
where the sum is over all pairs obtained by interchanging a fixed set of $r$ of the subscripts $j_1,\ldots,j_k$ with $r$ of the subscripts $i_1,\ldots,i_k$, maintaining the order in each.
\end{proposition}

We have the following simpler criterion to check if a point lies in $\Gr(k,n)$, which follows from \cite[Proof on page 133]{Ful}.
\begin{proposition}\label{prop:plucker}
A collection of numbers $(\Delta_I(N))_{I \in \binom{[n]}{k}})$, not all zero, defines a point in $\Gr(k,n)$ if and only if the Pl\"ucker relation with $r=1$ index swapped is satisfied:
\begin{equation}\label{eq:plucker}
\sum_{s=1}^k (-1)^s\Delta_{i_1,i_2,\ldots,i_{k-1},j_s}\Delta_{j_1,\ldots,j_{s-1},\hat j_s,j_{s+1},\ldots,j_k} = 0
\end{equation}
where $\hat j_r$ denotes omission.
\end{proposition}

The Grassmannian can be covered by affine charts.  Let $\Omega \subset \Gr(k,n)$ be the locus $\Omega \coloneqq \{X \in \Gr(k,n) \mid \Delta_{[k]}(X) \neq 0\}$.  Then every $X \in \Omega$ is uniquely represented by a $k \times n$ matrix $M$ whose columns $1,2,\ldots,k$ form the identity matrix.  For example, if $k = 3$ and $n = 7$, we have
$$
M = \begin{bmatrix}1 & 0 & 0 & m_{14} & m_{15} &m_{16} & m_{17} \\
0& 1 & 0 & m_{24} & m_{25} &m_{26} & m_{27} \\
0 & 0 & 1 & m_{34} & m_{35} &m_{36} & m_{37} \\
\end{bmatrix}.
$$
The entries $m_{ij}$ for $i \in [k]$ and $j\in[k+1,n]$ form coordinates on $\Omega$, identifying $\Omega$ with the affine space $\C^{k(n-k)}$.  If instead of placing the identity matrix in the columns $\{1,2,\ldots,k\}$ we placed it in the columns indexed by $I \in \binom{[n]}{k}$, we obtain the chart $\Omega_I$.  The collection of $\binom{n}{k}$ affine charts $\{\Omega_I \mid I \in \binom{[n]}{k}\}$ cover the Grassmannian $\Gr(k,n)$.

\begin{example}
The Pl\"ucker coordinates of the $2$-plane
$$
X = \rowspan \begin{bmatrix} 1 & 0 & a & b \\ 0 & 1 & c &d \end{bmatrix}
$$
are $\Delta_{12} = 1$, $\Delta_{13} = c$, $\Delta_{14} = d$, $\Delta_{23} = -a$, $\Delta_{24} = -b$, and $\Delta_{34} = ad-bc$.  They satisfy the one Pl\"ucker relation $\Delta_{13}\Delta_{24} = \Delta_{12}\Delta_{34} + \Delta_{14}\Delta_{23}$.
\end{example}

\begin{example}
Suppose $k = 1$.  Then $\Gr(1,n)$ is the set of one-dimensional subspaces of $\C^n$.  The Pl\"ucker embedding $\Gr(1,n) \to \P^{n-1}$ is an isomorphism.
\end{example}

The real Grassmannian $\Gr(k,n)_{\R}$ parametrizes $k$-dimensional subspaces of $\R^n$.  One can identity $\Gr(k,n)_{\R}$ with the subset of $\Gr(k,n)$ consisting of points $X$ represented by Pl\"ucker coordinates $\Delta_I(X)$ that are all real numbers.  In other words, if all the $k \times k$ minors of a full rank $k\times n$ matrix $M$ are real, then there exists $g \in \GL(k)$ so that $g \cdot M$ is a full rank $k \times n$ matrix with real entries.

\subsection{The totally nonnegative Grassmannian}\label{sec:TNNGrass}
The \defn{totally nonnegative Grassmannian} \cite{Pos}, denoted $\Gr(k,n)_{\geq 0}$, is the subset of $X \in \Gr(k,n)$ represented by Pl\"ucker coordinates $\Delta_I(X)$ that are nonnegative real numbers.  The \defn{totally positive Grassmannian} or \defn{positive Grassmannian} for short is the subset $\Gr(k,n)_{>0} \subset \Gr(k,n)$ represented by Pl\"ucker coordinates $\Delta_I(X)$ that are all positive real numbers.

There is a natural right action of $\GL(n)$ on $\Gr(k,n)$, and we have the following compatibility of totally nonnegative parts, which follows immediately from \eqref{eq:CB}.
\begin{proposition}
Suppose $g \in \GL(n)_{\geq 0}$ and $X \in \Gr(k,n)_{ \geq 0}$.  Then $X \cdot g \in \Gr(k,n)_{\geq 0}$.
\end{proposition}

For any $I \in \binom{[n]}{k}$, we have a point $e_I = \spn(e_i \mid i \in I) \in \Gr(k,n)$ with Pl\"ucker coordinates $\Delta_J(e_I) = \delta_{I,J}$ for $J \in \binom{[n]}{k}$.  By definition, the point $e_I$ lies in $\Gr(k,n)_{\geq 0}$.  The torus $(\C^*)^n \subset \GL(n)$ acts on $\Gr(k,n)$ and the points $e_I$ are exactly the torus fixed points.

\begin{theorem}\label{thm:same}
We have $\Gr(k,n)_{\geq 0} = \overline{\Gr(k,n)_{>0}} = \overline{ e_{[k]}\cdot \GL(n)_{\geq 0}}$  in the Hausdorff topology.
\end{theorem}
%

The proof of Theorem \ref{thm:same} will be given in Section \ref{sec:LusPos}.

As a Corollary, we obtain the following classical result \cite{Whi}.

\begin{corollary}
We have $\GL(n)_{\geq 0} = \overline{\GL(n)_{>0}}$ in the Hausdorff topology on $\GL(n,\R)$.
\end{corollary}
\begin{proof}
By Theorem \ref{thm:same}, $\Gr(n,2n)_{\geq 0} = \overline{\Gr(n,2n)_{>0}}$.  Since $\Omega_{[n]} \subset \Gr(n,2n)$ is open, we have $\Gr(n,2n)_{\geq 0} \cap \Omega_{[n]} = \overline{\Gr(n,2n)_{>0} \cap \Omega}$.  But $\Gr(n,2n)_{\geq 0} \cap \Omega_{[n]}$ can be identified with $\GL(n)_{\geq 0}$, by the map
$$
(I_{n \times n} \mid A) \mapsto A', \qquad \qquad a'_{i,j} = (-1)^{n-i}a_{i,n+1-j}
$$ 
where $(I_{n\times n} \mid A)$ is the $n \times 2n$ matrix representing a point in $\Gr(n,2n)_{\geq 0} \cap \Omega_{[n]}$.
\end{proof}

\begin{remark}\label{rem:same}
Lusztig \cite{Luspartial} defined the totally nonnegative part of a generalized partial flag variety $G/P$.  In the case of the Grassmannian, his definition reduces to the subset $\overline{e_{[k]} \cdot \GL(n)_{>0}}$ of the Grassmannian.  By Theorem \ref{thm:same}, his definition agrees with the one we use.
\end{remark}

Let the cyclic group $\Z/n\Z$ act on $k \times n$ matrices with generator $\chi \in \Z/n\Z$ acting by the map
$$
\chi: \left[v_1,v_2,\ldots,v_n\right] \longmapsto \left[(-1)^{k-1}v_n,v_1,v_2,\ldots,v_{n-1} \right],
$$
where $v_1,v_2,\ldots,v_n \in \C^k$ denote column vectors.  It is easy to see that this action descends to an action of the cyclic group on $\Gr(k,n)$.  A straightforward computation gives

\begin{proposition}
$X \mapsto \chi(X)$ gives an action of $\Z/n\Z$ on $\Gr(k,n)_{\geq 0}$, and on $\Gr(k,n)_{>0}$.
\end{proposition}

\section{Perfect matchings in planar bipartite graphs}\label{sec:dimer}
The aim of this section is to generalize the construction $N \mapsto M(N)$ of Section \ref{sec:LGV} to produce points in $\Gr(k,n)_{\geq 0}$.
We will use the following nonstandard convention.  A ``network" will refer to a weighted graph.  A ``graph'' will refer to an unweighted graph.  Thus a network has an underlying graph.  In addition, $G$ will denote an unweighted graph while $N$ will denote a network.

\subsection{Matchings for bipartite networks in a disk}

Let $N$ be a weighted bipartite network embedded in the disk with $n$ boundary vertices, labeled $1,2,\ldots,n$ in clockwise order.  Each vertex is colored either black or white, and all edges join black vertices to white vertices.  We assume that all boundary vertices have degree 1, and that edges cannot join boundary vertices to boundary vertices.  The color of the boundary vertices is thus determined by the color of the interior vertices, and we do not indicate the color of a boundary vertex in our figures.

We let $d$ be the number of interior white vertices minus the the number of interior black vertices.  We let $d' \in [n]$ be the number of boundary vertices incident to an interior black vertex.


An \defn{almost perfect matching} $\Pi$ is a subset of edges of $N$ such that 
\begin{enumerate}
\item
each interior vertex is used exactly once
\item
boundary vertices may or may not be used.
\end{enumerate}
The boundary subset $I(\Pi) \subset \{1,2,\ldots,n\}$ is the set of black boundary vertices that are used by $\Pi$ union the set of white boundary vertices that are not used by $\Pi$.  By our assumptions we have $|I(\Pi)| =  k \coloneqq d' + d$.

We will always assume that almost perfect matchings of $N$ exist.  Therefore, we may suppose that isolated interior vertices do not exist.

Define the \defn{boundary measurement}, or \defn{dimer partition function} as follows.  For $I \subset [n]$ a $k$-element subset,
$$
\Delta_I(N) = \sum_{\Pi \mid I(\Pi) = I} \wt(\Pi)
$$
where $\wt(\Pi)$ is the product of the weight of the edges in $\Pi$.  Our first aim is to prove that boundary measurements define a point in the Grassmannian.  The following theorem improves on a result of Kuo \cite{Kuo}.

\begin{theorem}\label{thm:matchingplucker}
Suppose $N$ has nonnegative real weights, and that almost perfect matchings of $N$ exist.  Then the homogeneous coordinates $\{\Delta_I(N) \mid I \in \binom{[n]}{k}\}$ defines a point $X(N)$ in the Grassmannian $\Gr(k,n)$.
\end{theorem}

\begin{example}
Let us consider the lollipop graph $N$ below.  Note that all boundary vertices must have degree 1, so we cannot have graphs smaller than the lollipop graphs.  Then the point $X(N) \in \Gr(k,n)$ is a torus-fixed point.  The network $N$ represents the point $e_{\{3,4\}} = \spn(e_3,e_4) \in \Gr(2,4)$.  There is a single almost perfect matching $\Pi$, consisting of all four edges.  This matching satisfies $I(\Pi)=\{3,4\}$.
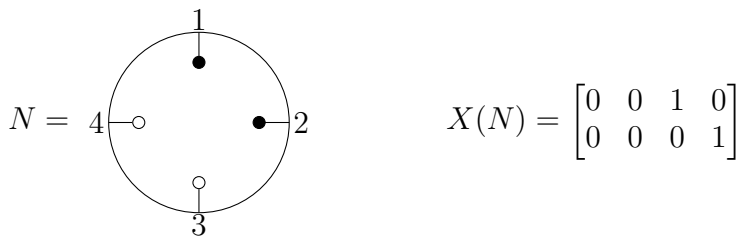
\begin{equation*}
N = \begin{tikzpicture}[baseline=-0.5ex,scale=0.8]
\node at (0,1.7) {$1$};
\node at (1.7,0) {$2$};
\node at (0,-1.7) {$3$};
\node at (-1.7,0) {$4$};
\draw (0,0) circle (1.5cm);
\draw (-1.5,0) -- (-1,0);
\draw (1.5,0) -- (1,0);
\draw (0,1) -- (0,1.5);
\draw (0,-1) -- (0,-1.5);
\filldraw[black] (1,0) circle (0.1cm);
\filldraw[black] (0,1) circle (0.1cm);
\filldraw[white] (0,-1) circle (0.1cm);
\draw (0,-1) circle (0.1cm);
\filldraw[white] (-1,0) circle (0.1cm);
\draw (-1,0) circle (0.1cm);
\end{tikzpicture}
\qquad \qquad X(N) = \begin{bmatrix} 0 & 0 & 1 & 0 \\ 0 & 0 &0 & 1\end{bmatrix}
\end{equation*}

\end{example}
\begin{example}
Let us compute the boundary measurements of the square graph for $\Gr(2,4)$.
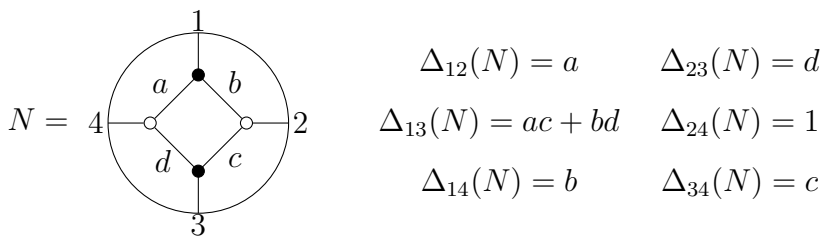
\begin{equation*}
N = \begin{tikzpicture}[baseline=-0.5ex,scale=0.8]
\node at (0,1.7) {$1$};
\node at (1.7,0) {$2$};
\node at (0,-1.7) {$3$};
\node at (-1.7,0) {$4$};
\draw (0,0) circle (1.5cm);
\draw (-1.5,0) -- (-0.8,0);
\draw (1.5,0) -- (0.8,0);
\draw (0,0.8) -- (0,1.5);
\draw (0,-0.8) -- (0,-1.5);
\draw (-0.8,0) -- node[above left = -2pt] {$a$} (0,0.8) -- node[above right = -2pt] {$b$}(0.8,0) -- node[below right = -2pt] {$c$} (0,-0.8) -- node[below left = -3pt] {$d$} (-0.8,0);
\whitedot{(0.8,0)}
\blackdot{(0,0.8)}
\blackdot{(0,-0.8)}
\whitedot{(-0.8,0)}
%
\begin{scope}[shift={(1,0)}]
\node at (4,1) {$\Delta_{12}(N) = a$};
\node at (4,0) {$\Delta_{13}(N) = ac + bd$};
\node at (4,-1) {$\Delta_{14}(N) = b$};
\node at (8,1) {$\Delta_{23}(N) = d$};
\node at (8,0) {$\Delta_{24}(N) = 1$};
\node at (8,-1) {$\Delta_{34}(N) = c$};
\end{scope}
\end{tikzpicture}
\end{equation*}

\end{example}

\subsection{Double dimers}
\label{sec:doubledimer}
To prove Theorem \ref{thm:matchingplucker}, we must show that $\Delta_I(N)$ satisfy the Pl\"ucker relations, which are some quadratic identities in $\Delta_I$.  We thus proceed to study ordered pairs of almost perfect matchings in $N$.

A \defn{$(k,n)$-partial non-crossing pairing} is a pair $(\tau, T)$ where $\tau$ is a matching of a subset $S = S(\tau) \subset \{1,2,\ldots,n\}$ of even size, such that when the vertices are arranged in order on a circle, and the edges are drawn in the interior, then the edges do not intersect; and $T$ is a subset of $[n] \setminus S$ satisfying $|S| + 2 |T| = 2k$.  Let $\AA_{k,n}$ denote the set of $(k,n)$-partial non-crossing pairings.

A subgraph $\Sigma \subset N$ is a \defn{Temperley-Lieb subgraph} if it is a union of connected components each of which is: (a) a path between boundary vertices, or (b) an interior cycle, or (c) a single edge (called a \defn{doubled edge}), such that every interior vertex is used.  The set of boundary vertices used by the paths in a Temperley-Lieb subgraph is denoted $S(\Sigma)$.  Thus each Temperley-Lieb subgraph $\Sigma$ gives a partial non-crossing pairing on $S(\Sigma) \subset \{1,2,\ldots,n\}$.

Let $(\Pi,\Pi')$ be a double-dimer (that is, a pair of dimer configurations) in $N$ (see for example \cite{KW}).  Then the union $\Sigma = \Pi \cup \Pi'$ is a Temperley-Lieb subgraph:
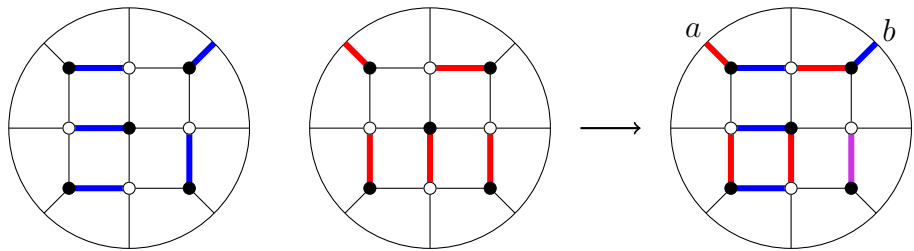
\begin{center}
\begin{tikzpicture}[scale=0.8]
\coordinate (a4) at (45:2);
\coordinate (a3) at (90:2);
\coordinate (a2) at (135:2);
\coordinate (a1) at (180:2);
\coordinate (a8) at (225:2);
\coordinate (a7) at (270:2);
\coordinate (a6) at (315:2);
\coordinate (a5) at (0:2);
\coordinate (x11) at (-1,1);
\coordinate (x12) at (0,1);
\coordinate (x13) at (1,1);
\coordinate (x21) at (-1,0);
\coordinate (x22) at (0,0);
\coordinate (x23) at (1,0);
\coordinate (x31) at (-1,-1);
\coordinate (x32) at (0,-1);
\coordinate (x33) at (1,-1);

\draw (a1) -- (x21);
\draw (a2) -- (x11);
\draw (a3) -- (x12);
\draw (a4) -- (x13);
\draw (a5) -- (x23);
\draw (a6) -- (x33);
\draw (a7) -- (x32);
\draw (a8) -- (x31);
\draw (x11) -- (x13);
\draw (x21) -- (x23);
\draw (x31) -- (x33);
\draw (x11) -- (x31);
\draw (x12) -- (x32);
\draw (x13) -- (x33);
\draw (0,0) circle (2cm);
\draw[blue,line width = 0.08cm] (x11) -- (x12);
\draw[blue,line width = 0.08cm] (x13) -- (a4);
\draw[blue,line width = 0.08cm] (x21) -- (x22);
\draw[blue,line width = 0.08cm] (x31) -- (x32);
\draw[blue,line width = 0.08cm] (x23) -- (x33);

\filldraw[black] (x11) circle (0.1cm);
\filldraw[black] (x13) circle (0.1cm);
\filldraw[black] (x22) circle (0.1cm);
\filldraw[black] (x31) circle (0.1cm);
\filldraw[black] (x33) circle (0.1cm);

\filldraw[white] (x12) circle (0.1cm);
\draw (x12) circle (0.1cm);
\filldraw[white] (x21) circle (0.1cm);
\draw (x21) circle (0.1cm);
\filldraw[white] (x23) circle (0.1cm);
\draw (x23) circle (0.1cm);
\filldraw[white] (x32) circle (0.1cm);
\draw (x32) circle (0.1cm);

\begin{scope}[{shift={(5,0)}}]
\coordinate (a4) at (45:2);
\coordinate (a3) at (90:2);
\coordinate (a2) at (135:2);
\coordinate (a1) at (180:2);
\coordinate (a8) at (225:2);
\coordinate (a7) at (270:2);
\coordinate (a6) at (315:2);
\coordinate (a5) at (0:2);
\coordinate (x11) at (-1,1);
\coordinate (x12) at (0,1);
\coordinate (x13) at (1,1);
\coordinate (x21) at (-1,0);
\coordinate (x22) at (0,0);
\coordinate (x23) at (1,0);
\coordinate (x31) at (-1,-1);
\coordinate (x32) at (0,-1);
\coordinate (x33) at (1,-1);
\draw (a1) -- (x21);
\draw (a2) -- (x11);
\draw (a3) -- (x12);
\draw (a4) -- (x13);
\draw (a5) -- (x23);
\draw (a6) -- (x33);
\draw (a7) -- (x32);
\draw (a8) -- (x31);
\draw (x11) -- (x13);
\draw (x21) -- (x23);
\draw (x31) -- (x33);
\draw (x11) -- (x31);
\draw (x12) -- (x32);
\draw (x13) -- (x33);
\draw (0,0) circle (2cm);
\draw[red,line width = 0.08cm] (a2) -- (x11);
\draw[red,line width = 0.08cm] (x13) -- (x12);
\draw[red,line width = 0.08cm] (x21) -- (x31);
\draw[red,line width = 0.08cm] (x22) -- (x32);
\draw[red,line width = 0.08cm] (x23) -- (x33);

\filldraw[black] (x11) circle (0.1cm);
\filldraw[black] (x13) circle (0.1cm);
\filldraw[black] (x22) circle (0.1cm);
\filldraw[black] (x31) circle (0.1cm);
\filldraw[black] (x33) circle (0.1cm);

\filldraw[white] (x12) circle (0.1cm);
\draw (x12) circle (0.1cm);
\filldraw[white] (x21) circle (0.1cm);
\draw (x21) circle (0.1cm);
\filldraw[white] (x23) circle (0.1cm);
\draw (x23) circle (0.1cm);
\filldraw[white] (x32) circle (0.1cm);
\draw (x32) circle (0.1cm);
\end{scope}

\draw[->,thick] (7.5,0) -- (8.5,0);

\begin{scope}[{shift={(11,0)}}]
\coordinate (a4) at (45:2);
\coordinate (a3) at (90:2);
\coordinate (a2) at (135:2);
\coordinate (a1) at (180:2);
\coordinate (a8) at (225:2);
\coordinate (a7) at (270:2);
\coordinate (a6) at (315:2);
\coordinate (a5) at (0:2);
\coordinate (x11) at (-1,1);
\coordinate (x12) at (0,1);
\coordinate (x13) at (1,1);
\coordinate (x21) at (-1,0);
\coordinate (x22) at (0,0);
\coordinate (x23) at (1,0);
\coordinate (x31) at (-1,-1);
\coordinate (x32) at (0,-1);
\coordinate (x33) at (1,-1);

\node at (135:2.3) {$a$};
\node at (45:2.3) {$b$};

\draw (a1) -- (x21);
\draw (a2) -- (x11);
\draw (a3) -- (x12);
\draw (a4) -- (x13);
\draw (a5) -- (x23);
\draw (a6) -- (x33);
\draw (a7) -- (x32);
\draw (a8) -- (x31);
\draw (x11) -- (x13);
\draw (x21) -- (x23);
\draw (x31) -- (x33);
\draw (x11) -- (x31);
\draw (x12) -- (x32);
\draw (x13) -- (x33);
\draw (0,0) circle (2cm);
\draw[red,line width = 0.08cm] (a2) -- (x11);
\draw[red,line width = 0.08cm] (x13) -- (x12);
\draw[red,line width = 0.08cm] (x21) -- (x31);
\draw[red,line width = 0.08cm] (x22) -- (x32);
\definecolor{mycolor}{rgb}{0.8,0.2,0.9}
\draw[mycolor,line width = 0.08cm] (x23) -- (x33);
\draw[blue,line width = 0.08cm] (x11) -- (x12);
\draw[blue,line width = 0.08cm] (x13) -- (a4);
\draw[blue,line width = 0.08cm] (x21) -- (x22);
\draw[blue,line width = 0.08cm] (x31) -- (x32);

\filldraw[black] (x11) circle (0.1cm);
\filldraw[black] (x13) circle (0.1cm);
\filldraw[black] (x22) circle (0.1cm);
\filldraw[black] (x31) circle (0.1cm);
\filldraw[black] (x33) circle (0.1cm);

\filldraw[white] (x12) circle (0.1cm);
\draw (x12) circle (0.1cm);
\filldraw[white] (x21) circle (0.1cm);
\draw (x21) circle (0.1cm);
\filldraw[white] (x23) circle (0.1cm);
\draw (x23) circle (0.1cm);
\filldraw[white] (x32) circle (0.1cm);
\draw (x32) circle (0.1cm);
\end{scope}
\end{tikzpicture}
\end{center}
When $\Sigma$ arises from a double-dimer, the set $S(\Sigma)$ is given by $S =(I(\Pi) \setminus I(\Pi')) \cup (I(\Pi') \setminus I(\Pi))$, and we obtain a non-crossing pairing on $S$.  For example, in the above picture we have that $a$ is paired with $b$ and $S = \{a,b\}$.  Note that a Temperley-Lieb subgraph $\Sigma$ can arise from a double-dimer $(\Pi,\Pi')$ in many different ways: it does not remember which edge in a path came from which of the two original dimer configurations.

For each $(k,n)$-partial non-crossing pairing $(\tau,T) \in \AA_{k,n}$, define the \defn{Temperley-Lieb immanant}
$$
F_{\tau,T}(N) \coloneqq \sum_{\Sigma} \wt(\Sigma)
$$
to be the sum over Temperley-Lieb subgraphs $\Sigma$ which give boundary path pairing $\tau$, and $T$ contains black boundary vertices belonging to a doubled-edge in $\Sigma$, together with white boundary vertices not belonging to a doubled-edge in $\Sigma$.  Here $\wt(\Sigma)$ is the product of all weights of edges in $\Sigma$ times $2^{\# {\rm cycles}}$; also, the weight of a doubled-edge in $\Sigma$ is the square of the weight of that edge.  The function $F_{\tau,T}$, introduced in \cite{Lamweb}, is a Grassmann-analogue of Rhoades and Skandera's Temperley-Lieb immanants \cite{RS}.  

Given $I, J \in \binom{[n]}{k}$, we say that a $(k,n)$-partial non-crossing pairing $(\tau,T)$ is compatible with $I,J$ if:
\begin{enumerate}
\item $S(\tau) = (I \setminus J) \cup (J \setminus I)$, and each edge of $\tau$ matches a vertex in $(I \setminus J)$ with a vertex in $(J \setminus I)$, and
\item 
$T = I \cap J$.
\end{enumerate}

\begin{theorem}[\cite{Lamweb}]\label{thm:TL}
For $I, J \in \binom{[n]}{k}$, we have
$$
\Delta_I(N) \Delta_J(N) = \sum_{\tau,T} F_{\tau,T}(N)
$$
where the summation is over all $(k,n)$-partial non-crossing pairings $\tau$ compatible with $I, J$.
\end{theorem}
\begin{proof}
The only thing left to prove is the compatibility property.

Let $\Pi, \Pi'$ be almost perfect matchings of $N$ such that $I(\Pi) = I$ and $I(\Pi') = J$.  Let $p$ be one of the boundary paths in $\Pi \cup \Pi'$, with endpoints $s$ and $t$.  If $s$ and $t$ have the same color, then the path is even in length.  If $s$ and $t$ have different colors, then the path is odd in length.  In both cases one of $s$ and $t$ belongs to $I \setminus J$ and the other belongs to $J \setminus I$.  
\end{proof}

\begin{example}
Suppose $n = 6$.  Then $\Delta_{124}(N) \Delta_{356}(N) = F_{\tau_1,\emptyset}+F_{\tau_2,\emptyset}$, where $\tau_1$ and $\tau_2$ are the following non-crossing matchings:
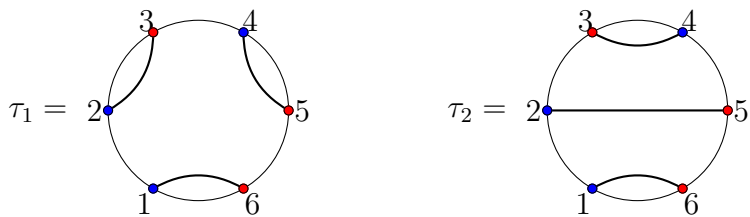
\begin{equation*}
\tau_1 = 
\begin{tikzpicture}[scale=0.6,baseline=-0.5ex]
\coordinate (a4) at (60:2);
\coordinate (a3) at (120:2);
\coordinate (a2) at (180:2);
\coordinate (a1) at (-120:2);
\coordinate (a6) at (-60:2);
\coordinate (a5) at (0:2);

\node at (60:2.3) {$4$};
\node at (120:2.3) {$3$};
\node at (180:2.3) {$2$};
\node at (240:2.4) {$1$};
\node at (-60:2.4) {$6$};
\node at (0:2.3) {$5$};

\draw (0,0) circle (2);

\draw[thick] (a2) to [bend right] (a3);
\draw[thick] (a4) to [bend right] (a5);
\draw[thick] (a6) to [bend right] (a1);

\filldraw[blue] (a1) circle (0.1cm);
\draw (a1) circle (0.1cm);
\filldraw[blue] (a2) circle (0.1cm);
\draw (a2) circle (0.1cm);
\filldraw[blue] (a4) circle (0.1cm);
\draw (a4) circle (0.1cm);
\filldraw[red] (a3) circle (0.1cm);
\draw (a3) circle (0.1cm);
\filldraw[red] (a5) circle (0.1cm);
\draw (a5) circle (0.1cm);
\filldraw[red] (a6) circle (0.1cm);
\draw (a6) circle (0.1cm);
\end{tikzpicture}
\qquad \qquad 
\tau_2 = 
\begin{tikzpicture}[scale=0.6,baseline=-0.5ex]
\coordinate (a4) at (60:2);
\coordinate (a3) at (120:2);
\coordinate (a2) at (180:2);
\coordinate (a1) at (-120:2);
\coordinate (a6) at (-60:2);
\coordinate (a5) at (0:2);

\node at (60:2.3) {$4$};
\node at (120:2.3) {$3$};
\node at (180:2.3) {$2$};
\node at (240:2.4) {$1$};
\node at (-60:2.4) {$6$};
\node at (0:2.3) {$5$};

\draw (0,0) circle (2);

\draw[thick] (a2) to (a5);
\draw[thick] (a3) to [bend right] (a4);
\draw[thick] (a6) to [bend right] (a1);

\filldraw[blue] (a1) circle (0.1cm);
\draw (a1) circle (0.1cm);
\filldraw[blue] (a2) circle (0.1cm);
\draw (a2) circle (0.1cm);
\filldraw[blue] (a4) circle (0.1cm);
\draw (a4) circle (0.1cm);
\filldraw[red] (a3) circle (0.1cm);
\draw (a3) circle (0.1cm);
\filldraw[red] (a5) circle (0.1cm);
\draw (a5) circle (0.1cm);
\filldraw[red] (a6) circle (0.1cm);
\draw (a6) circle (0.1cm);
\end{tikzpicture}
\end{equation*}
\end{example}

\subsection{Proof of Theorem \ref{thm:matchingplucker}}
We shall use Proposition \ref{prop:plucker}.
%

Use Theorem \ref{thm:TL} to expand \eqref{eq:plucker} with $\Delta_I =\Delta_I(N)$ as a sum of $F_{\tau,T}(N)$ over pairs $(\tau,T)$ (with multiplicity).  We note that the set $T$ is always the same in any term that comes up.  We assume that $i_1 < i_2 < \cdots < i_{k-1}$ and $j_1 < j_2 < \cdots < j_{k+1}$.

So each term $F_{\tau,T}$ is labeled by $(I,J,\tau)$ where $I,J$ is compatible with $\tau$, and $I,J$ occur as a term in \eqref{eq:plucker}.  We provide an involution on such terms.  By the compatibility condition, all but one of the edges in $\tau$ uses a vertex in $\{i_1,i_2,\ldots,i_{k-1}\}$.  The last edge is of the form $(j_a,j_b)$, where $j_a \in I$ and $j_b \in J$.  The involution swaps $j_a$ and $j_b$ in $I, J$ but keeps $\tau$ the same.

Finally we show that this involution is sign-reversing.  Let $I' = I \cup \{j_b\} - \{j_a\}$ and $J' = J \cup \{j_a\} - \{j_b\}$.  Then the sign associated to the term labeled by $(I,J,\tau)$ is equal to $(-1)$ to the power of $\#\{r \in [k] \mid i_r > j_a\} + a$.   Note that by the non-crossingness of the edges in $\tau$ there must be an even number of vertices belonging to $(I \setminus J) \cup (J \setminus I)$ strictly between $j_a$ and $j_b$.  Thus $j_b - j_a = (b-a) + (\#\{r \in [k] \mid i_r > j_b\}-\#\{r \in [k] \mid i_r > j_a\}) \mod 2$ is odd.  So the involution changes the sign.  This completes the proof of Theorem \ref{thm:matchingplucker}.

\subsection{Gauge equivalence}
Let $N$ be a planar bipartite network.  If $e_1,e_2,\ldots,e_d$ are adjacent to an \defn{interior} vertex $v$, we can multiply all of their edge weights by the same constant $c \in \R_{> 0}$, and still get the same point $X(N)$.  Note that we cannot do this at a boundary vertex.

Let $F$ be any face of the network $N$.  This can be a face completely bounded by edges of $N$, or a face that also touches the boundary of the disk.  Take the clockwise orientation of the edges bounding the face, and define the face weight
\begin{equation}\label{eq:yF}
y_F \coloneqq \prod_{e \text{ bounding } F} \wt(e)^{\pm 1}
\end{equation}
where we have $+1$ if the edge goes out of a black vertex and into a white vertex, and $-1$ if the edge goes out of a white vertex and into a black vertex.

\begin{lemma}
Face weights are preserved by gauge equivalence.
\end{lemma}

Here is some more abstract language to formulate the above. A \defn{line bundle} $V = V_G$ on a graph $G$ is the association of a one-dimensional vector space $V_v$ to each vertex $v$ of $G$.  A \defn{connection} $\Phi$ on $V$ is a collection of invertible linear maps $\phi_{uv}: V_u \to V_v$ for each edges $u,v$ satisfying $\phi_{uv} = \phi_{vu}^{-1}$.  If we fix a basis of each $V_v$, then the connection $\Phi$ is equivalent to giving $G$ a weighting, that is, it is equivalent to a network $N$ with underlying graph $G$.  Two connections $\Phi$ and $\Phi'$ are isomorphic if they are related by change of basis at each $V_v$.

\begin{lemma}
Gauge equivalence for $N$ corresponds to changing bases for $\{V_v\}$.  Isomorphism classes of connections on $V$ are in bijection with gauge equivalence classes of planar bipartite networks $N$ with underlying graph $G$.  Isomorphism classes of connections are in bijection with face weights $\{y_F \in \R_{>0}\}$, which can be chosen arbitrarily subject to the condition that $\prod_F y_F = 1$.
\end{lemma}
\begin{proof}
Only the last statement is not clear, and it basically follows from Euler's formula.
\end{proof}

Let $\L_G$ be the moduli space of connections on $V_G$ (that is, the space of isomorphism classes of connections), and let $(\L_G)_{>0}$ be the positive points so that $(\L_G)_{>0} \simeq \R^{\#F - 1}_{>0}$ can be identified with the space of positive real weighted networks $N$ with underlying graph $G$, modulo gauge equivalence.  Here $\#F$ denotes the number of faces of $G$.

\subsection{Relations for bipartite graphs}
\label{sec:relations}
We have the following local moves, replacing a small local part of $N$ by another specific network to obtain $N'$:
\begin{enumerate}
\item[(M1)]
Spider move \cite{GK}, square move \cite{Pos}, or urban renewal \cite{Propp}: assuming the leaf edges of the spider have been gauge fixed to 1, the transformation is
\begin{equation}\label{eq:square}
a'=\frac{a}{ac+bd} \qquad b'=\frac{b}{ac+bd} \qquad c'=\frac{c}{ac+bd} \qquad d'=\frac{d}{ac+bd}
\end{equation}
\begin{center}
\begin{tikzpicture}[scale=0.7]
\draw (-2,0) -- (0,1)--(2,0)--(0,-1)-- (-2,0);
\draw (0,1) -- (0,2);
\draw (0,-1) -- (0,-2);
\node at (-1.2,0.7) {$a$};
\node at (-1.2,-0.7) {$d$};
\node at (1.2,0.7) {$b$};
\node at (1.2,-0.7) {$c$};

\filldraw[black] (0,1) circle (0.1cm);
\filldraw[black] (0,-1) circle (0.1cm);
\filldraw[white] (-2,0) circle (0.1cm);
\draw (-2,0) circle (0.1cm);
\filldraw[white] (0,2) circle (0.1cm);
\draw (0,2) circle (0.1cm);
\filldraw[white] (2,0) circle (0.1cm);
\draw (2,0) circle (0.1cm);
\filldraw[white] (0,-2) circle (0.1cm);
\draw (0,-2) circle (0.1cm);
\end{tikzpicture}
\hspace{30pt}
\begin{tikzpicture}[scale=0.7]
\draw (0,-2) -- (1,0)-- (0,2)-- (-1,0)-- (0,-2);
\draw (1,0) -- (2,0);
\draw (-1,0) -- (-2,0);
\node at (0.7,-1.2) {$a'$};
\node at (-0.7,-1.2) {$b'$};
\node at (0.7,1.2) {$d'$};
\node at (-0.7,1.2) {$c'$};

\filldraw[black] (1,0) circle (0.1cm);
\filldraw[black] (-1,0) circle (0.1cm);
\filldraw[white] (-2,0) circle (0.1cm);
\draw (-2,0) circle (0.1cm);
\filldraw[white] (0,2) circle (0.1cm);
\draw (0,2) circle (0.1cm);
\filldraw[white] (2,0) circle (0.1cm);
\draw (2,0) circle (0.1cm);
\filldraw[white] (0,-2) circle (0.1cm);
\draw (0,-2) circle (0.1cm);
\end{tikzpicture}
\end{center}
\item[(M2)]
Valent two vertex removal.  If $v$ has degree two, we can gauge fix both incident edges $(v,u)$ and $(v,u')$ to have weight 1, then contract both edges (that is, we remove both edges, and identify $u$ with $u'$).  Note that if $v$ is a valent two-vertex adjacent to boundary vertex $b$, with edges $(v,b)$ and $(v,u)$, then removing $v$ produces an edge $(b,u)$, and the color of $b$ flips.
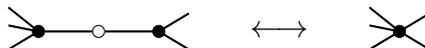
\begin{center}
\begin{tikzpicture}[scale=0.8]
\draw[thick] (0,0) -- (-0.5,0.1);
\draw[thick] (0,0) -- (-0.5,-0.3);
\draw[thick] (0,0) -- (-0.5,0.4);
\draw[thick] (2,0) -- (2.5,0.3);
\draw[thick] (2,0) -- (2.5,-0.3);
\draw[thick] (0,0) -- (2,0);
\blackdot{(0,0)}
\whitedot{(1,0)}
\blackdot{(2,0)}

\node at (4,0) {$\longleftrightarrow$};

\begin{scope}[shift={(1,0)}]
\blackdot{(5,0)}
\draw[thick] (5,0) -- (4.5,0.1);
\draw[thick] (5,0) -- (4.5,-0.3);
\draw[thick] (5,0) -- (4.5,0.4);
\draw[thick] (5,0) -- (5.5,0.3);
\draw[thick] (5,0) -- (5.5,-0.3);
\end{scope}

\end{tikzpicture}
\end{center}

\item[(R1)]
Multiple edges with the same endpoints can be reduced to a single edge with the sum of original weights.
\item[(R2)]
Leaf removal.  Suppose $v$ is leaf, and $(v,u)$ the unique edge incident to it.  Then we can remove both $v$ and $u$, and all edges incident to $u$.  However, if there is a boundary edge $(b,u)$ where $b$ is a boundary vertex, then that edge is replaced by a boundary edge $(b,w)$ where $w$ is a new vertex with the same color as $v$.

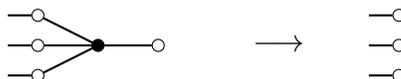
\begin{center}
\begin{tikzpicture}[scale=0.8]
\draw[thick] (-1.5,0)--(-1,0);
\draw[thick] (0,0) -- (-1,0);
\draw[thick] (-1.5,0.5)--(-1,0.5);
\draw[thick] (0,0) -- (-1,-0.5);
\draw[thick] (-1.5,-0.5)--(-1,-0.5);
\draw[thick] (0,0) -- (-1,0.5);
\draw[thick] (0,0) -- (1,0);
\blackdot{(0,0)}
\whitedot{(1,0)}
\whitedot{(-1,0)}
\whitedot{(-1,0.5)}
\whitedot{(-1,-0.5)}

\node at (3,0) {$\longrightarrow$};

\draw[thick] (5,0) -- (4.5,0);
\draw[thick] (5,0.5) -- (4.5,0.5);
\draw[thick] (5,-0.5) -- (4.5,-0.5);
\whitedot{(5,0)}
\whitedot{(5,0.5)}
\whitedot{(5,-0.5)}

\end{tikzpicture}
\end{center}

\item[(R3)]
Dipoles (two degree one vertices joined by an edge) can be removed.
\end{enumerate}

The following results are checked case-by-case.
\begin{proposition}
Each of the moves (M1) and (M2), and each of the reductions (R1), (R2), (R3) preserve $X(N)$.
\end{proposition}

\begin{proposition}\label{prop:GG'}
Suppose $G$ and $G'$ are related by (M1) and (M2).  Then the moves induce a homeomorphism $(\L_G)_{>0} \simeq (\L_{G'})_{>0}$.
\end{proposition}



%

\section{Plabic graphs}\label{sec:flows}
So far we have only discussed planar bipartite graphs.  Postnikov \cite{Pos} gives a more general theory in the setting of ``plabic graphs".  Here we will not introduce Postnikov's original notion of boundary measurement, but work with the setting of flows in perfectly oriented networks, as studied in \cite{Tal, PSW}.

A \defn{bicolored graph} is a finite undirected graph $G$ with $n$ distinguished vertices labeled $1,2,\ldots,n$ called \defn{boundary vertices}.  The non-boundary vertices are called \defn{interior vertices} and each interior vertex is colored either black or white.  Each boundary vertex has degree one and is not colored.  We allow both loops and multiple edges. 

A \defn{perfect orientation} $O$ of a bicolored graph $G$ is a choice of direction for each edge of the graph $G$ such that interior black vertices have outdegree $1$ (and any indegree) and interior white vertices have indegree $1$ (and any outdegree).  If $(G,O)$ is perfectly oriented with $n$ boundary edges, then the number of boundary sources $k$ is given by the formula (see \cite[Definition 11.5]{Pos})
\begin{equation}\label{eq:kO}
k \coloneqq \frac{1}{2} \left(n + \sum_{v \text{ black}} (\deg(v)-2) + \sum_{v \text{ white}} (2 -\deg(v)) \right).
\end{equation}

If a bicolored graph $G$ is embedded into a disk so that the boundary vertices are arranged in order on the boundary of the disk then we call $G$ a \defn{plabic graph} \cite{Pos}.  A \defn{plabic network} $N$ is a plabic graph where each edge has been given a positive real edge weight.

\begin{proposition}\label{prop:O}
Suppose $G$ is a planar bipartite graph.  Then there is a natural bijection between perfect orientations $O$ of $G$ and almost perfect matchings $\Pi$ of $G$.  In particular, a planar bipartite graph $G$ has an almost perfect matching if and only if it has a perfect orientation.
\end{proposition}
\begin{proof}
Let $\Pi$ be an almost perfect matching.  We construct a perfect orientation $O$ of $G$ as follows.  Suppose $e \notin \Pi$.  Then we orient the edge $e$ from white to black.  Suppose $e \in \Pi$.  Then we orient the edge $e$ from black to white.  It is straightforward to see that this is a bijection.
\end{proof}

A \defn{flow} $\FF$ in a perfectly oriented plabic graph $(G,O)$ is a subset of the edges of $G$, such that at each interior vertex the number of incoming edges in $\FF$ equals the number of outgoing edges in $\FF$.  If $(G,O)$ is perfectly oriented, it follows immediately from the definition that a flow $\FF$ is a union of oriented cycles and oriented paths between boundary vertices.  
The weight $\wt(\FF)$ of a flow $\FF$ is the product of the weights of the set of edges belonging to $\FF$.
Define the boundary subset $I(\FF) \in \binom{[n]}{k}$ by
$$
I(\FF)\coloneqq \{\text{ boundary sources not used }\} \cup \{\text{ boundary sinks used }\}.
$$
The weight of a flow is the product of the edge weights used in the flow.
Define the boundary measurements of $(N,O)$ to be
$$
\Delta_I(N,O) = \sum_{\FF \mid I(\FF) = I} \wt(\FF).
$$
The following result is the oriented analogue of Theorem \ref{thm:matchingplucker}.  It can be proved in a similar manner.

\begin{theorem}\label{thm:flow}
Suppose $(N,O)$ is a perfectly oriented planar bicolored network with positive edge weights.  Then $\{\Delta_I(N,O) \mid I \in \binom{[n]}{k}\}$ define a point $X(N,O)$ in the totally nonnegative Grassmannian $\Gr(k,n)_{\geq 0}$.
\end{theorem}

\begin{proposition}
Suppose $N$ is a planar bipartite network, and $O$ is a perfect orientation of $N$.  Define $(\tilde N,O)$ to be the oriented network where the edge weights on black to white edges of $O$ have been inverted.  Then
$$
X(N) = X(\tilde N, O).
$$
\end{proposition}
\begin{proof}
Let $\FF$ be a flow in $(\tilde N, O)$.  Reversing the all the edges of $\FF$ gives another perfect orientation $O$ of $N$.  By Proposition \ref{prop:O}, we obtain a bijection $\FF \mapsto \Pi$ between flows of $(\tilde N, O)$ and almost perfect matchings of $N$.  We then calculate that $\wt(\FF) = \wt(\Pi)/\wt(\Pi_O)$, where $\Pi_O$ is the almost perfect matching associated to the chosen perfect orientation $O$.
\end{proof}

The boundary measurements of a planar bipartite graph are invariant under the relations discussed in Section \ref{sec:relations}.  For perfectly oriented plabic networks we have the further relation:

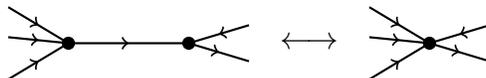
\begin{center}
\begin{tikzpicture}[scale=0.8]

\begin{scope}[decoration={
    markings,
    mark=at position 0.5 with {\arrow{>}}}
    ] 
\draw[thick,postaction={decorate}] (-1,0.1) -- (0,0);
\draw[thick,postaction={decorate}] (-1,-0.6)--(0,0);
\draw[thick,postaction={decorate}]  (-1,0.6)--(0,0);
\draw[thick,postaction={decorate}]  (3,0.3) -- (2,0);
\draw[thick,postaction={decorate}] (2,0) -- (3,-0.3);
\draw[thick,postaction={decorate}] (0,0) -- (2,0);

\blackdot{(0,0)}
\blackdot{(2,0)}

\node at (4,0) {$\longleftrightarrow$};

\begin{scope}[shift={(1,0)}]
\blackdot{(5,0)}
\draw[thick,postaction={decorate}] (4,0.1)--(5,0);
\draw[thick,postaction={decorate}] (4,-0.6)--(5,0);
\draw[thick,postaction={decorate}] (4,0.6)--(5,0);
\draw[thick,postaction={decorate}] (6,0.3)--(5,0);
\draw[thick,postaction={decorate}] (5,0) -- (6,-0.3);
\end{scope}
\end{scope}
\end{tikzpicture}
\end{center}
allowing us to merge or unmerge adjacent vertices of the same color, when the edge weight of the connecting edge is equal to 1.

\section{Bounded affine permutations}\label{sec:perms}
For more details on the material of this section, we refer the reader to \cite{Pos,KLS}.

\subsection{Affine permutations}
Fix $n \geq 2$.  An \defn{affine permutation} is a bijection $f: \Z \to \Z$ satisfying the periodicity condition $f(i+n) = f(i) + n$ for all $i \in \Z$.  Affine permutations form a group under composition denoted $\tS_n$.  The quantity $\sum_{i=1}^n (f(i)-i)$ is always divisible by $n$, and we let $\tS_n^k$ denote the subset of $\tS_n$ satisfying the condition
$$\sum_{i=1}^n (f(i)-i) = kn.$$
We call $f \in \tS_n^k$ a $(k,n)$-affine permutation.  We will give an affine permutation by giving it in \emph{window notation}: $[f(1),f(2),\ldots,f(n)]$.

The subset $\tS^0_n$ is the Coxeter group $W_n$ of affine type $A$, with generators $s_0,s_1,\ldots,s_{n-1}$, and relations 
\begin{align*}
s_i^2 &= 1 \\
s_i s_j &= s_j s_i & \mbox{if $|i-j| >1$}\\
s_i s_{i+1} s_i &= s_{i+1} s_i s_{i+1} 
\end{align*}
where all indices are taken modulo $n$.  The {\it length} $\ell(w)$ of $w \in W_n$ is the length of the shortest expression of $w$ as a product of the $s_i$.  We let $\leq$ denote the Bruhat partial order in $W_n$.

The group $W_n$ acts on the set of $(k,n)$-affine permutations by both left and right multiplications.  If $g = fs_i$, then $g$ is obtained from $f$ by swapping $f(i+rn)$ and $f(i+rn+1)$ for all $r \in \Z$; that is, right multiplication by $s_i$ swaps positions $i$ and $i+1$.  Similarly, if $g = s_if$, then $g$ is obtained from $f$ by swapping the values $i+rn$ and $i+rn+1$ for all $r \in \Z$.

For each $k$ there is a distinguished $(k,n)$-affine permutation $\id$ given by $\id(i) = i+k$ for all $i \in \Z$.  Every $f \in \tS_n^k$ is of the form $f = w \cdot \id$ for a unique $w \in W_n$.  The length of $f \in \tS_n^k$ is then defined to be the length of $w$, and if $f = w  \cdot\id$ and $g = v  \cdot\id$, we define $f \leq g$ if and only if $w \leq v$.  The poset $\tS_n^k$ has $\id$ as its unique minimal element, which has length 0.  Note that these definitions can also be made (with the same result) using right multiplication by $W_n$.  The length $\ell(f)$ of an affine permutation can also be computed as the cardinality of the set of inversions:
$$
\ell(f) = |\{(i,j) \in [n] \times \Z \mid i < j \text{ and } f(i) > f(j)\}|.
$$

\subsection{Bounded affine permutations}

A \defn{$(k,n)$-bounded affine permutation} is a $(k,n)$-affine permutation $f \in \tS_n^k$ satisfying the additional bounded condition:
$$i \leq f(i) \leq i+n.$$
The set $\Bound(k,n)$ of $(k,n)$-bounded affine permutations forms a lower order ideal in $\tS_n^k$ (\cite{KLS}).  We define the partial order on $\Bound(k,n)$ to be the {\bf dual of the induced order} from $\tS_n^k$.  Thus $f \leq g$ in $\Bound(k,n)$ if and only if $g$ is less than $f$ in Bruhat order.  Unless otherwise specified, we always use this partial order when referring to bounded affine permutations.

\subsection{Grassmann necklaces}
Let $I =\{i_1<i_2<\cdots<i_k\}$ and $J =\{j_1<j_2<\cdots <j_k\}$ be two $k$-element subsets of $[n]$.  We define a partial order  $\leq$ on $\binom{[n]}{k}$ by $I \leq J$ if $i_r \leq j_r$ for all $r=1,2,\ldots,k$.

We write $\leq_a$ for the cyclically rotated ordering $a < a+1 < \cdots < n < 1 < \cdots < a-1$ on $[n]$.  Replacing $\leq$ by $\leq_a$, we also have the cyclically rotated version partial order $I \leq_a J$ on $\binom{[n]}{k}$.

A \defn{$(k,n)$-Grassmann necklace} \cite{Pos} is a collection of $k$-element subsets $\I = (I_1,I_2,\ldots,I_n)$ satisfying the following property: for each $a \in [n]$:
\begin{enumerate}
\item
$I_{a+1} = I_a$ if $a \notin I_a$
\item
$I_{a+1} = I_a -\{a\} \cup \{a'\}$ if $a \in I_a$.
\end{enumerate}
There is a partial order on the set of $(k,n)$-Grassmann necklaces, given by $\I \leq \J$ if $I_a \leq_a J_a$ for all $a =1,2,\ldots,n$.

Given $f \in \Bound(k,n)$, we define a sequence $\I(f) = (I_1,I_2,\ldots,I_n)$ of $k$-element subsets by the formula
$$
I_a = \{f(b) \mid b < a \text{ and } f(b) \geq a\} \mod n
$$
where $\mod n$ means that we take representatives in $[n]$.

\begin{example}
Let $k =2$ and $n = 6$.  Suppose $f = [2,4,6,5,7,9]$.  Then $\I(f) = (13,23,34,46,56,16)$.
\end{example}

\begin{theorem}\label{thm:Grassorder}
The map $f \mapsto \I(f)$ is a bijection between $(k,n)$-bounded affine permutations and $(k,n)$-Grassmann necklaces.  We have $f \geq f'$ in $\Bound(k,n)$ if and only if $\I(f) \leq \I(f')$.
\end{theorem}

The inverse map $\I \mapsto f(\I)$ is given as follows.  Suppose $a \notin I_a$.  Then define $f(a) = a$.  Suppose $a \in I_a$ and $I_{a+1} = I_a -\{a\} \cup \{a'\}$.  Then define $f(a) = b$ where $b \equiv a' \mod n$ and $a < b \leq a+n$.  We leave it to the reader to check that this is inverse to the map $f \mapsto \I$, and proves the ``bijection" statement of Theorem \ref{thm:Grassorder}.  The comparison of partial orders is best understood via rank matrices.

\subsection{Cyclic rank matrices}
A formal characterization of cyclic rank matrices is discussed in \cite{KLS}, see also \cite{Pos}.  Here we only consider cyclic rank matrices of points $X \in \Gr(k,n)$.  Let $v_1,v_2,\ldots,v_n \in \C^k$ be the $n$ columns of a $k\times n$ matrix representing $X$.  Set $v_{i+n} \coloneqq (-1)^{k-1} v_{i}$ to define $v_i$ for $i \in \Z$.  The \defn{cyclic rank matrix} of $X$ is the function
$$
r_{X}(i,j) \coloneqq \dim \spn (v_{i},v_{i+1},\ldots,v_j) \in \{0,1,\ldots,k\}
$$
defined for $i \leq j$. 

We also define the bounded affine permutation $f_X$ by 
\begin{equation}\label{eq:fX}
f_X(i) \coloneqq \min \{j \geq i \mid v_i \in \spn(v_{i+1},\ldots,v_j)\}.
\end{equation}
Thus $f_X(i) = i$ if $v_i = 0$, and $f_X(i) = i+n$ if $v_i$ does not lie in the span of the other $n-1$ columns.  It is clear that $f_X$ is bounded and periodic; the fact that it is a bijection from $\Z$ to $\Z$ is left as an exercise.

Let us also define the Grassmann necklace $\I_X = (I_1,I_2,\ldots,I_n)$ by
$$
I_a \coloneqq \min_{\leq_a} \left\{J \in \binom{[n]}{k} \mid \Delta_J(X) \neq 0 \right\}
$$
where $\min_{\leq_a}$ is the lexicographical minimum with respect to the partial order $\leq_a$.

\begin{proposition}\label{prop:three}
Let $X \in \Gr(k,n)$.  Then $f_X \in \Bound(k,n)$ and $\I_X$ is a $(k,n)$-Grassmann necklace, related by the bijection of Theorem \ref{thm:Grassorder}.  Furthermore, any one of $f_X, \I_X$, and $r_X$ determine the other two.
\end{proposition}
\begin{proof}
We only sketch a proof of the last statement.  The condition $v_i \in \spn(v_{i+1},\ldots,v_j)$ is equivalent to $\dim \spn(v_{i},v_{i+1},\ldots,v_j) = \dim \spn(v_{i+1},\ldots,v_j)$.  Thus $f_X$ is determined by $r_X$.  Conversely, $f_X$ can be used to determine when the rank matrix increases, that is, when $r(i,j) - r(i+1,j)$ is equal to 0 or to 1.  This shows that $f_X$ and $r_X$ determine each other.
The lexicographically minimal $J$ such that $\Delta_J(X) \neq 0$ is determined by the values $r(1,1),r(1,2),r(1,3),\ldots,r(1,n)$.  Specifically, $j \in J$ if and only if $r(1,j) > r(1,j-1)$, where we take $r(1,0) = 0$.  The converse is similar.
\end{proof}

\begin{proof}[Sketch proof of Theorem \ref{thm:Grassorder}]
Define a partial order on cyclic rank matrices by $r \leq r'$ if and only if $r(i,j) \leq r'(i,j)$ for all $i,j$.  Then it is a standard result in combinatorics \cite{BB} that $f \geq f'$ in $\Bound(k,n)$ if and only if $r_{f'} \leq r_f$, where the rank matrices are related to the bounded affine permutations by the correspondence of Proposition \ref{prop:three}.  (We will see later that for every $f \in \Bound(k,n)$ there exists $X \in \Gr(k,n)$ such that $f_X = f$, so there is no loss of generality.)  But it is also clear from the Proof of Proposition \ref{prop:three} that $\I \leq \I'$ if and only if $r_{\I'} \leq r_{\I}$, so the claim follows.
\end{proof}

\begin{example}
Let $k = 3$ and $n = 6$.  Consider the point
$$
X = 
\begin{bmatrix}
 1 & 1 & 0 & 0 & 0 & 0 \\
 0 & 1 & 4 & 6 & 0 & 0 \\
 0 & 0 & 1 & 2 & 2 & 1 \\
\end{bmatrix}
\in \Gr(3,6)_{\geq 0}.
$$
Then $f_X =[4, 7, 5, 8, 6, 9]$, because, for example, $v_2 \in \spn(v_3,v_4,v_5,v_6,v_7)$ but $v_2 \notin \spn(v_3,v_4,v_5,v_6)$.  We have
$\I_X = (123,234,341,451,512,612)$.
We have $r_X(1,2) = 2$ but $r_X(5,6) = 1$.
\end{example}

\section{Totally nonnegative Grassmann cells}\label{sec:representable}
In this section, we decompose $\Gr(k,n)_{\geq 0}$ into \emph{positroid cells}, and show that every point in $\Gr(k,n)_{\geq 0}$ is represented by a network $N$ (the analogue of Theorem \ref{thm:LGV}).  The main results in this section are due to Postnikov \cite{Pos}.  Our proof relies on a bridge--lollipop reduction procedure which we believe to be new.

\subsection{Trips and zig-zag paths}\label{sec:trip}
Let $G$ be a planar bipartite graph.  In the following we will sometimes think of an edge in $G$ as two directed edges, one in each direction.

We decompose $G$ into directed paths and cycles as follows.  Given a directed edge $e: u \to v$, if $v$ is black we pick the edge $e': v \to w$ after $e$ by turning (maximally) right at $v$; if $v$ is white, we turn (maximally) left at $v$.  This decomposes $G$ into a union of directed paths and cycles, such that every edge is covered twice (once in each direction).  These paths and cycles are called \defn{zig-zag paths}, or \defn{trips}.  
%

The \defn{trip permutation} $\pi_G: [n] \to [n]$ is the permutation given by $\pi_G(i) =j$ if the trip that starts at $i$ ends at $j$.  For example in the following square graph, we have $\pi_G(1) = 3, \pi_G(2) = 4, \pi_G(3) = 1, \pi_G(4) = 2$.
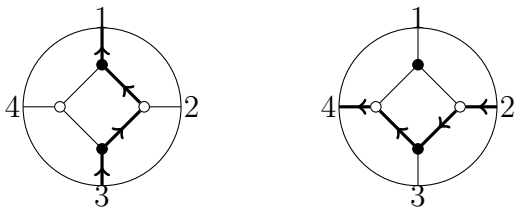
\begin{center}
\begin{tikzpicture}[baseline=-0.5ex,scale=0.7]
\node at (0,1.7) {$1$};
\node at (1.7,0) {$2$};
\node at (0,-1.7) {$3$};
\node at (-1.7,0) {$4$};
\draw (0,0) circle (1.5cm);
\draw (-1.5,0) -- (-0.8,0);
\draw (1.5,0) -- (0.8,0);
\draw (0,0.8) -- (0,1.5);
\draw (0,-0.8) -- (0,-1.5);
\draw (-0.8,0) -- (0,0.8) -- (0.8,0) -- (0,-0.8) -- (-0.8,0);
\begin{scope}[decoration={
    markings,
    mark=at position 0.5 with {\arrow{>}}}
    ]
\draw[very thick,postaction={decorate}] (0,-1.5)-- (0,-0.8);
\draw[very thick,postaction={decorate}]  (0,-0.8)--(0.8,0);
\draw[very thick,postaction={decorate}](0.8,0)--(0,0.8);
\draw[very thick,postaction={decorate}](0,0.8)--(0,1.5);
\end{scope}
\whitedot{(0.8,0)}
\blackdot{(0,0.8)}
\blackdot{(0,-0.8)}
\whitedot{(-0.8,0)}

\begin{scope}[shift={(6,0)}]
\node at (0,1.7) {$1$};
\node at (1.7,0) {$2$};
\node at (0,-1.7) {$3$};
\node at (-1.7,0) {$4$};
\draw (0,0) circle (1.5cm);
\draw (-1.5,0) -- (-0.8,0);
\draw (1.5,0) -- (0.8,0);
\draw (0,0.8) -- (0,1.5);
\draw (0,-0.8) -- (0,-1.5);
\draw (-0.8,0) -- (0,0.8) -- (0.8,0) -- (0,-0.8) -- (-0.8,0);
\begin{scope}[decoration={
    markings,
    mark=at position 0.5 with {\arrow{>}}}
    ]
\draw[very thick,postaction={decorate}] (1.5,0)-- (0.8,0);
\draw[very thick,postaction={decorate}]  (0.8,0)--(0,-0.8);
\draw[very thick,postaction={decorate}](0,-0.8)--(-0.8,0);
\draw[very thick,postaction={decorate}](-0.8,0)--(-1.5,0);
\end{scope}
\whitedot{(0.8,0)}
\blackdot{(0,0.8)}
\blackdot{(0,-0.8)}
\whitedot{(-0.8,0)}
\end{scope}
\end{tikzpicture}
\end{center}

\begin{proposition}
Trip permutations are preserved by the moves (M1) and (M2).
\end{proposition}
\begin{proof}
This is checked case by case.
\end{proof}

A leafless planar bipartite graph $G$ is \defn{reduced} or \defn{minimal} if 
\begin{enumerate}
\item
there are no trips that are cycles,
\item
no trip uses an edge twice (once in each direction) except for the case of a boundary leaf, and
\item
no two trips $T_1$ and $T_2$ share two edges $e_1, e_2$ such that the edges appear in the same order in both trips.
\end{enumerate}

Note that $T_1$ and $T_2$ can share two edges $e_1,e_2$ if they appear in a different order.

\begin{remark}
The conditions imply that if $\pi_G(i) = i$ then the boundary vertex $i$ must be connected to a boundary leaf.
\end{remark}

\begin{remark}
The trip permutations allow us to associate a $k$-element subset $I_F \subset [n]$ to each face $F$ of a planar bipartite graph $G$.  These face labels play an important role in certain aspects of the subject \cite{OPS,OS,MS,FP}.
\end{remark}

%

%
\subsection{The bounded affine permutation of a reduced planar bipartite graph}

Let $G$ be a reduced planar bipartite graph.  We define a bounded affine permutation $f_G \in \Bound(k,n)$ as follows: we always have $f_G(i) = \pi_G(i) \mod n$, where $\pi_G$ is the trip permutation of $G$ defined in Section \ref{sec:trip}.  Given the bounded condition, the only time there is ambiguity is if the trip that starts at $i$ ends at $i$, that is, $\pi_G(i) = i$.  In this case, we have $f_G(i) = i$ if $i$ is incident to a black vertex and $f_G(i) = i+n$ if $i$ is incident to a white vertex.

It is not difficult to check that if $G$ and $G'$ are related by the moves (M1) and (M2) then $f_G = f_{G'}$.  
We omit the proof of the following important result.

\begin{theorem}[\cite{Pos}]\label{thm:reduced}
Every planar bipartite graph is move-equivalent to a reduced graph.  A planar bipartite graph is reduced if and only if it has the minimal number of faces in its move-equivalence class.  Any two reduced planar bipartite graphs in the same move-equivalence class are related by the equivalences (M1) and (M2).  Two reduced planar bipartite graphs $G$ and $G'$ are in the same move-equivalence class if and only if $f_G = f_{G'}$.
\end{theorem}
%

Theorem \ref{thm:reduced} is an analogue of the well-known fact that any two reduced words for a permutation are related by commutation moves and braid moves.  Another proof of (part of) Theorem \ref{thm:reduced} appears in the recent work of Oh and Speyer \cite{OS}.

\subsection{Matroids and positroids}\label{sec:Schubmat}
Some basic facts about matroids will be reviewed in Section \ref{sec:matroids}.  For now, we will think of matroids as collections of $k$-element subsets, called bases, satisfying the exchange axiom.

If $X \in \Gr(k,n)$ we define 
\begin{equation}\label{eq:matroid}
\M_X = \left\{I \in \binom{[n]}{k} \mid \Delta_I(X) \neq 0\right\}
\end{equation}
to be the matroid of $X$.

Let $\S_I \coloneqq \{J \in \binom{[n]}{k} \mid I \leq J\}$ be the Schubert matroid with minimal element $I$.  Let $\S_{I,a} \coloneqq \{J \in \binom{[n]}{k} \mid I \leq_a J\}$ be a cyclically rotated Schubert matroid.  We leave as an exercise for the reader to check that these are indeed matroids.

Given $X \in \Gr(k,n)$ we write $X \in \oX_I$ if $I$ is the lexicographically minimal subset such that $\Delta_I(X) \neq 0$ (we will define the Schubert cell $\oX_I$ and the Schubert variety $X_I$ in Section \ref{sec:Schub}).  The following result is one version of the greedy property of matroids.

\begin{lemma}
If $X \in \oX_I$ then $\M_X \subset \S_I$.
\end{lemma}

  If $X \in \Gr(k,n)_{\geq 0}$, then we call $\M_X$ a \defn{positroid}.  Denote the set of positroids by $\PP(k,n)$.  Given a positroid $\M \in \PP(k,n)$, we let the positroid cell $\Pi_{\M, >0}$ be
$$
\Pi_{\M, >0} \coloneqq \{X \in \Gr(k,n)_{\geq 0}\mid \M_X = \M\}.
$$
Given a positroid $\M \in \PP(k,n)$, we obtain a Grassmann necklace $\I(\M)$ defined by 
\begin{equation}\label{eq:IM}
I_a = \min_{\leq a} \{J \in \M\}
\end{equation}
where $\min_{\leq a}$ is the lexicographical minimum with respect to the cyclic order $\leq a$ on $[n]$.  We also define the bounded affine permutation $f_M \in \Bound(k,n)$ by $\I(f) = \I(\M)$.


\subsection{Adding Bridges}
Let $G$ be a planar bipartite graph.  We define the operation of \defn{adding a bridge at $i$}, black at $i$ and white at $i+1$.  It modifies a bipartite graph near the boundary vertices $i$ and $i+1$:
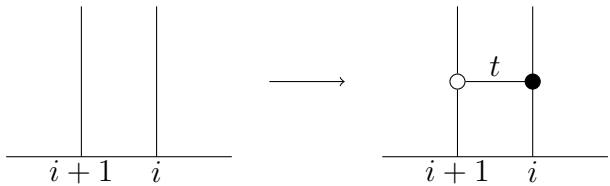
\begin{center}
\begin{tikzpicture}
\draw (1,0) -- (4,0);
\draw (2,0) -- (2,2);
\draw (3,0) -- (3,2) ;
\node at (2,-0.2) {$i+1$};
\node at (3,-0.2) {$i$};

\draw[->] (4.5,1) -- (5.5,1);
\begin{scope}[{shift={(5,0)}}]
\draw (1,0) -- (4,0);
\draw (2,0) -- (2,2);
\draw (3,0) -- (3,2) ;
\draw (2,1) -- (3,1);
\filldraw [white] (2,1) circle (0.1cm);
\filldraw [black] (3,1) circle (0.1cm);
\draw (2,1) circle (0.1cm);
\node at (2,-0.2) {$i+1$};
\node at (3,-0.2) {$i$};
\node at (2.5,1.2) {$t$};
\end{scope}
\end{tikzpicture}
\end{center}

The bridge edge is the edge labeled $t$ in the above picture.  Note that in general this modification might create a graph that is not bipartite -- for example, if in the original graph $i$ is connected to a black vertex.  However, by adding valent two vertices using the local move (M2), we can always assume that we obtain a bipartite graph.
There is an operation of ``adding a bridge at $i$, white at $i$ and black at $i+1$", as well.
%
%

Adding a bridge is the network analogue of multiplication by the Chevalley generators $x_i(a)$ and $y_i(b)$ of Section \ref{sec:TP}.

\begin{lemma}\label{lem:networkbridge}
Let $N$ be a network.  Now let $N'$ be obtained by adding a bridge with edge weight $a$ from $i$ to $i+1$ which is white at $i$ and black at $i+1$.  Then the boundary measurements change as follows:
$$
\Delta_I(N') = \begin{cases} 
\Delta_I(N) + a\Delta_{I - \{i+1\} \cup \{i\}}(N) & \mbox{if $i+1 \in I$ but $i \notin I$} \\
\Delta_I(N) & \mbox{otherwise.}
\end{cases}
$$
Thus $X(N') = X(N) \cdot x_i(a)$.

If $N''$ is obtained by adding a bridge, black at $i$ and white at $i+1$, then 
$$
\Delta_I(N') = \begin{cases} 
\Delta_I(N) + a\Delta_{I - \{i\} \cup \{i+1\}}(N) & \mbox{if $i \in I$ but $i+1 \notin I$} \\
\Delta_I(N) & \mbox{otherwise.}
\end{cases}
$$
Thus $X(N'') = X(N) \cdot y_i(a)$.
\end{lemma}

For $i = n$, we should think of $x_n(a)$ (resp. $y_n(a)$ as the operation obtained from $x_1(a)$ (resp. $y_1(a)$) by conjugating by the generator of the $\Z/n\Z$ action on $\Gr(k,n)$.

\begin{remark}
Thinking of adding bridges as the $\GL(n)_{\geq 0}$ action on $\Gr(k,n)_{\geq 0}$ breaks the cyclic symmetry of planar bipartite graphs (the operations $x_n(a)$ and $y_n(a)$ do not come from elements of $\GL(n)_{\geq 0}$).  It is more natural to consider adding bridges to be the action of the totally nonnegative part of the polynomial loop group $\GL_n(\R[t,t^{-1}])$ on $\Gr(k,n)_{\geq 0}$.  In \cite{LPasha1,LPasha2}, the analogue of Theorems \ref{thm:LW} and \ref{thm:LGV} are established for the polynomial loop group.  In particular, elements $g \in \GL_n(\R[t^{-1},t])_{\geq 0}$ are represented by networks on a cylinder.  The action of $\GL_n(\R[t^{-1},t])_{\geq 0}$ on $\Gr(k,n)_{\geq 0}$ corresponds to gluing a cylinder to a disk along one boundary of the cylinder, and thus obtaining a disk.  I expect there to be rich generalizations of the topics discussed here to networks on surfaces; see \cite{GSVsurface,LPashasurface,GK}.
\end{remark}

\subsection{Adding a lollipop}
We also need the operation of \defn{adding a lollipop}, which can be either white or black.  This inserts a new boundary vertex connected to an interior leaf.  The new boundary vertices are then relabeled:
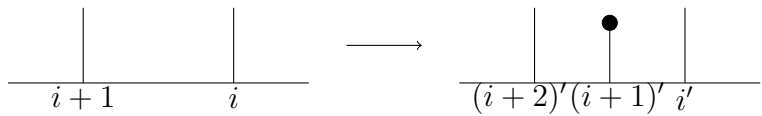
\begin{center}
\begin{tikzpicture}
\draw (1,0) -- (5,0);
\draw (2,0) -- (2,1);
\draw (4,0) -- (4,1) ;
\node at (2,-0.2) {$i+1$};
\node at (4,-0.2) {$i$};

\draw[->] (5.5,0.5) -- (6.5,0.5);
\begin{scope}[{shift={(6,0)}}]
\draw (1,0) -- (5,0);
\draw (2,0) -- (2,1);
\draw (4,0) -- (4,1) ;
\draw (3,0) -- (3,0.8);
\filldraw [black] (3,0.8) circle (0.1cm);
\node at (1.8,-0.2) {$(i+2)'$};
\node at (3.1,-0.2) {$(i+1)' $};
\node at (4,-0.2) {$i'$};
\end{scope}
\end{tikzpicture}
\end{center}

\subsection{Reduction of TNN Grassmann cells}

Let $X \in \Gr(k,n)_{\geq 0}$.  Suppose $f_X$ has a fixed point $f_X(i) = i$.  Then by \eqref{eq:fX}, the $i$-th column $v_i$ of any representative of $X$ must be the 0 vector.  We have a projection map $p_i: \R^n \to \R^{n-1}$ removing the $i$-th coordinate.

\begin{lemma}\label{lem:proj1}
The projection map induces a bijection between $\{X \in \Gr(k,n)_{\geq 0} \mid f_X(i) = i\}$ and $\Gr(k,n-1)_{\geq 0}$.
\end{lemma}

Now suppose $f_X$ satisfies $f_X(i) = i+n$.  Then by \eqref{eq:fX}, the $i$-th column $v_i$ of any representative of $X$ is not in the span of the other columns.  Treating $X$ as a $k$-dimensional subspace of $\R^n$, we have that $p_i(X)$ is a $(k-1)$-dimensional subspace of $\R^n$.  

\begin{lemma}\label{lem:proj2}
The projection map gives a bijection between $\{X \in \Gr(k,n)_{\geq 0} \mid f_X(i) = i+n\}$ and $\Gr(k-1,n-1)_{\geq 0}$.
\end{lemma}
\begin{proof}
By cyclic rotation we assume that $i = 1$.  By left multiplying by $g \in \GL(k,\R)$, we may assume that the first column is $(1,0,\ldots,0)^{T}$ and that the first row is $(1,0,\ldots,0)$.  Removing the first row and column gives a $(k-1) \times (n-1)$ matrix, representing a point in $\Gr(k-1,n-1)_{\geq 0}$.  It is not hard to see that this is a bijection.
\end{proof}



We now give a bridge (or Chevalley generator) reduction of TNN points in the Grassmannian.  Let $X$ be a TNN point of the Grassmannian.  Suppose the bounded affine permutation $f_X$ satisfies $i < i+1 \leq f(i) < f(i+1) \leq i+n$.  Then we say that $X$ has a bridge at $i$.

\begin{proposition}\label{prop:reduce}
Suppose $X \in \Gr(k,n)_{\geq 0}$ has a bridge at $i$.  Then the quantity $$a = \Delta_{I_{i+1}}(X)/\Delta_{I_{i+1} \cup \{i\} - \{i+1\}}(X)$$ is positive and well defined, and $X' = X \cdot x_i(-a) \in \Gr(k,n)_{\geq 0}$ has a positroid strictly smaller than $\M_X$.  We also have $f_{X'} = f_{X}s_i$.
\end{proposition}

\begin{proof}
Let $v_i$ be the columns of a $k \times n$ matrix which represents $X$.

If $f(i) = i+1$, then by \eqref{eq:fX}, the columns $v_i$ and $v_{i+1}$ are parallel, and since $f(i+1) \neq i+1$ both $v_i$ and $v_{i+1}$ are non-zero.  In this case $a$ is just the ratio $v_{i+1}/v_i$, and $X'$ is what we get by changing the $(i+1)$-st column to $0$.  All the claims follow.

We now assume that $f(i) > i+1$.  For simplicity of notation, assume $i=1$.  Let $f(i) = j$ and $f(i+1) =  k$.
Since $f(i) \notin \{i,i+n\}$, we have $i \in I_i$ and $i \notin I_{i+1}$.  We also have $i+1 \in I_i \cap I_{i+1}$.  We let $I_i = \{i,i+1\} \cup I$, $I_{i+1} = (i+1) \cup I \cup \{j\}$, and $I_{i+2} = I \cup \{j,k\}$ for some $I \subset [n]-\{i,i+1\}$.  Note that if $k = n+i$, then $I_{i+2} = I \cup \{j,i\}$; this immediately gives $\Delta_{i \cup I \cup j} \neq 0$.

Suppose $k \neq n+i$.  Then we have a Pl\"ucker relation
$$
\Delta_{i \cup I\cup j} \Delta_{(i+1) \cup I \cup k} = \Delta_{i \cup I \cup k} \Delta_{(i+1) \cup I \cup j} + \Delta_{i \cup (i+1) \cup I} \Delta_{I \cup j \cup k}
$$
where all subsets are ordered according to $\leq_i$.  (The easiest way to see that the signs are correct is just to take $i = 1$.)  Since the RHS is positive, $\Delta_{i \cup I \cup j} \neq 0$.

Now $X'$ is obtained from $X$ by adding $-a$ times $v_i$ to $v_{i+1}$.  So 
\begin{equation}\label{eq:Delta}
\Delta_J(X') = \begin{cases} \Delta_J(X) -a \Delta_{J -\{i+1\} \cup \{i\}}(X) &\mbox{ if $i+1 \in J$ and $i \notin J$} \\
\Delta_J(X) & \mbox{otherwise.}
\end{cases}
\end{equation}
The formulae above are the minors of this specific representative of $X'$; the Pl\"ucker coordinates of the actual point in the Grassmannian are only determined up to a scalar.  By Lemma \ref{lem:reduce} below, we see that $X' \in \Gr(k,n)_{\geq 0}$, and that $J \in \M_{X'}$ only if $J \in \M_X$.  However, $\Delta_{I_{i+1}}(X') = 0$, so $\M_{X'} \subsetneq \M_X$.  

Finally, let $v'_i$ be the columns for the matrix obtained from $v_i$ by right multiplication by $x'(-a)$.  Then $\spn(v_i) = \spn(v'_i)$ and $\spn(v_i,v_{i+1}) = \spn(v'_i,v'_{i+1})$, so $f_{X'}(r) = f_X(r)$ unless $r \in \{i,i+1\} \mod n$.  But $f_{X'} \neq f_X$ since $\Delta_{I_{i+1}}(X') = 0$.  Thus $f_{X'}$ must be obtained from $f_X$ by swapping the values of $f(i)$ and $f(i+1)$.
\end{proof}

\begin{lemma}\label{lem:reduce}
Let $X \in \Gr(k,n)_{\geq 0}$ be as in Proposition \ref{prop:reduce}, with $f(i) > i+1$.  For simplicity of notation suppose $i = 1$.  
Write $I_2 = 2 \cup I \cup j$.  Suppose $J \subset \{3,\ldots,n\}$ satisfies $1 \cup J \in \M_X$.  Then $\Delta_{1\cup I \cup j}(X)\Delta_{2 \cup J}(X) \geq \Delta_{1 \cup J}(X) \Delta_{2 \cup I \cup j}(X)$.
\end{lemma}
\begin{proof}
Let $\M$ be the positroid of $X$.  
We let $I_1 = \{1,2\} \cup I$, $I_2 = 2 \cup I \cup \{j\}$, and $I_3 = I \cup \{j,k\}$, as in the proof of Proposition \ref{prop:reduce}. We have already shown in the proof of Proposition \ref{prop:reduce} that $(1 \cup I \cup j) \in \M$.

We proceed by induction on the size of $r = |(I \cup j) \setminus J|$.  The case $r = 0$ is tautological.  So suppose $r \geq 1$.  We may assume that $1 \cup J \in \M$ for otherwise the claim is trivial.  Applying the exchange lemma to $1 \cup J$ the element $a = \max(J \setminus (I \cup j)) \in J$ and the other base $1 \cup I \cup j$, we obtain $L = J -\{a\} \cup \{b\}$ such that $1 \cup L \in \M$.  

We claim that $b < a$.  To see this, note that $I_1 \leq (1 \cup J)$, which implies that $a > I \setminus J$.  So the only way that $b$ could be greater than $a$ is if $b = j$, and $a < j$.  But by assumption we also have $I_3 = I \cup \{j,k\} \leq_3 (1 \cup J)$ with $k \geq_2 j$.  This is impossible since both $k$ and $j$ are greater than $a$, but we have $J \setminus I \subset [3,a]$ -- the only element of $(1 \cup J) \setminus I$ that is greater than $j$ or $k$ in $\leq_3$ order is $1$. Thus $b < a$. 

So by induction we have that $\Delta_{2\cup L}/\Delta_{1 \cup L} \geq \Delta_{2 \cup I}/\Delta_{1 \cup I}$, where in particular we have $(1 \cup L), (2 \cup L) \in \M$.  It suffices to show that $\Delta_{2\cup J}/\Delta_{1 \cup J} \geq \Delta_{2\cup L}/\Delta_{1 \cup L}$.

We apply the Pl\"ucker relation to $\Delta_{2\cup J} \Delta_{1 \cup L}$, swapping $L$ with $(k-1)$ of the indices in $2\cup J$ to get
$$
\Delta_{1\cup L}\Delta_{2 \cup J}= \Delta_{1 \cup J} \Delta_{2 \cup L} +  \Delta_{12 j_1 j_2 \cdots \hat a \cdots j_{k-1}} \Delta_{\ell_1 \ell_2 \cdots a \cdots \ell_{k-1}} .
$$
We note that  $\ell_1 < \ell_2 < \cdots < a < \cdots < \ell_{k-1}$ is actually correctly ordered, since $L$ is obtained from $J$ by changing $a$ to a smaller number.  So all factors in the above expression are nonnegative.  The claim follows.
\end{proof}
%

\subsection{Network realizability of $\Gr(k,n)_{\geq 0}$}
Let $M_G:(\L_G)_{>0} \to \Gr(k,n)_{\geq 0}$ be the map that takes a network $N$ representing a point in $(\L_G)_{>0}$ to the point $X(N)$.  Let $\Pi_{G,>0}$ denote the image of $M_G$.
\begin{theorem}\label{thm:main}
\
\begin{enumerate}
\item
Every $X \in \Gr(k,n)_{\geq 0}$ is representable by a network $N$. 
\item
The map $\M \mapsto f_\M$ is a bijection between $\PP(k,n)$ and $\Bound(k,n)$.  The map $\M \mapsto \I(\M)$ is a bijection between positroids and Grassmann necklaces.
\item
For each positroid cell $\Pi_{\M,>0}$ there is a reduced bipartite graph $G$ such that $M_G: (\L_G)_{>0} \to \Pi_{G,>0} \coloneqq \Pi_{\M,>0}$ is bijective.  The bounded affine permutation of $G$ is equal to $f_\M$.
\item
$\Pi_\M \simeq \R^d_{>0}$ has dimension equal to $d = k(n-k) - \ell(f_\M)$.
\end{enumerate}
\end{theorem}
\begin{proof}
We establish the first statement completely first.  We proceed by induction on $n$, and then by induction on $|\M|$.

Suppose $n = 1$, then $X$ is representable by a network $N$ with a single boundary vertex joined to a single interior vertex, which can be either black or white.  This represents the unique points in $\Gr(0,1)_{\geq 0}$ and $\Gr(1,1)_{\geq 0}$.  This is the base case.

Now suppose $X \in \Gr(k,n)_{\geq 0}$.  If $f_X(i) \in \{i,i+n\}$, then we can apply the reductions of Lemma \ref{lem:proj1} and Lemma \ref{lem:proj2} to get some $X'$ which by induction is represented by a network $N'$.  To obtain $N$ from $N'$ we insert a lollipop (with any edge weight, they are all gauge equivalent) at position $i$.  Note that $f_{X'}$ is determined completely by $f_X$.

Thus we may suppose that $f_X(i) \notin \{i,i+n\}$.  But then we can find some $i$ such that $f_X(i) < f_X(i+1)$ satisfying the conditions of Proposition \ref{prop:reduce}.  Let $X' \in \Gr(k,n)_{\geq 0}$ be the TNN point of Proposition \ref{prop:reduce}.  Then by induction on $\M$, we may assume that $X'$ is represented by a network $N'$.  Let $N$ be the network obtained from $N'$ by adding a bridge between $i$ and $i+1$, white at $i$ and black at $i+1$.  Lemma \ref{lem:networkbridge} then says that $N$ represents $X$.  

Thus every $X \in \Gr(k,n)_{\geq 0}$ is representable by a network $N$.  We note that the entire recursion depends only on $f_X$: we can choose the underlying graph $G$ of $N$ to depend on $f_X$ only.  Thus for each bounded affine permutation $f$, there is a graph $G(f)$ which parametrizes all of $\{X \in  \Gr(k,n)_{\geq 0} \mid f_X = f\}$.  But the matroid of $X(N)$ depends only on $G$ (as long as all edge weights are positive), so we have a bijection between positroids and bounded affine permutations, and in turn Grassmann necklaces.

We note that adding a bridge adds one face and hence one parameter to $(\L_G)_{>0}$.  Adding lollipops do not change the number of faces.  So $(\L_{G(f)})_{>0} \simeq \R_{>0}^d$ where $d$ is the number of bridges used in the entire recursion.  Furthermore, the edge weights of the bridges determine the graph up to gauge equivalence, or, equivalently, these edge weights are coordinates on $(\L_{G(f)})_{>0}$.  But the labels of the bridges are uniquely recovered $X = X(N)$ by the recursive algorithm above.  So the map $M_G: (\L_G)_{>0} \to \Pi_{\M, >0}$ is a bijection, where $G= G(f_\M)$.  By Theorem \ref{thm:reduced}, $G$ is reduced since $M_G: (\L_G)_{>0} \to \Gr(k,n)$ is injective (or the reduced statement can be proved directly).

Finally, we note that the dimension claim is true for $n=1$, and we have $\ell(fs_i) = \ell(f) + 1$ when $f(i) < f(i+1)$.  Now suppose we have $X$ such that $f_X(i) = i$ and $X'$ is obtained by the projection $p_i$.  Then $\{(i,j) \mid i < j \text{ and } f_X(i) > f_X(j)\}=\emptyset$, but $|\{(j,i) \mid j < i \text{ and } f_X(j) > f_
X(i)\}| = k$.  So $\ell(f_X) = \ell(f_{X'}) + k$.  A similar relation holds when $f_X(i) = i+n$.  Thus the formula for the dimension of $\Pi_{\M,>0}$ holds by induction.
\end{proof}


\begin{remark}
There are a number of explicit constructions of graphs $G(f)$ that represent each $f \in \Bound(k,n)$, see \cite{Pos,Kar}.
\end{remark}

Using Theorem \ref{thm:main}, we define the \defn{positroid cell} $\Pi_{f, >0} \coloneqq \Pi_{\M, >0}$, where $f_\M = f$.

\begin{corollary}\label{cor:G}
For any reduced planar bipartite graph $G$, we have
$M_G: (\L_G)_{>0} \to \Pi_{G, >0} = \Pi_{f_G, > 0}$.  
\end{corollary}
\begin{proof}
This follows from combining Theorem \ref{thm:main}(3) with Proposition \ref{prop:GG'} and Theorem \ref{thm:reduced}.
\end{proof}

\begin{theorem}
Suppose $N$ and $N'$ are planar bipartite networks with $X(N) = X(N')$.  Then $N$ and $N'$ are related by local moves and gauge equivalences.  
\end{theorem}
\begin{proof}
By Theorem \ref{thm:reduced}, we may first replace $N$ and $N'$ by networks whose underlying planar bipartite graphs are reduced, without changing $X(N)$ and $X(N')$.  Again by Theorem \ref{thm:reduced}, we may assume that $N$ and $N'$ and have the same underlying reduced planar bipartite graph $G$, which we may choose to be the graph $G$ in Theorem \ref{thm:main}(3).  Thus Theorem \ref{thm:main}(3) says that $N$ and $N'$ are related by gauge equivalences.
\end{proof}

\section{Positroids and $\Gr(k,n)_{\geq 0}$ as a stratified space}\label{sec:positroids}
%
In this section, we give a number of different descriptions of positroids due to Oh \cite{Oh}, Lam and Postnikov \cite{LPmembrane}, and Ardila, Rincon, and Williams \cite{ARW}.  We also describe the closure partial order on positroid cells, originally determined by Postnikov \cite{Pos} and Rietsch \cite{Rie}.  The description here in terms of Bruhat order is from \cite{KLS}.

\subsection{Closures of positroid cells}
Define $\Pi_{f,\geq 0} \coloneqq \cl(\Pi_{f,> 0})$ to be the closure of $\Pi_{f,>0}$ in the Hausdorff topology on $\Gr(k,n)$ (not to be confused with the Zariski topology that we shall mostly use).

\begin{theorem}\label{thm:partialorder}
Let $f \in \Bound(k,n)$.  Then $\Pi_{f,\geq 0} = \bigsqcup_{g \leq f} \Pi_{g,>0}$.
\end{theorem}

We first give a proof of the direction $\supseteq$.  We hope the reader notices the strong similarity with arguments in Bruhat order.
\begin{proposition}\label{prop:half}
We have $\Pi_{f,\geq 0} \supseteq \bigsqcup_{g \leq f} \Pi_{g,>0}$.
\end{proposition}
\begin{proof}
By induction, it is enough to show that $\Pi_{g, >0} \subset \Pi_{f,\geq 0}$ when $g \lessdot f$ in $\B(k,n)$ (thus $g$ covers $f$ in Bruhat order of $\tS_n^k$).  It is a standard exercise to show that this happens if and only if $g$ is obtained from $f$ by swapping $f(i)$ with $f(j)$, where
\begin{equation}\label{eq:fg}
i < j, \qquad f(i) < f(j), \qquad \{f(a) \mid i < a < j\} \cap [f(i),f(j)] = \emptyset.
\end{equation}
Let $G$ be a reduced planar bipartite graph with $f_G = f$.  Then \eqref{eq:fg} implies that the trip $T_i$ starting at $i$ and the trip $T_j$ starting at $j$ must cross one another.  In particular, $T_i$ and $T_j$ must share an edge $e$, where they travel in opposite directions along $e$.  By the move (M2), we can assume that this edge $e$ is unique, and that the graph $G' = G \setminus \{e\}$ is reduced.  Then it follows from the definitions that $f_{G'} = g$.  A network $N'$ with underlying graph $G'$ can thus be thought of as a network $N(0)$ with underlying graph $G$, but edge $e$ having weight 0.  Let $N(a)$ be the same network but letting edge $e$ have weight $a$.  Then $X(N(0)) = \lim_{a \to 0} X(N(a))$, and so by Corollary \ref{cor:G} we have $\Pi_{g, >0} \subset \Pi_{f,\geq 0}$.
\end{proof}

Let $I \in \binom{[n]}{k}$.  Define $t_I \in \B(k,n)$ by $t_I(i) = i+n$ if $i \in I$ and $t_I(i) = i$ if $i \notin I$.  Recall that the rotated Schubert matroid $\S_{I,a}$ was defined in Section \ref{sec:Schubmat}.
\begin{lemma}\label{lem:tI}
Let $f \in \Bound(k,n)$ and $I \in \binom{[n]}{k}$.  We have $f \geq t_I$ if and only if $I \in \S_{I_1,1} \cap \S_{I_2,2} \cap \cdots \cap \S_{I_n,n}$, where $\I=(I_1,\ldots,I_n)$ is the Grassmann necklace of $f$.
\end{lemma}
\begin{proof}
The Grassmann necklace of $t_I$ is $(I,I,\ldots,I)$.  The result then follows from Theorem \ref{thm:Grassorder}.
\end{proof}

\subsection{Oh's Theorem}
Our approach gives a new proof of Oh's theorem.

\begin{theorem}[\cite{Oh}]\label{thm:oh}
Positroids are intersections of cyclically rotated Schubert matroids: if $\I(\M) = (I_1,I_2,\ldots,I_n)$ then 
$$\M = \S_{I_1,1} \cap \S_{I_2,2} \cap \cdots \cap \S_{I_n,n}.$$
\end{theorem}
\begin{proof}
The inclusion $\subseteq$ follows from the definition \eqref{eq:IM}.  For the reverse inclusion, suppose $I$ belongs to the right hand side.  Let $f  = f_\M \in \Bound(k,n)$ be the bounded affine permutation corresponding to the positroid $\M$.  By Lemma \ref{lem:tI}, we have $f \geq t_I$.  By Proposition \ref{prop:half}, we have $\Pi_{t_I,>0} \subset \Pi_{f,\geq 0}$.  But $\Pi_{t_I,>0} $ is simply the point $e_I \in \Gr(k,n)$ with the single non-vanishing Pl\"ucker coordinate $\Delta_I$.  Thus the Pl\"ucker coordinate cannot vanish on $\Pi_\M$ (otherwise it would vanish on the closure as well).  It follows that $I \in \M$, as required.
\end{proof}

Recall that Theorem \ref{thm:main} gives a bijection $f \mapsto \M(f)$ between $\Bound(k,n)$ and $\PP(k,n)$.  Theorem \ref{thm:oh} has the following immediate corollary.
\begin{corollary}\label{cor:Morder}
We have $f \geq g$ if and only if $\M(f) \supseteq \M(g)$.
\end{corollary}

\begin{proof}[Proof of Theorem \ref{thm:partialorder}]
By Proposition \ref{prop:half}, we have the inclusion $\Pi_{f,\geq 0} \supseteq \bigsqcup_{g \leq f} \Pi_{g,>0}$.

Suppose $X \in \overline{\Pi_{f,> 0}}$.  Then $X \in \Gr(k,n)_{\geq 0}$ so $X \in \Pi_{g, >0}$ for some $g \in \Bound(k,n)$.  The Pl\"ucker coordinates $\Delta_I(X)$ are non-zero for $I \in \M(g)$.  Suppose $J \notin \M(f)$.  Then the Pl\"ucker coordinate $\Delta_J$ vanishes on $\Pi_{f,>0}$ and therefore it also vanishes on $\overline{\Pi_{f,> 0}}$.  We conclude that $\M(g) \subseteq \M(f)$.  But by Corollary \ref{cor:Morder} this implies $f \geq g$.  Thus $\overline{\Pi_{f,> 0}} = \bigsqcup_{g \leq f} \Pi_{g,>0}$.
\end{proof}

We also have the following somewhat surprising Corollary.
\begin{corollary}
Suppose $f,g \in \Bound(k,n)$.  Then $f \geq g$ if and only if whenever $g \geq t_I$ we have $f \geq t_I$ as well, for $I \in \binom{[n]}{k}$.
\end{corollary}

\subsection{Proof of Theorem \ref{thm:same}}\label{sec:LusPos}
We can now prove the equivalence of Lusztig's and Postnikov's defintions of the totally nonnegative Grassmannian.
The first equality of Theorem \ref{thm:same} is just the special case $f = \id$ of Theorem \ref{thm:partialorder}.

Now let $f \in \Bound(k,n)$ be given by $f(i) = i+n$ for $1 \leq i \leq k$ and $f(i) = i$ for $k+1 \leq i \leq n$.  Then $\Pi_{f,>0}$ is the single point $e_{[k]} \in \Gr(k,n)$.   Let $w \in S_n$ be the permutation such that $fw = \id$.  Then $w = (r+1)(r+2)\cdots n 1 2 \cdots r$ in one-line notation.  Let $i_1 i_2 \cdots i_\ell$ be a reduced word for $w$.  Then by the proof of Theorem \ref{thm:main}, adding the bridges indexed by $i_1, i_2, \ldots, i_\ell$ to the lollipop graph of $e_{[r]}$ gives a planar bipartite graph $G$ such that $M_G: (\L_G)_{>0} \to \Gr(k,n)_{>0}$ is bijective.  Thus for $X \in \Gr(k,n)_{>0}$, there are (unique) parameters $a_1,a_2,\ldots,a_\ell \in \R_{>0}$ such that the matrix $g = x_{i_1}(a_1) \cdots  x_{i_\ell}(a_\ell)$ satisfies $e_{[r]} \cdot g = X$. This shows that $\Gr(k,n)_{>0} \subset \GL(n)_{ \geq 0} \cdot e_{[k]}$.

\subsection{Supermodularity of Pl\"ucker coordinates}\label{sec:logconcave}
%
%
%
%
%

Let $I = \{i_1<i_2<\ldots,i_k\}$ and $J=\{j_1<\cdots<j_k\} \in \binom{[n]}{k}$.  Suppose the multiset $I \cup J$, when sorted in increasing order, is equal to $\{a_1 \leq b_1 \leq a_2 \leq \cdots \leq a_k \leq b_k\}$.  Then we define $\sort_1(I,J) = \{a_1,\ldots,a_k\}$ and $\sort_2(I,J) = \{b_1,\ldots,b_k\}$.  Also define $\min(I,J) \coloneqq \{\min(i_1,j_1),\ldots, \min(i_k,j_k)\}$ and $\max(I,J) \coloneqq \{\max(i_1,j_1),\ldots, \max(i_k,j_k)\}$.  For example, if $I = \{1,3,5,6,7\}$ and $J = \{2,3,4,8,9\}$ then $\sort_1(I,J) = \{1,3,4,6,8\}, \sort_2(I,J) = \{2,3,5,7,9\}, \min(I,J) = \{1,3,4,6,7\},$ and $\max(I,J) = \{2,3,5,8,9\}$.


\begin{proposition}\label{prop:inequalities}
Let $X \in \Gr(k,n)_{\geq 0}$.  Then 
$$
\Delta_I(X) \Delta_J(X) \leq \Delta_{\min(I,J)}(X) \Delta_{\max(I,J)}(X) \leq \Delta_{\sort_1(I,J)}(X) \Delta_{\sort_2(I,J)}(X).
$$
\end{proposition}
\begin{proof}
We use Theorem \ref{thm:TL} and show that any $(\tau,T)$ compatible with $I,J$ is also compatible with $\min(I,J),\max(I,J)$ and with $\sort_1(I,J),\sort_2(I,J)$.  We also note that $\sort_i(\min(I,J),\max(I,J)) = \sort_i(I,J)$.
\end{proof}

Similar inequalities occur in the very different context of Schur positivity \cite{LPP}.  See also \cite{FP} for related ideas.

The operations $\min(I,J)$ and $\max(I,J)$ have another interpretation.  To each $I \in \binom{[n]}{k}$ we have an associated partition $\lambda(I) \subseteq (n-k)^{k}$ (see Section \ref{sec:sym}).  Thinking of $\lambda$ and $\mu$ as Young diagrams, write $\lambda \cup \mu$ for the partition that is the union of the boxes in $\lambda$ and $\mu$, and similarly define $\lambda \cap \mu$.  Then $\lambda(\max(I,J)) = \lambda(I) \cup \lambda(J)$ and $\lambda(\min(I,J)) = \lambda(I) \cap \lambda(J)$.  This makes the poset of partitions $\lambda \subseteq (n-k)^{k}$ under inclusion (resp. the poset $(\binom{[n]}{k}, \leq)$) a distributive lattice under the operations $(\cup,\cap)$ (resp. $(\max,\min)$).  

\begin{corollary}
Every positroid $\M$ is a distributive lattice.
\end{corollary}

A \defn{supermodular} function $f: L \to \R$  on a lattice $(L,\vee,\wedge)$ is a function satisfying $f(x \vee y) + f(x \wedge y) \geq f(x) + f(y)$.  A \defn{log-supermodular} function $g: L \to \R_{>0}$ is a function such that $\log g$ is supermodular.

\begin{corollary}
For $X \in \Pi_{f,>0}$, the function $I \mapsto \Delta_I(X)$ is a log-supermodular function from the lattice $(\M(f),\max,\min)$ to $\R_{> 0}$.
\end{corollary}

We can also think of the function $I \mapsto \Delta_I(X)$ as a function $h_X$ on the vectors $e_I \in \R^{n}$ (the $0$-$1$ vector with $1$-s in locations specified by $I$).  Then the inequality $\Delta_I(X) \Delta_J(X) \leq \Delta_{\sort_1(I,J)}(X) \Delta_{\sort_2(I,J)}(X)$ implies that $h_X$ is log-concave: $h_X(x)h_X(y) \leq h_X((x+y)/2)^2$, whenever $x,y,(x+y)/2$ are all of the form $e_I$.

\subsection{Alcoved polytopes and sort-closed sets}
The class of positroids $\PP(k,n)$ is exactly the same as the class of sort-closed matroids that had previously been studied in a different setting \cite{LPalcove,Blu}.

A matroid $\M$ is \defn{sort-closed} if $I,J \in \M$ implies $\sort_1(I,J), \sort_2(I,J) \in \M$.
\begin{thm}[\cite{LPmembrane}]\label{thm:sortclosed}
A matroid $\M$ is a positroid if and only if it is sort-closed.
\end{thm}

The ``only if" direction of Theorem \ref{thm:sortclosed} follows immediately from Proposition \ref{prop:inequalities}.  The ``if direction" of Theorem \ref{thm:sortclosed} follows from a characterization of sort-closed collections as integer points in alcoved polytopes, see \cite{LPalcove}.

Theorem \ref{thm:sortclosed} can also be stated as follows: a matroid polytope is a positroid polytope if and only if it is also an alcoved polytope.  In \cite{LPmembrane}, Postnikov and I take this as a starting point to investigate \emph{polypositroids}, the positive analogue of polymatroids.

\subsection{Positively oriented matroids}
A theorem of Ardila, Rincon, and Williams gives yet another characterization of positroids: they are exactly the underlying matroids of positively orientable matroids.  

A \defn{chirotope} of rank $k$ oriented matroid $\M$ on $[n]$ is a function $\chi:[n]^k \to \{-1,0,1\}$ satisfying the axioms
\begin{enumerate}
\item
The map $\chi$ is alternating:
$$
\chi(i_{\sigma(1)},i_{\sigma(2)},\ldots,i_{\sigma(k)}) = \sign(\sigma) \chi(i_1,i_2,\ldots,i_k)
$$
where $\sign(\sigma)$ is the sign of the permutation $\sigma$.
\item
For any $a_1,a_2,a_3,a_4,i_3,i_4,\ldots,i_k \in [n]$, we have
$$
\text{if } \varepsilon \coloneqq \chi(a_1,a_2,i_3,\ldots,i_k) \chi(a_3,a_4,i_3,\ldots,i_k) \in \{-1,1\},
$$
then either
\begin{align*}
\chi(a_3,a_2,i_3,\ldots,i_k) \chi(a_1,a_4,i_3,\ldots,i_k) &= \varepsilon, \text{or}
\\
\chi(a_2,a_4,i_3,\ldots,i_k) \chi(a_1,a_3,i_3,\ldots,i_k) &= \varepsilon.
\end{align*}
\end{enumerate}

Suppose $\chi$ is a chirotope of rank $k$ on $[n]$.  Then the set $\M_\chi = \{I \in \binom{[n]}{k} \mid \chi(I) \neq 0\}$ is the underyling matroid of $\chi$.

A chirotope $\chi$ is positively orientable if there exists a subset $A \subseteq [n]$ so that $$(-1)^{|A \cap \{i_1,i_2,\ldots,i_k\}|} \chi(i_1,i_2,\ldots,i_k) \geq 0,$$ whenever $i_1 < i_2 < \cdots < i_k$.  It is clear that any point $X \in \Gr(k,n)_{\geq 0}$ gives a positively oriented matroid.

\begin{theorem}[\cite{ARW}]
Suppose the chirotope $\chi$ is positively orientable.  Then the underlying matroid $\M_\chi$ is a positroid.
\end{theorem}

We remark that da Silva had earlier conjectured that positively orientable matroids are realizable.

\subsection{The topology of $\Gr(k,n)_{\geq 0}$}
Let $\hBound(k,n) \coloneqq \Bound(k,n) \cup \{f_\emptyset\}$, where $f_\emptyset$ is a new minimal element.  Thus $\hBound(k,n)$ has unique minimum $f_\emptyset$ and unique maximum $\id$.  By convention, $\Pi_{f_\emptyset, \geq 0} \coloneqq \emptyset$.  

Recall that a poset is \defn{thin} if length two intervals are diamonds, and \defn{Eulerian} if in each interval $[x,y]$ where $x \neq y$, the number of odd rank elements equals the number of even elements.  We refer the reader to \cite{BB} for the definition of \defn{shellable}.

\begin{theorem}[\cite{Wil}] \label{thm:Wil}
The poset $\hBound(k,n)$ is thin, Eulerian, and shellable.
\end{theorem}
The weaker statement that $\Bound(k,n)$ is thin, Eulerian, and shellable (that is, every interval is shellable) follows from general results in Coxeter group theory and the fact that $\Bound(k,n)$ is dual to a convex subposet of a Bruhat order \cite{KLS}.

Lusztig \cite{Lusintro} showed that $\Gr(k,n)_{\geq 0}$ is contractible, and Postnikov, Speyer, and Williams \cite{PSW} showed that the stratification $\Gr(k,n)_{\geq 0} = \bigcup_f \Pi_{f,\geq 0}$ is a CW complex.  It follows from Theorem \ref{thm:Wil} and results of Bj\"orner \cite{Bjo} that $\hBound(k,n)$ is the face poset of some regular CW complex homeomorphic to a ball.  It is conjectured that the $\Gr(k,n)_{\geq 0} = \bigcup_f \Pi_{f,\geq 0}$ itself is a regular CW complex homeomorphic to a ball.  Rietsch and Williams \cite{RW} showed that this statement is true up to homotopy-equivalence.

\section{Positroid varieties}
\label{sec:Schub}
So far we have concerned ourselves with the combinatorics of planar bipartite graphs and the behavior of points in the TNN Grassmannian $\Gr(k,n)_{\geq 0}$.  However, to go further it is very helpful to be able to use the language of algebraic geometry.  This leads us to the study of the \defn{positroid varieties} that form a stratification of the complex Grassmannian $\Gr(k,n)$ \cite{KLS}.

\subsection{Schubert varieties}
We refer the reader to \cite{Ful} for the material of this section.  Let $I \in \binom{[n]}{k}$ be a $k$-element subset of $[n]$.  Let $F_\bullet = \{ 0 = F_0 \subset F_1 \subset \cdots F_{n-1} \subset F_n = \C^n\}$ be a flag in $\C^n$, so that $\dim F_i = i$.  
The Schubert cell $\oX_I(F_\bullet)$ is given by
\begin{equation}\label{eq:schub}
\oX_I(F_\bullet) \coloneqq \{X \in \Gr(k,n) \mid \dim(X \cap F_j) =\#(I \cap [n-j+1,n]) 
\text{ for all } j \in [n]\}.
\end{equation}
The Schubert variety $X_I(F_\bullet)$ is given by
\begin{equation}
X_I(F_\bullet) \coloneqq \{X \in \Gr(k,n) \mid \dim(X \cap F_j) \geq \#(I \cap [n-j+1,n]) 
\text{ for all } j \in [n]\}.
\end{equation}
We have $X_I(F_\bullet) = \overline{\oX_I(F_\bullet)}$.  Also, $X_{[k]}(F_\bullet) = \Gr(k,n)$ and $\codim(X_I(F_\bullet)) = i_1+i_2+\cdots+i_k - (1+2+\cdots+k)$, where $I = \{i_1,i_2\ldots,i_k\}$.  Here and elsewhere, we always mean complex (co)dimension when referring to complex subvarieties.

Let $E_\bullet$ be the standard flag defined by $E_i = \spn(e_n,e_{n-1},\ldots,e_{n-i+1})$.  Then we set the standard Schubert varieties to be $X_I \coloneqq X_I(E_\bullet)$.  Suppose $v_1,v_2,\ldots,v_n$ are the columns of a $k\times n$ matrix (with respect to the basis $e_1,e_2,\ldots,e_n$) representing $X \in \Gr(k,n)$.  Then the condition $\dim(X \cap E_j) = d$ is equivalent to the condition $\dim \spn(v_1,\ldots,v_{n-j}) = k-d$.  Thus the Schubert variety $X_I(E_\bullet)$ is cut out by rank conditions on initial sequences of columns of $X$.

\subsection{Positroid varieties}
Let the generator $\chi$ of the cyclic group $\Z/n\Z$ act on $[n]$ by the formula $\chi(i) = i+1 \mod n$ (cf. Section \ref{sec:TNNGrass}).  Then $\chi$ also acts on subsets of $[n]$.  For $\I=(I_1,\ldots,I_n) \in \binom{[n]}{k}^n$, define the \defn{open positroid variety} $\oPi_\I \subset \Gr(k,n)$ by 
\begin{equation}\label{eq:posdef}
\oPi_\I \coloneqq \oX_{I_1} \cap \chi(\oX_{\chi^{-1}(I_2)}) \cap \cdots \cap \chi^{n-1}(\oX_{\chi^{1-n}(I_n)}).
\end{equation}
If $f \in \Bound(k,n)$ then we set $\oPi_f = \oPi_{\I(f)}$, where $\I(f)$ is the Grassmann necklace of $f$.  For any $X \in \Gr(k,n)$, we have defined in \eqref{eq:fX} $f_X \in \Bound(k,n)$.  It follows from the definitions that $X \in \oPi_{f_X}$, and that $\oPi_\I$ is empty unless $\I$ is a Grassmann necklace.

\begin{proposition}
The subvariety $\oPi_\I$ is nonempty if and only if $\I$ is a Grassmann necklace.
\end{proposition}
\begin{proof}
Suppose $f \in \Bound(k,n)$.  We need to show that $\oPi_f$ is non-empty.  But this follows from our construction of points in $\Gr(k,n)_{\geq 0}$ (Theorem \ref{thm:main}).
%
\end{proof}

Define the \defn{positroid variety} $\Pi_f$ to be the Zariski closure of $\oPi_f$ in $\Gr(k,n)$.  It is shown in \cite{KLS} that 
\begin{equation*}
\Pi_\I = \X_{I_1} \cap \chi(\X_{\chi^{-1}(I_2)}) \cap \cdots \cap \chi^{n-1}(\X_{\chi^{1-n}(I_n)}).
\end{equation*}
From the definitions, we have $\Pi_{f,>0} = \oPi_f \cap \Gr(k,n)_{\geq 0}$, and $\Pi_{f,\geq 0} = \Pi_f \cap \Gr(k,n)_{\geq 0}$.

\begin{proposition}\label{prop:irred}
The positroid variety $\Pi_f$ is irreducible.
\end{proposition}
In \cite{KLS}, it is shown that $\Pi_f$ is the image of a Richardson variety $X_v^w \subseteq \Fl(n)$ under a projection map $\pi: \Fl(n) \to \Gr(k,n)$ from the full flag variety to the Grassmannian.  The irreducibility then follows from the fact that Richardson varieties are irreducible.

It is surprisingly difficult (at least for me) to prove Proposition \ref{prop:irred} directly.  Indeed, the intersection \eqref{eq:posdef} is usually not transverse, and ideals generated by Pl\"ucker coordinates are in general not prime.

\begin{theorem}\label{thm:Zariskidense}
$\Pi_f$ has codimension $\ell(f)$, and $\Pi_{f,>0}$ is a Zariski-dense subset of $\Pi_f$.  We have $\Pi_f = \bigsqcup_{g \geq f} \oPi_g$.
\end{theorem}
\begin{proof}
The first statement is proved in \cite{KLS} by the identification mentioned above of $\Pi_f$ with the projection $\pi(X_v^w)$ of a Richardson variety.
We have shown in Theorem \ref{thm:main} that $\Pi_{f,>0} \simeq \R^{k(n-k)-\ell(f)}$.  The Zariski closure of $\Pi_{f,>0}$ must thus be a subvariety of $\Pi_f$ with dimension at least $k(n-k)-\ell(f)$, which is equal to the dimension of $\Pi_f$.  Since $\Pi_f$ is irreducible by Proposition \ref{prop:irred}, the first claim follows.  For the second claim, the inclusion $\Pi_f \supseteq \bigsqcup_{g \geq f} \oPi_g$ follows from Theorem \ref{thm:partialorder}.  The reverse inclusion is proved in the same way as in Theorem \ref{thm:partialorder}.
\end{proof}

In fact, a stronger version of Proposition \ref{prop:irred} holds.
\begin{theorem}[\cite{KLS}] \label{thm:prime} Let $\M \in \PP(k,n)$ be a positroid.  Then the homogeneous ideal $\langle \Delta_I \mid I \in \M \rangle$ is a prime ideal.
\end{theorem}
We will return to this ideal in Section \ref{sec:coord}.  The proof of Theorem \ref{thm:prime} depends on the technology of Frobenius splittings which we do not discuss here; see \cite{KLS2}.  It would be interesting to give a direct proof of Theorem \ref{thm:prime}.

A projective variety $Y \subseteq \P^n$ is \defn{projectively normal} if it is normal and the restriction map $\Gamma(\P^n,\O(k)) \to \Gamma(Y,\O(k))$ is surjective for all $k$.

\begin{thm}[\cite{KLS}]\label{thm:normal}
Positroid varieties are projectively normal, Cohen-Macaulay, and have rational singularities.
\end{thm}

See also Billey and Coskun \cite{BC}.

In brief, positroid varieties are in general singular, but the singularities are relatively mild.  Projective normality will be the most important property for us.  A normal variety has a good theory of Weil divisors.  In particular, we have a well-behaved notion of the divisors of poles and of zeros of a rational function, or rational form on a positroid variety $\Pi_f$.  This will be important in Section \ref{sec:form}.  Also, in Section \ref{sec:coord} we will discuss the homogeneous coordinate ring of a positroid variety by restricting sections from the Pl\"ucker embedding.  Projective normality implies that the resulting graded ring is intrinsic to $\Pi_f$.

The singularities of positroid varieties will be important to us again in Section \ref{sec:form2}.

\section{Cohomology class of a positroid variety}\label{sec:cohom}
In this section, we describe the cohomology class of a positroid variety in terms of affine Stanley symmetric functions.  We follow \cite{KLS} and \cite{LamAffineStanley}.
\subsection{The cohomology ring of the Grassmannian}
We shall work with singular cohomology with integer coefficients. 
Let $X$ be a smooth complex projective variety, and $Y \subset X$ a closed irreducible subvariety.  Then we have a cohomology class $[Y] \in H^{2d}(X,\Z)$ where $d$ is the codimension of $Y$.  Recall that two subvarieties $Y,Z \subset X$ intersect transversally, if the intersection $Y \cap Z$ is smooth and each component has dimension $\dim(Y) + \dim(Z) - \dim(X)$. 

\begin{theorem}[{\cite[Appendix B]{Ful}}]\label{thm:Ful}
Let $X$ be a nonsingular variety.  Let $Y, Z \subset X$ be closed irreducible subvarieties.  Suppose $Y$ and $Z$ intersect transversally.  Then we have
$$
[Y] \cdot [Z] = [Y \cap Z]
$$
in the cohomology ring $H^*(X)$.
\end{theorem}
When $Y \cap Z$ is a finite set of $r$ (reduced) points, we have $[Y\cap Z] = r[\pt] \in H^*(X)$.

Let $E_\bullet$ be the standard flag in $\C^n$.  The cohomology ring $H^*(\Gr(k,n))$ vanishes in odd degrees, and the set $\{[X_I(E_\bullet)] \mid \codim(X_I) = d\}$ of Schubert classes forms a $\Z$-basis of $H^{2d}(\Gr(k,n))$. 

\subsection{Symmetric function realization}\label{sec:sym}
Let $\Lambda = \Lambda_\Z$ denote the ring of symmetric functions over $\Z$.  It has bases of monomial symmetric functions $m_\lambda$, homogeneous symmetric functions $h_\lambda$, and Schur functions $s_\lambda$, each of which are indexed by partitions $\lambda$.  We refer the reader to \cite{Mac,EC2} for background material on symmetric functions.

There is a bijection between partitions $\lambda \subseteq (n-k)^{k}$ contained in a $k \times (n-k)$ rectangle and subsets $I \in \binom{[n]}{k}$ given by $I(\lambda) = \{\lambda_k+1,\lambda_{k-1}+2, \ldots,\lambda_1+k\}$.  So for example $I(3,2,0) = \{1,4,6\}$ if $k = 3$.

The ring $H^*(\Gr(k,n))$ is isomorphic to the quotient of the ring $\Lambda$ of symmetric functions by an ideal $I_{k,n}$ (see \cite{Ful}).  Let $\eta: \Lambda_\Z \to H^*(\Gr(k,n),\Z)$ be the quotient map.  Then we have
$$
\eta(s_\lambda) = \begin{cases} [X_{I(\lambda)}] & \mbox{if $\lambda \subset (n-k)^k$,} \\
0 & \mbox{otherwise.}
\end{cases}
$$
We will often identify a symmetric function $f \in \Lambda$ with its image $\eta(f) \in H^*(\Gr(k,n))$.  Thus $[\Gr(k,n)] = s_{(0)}$ and $[\pt] = s_{(n-k)^k}$.  Let $\lambda^c$ denote the 180 degree rotation of the complement of $\lambda$ inside the $(n-k)^k$ rectangle.  Then $\lambda^c(J) = \lambda(I)$ where $I = J^c \coloneqq \{(n+1) -j \mid j \in  J\}$.  Inside $H^*(\Gr(k,n))$, we have the equality 
\begin{equation}\label{eq:dual}
s_\lambda \, s_{\mu} = \begin{cases} 1 & \mbox{$\mu=\lambda^c$} \\
0 & \mbox{otherwise}
\end{cases}
\end{equation}
for $|\lambda|+|\mu| = k(n-k)$.  

\subsection{Affine Stanley symmetric functions}\label{sec:affineStanley}
We use notation from Section \ref{sec:perms}.
An element $v \in W_n$ is called \defn{cyclically decreasing} if it has a reduced word $v = s_{i_1} s_{i_2} \cdots s_{i_k}$ such that $i_1, i_2, \ldots, i_k$ are distinct, and if both $i$ and $i+1$ occur then $i+1$ occurs before $i$.  For example, $s_4 s_3 s_1 s_0 s_6$ is cyclically decreasing if $n = 7$.  A \defn{cyclically decreasing factorization} of $v$ is a factorization $v = v_1 v_2 \cdots v_r$ where $\ell(v) = \ell(v_1) + \ell(v_2) + \cdots + \ell(v_r)$ and each $v_i$ is cyclically decreasing.  For $v \in W_n$, we define the \defn{affine Stanley symmetric function}
$$
\tF_v(x_1,x_2,\ldots) = \sum_{v = v_1 v_2 \cdots v_r} x_1^{\ell(v_1)} x_2^{\ell(v_2)} \cdots x_r^{\ell(v_r)}.
$$
It follows easily from the definitions that $\tF_v = \tF_w$ if $v$ is obtained from $w$ by the Coxeter automorphism that sends $s_i$ to $s_{i+r}$ for all $i$, and a fixed $r$.

Recall that $\id$ denotes the bounded affine permutation given by $\id(i) = i+k$ and each $(k,n)$-affine permutation $f \in \tS_n^k$ has an expressions as $f = \id v = w \id$ for $v,w \in W_n$.  The elements $v,w$ are related by the Coxeter automorphism that sends $s_i$ to $s_{i+k}$ for all $i$.  We define $\tF_f \coloneqq \tF_v = \tF_w$.

The basic result on affine Stanley symmetric functions is the following, generalizing work of Stanley \cite{Sta}.
\begin{theorem}\cite{LamAffineStanley}
For any $f \in \tS_n$, the generating function $\tF_f$ is a symmetric function.
\end{theorem}

The positroid variety $\Pi_f \subset \Gr(k,n)$ has a cohomology class $[\Pi_f] \in H^*(\Gr(k,n))$.  The following result further confirms that the bounded affine permutation $f \in \Bound(k,n)$ is the correct object to index a positroid variety $\Pi_f$.

\begin{theorem}[\cite{KLS}]\label{thm:KLS}
We have $[\Pi_f] \equiv \tF_f \in H^*(\Gr(k,n))$.
\end{theorem}

We do not prove Theorem \ref{thm:KLS} here.  The main steps in its proof (\cite{KLS,HL}) are: (1) an interpretation of $\tF_f$ as a cohomology class in the affine Grassmannian of $\GL(n)$ \cite{LamJAMS,book}; (2) the consideration of the torus-equivariant cohomology class of $[\Pi_f]$; and (3) a map that pulls back cohomology classes from the affine Grassmannian to $\Gr(k,n)$.

\begin{example}
We list the cohomology classes of all positroid varieties, up to cyclic rotation, of $\Gr(2,4)$.
\begin{center}
\begin{tabular}{|c|c|c|c|}
\hline
$f \in \Bound(2,4)$ & reduced word & $\tF_f \in \Lambda
$ & $[\Pi_f] \in H^*(\Gr(2,4))$ \\
\hline
$[3456]$ & $\id$ & $1$ &$1$ \\
$[3546]$ & $\id s_2$ &$s_1$ & $s_1$ \\
$[2547]$ &  $\id s_2 s_0$&  $s_{11}+s_{2}$ & $s_{11}+s_{2}$ \\
$[3564]$ & $\id s_2 s_3$ & $s_{11}$ &$s_{11}$\\
$[5346]$ & $\id s_2 s_1$ & $s_{2}$ & $s_{2}$ \\
$[5247]$ & $\id s_2 s_0 s_1$ & $s_{21}$ &  $s_{21}$ \\
$[5364]$ & $\id s_2 s_1 s_3$ & $s_{21}$ &  $s_{21}$\\
$[3654]$ & $\id s_2s_3s_2$ & $s_{21}$ &  $s_{21}$\\
$[5274]$ & $\id s_2s_0s_1s_3$  & $s_{22} +s_{211}- s_{1111}$ & $s_{22}$\\
$[5634]$ & $\id s_2s_1s_3s_2$ & $s_{22}$ & $s_{22}$\\
\hline
\end{tabular}
\end{center}

\end{example}

\subsection{The case $k = 1$}
Suppose $k = 1$.  Then positroid varieties are simply coordinate hyperspaces in $\P^{n-1} = \Gr(1,n)$.  Every $f \in \B(1,n)$ can be written in the form
$$
f = \id s_{[a_1,b_1]} s_{[a_2,b_2]} \cdots s_{[a_r,b_r]}
$$
where $s_{[a,b]} \coloneqq s_{b} s_{b-1} \cdots s_a$, and the $[a_i,b_i] \subsetneq [n]$ are disjoint and non-adjacent cyclic intervals.  It follows from the definition that
$$
\tF_f = h_{|[a_1,b_1]|} h_{|[a_2,b_2]|} \cdots h_{|[a_r,b_r]|} \equiv h_{\ell(f)} \mod I_{1,n},
$$
as expected.

\subsection{The case $k = 2$}\label{sec:k2}
We work out $\tF_f \in H^*(\Gr(2,n))$ completely in this section.  Let $X \in \Gr(2,n)$ be represented by a $2 \times n$ matrix with column vectors $v_1,v_2,\ldots,v_n \in \C^2$.  Positroid varieties are cut out by rank conditions of the form
$$
\rank( \spn(v_a,v_{a+1},\ldots,v_b)) \leq 1, \qquad \text{or} \qquad \rank( \spn(v_a,v_{a+1},\ldots,v_b)) = 0,
$$
for cyclic intervals $[a,b]$.  The latter condition just says that $v_a = v_{a+1} = \cdots = v_b = 0$.  Any two rank conditions of the first type for cyclic intervals $[a,b]$ and $[c,d]$ that overlap glue to give a rank condition of the same type on $[a,b] \cup [c,b]$.  It follows that a positroid variety is determined by setting $v_i = 0$ for $i \in A \subsetneq [n]$, and imposing that the vectors $\{v_a,v_{a+1},\ldots,v_b\}$ are parallel, for a non-trivial partition $[n] \setminus A = \bigcup_i [a_i,b_i]$ into disjoint cyclic intervals.  (The cyclic order on $[n] \setminus A$ is inherited from that of $[n]$.)

Let us say that $f \in \B(2,n)$ has type $(\alpha;\beta_1,\beta_2,\ldots,\beta_r)$ if $\alpha = |A|$ and $\beta_i = |[a_i,b_i]|$.  Here $\alpha \in [n-1]$ and $\beta_i \geq 1$, and $\alpha + \beta_1 + \cdots + \beta_r = n$, and $r \geq 2$.

For a partition $\lambda = (\lambda_1,\lambda_2)$, let $\lambda^{+\alpha} \coloneqq (\lambda_1 + \alpha,\lambda_2 + \alpha)$.

\begin{proposition}\label{prop:k2}
Suppose $f$ has type $(\alpha;\beta_1,\beta_2,\ldots,\beta_r)$.  Then
$$
\tF_f \equiv (h_{\beta_1 -1}h_{\beta_2 -1} \cdots h_{\beta_r-1} )^{+\alpha} \mod I_{2,n}
$$
where $p \mapsto p^{+\alpha}$ is the linear operator that is induced by $\lambda \mapsto \lambda^{+\alpha}$.
\end{proposition}
\begin{proof}
First, we consider the case $\alpha = 0$.  Let the partition of $[n]$ be into cyclic intervals $\pi_1,\pi_2,\ldots,\pi_r$.  Using \eqref{eq:fX}, we calculate that the bounded affine permutation $f$ is given by 
$$
f(i) = \begin{cases} i+1 & \mbox{$i+1$ belongs to the same part as $i$,} \\
i+\pi_a + 1 & \mbox{$i+1$ belongs to the part $\pi_a$.}
\end{cases}
$$
We then have an expression
$$
f = \id \, s_{\pi'_1} s_{\pi'_2} \cdots s_{\pi'_r}
$$
where if $\pi = [a,b]$ then $\pi'=[a-1,b-2]$.  It follows that
$$
\tF_f \equiv h_{\beta_1 -1}h_{\beta_2 -1} \cdots h_{\beta_r-1}  \mod I_{2,n},
$$
so the formula holds for $\alpha = 0$.  Now suppose $\alpha > 0$.  Then $\Pi_f$ is a positroid variety of the subGrassmannian $\Gr(2,V) \subset \Gr(2,n)$ where $V = \spn(e_i \mid i \in [n] \setminus A)$.  We can first calculate the cohomology class of $\Pi_f$ in $H^*(\Gr(2,V))$.  This also determines the homology class $[\Pi_f]_* \in H_*(\Gr(2,V))$, which we can pushforward via the injection $\iota: \Gr(2,V) \hookrightarrow \Gr(2,n)$.  Finally, this determines the cohomology class of $\Pi_f$ in $H^*(\Gr(2,n))$.  To see that the injection $\iota: \Gr(2,V) \hookrightarrow \Gr(2,n)$ induces the map $p \mapsto p^{+\alpha}$, we need only check what it does to Schubert classes.
\end{proof}

\section{Tableaux, promotion, and canonical bases}
\label{sec:tableaux}
\subsection{Highest weight representations}
A partition $\lambda = (\lambda_1 \geq \lambda_2 \geq \cdots \geq \lambda_\ell > 0)$ is a weakly decreasing sequence of positive integers.  We say that $\lambda = (\lambda_1 \geq \lambda_2 \geq \cdots \geq \lambda_\ell > 0)$ has $\ell$ parts and size $|\lambda| = \lambda_1 + \lambda_2 + \cdots + \lambda_\ell$.  We have the following dominance order on partitions: $\lambda \geq \mu$ if and only if $|\lambda| = |\mu|$ and $\lambda_1 \geq \mu_1$, $\lambda_1 + \lambda_2 \geq \mu_1+\mu_2$, and so on.

For a partition $\lambda$ with at most $n$ parts, we have an irreducible, finite-dimensional representation $V(\lambda)$ of $\GL(n)$ with highest weight $\lambda$.  We state some basic facts concerning $V(\lambda)$.

The Young diagram of $\lambda$ is the collection of boxes in the plane with $\lambda_1$ boxes in the first row, $\lambda_2$ boxes in the second row, and so on, where all boxes are upper-left justified.  A semistandard tableaux $T$ of shape $\lambda$ is a filling of the Young diagram of $\lambda$ by the numbers $1,2,\ldots,n$ so that each row is weakly-increasing, and each column is strictly increasing.  The weight $\wt(T)$ of a tableau $T$ is the composition $(\alpha_1,\alpha_2,\ldots,\alpha_n)$ where $\alpha_i$ is equal to the number of $i$-s in $T$.  For example, 
$$
\tableau[sbY]{
1&1&3&4&4 \\
2&3&4&5 \\
4&4
}
$$
is a semistandard tableau with shape $(5,4,2)$ with weight $(2,1,2,5,1)$.  Let $B(\lambda)$ denote the set of semistandard tableaux of shape $\lambda$.  (Note that this set depends on $n$, which is suppressed from the notation.)  The dimension $\dim(V(\lambda))$ is equal to the cardinality of $B(\lambda)$.  A vector $v$ in a $\GL(n)$-representation $V$ is called a weight vector with weight $(\alpha_1,\alpha_2,\ldots,\alpha_n)$ if the diagonal matrix $\diag(x_1,x_2,\ldots,x_n)$ sends $v$ to $(x_1^{\alpha_1} x_2^{\alpha_2} \cdots x_n^{\alpha_n})v$.

Let $U_q(\sl_n)$ denote the quantized enveloping algebra of $\sl_n$ and $V_q(\lambda)$ denote a highest weight representation.  Lusztig \cite{Luscan} and Kashiwara \cite{Kascan} have constructed a \defn{canonical basis}, or \defn{global basis} of the $U_q(\sl_n)$-module $V_q(\lambda)$.  We shall only use the evaluation of this basis at $q=1$, giving a basis of $V(\lambda)$.

\begin{quote}
There exists a basis $\{G(T) \mid T\in B(\lambda)\}$ of $V(\lambda)$ such that each $G(T)$ is a weight vector with weight $\wt(T)$.
\end{quote}

We shall also let $\{G(T)^* \mid T \in B(\lambda)\}$ denote the dual basis of $V(\lambda)^*$, called the \defn{dual canonical basis}.
%
%

\subsection{Promotion on rectangular tableaux}
Let $\omega_k = (1,1,\ldots, 1)$ be the partition with $k$ $1$'s.  Then $V(\omega_k)$ is isomorphic to the $k$-th exterior power $\Lambda^k(\C^n)$ of the standard representation $\C^n$ of $\GL(n)$.  For an integer $d \geq 1$, the representation $V(d\omega_k)$ for a rectangular partition has very special properties.  The set $B(d\omega_k)$ is the set of semistandard Young tableaux with $k$ rows and $d$ columns.  For example,
$$
\tableau[sbY]{
1&1&3&4&4 \\
2&3&4&5&5 \\
4&4&6&6&6
}
$$
belongs to $B(5\omega_3)$.  

The set $B(d\omega_k)$ has an additional operation called \defn{promotion}, which is a bijection $\chi:B(d\omega_k) \to B(d\omega_k)$.  Promotion is defined as follows: first remove all occurrences of the letter $n$ in $T$.  Then slide the boxes to the bottom right of the rectangle, always keeping the  rows weakly-increasing and columns strictly-increasing.   Once all slides are complete, we add one to all letters, and fill the empty boxes with the letter $1$ to obtain $\chi(T)$.  For example,
$$
\raisebox{0.6cm}{$T =$} \;\; \tableau[sbY]{
1&1&3&4&4 \\
2&3&4&5&5 \\
4&4&6&6&6
}
\;\;
\raisebox{0.6cm}{$\to$}
\;\;
\tableau[sbY]{
1&1&3&4&4 \\
2&3&4&5&5 \\
4&4
}
\;\;
\raisebox{0.6cm}{$\to$}
\;\;
\tableau[sbY]{
\bl&\bl&\bl&1&1 \\
2&3&3&4&4 \\
4&4&4&5&5
}
\;\;
\raisebox{0.6cm}{$\to$}
\;\;
\tableau[sbY]{
1&1&1&2&2 \\
3&4&4&5&5 \\
5&5&5&6&6
}
\;\;\raisebox{0.6cm}{$= \chi(T)$.}
$$

\begin{theorem}[\cite{Shi,Rho}]
The bijection $\chi:B(d\omega_k) \to B(d\omega_k)$ has order $n$.
\end{theorem}

\begin{example}
The action of $\chi$ cycles through the following six tableaux:
$$
\tableau[sbY]{
1&1&3&4&4 \\
2&3&4&5&5 \\
4&4&6&6&6
}
\;\;
\tableau[sbY]{
1&1&1&2&2 \\
3&4&4&5&5 \\
5&5&5&6&6
}
\;\;
\tableau[sbY]{
1&1&2&3&3 \\
2&2&4&5&5 \\
6&6&6&6&6
}
\;\;
\tableau[sbY]{
1&1&1&1&1 \\
2&2&3&4&4 \\
3&3&5&6&6
}
\;\;
\tableau[sbY]{
1&1&2&2&2 \\
2&2&3&3&5 \\
4&4&4&5&6
}
\;\;
\tableau[sbY]{
1&2&2&3&3 \\
3&3&3&4&4 \\
5&5&5&6&6
}
$$
\end{example}


\subsection{(Opposite) Demazure crystals}
Let $I = \{i_1< i_2< \cdots < i_k\} \in \binom{[n]}{k}$.  Define the tableau $T_I \in B(d\omega_k)$ to be the unique rectangular-shaped tableaux whose first row is filled with $i_1$, second row is filled with $i_2$, and so on.  

Define the \defn{Demazure subcrystal} $B_I(d\omega_k)$ to be the set of tableaux $T \in B(d\omega_k)$ such that $T(a,b) \geq T_I(a,b)$ for any cell $(a,b)$.  In other words, $T \in B_I(d\omega_k)$ if it is entry-wise greater than or equal to $T_I$.

\begin{example}\label{ex:Schubmat}
Suppose that $d = 1$.  Then $B(\omega_k)$ can be identified with the set $\binom{[n]}{k}$ of $k$-element subsets of $[n]$.  Then $B_I(\omega_k) = \{J \in \binom{[n]}{k} \mid I \leq J\}$ is simply the Schubert matroid $\S_I$.
\end{example}
\begin{example}
Suppose $n = 4$ and $I = \{1,3\}$.  Then $B(2\omega_2)$ consists of the following tableaux:
$$
\tableau[sbY]{
1&1\\
3&3
}
\;
\tableau[sbY]{
1&1\\
3&4
}
\;
\tableau[sbY]{
1&1\\
4&4
}
\;
\tableau[sbY]{
1&2\\
3&3
}
\;
\tableau[sbY]{
1&2\\
3&4
}
\;
\tableau[sbY]{
1&2\\
4&4
}
\;
\tableau[sbY]{
1&3\\
3&4
}
\;
\tableau[sbY]{
1&3\\
4&4
}
\;
\tableau[sbY]{
2&2\\
3&3
}
\;
\tableau[sbY]{
2&2\\
3&4
}
\;
\tableau[sbY]{
2&2\\
4&4
}
\;
\tableau[sbY]{
2&3\\
3&4
}
\;
\tableau[sbY]{
2&3\\
4&4
}
\;
\tableau[sbY]{
3&3\\
4&4
}
$$
\end{example}

\begin{remark}
The usual definition of the (opposite) Demazure crystal is to consider all tableaux $T = \tf_{j_1}\tf_{j_2} \cdots \tf_{j_r} \cdot T_I$  that can be obtained from $T_I$ by Kashiwara's lowering crystal operators $\tf_i$.  While it is not obvious, our definition agrees with the usual definition.
\end{remark}

\section{The homogeneous coordinate ring of a positroid variety}\label{sec:coord}
\subsection{Homogeneous coordinate ring of the Grassmannian}
Let $\hGr(k,n)$ denote the affine cone over the Grassmannian $\Gr(k,n)$.  A point $X \in \hGr(k,n)$ is determined by a set $\Delta_I(X)$ of Pl\"ucker coordinates satisfying the Pl\"ucker relations and we allow the possibility that all $\Delta_I(X)$ are simultaneously zero.  The subset $\hGr(k,n)_{\geq 0} \subset \hGr(k,n)$ is the set of points with nonnegative Pl\"ucker coordinates.  It is sometimes convenient to work with $\hGr(k,n)$ instead of $\Gr(k,n)$ because we can talk about functions on $\hGr(k,n)$.  For $\Gr(k,n)$ we can only talk about homogeneous coordinates, or sections of line bundles.


Let $R(k,n)$ denote the coordinate ring of $\hGr(k,n)$, or equivalently, the homogeneous coordinate ring $\bigoplus_{d=0}^\infty \Gamma(\Gr(k,n),\O(d))$ of $\Gr(k,n)$.  Thus,
$$
R(k,n) = \C[\Delta_I]/(\text{Pl\"ucker relations})
$$
is a graded ring where the degree of $\Delta_I$ is taken to be 1.  For example,
$$
R(2,4) = \C[\Delta_{12},\Delta_{13},\Delta_{14},\Delta_{23},\Delta_{24},\Delta_{34}]/(\Delta_{13}\Delta_{24} -\Delta_{12}\Delta_{34} - \Delta_{14}\Delta_{23}).
$$
We also note that $R(k,n)$ is a unique factorization domain.  In particular, a codimension one irreducible subvariety of $\hGr(k,n)$ is cut out by a single polynomial.

The degree $d$ component $R(k,n)_d = \Gamma(\Gr(k,n),\O(d))$ of the graded ring $R(k,n)$ is isomorphic, as a $\GL(n)$-representation, to the dual $V(d\omega_k)^*$ of the highest weight representation $V(d\omega_k)$.  

The multiplicative structure of $R(k,n)$ can be described as follows.  For two highest weights $\lambda$ and $\mu$, there is a natural inclusion of $\GL(n)$-modules $V(\lambda+\mu) \to V(\lambda) \otimes V(\mu)$.  In particular, we have a map
$$
\eta^{d,d'}_{k}: V((d+d')\omega_{k}) \rightarrow  V(d\omega_\ell) \otimes V(d'\omega_k).
$$
Under the identification $R(k,n)_d \simeq V(d\omega_k)^*$, this map is dual to the multiplication map $R(k,n)_d \otimes R(k,n)_{d'} \to R(k,n)_{d+d'}$.

\subsection{Temperley-Lieb invariants}
In Section \ref{sec:doubledimer} we introduced functions $F_{\tau,T}(N)$ of a planar bipartite network $N$.  Let $\hX(N) = \{\Delta_I(N) \mid I \in \binom{[n]}{k}\} \in \hGr(k,n)$ denote the point in the cone over the Grassmannian corresponding to $N$.  

\begin{proposition}[\cite{Lamweb}]\label{prop:tau}
The function $F_{\tau,T}(N)$ depends only on $\hX(N) \in \hGr(k,n)$ and thus descends to a function $F_{\tau,T}$ on $\hGr(k,n)$.  Furthermore, the set $\{F_{\tau,T} \mid (\tau,T) \in \AA_{k,n}\}$ forms a basis for $R(k,n)_2$.
\end{proposition}
\begin{proof}[Sketch of proof.]
Call $\Delta_{I_1}\Delta_{I_2}$ a \defn{standard monomial} if $I_1$ and $I_2$ form the columns of a semistandard tableaux.  The main calculation (see \cite{Lamweb} for details) is to check that the formula given in Theorem \ref{thm:TL} can be inverted, expressing $F_{\tau,T}(N)$ in terms of the standard monomials $\Delta_{I_1}(N)\Delta_{I_2}(N)$.  This proves the first statement.  The second statement follows from the fact that standard monomials form a basis for $R(k,n)_2$.
\end{proof}

It follows immediately from the definitions and Proposition \ref{prop:tau} that $F_{\tau,T}$ takes positive values on $\hGr(k,n)_{\geq 0}$.  A partial converse to this is also true: any weight vector in $R(k,n)_2$ that is nonnegative on $\hGr(k,n)_{\geq 0}$ is a nonnegative linear combination of the $F_{\tau,T}$.

\subsection{Dual canonical basis of the Grassmannian}
The dual canonical basis of Section \ref{sec:tableaux} gives rise to a basis of $R(k,n)_d$ with remarkable properties.  The following result will be established in \cite{Lam+}.  Part (3) is due to Lusztig \cite{LusTP} and (4) depends on a result of Rhoades \cite{Rho}.

\begin{theorem}\label{thm:canbasis}
The vector space $R(k,n)_d$ has a dual canonical basis $\{G(T)^* \mid T \in B(d\omega_k)\}$ with the following properties:
\begin{enumerate}
\item For $d = 1$, we have $G(T)^* = \Delta_I$, where $I$ is the set of entries in the one-column tableau $T$.
\item For $d= 2$, the set $\{G(T)^* \mid T \in B(2\omega_k)\}$ is exactly the set $\{F_{(\tau,T)} \mid (\tau,T) \in \AA_{k,n}\}$.
\item For any $T \in B(d\omega_k)$, the function $G(T)^*$ is a nonnegative function on $\hGr(k,n)_{\geq 0}$.
\item For any $T \in B(d\omega_k)$, we have $\chi^*(G(T)^*) = G(\chi(T))^*$, where $\chi^*$ is the pullback map induced by $\chi: \Gr(k,n) \to \Gr(k,n)$.
\item For $f \in \Bound(k,n)$ the vectors $G(T)^*$ that do not restrict to identically zero on $\Pi_f$ form a basis for the homogeneous coordinate ring of $\Pi_f$.
\item For $f \in \Bound(k,n)$, if $G(T)^*$ is not identically zero on $\Pi_f$, then it takes strictly positive values everywhere on $\Pi_{f,>0}$.
\end{enumerate}
\end{theorem}

We will make (5) much more explicit shortly.

The bijection $\theta:\AA_{k,n} \to B(2\omega_k)$ of Theorem \ref{thm:canbasis}(2) is given as follows.  Given $(\tau, T)$, the tableau $\theta(\tau,T)$ has columns $I_1,I_2$, where $I_1 \cap I_2 = T$, and for each strand $(a,b) \in \tau$ with $a<b$, we have $a \in I_1$ and $b \in I_2$.

\begin{example}
The bijection $\theta$ sends the following five non-crossing pairings in $\AA_{3,6}$
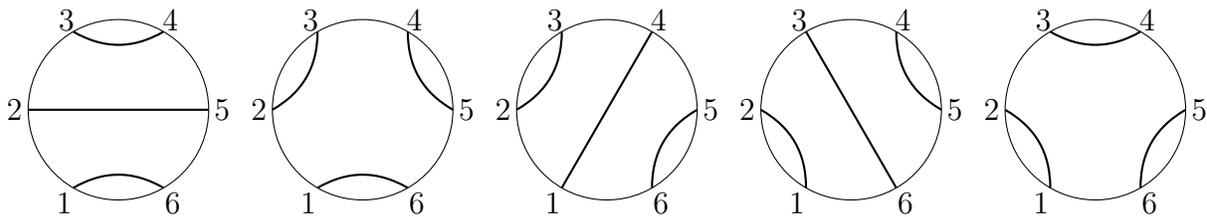
\begin{equation*}
\begin{tikzpicture}[scale=0.6,baseline=-0.5ex]
\coordinate (a4) at (60:2);
\coordinate (a3) at (120:2);
\coordinate (a2) at (180:2);
\coordinate (a1) at (-120:2);
\coordinate (a6) at (-60:2);
\coordinate (a5) at (0:2);

\node at (60:2.3) {$4$};
\node at (120:2.3) {$3$};
\node at (180:2.3) {$2$};
\node at (240:2.4) {$1$};
\node at (-60:2.4) {$6$};
\node at (0:2.3) {$5$};

\draw (0,0) circle (2);

\draw[thick] (a1) to [bend left] (a6);
\draw[thick] (a2) to (a5);
\draw[thick] (a3) to [bend right] (a4);
\end{tikzpicture}
\begin{tikzpicture}[scale=0.6,baseline=-0.5ex]
\coordinate (a4) at (60:2);
\coordinate (a3) at (120:2);
\coordinate (a2) at (180:2);
\coordinate (a1) at (-120:2);
\coordinate (a6) at (-60:2);
\coordinate (a5) at (0:2);

\node at (60:2.3) {$4$};
\node at (120:2.3) {$3$};
\node at (180:2.3) {$2$};
\node at (240:2.4) {$1$};
\node at (-60:2.4) {$6$};
\node at (0:2.3) {$5$};

\draw (0,0) circle (2);

\draw[thick] (a1) to [bend left] (a6);
\draw[thick] (a2) to [bend right] (a3);
\draw[thick] (a4) to [bend right] (a5);
\end{tikzpicture}
\begin{tikzpicture}[scale=0.6,baseline=-0.5ex]
\coordinate (a4) at (60:2);
\coordinate (a3) at (120:2);
\coordinate (a2) at (180:2);
\coordinate (a1) at (-120:2);
\coordinate (a6) at (-60:2);
\coordinate (a5) at (0:2);

\node at (60:2.3) {$4$};
\node at (120:2.3) {$3$};
\node at (180:2.3) {$2$};
\node at (240:2.4) {$1$};
\node at (-60:2.4) {$6$};
\node at (0:2.3) {$5$};

\draw (0,0) circle (2);

\draw[thick] (a1) to (a4);
\draw[thick] (a2) to [bend right] (a3);
\draw[thick] (a5) to [bend right] (a6);
\end{tikzpicture}
\begin{tikzpicture}[scale=0.6,baseline=-0.5ex]
\coordinate (a4) at (60:2);
\coordinate (a3) at (120:2);
\coordinate (a2) at (180:2);
\coordinate (a1) at (-120:2);
\coordinate (a6) at (-60:2);
\coordinate (a5) at (0:2);

\node at (60:2.3) {$4$};
\node at (120:2.3) {$3$};
\node at (180:2.3) {$2$};
\node at (240:2.4) {$1$};
\node at (-60:2.4) {$6$};
\node at (0:2.3) {$5$};

\draw (0,0) circle (2);

\draw[thick] (a1) to [bend right] (a2);
\draw[thick] (a3) to (a6);
\draw[thick] (a4) to [bend right] (a5);
\end{tikzpicture}
\begin{tikzpicture}[scale=0.6,baseline=-0.5ex]
\coordinate (a4) at (60:2);
\coordinate (a3) at (120:2);
\coordinate (a2) at (180:2);
\coordinate (a1) at (-120:2);
\coordinate (a6) at (-60:2);
\coordinate (a5) at (0:2);

\node at (60:2.3) {$4$};
\node at (120:2.3) {$3$};
\node at (180:2.3) {$2$};
\node at (240:2.4) {$1$};
\node at (-60:2.4) {$6$};
\node at (0:2.3) {$5$};

\draw (0,0) circle (2);

\draw[thick] (a1) to [bend right] (a2);
\draw[thick] (a3) to [bend right] (a4);
\draw[thick] (a5) to [bend right] (a6);
\end{tikzpicture}
\end{equation*}
to the five tableaux in $B(2\omega_3)$:
$$
\tableau[sbY]{
1&4\\
2& 5\\
3&6
}
\;\;
\tableau[sbY]{
1&3\\
2& 5\\
4&6
}
\;\;
\tableau[sbY]{
1&3\\
2& 4\\
5&6
}
\;\;
\tableau[sbY]{
1&2\\
3& 5\\
4&6
}
\;\;
\tableau[sbY]{
1&2\\
3& 4\\
5&6
}
$$
\end{example}

The following result can be found in \cite{PPR}.
\begin{theorem}
Under the bijection $\theta$, the obvious cyclic action on $\AA_{k,n}$ corresponds to the promotion operator on $B(2\omega_k)$.
\end{theorem}

\subsection{Demazure modules and Schubert varieties}
%
Let $\I(X_I) \subset R(k,n)$ denote the homogeneous ideal of the Schubert variety $X_I$ (see Section \ref{sec:Schub}) and let $\I(X_I)_d \subset R(k,n)_d$ denote the degree $d$ component.  Let $R(X_I)_d = \Gamma(X_I,\O(d))$ denote the degree $d$ part of the homogeneous coordinate ring of $X_I$.  Since sections on the Grassmannian restrict to sections on Schubert varieties, the space $R(X_I)_d$ is naturally a quotient of $R(k,n)_d=V(d\omega_k)^*$.

For $I \in \binom{[n]}{k}$, we have an extremal weight vector $G(T_I) \in V(d\omega_k)$.  The vector $G(T_I)$ spans the weight space of $V(d\omega_k)$ with weight $\alpha$ given by $\alpha_i = d$ if $i \in I$ and $\alpha_i = 0$ otherwise.  The \defn{(opposite) Demazure module} $V_I(d\omega_k)$ is defined to be the $B_-$ - submodule of $V_{d\omega_k}$ generated by the vector $G(T_I)$. 

The following result is due to Kashiwara \cite{Kasdem}.
\begin{theorem}\label{thm:Kas}
The $B_-$ - submodule $V_I(d\omega_k)$ has a basis $\{G(T) \mid T \in B_I(d\omega_k)\}$.
\end{theorem}

The following result is a consequence of Theorem \ref{thm:Kas}.
\begin{proposition}\label{prop:SchubDemazure} \
We have 
\begin{enumerate}
\item
$\I(X_I)_d = V_{I}(d\omega_k)^\perp \subset V(d\omega_k)^* = R(k,n)_d$ has a basis given by $\{G(T)^* \mid T \notin B_I(d\omega_k)\}$.
\item
$R(X_I)_d$ has a basis given by (the image of) $\{G(T)^* \mid T \in B_I(d\omega_k)\}$.
\end{enumerate}
\end{proposition}

\subsection{Cyclic Demazure modules and positroid varieties}
Let $f$ be a $(k,n)$-bounded affine permutation.  Define $\I(\Pi_f)_d \subset R(k,n)_d$ by
$$
\I(\Pi_f)_d\coloneqq \I(\Pi_f) \cap R(k,n)_d
$$
to be the degree $d$ homogeneous component of $\I(\Pi_f)$.  Since $\I(\Pi_f)$ is a homogeneous ideal, it is spanned by the subspaces $\I(\Pi_f)_d$.  The aim of this section is to give a representation-theoretic description of $\I(\Pi_f)_d$ as a subspace of $R(k,n)_d \simeq V(d\omega_k)^*$.

Let $f \in \Bound(k,n)$ have $(k,n)$-Grassmann-necklace $\I(f) = (I_1,I_2,\ldots,I_n)$.  Define the \defn{cyclic Demazure crystal} $B_f(d\omega_k)$ to be intersection 
$$
B_f(d\omega_k)\coloneqq B_{I_1}(d\omega_k) \cap \chi(B_{\chi^{-1}(I_2)}(d\omega_k)) \cap \cdots \cap \chi^{n-1}(B_{\chi^{1-n}(I_n)}(d\omega_k)).
$$
If we identify $B(\omega_k)$ with the set $\binom{[n]}{k}$ of $k$-element subsets of $[n]$, then by Example \ref{ex:Schubmat}, $B_f(\omega_k)$ is simply the positroid $\M(f)$.  Also, define the \defn{cyclic Demazure module} $V_f(d\omega_k)$ to be intersection
$$
V_f(d\omega_k) \coloneqq V_{I_1}(d\omega_k) \cap \chi(V_{\chi^{-1}(I_2)}(d\omega_k)) \cap \cdots \cap \chi^{n-1}(V_{\chi^{1-n}(I_n)}(d\omega_k)).
$$

Let $R(\Pi_f)$ denote the homogeneous coordinate ring of the positroid variety $\Pi_f$.  The following results will be established in \cite{Lam+}.

\begin{theorem}[\cite{Lam+}]
The subspace $V_f(d\omega_k)$ has a basis $\{G(T) \mid T \in B_f(d\omega_k)\}$.
\end{theorem}

\begin{theorem}[\cite{Lam+}]\label{thm:Posideal}\
\begin{enumerate}
\item
$\I(\Pi_f)_d$ is isomorphic to $V_f(d\omega_k)^\perp$ and has a basis given by $\{G(T)^* \mid T \notin B_f(d\omega_k)\}$.
\item
$R(\Pi_f)_d$ has a basis given by the images of $\{G(T)^* \mid T \in B_f(d\omega_k)\}$.
\end{enumerate}
\end{theorem}
\begin{example}
Suppose $k = 1$.  In this case $B_I(d\omega_1)$ is the set of one-row (of length $d$) tableaux with entries in $1,2,\ldots,i$, where $I = \{i\}$.  By choosing the $(1,n)$-Grassmann necklace appropriately, $B_f(\omega_1)$ can be arranged to be any subset of $\{1,2,\ldots,n\}$.  For example, if $n = 4$, $(I_1,I_2,I_3,I_4) = (1,3,3,1)$ gives $B_f(\omega_1) = \{1,3\}$.  The set $B_f(d\omega_1)$ is simply the set of one-row tableaux with entries in $B_f(\omega_1)$.
\end{example}

\begin{example}\label{ex:cyclicDem}
Take $k = 2$ and $n = 4$.  Let us consider the positroid variety $\Pi_f$ where $f = [2547] \in \B(2,4)$.  The Grassmann necklace is $\I(f) = (13,23,13,41)$.  The set $B_f(2\omega_2)$ is given by the set of tableaux
$$
\tableau[sbY]{
1&1\\
3&3
}
\;\;
\tableau[sbY]{
1&2\\
3&3
}
\;\;
\tableau[sbY]{
1&2\\
3&4
}
\;\;
\tableau[sbY]{
1&1\\
3&4
}
\;\;
\tableau[sbY]{
2&2\\
3&3
}
\;\;
\tableau[sbY]{
2&2\\
3&4
}
\;\;
\tableau[sbY]{
1&1\\
4&4
}
\;\;
\tableau[sbY]{
1&2\\
4&4
}
\;\;
\tableau[sbY]{
2&2\\
4&4
}
$$
The positroid cell $\Pi_{f,>0}$ is represented by the planar bipartite graph 
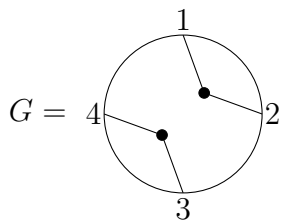
\begin{equation*} G=
\begin{tikzpicture}[baseline=-0.5ex,scale=0.7]
\node at (0,1.8) {$1$};
\node at (1.7,0) {$2$};
\node at (0,-1.8) {$3$};
\node at (-1.7,0) {$4$};
\draw (0,0) circle (1.5cm);
\draw (-1.5,0) -- (-0.4,-0.4);
\draw (1.5,0) -- (0.4,0.4);
\draw (0.4,0.4) -- (0,1.5);
\draw (-0.4,-0.4)-- (0,-1.5);
\blackdot{(-0.4,-0.4)}
\blackdot{(0.4,0.4)}
\end{tikzpicture}
\end{equation*}
Under the bijection $\theta: \AA_{k,n} \to B(2\omega_k)$ described after Theorem \ref{thm:canbasis}, the third tableau in $B_f(2\omega_2)$ is sent to the non-crossing matching
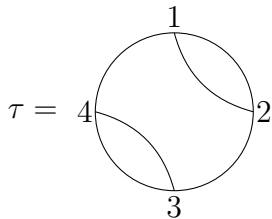
\begin{equation*} \tau = 
\begin{tikzpicture}[baseline=-0.5ex,scale=0.7]
\node at (0,1.8) {$1$};
\node at (1.7,0) {$2$};
\node at (0,-1.8) {$3$};
\node at (-1.7,0) {$4$};
\draw (0,0) circle (1.5cm);
\draw (-1.5,0) to [bend left] (0,-1.5);
\draw (1.5,0) to [bend left] (0,1.5);
\end{tikzpicture}
\end{equation*}
and so one can check from the definition that $F_{\tau,\emptyset}$ is non-vanishing (in fact, always positive) on $\Pi_{f,>0}$.  On the other hand, if $\tau' =\{(1,4),(2,3)\}$, then $F_{\tau',\emptyset}$ vanishes on $\Pi_{f,>0}$ since the graph $G$ has no Temperley-Lieb subgraphs with non-crossing matching $\tau'$.  We have that $\theta(\tau',\emptyset)$ is the tableau with columns $12$ and $34$, and this tableau is not in $B_f(2\omega_2)$, consistent with Theorem \ref{thm:Posideal}.
\end{example}

Since $B_f(\omega_k)$ is simply a positroid, Theorem \ref{thm:Posideal} is a higher degree analogue of Theorem \ref{thm:oh}.  Looking at whether dual canonical basis elements vanish or not is a higher degree analogue of the concept of a matroid.  

\begin{problem}
Find a formula for the character of $V_f(d\omega_k)$.  Equivalently, compute the weight generating function of $B_f(d\omega_k)$.
\end{problem}

For the bounded affine permutation $f = [2547]$ of Example \ref{ex:cyclicDem}, we have
$$
{\rm ch}(V_f(2\omega_2)) = x_1^2x_3^2 + x_1^2x_3x_4 + x_1^2 x_4^2+ x_1x_2x_3^2 + x_1x_2x_3x_4+ x_1x_2x_4^2 + x_2^2 x_3^2 + x_2^2x_3x_4 + x_2^2 x_4^2.
$$
%


It may seem from the above results that we might expect many ideals of subvarieties of the Grassmannian to have a basis given by a subset of the dual canonical basis, but this is not the case.  
\begin{example}\label{ex:13}
Let $X \subset \Gr(2,4)$ be given by the single equation $\{\Delta_{13} = 0\}$ (which is not a positroid variety).  Then the degree two part of $\I(X)$ has a one-dimensional weight space for the weight $(1,1,1,1)$.  It is spanned by the vector $\Delta_{13}\Delta_{24}$.  This vector is a sum of two elements of the dual canonical basis by Theorem \ref{thm:TL}.
\end{example}

%

Quantum versions of Grassmannians and Schubert varieties have been studied by many authors, see for example \cite{LR}.  In that setting, positroid varieties correspond to certain torus-invariant prime ideals, classified in \cite{MC,Y}.  

\begin{problem}
Find the quantum version of Theorem \ref{thm:Posideal}.
\end{problem}
Note however that the cyclic symmetry acts on the quantum Grassmannian in a more subtle way than it does on the Grassmannian \cite{LLtwist}.

%
%
%
%

\section{Canonical form}\label{sec:form}
Each positroid variety $\Pi_f$ has a distinguished rational differential top form $\omega_f$ with remarkable properties.  This differential form has simple (logarithmic) poles along the boundary $\partial \Pi_f \coloneqq \bigcup_{g > f} \Pi_f$, and no zeroes.  We will describe the rational form $\omega_f$ in an explicit combinatorial way, but we first begin with two more abstract descriptions.

For $X$ a normal variety, we say that $D$ is an anticanonical divisor on $X$ if $D \cap X_{\reg}$ is an anticanonical divisor on $X_{\reg}$, where $X_{\reg}$ denotes the smooth locus of $X$.  By Theorem \ref{thm:normal}, $\Pi_f$ is normal.  Let $\Pi_1,\Pi_2,\ldots,\Pi_r$ be the irreducible components of $\partial \Pi_f$.  In \cite{KLS2}, we showed that the divisor $\sum_{i=1}^r [\Pi_r]$ is anticanonical on $\Pi_f$.  In particular, there is a rational differential form $\omega_f$ whose divisor of poles is equal to $\sum_{i=1}^r [\Pi_r]$ (cf. \cite[Lemma 2.9]{LamWhittaker}).  The singular locus $\Pi_f - (\Pi_f)_{\reg}$ has codimension two in $\Pi_f$, and we may ignore it when considering poles or zeroes (which are codimension one phenomena).  The differential form $\omega_f$ so defined is unique up to scalar, since the ratio of two such forms would be a rational function on $\Pi_f$ with no poles or zeroes, and thus a constant.

The form $\omega_f$ is also natural from the point of view of cluster algebras.  The works of Leclerc \cite{Lec}, Muller and Speyer \cite{MS}, and Lenagan and Yakimov \cite{LY}, strongly suggest that the coordinate ring $\C[\oPi_f]$ of an open positroid variety is a cluster algebra.  Cluster varieties have (up to sign) a natural top form, which is the differential form $\frac{dx_1}{x_1} \wedge \cdots \wedge \frac{dx_n}{x_n}$ on any cluster torus with coordinates (that is, cluster variables) $(x_1,x_2,\ldots,x_n)$.  We will not discuss the cluster structure further, though it is certainly an important part of the story.

\medskip 

Let $G$ be a reduced planar bipartite graph with bounded affine permutation $f$.  A \defn{disconnected grove} of $G$ is a spanning subforest $F$ of $G$ such that every connected component of $F$ contains exactly one boundary vertex.  For a subset $E' \subset E$ and a collection of parameters $(t_e)_{e \in E'} \in \R_{>0}^{|E'|}$, let $N(x_e)$ be planar bipartite network with weights given by $t_e$, for $e \in E'$, and all other weights equal to 1.

\begin{lemma}\label{lem:E'}
Let $E' \subset E(G)$ of the edges of the $G$.  Then the following are equivalent:
\begin{enumerate}
\item
The complement $E(G) \setminus E'$ is a disconnected grove of $G$.
\item
The map $\phi_{E'}: (t_e)_{e \in E'} \in \R_{>0}^{|E'|} \mapsto X(N(x_e))$ is a homeomorphism onto $\Pi_{f,>0}$.
\end{enumerate}
\end{lemma}
\begin{proof}
Suppose (1) holds.  By Corollary \ref{cor:G}, it is enough to show that equation \eqref{eq:yF} maps $(t_e)_{e \in E'} \in \R_{>0}^{|E'|}$ homeomorphically onto $(\L_G)_{>0}$.  To see this, we proceed by induction on the number of faces of $G$.  If $G$ has a single face, then $(\L_G)_{>0}$ is a single point, the only disconnected grove of $G$ is $G$ itself, and $E'$ must be the empty set.  Thus the base case holds.  Now suppose the claim holds for all planar bipartite graphs with $a$ faces, and suppose that $G$ has $a+1$ faces.  It is easy to see that removing any edge $e \in E'$ from $G$ gives a graph $G'$ with one fewer face.  Suppose $F,F'$ are the faces of $G$ separated by $e$.  Then $y_F y_{F'} = y'_{F \cup F'}$, where $y$-s are the face weights of $G$ and $y'$-s are the face weights of $G'$.  If we know all face weights of $G'$ (by induction this is equivalent to knowing $t_{e'}$ for all $e' \in E \setminus \{e\}$), then the value of $t_e$ determines the face weights $y_F$ and $y_{F'}$, and conversely.  Thus (2) follows.
The proof that (2) implies (1) uses the same ideas.
\end{proof}

In fact, the map $\phi_{E'}: \R_{>0}^{|E'|} \to \Pi_{f,>0}$ extends to a birational isomorphism between $(\C^*)^{|E'|}$ and $\Pi_{f}$.  This follows from the fact that $\R_{>0}^{|E'|}$ (resp. $\Pi_{f,>0}$) is Zariski-dense in $(\C^*)^{|E'|}$ (resp. $\Pi_f$), and that the inverse of $\phi_{E'}$ is given by rational formulae.  
We can thus define a rational differential form of top degree
$$
\omega_{G} \coloneqq \prod_{e \in E'} \dlog t_e \coloneqq \prod_{e \in E'} \frac{dt_e}{t_e}
$$
on $\Pi_f$ via this birational isomorphism.  This form depends on an ordering of $E'$, but we shall only consider $\omega_{G}$ up to sign.  To see that $\omega_G$ does not depend on the choice of $E'$, we note that
$$
\omega_{G} = \pm \prod_{F} \frac{dy_F}{y_F}
$$
where the product is over all but one of the faces of $G$.  The equality follows from the fact that the transformation $(t_e) \mapsto (y_F)$ is an invertible monomial transformation (the proof of Lemma \ref{lem:E'} gives such an invertible monomial transformation).  Similarly, the map $\phi_{E'}$ only depends on $G$, so that we have a canonical map $\phi_G: (\C^*)^{\dim(\Pi_f)} \to \oPi_f \subseteq \Pi_f$.

Let $Y \subset X$ be an irreducible subvariety of codimension one.  Let $\omega$ be a rational form on $X$.  We now define the residue $\Res_Y \omega$ of $\omega$ along $Y$.  Suppose $X$ has local coordinates $h_1,h_2,\ldots,h_d$ and $Y$ is locally cut out by the equation $h_1= 0$.  Write $\omega = \frac{dh_1}{h_1} \wedge \omega'$, where $\omega'$ is of the form $g(h_1,h_2,\ldots,h_d) dh_2 \wedge dh_3 \wedge \cdots \wedge dh_d$ for a rational function $g$.  Then $\Res_Y \omega = \omega'|_Y$.  We refer the reader to \cite{GH} for further background on this.

\begin{theorem}\label{thm:Resform}
The rational form $\omega_f = \omega_G$ on $\Pi_f$ is, up to sign, independent of the choice of reduced planar bipartite graph $G$ representing $f \in \Bound(k,n)$.  This form has no zeroes, and it has simple poles on each $\Pi_{f'}$ where $f' \gtrdot f$.  Furthermore, $\Res_{\Pi_{f'}} \omega_f = \omega_{f'}$.
\end{theorem}
\begin{proof}
We first show that $\omega_G$ does not depend on $G$.  By Theorem \ref{thm:reduced}, if $G'$ is another reduced planar bipartite graph representing $f$, then $G'$ and $G$ are related by the moves (M1) and (M2).  It is easy to see that the move (M2) does not change $\omega_G$.  Let us consider the move (M1).  We are free to choose $E'$ as we desire, and we can pick $E'$ to contain the four edges (with weights $a,b,c,d$) surrounding the square face of used in (M1), see Section \ref{sec:relations}.  We then check that
$$
\dlog a \wedge \dlog b \wedge \dlog c \wedge \dlog d = \pm D^4 \dlog a' \wedge \dlog b' \wedge \dlog c' \wedge \dlog d'
$$
where $a',b',c',d'$ are given by \eqref{eq:square}, and $D = (ac+bd)$.  The factor $D^4$ is to account for the fact that the two graphs shown in (M1) have Pl\"ucker coordinates that differ by a factor of $D$ (even though they are the same point in the Grassmannian).  Thus we have a well-defined form $\omega_f$.

Now suppose that $f' \gtrdot f$, and $G$ is a reduced planar bipartite graph with no degree two vertices representing $f$.  From the proof of Theorem \ref{thm:main} we know that there is an edge $e$ of $G$ such that removing $e$ gives a reduced planar bipartite graph $G'$ representing $f'$.  Note that the number of faces of $G'$ is one less than the number of faces of $G$, so that we can pick $E' \subset E(G)$ satisfying the conditions of Lemma \ref{lem:E'} containing the edge $e$.  There is thus a morphism $\C \times (\C^*)^{|E'|-1} \to \Pi_f$, where the (distinguished) first coordinate is $t_e$, and $\{0\} \times (\C^*)^{|E'|-1}$ is sent to $\Pi_{f'}$.  Thus, in the local coordinates $(t_e)_{e \in E'}$, the subvariety $\Pi_{f'}$ is cut out by the equation $t_e = 0$.  By definition, we have
$$
\Res_{\Pi_{f'}} \omega_f = \prod_{e' \in E' \setminus{e}} \dlog t_e = \omega_{f'}.
$$
This proves the last statement of the Theorem.

It is clear that $\omega_f$ has no poles or zeroes on $(\C^*)^{\dim(\Pi_f)}$, and thus no poles or zeroes on the image of $\phi_G$, for any $G$.  Let $Z \subset \oPi_f$ be the union of the images of $\phi_G$.  To complete the proof it would suffice to show that $\oPi_f \setminus Z$ is codimension two in $\oPi_f$, for then all polar and zero divisors can be detected on $Z$.  For the case, $\Pi_f = \Gr(k,n)$ this statement is shown in \cite{Sco}.  In general, we expect this follows easily from the connection with cluster algebras \cite{Lec,LY,MS}.  

We sketch a roundabout argument.  First suppose $f = \id$ so $\Pi_f = \Gr(k,n)$.  Then the fact that $\omega_{\id}$ has no other poles or zeroes follows from the an alternative description of the form given in Proposition \ref{prop:OmegaGr}.  Let $\omega'_{\id}$ be the rational form on $\Pi_f$ from \cite{KLS} described in the beginning of this section.  Since $\omega_{\id}$ and $\omega'_{\id}$ have the same poles and zeroes, they must be equal up to a constant.  But it also follows from \cite{KLS} that $\omega'_{f'} = \Res_{\Pi_f} \omega'_f$ whenever $f' \gtrdot f$.  Thus $\omega_f$ and $\omega'_f$ must be equal up to a scalar for all $f \in \Bound(k,n)$.  The claim about poles and zeroes follows.
\end{proof}

Consider the rational form
$$
\eta = \frac{d^{k \times n}C}{\Delta_{12\cdots k}(C)\Delta_{2\cdots k(k+1)}(C) \cdots \Delta_{n12\cdots(k-1)}(C)}
$$
on the space $\Mat(k,n)$ of $k \times n$ matrices $C$.  Here, if $C =(c_{i,j})$ then $d^{k\times n}C = \prod_{i,j} dc_{i,j}$.  The form $\eta$ is $\GL(k)$-invariant: for $g \in \GL(k)$ acting as a map $g: \Mat(k,n) \to \Mat(k,n)$, we have $g^*\omega = \omega$.  We thus have a rational form 
$$
\omega = \frac{d^{k \times n}C/\GL(k)}{\Delta_{12\cdots k}\Delta_{2\cdots k(k+1)} \cdots \Delta_{n12\cdots(k-1)}}
$$
on $\Gr(k,n)$ (the quotient of the dense subset of full-rank $k \times n$ matrices by $\GL(k)$).  Concretely, we consider the affine chart $\Omega_{[k]}$ (see Section \ref{sec:Gr}).  Represent a point $X \in \Omega_{[k]}$ by a $k\times n$ matrix with an identity matrix in the first $k$ columns.  Let $\{x_{a,b} \mid (a,b) \in \{1,2,\ldots,k\}\times\{k+1,\ldots,n\}\}$ be the coordinates of the remaining entries.  Then
$$
\omega = \frac{\prod_{a,b} dx_{a,b}}{\Delta_{12\cdots k}(X)\Delta_{2\cdots k(k+1)}(X) \cdots \Delta_{n12\cdots(k-1)}(X)}
$$
and this form does not depend on our choice of affine chart.

\begin{proposition}\label{prop:OmegaGr}
We have $\omega_{\id} = \pm \omega$.
\end{proposition}
\begin{proof}
We work on the affine chart $\Omega_{[k]}$.  Use the ``rectangular grid" planar bipartite network $N$ representing the top cell $\Pi_{\id,>0}$ of the Grassmannian $\Gr(k,n)$, and call the face weights $y_{i,j}$ for $i =1,2,\ldots,k$ and $j = k+1,\ldots,n$ (see \eqref{eq:yF}).  Below is the network $N$ for $k = 3$ and $n = 8$.  

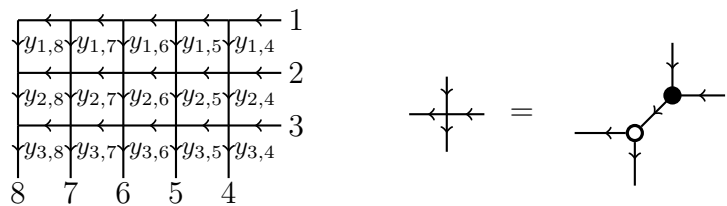
\begin{center}
\begin{tikzpicture}
\begin{scope}[very thick,decoration={
    markings,
    mark=at position 0.5 with {\arrow{>}}}
    ] 
\foreach \i in {0,0.7,1.4,2.1,2.8} {
   \foreach \j in {0,-0.7,-1.4} {
        \draw [thick,postaction=decorate] (\i,\j) -- ($(\i,\j)-(0.7,0)$);
       
    }
} 
\foreach \i in {4,5,6,7,8} {
   \foreach \j in {1,2,3} {
        \node at ($(-\i*0.7,\j*-0.7)-(3.15,-0.35)+(8.4,0)$) {{\footnotesize $y_{\j,\i}$}};
     }
}     
\foreach \i in {-0.7,0,0.7,1.4,2.1} {
   \foreach \j in {0,-0.7,-1.4} {
        \draw [thick,postaction=decorate] (\i,\j) -- ($(\i,\j)-(0,0.7)$);      
    }
}
\node at (3,0) {$1$};  
\node at (3,-0.7) {$2$};  
\node at (3,-1.4) {$3$};  
\node at (2.1,-2.3) {$4$};  
\node at (1.4,-2.3) {$5$};  
\node at (0.7,-2.3) {$6$};  
\node at (-0,-2.3) {$7$};  
\node at (-0.7,-2.3) {$8$};  

\begin{scope}[shift={(5,-1)}]
\begin{scope}[shift={(0,-0.25)}]
\draw[thick,postaction=decorate] (0.5,0)--(0,0);
\draw[thick,postaction=decorate] (0,0) -- (-0.5,0);
\draw[thick,postaction=decorate] (0,0.5)--(0,0);
\draw[thick,postaction=decorate] (0,0) -- (-0,-0.5);
\node at (1,0) {$=$};
\end{scope}
\draw[thick,postaction=decorate] (3,0.7)--(3,0);
\draw[thick,postaction=decorate] (3.7,0)--(3,0);
\draw[thick,postaction=decorate] (3,0)--(2.5,-0.5);
\draw[thick,postaction=decorate] (2.5,-0.5)--(2.5,-1.2);
\draw[thick,postaction=decorate] (2.5,-0.5)--(1.7,-0.5);
\blackdot{(3,0)}
\whitedot{(2.5,-0.5)}
\end{scope}

\end{scope}
\end{tikzpicture}
\end{center}

The orientation shown above gives $N$ the structure of an acyclic perfectly oriented network $(\tilde N,O)$ in the sense of Section \ref{sec:flows}.  A flow in $(\tilde N,O)$ is simply a family of non-intersecting paths from the sources $\{1,2,\ldots,k\}$ to the sinks $\{k+1,\ldots,n\}$.  Gauge fixing the edge weights appropriately, the weight of a path in $(\tilde N,O)$ is simply the product of the face weights $y_{i,j}$ over all the faces ``under" (that is, to the bottom right of) the path.

Let $x_{a,b}$ be the $(a,b)$-entry of the representative of $X = X(\tilde N,O) \in \Gr(k,n)$ with the identity matrix in the first $k$ columns.  Let $Y_{a,b} = \prod_{k \geq i \geq a \text{ and } k+1 \leq j \leq b} y_{i,j}$.  Then by Theorem \ref{thm:flow}
$$
x_{a,b} = Y_{a,b} + \text{ other terms }
$$
where the other terms do not involve $y_{a,b}$.   Using the fact that $dy_{i,j} \wedge dy_{i,j} = 0$, we have that
$$
\prod_{(a,b) \in [1,k]\times [k+1,n]}dx_{a,b}
= \pm \prod_{(a,b) \in [1,k]\times [k+1,n]} dY_{a,b} 
= \pm \prod_{(a,b) \in [1,k]\times [k+1,n]}Y_{a,b} \dlog y_{a,b}.
$$
Let $I = \{i,i+1,\ldots,k+i-1\}$.  Then $\Delta_{I}(\tilde N,O)$ is a weighted sum of families of non-intersecting paths from sources $A = [k] \setminus I $ to sinks $B = [k+1,n] \cap I$.  There is only one such non-intersecting path family, and it has weight equal to $ Y_{a_1,b_1} Y_{a_2,b_2} \cdots Y_{a_r,b_r}$, where $A = \{a_1<a_2<\cdots<a_r\}$ and $B = \{b_1<b_2< \cdots < b_r\}$.  Note that each $Y_{a,b}$ occurs exactly once in such a product.  Thus
\begin{align*}
\omega &= \frac{1}{\Delta_{12\cdots k}\Delta_{2\cdots k(k+1)} \cdots \Delta_{n12\cdots(k-1)}}\prod_{(a,b) \in [1,k]\times [k+1,n]}dx_{a,b} \\
&= \pm \prod_{(a,b) \in [1,k]\times [k+1,n]} \dlog y_{a,b} = \pm \omega_G = \pm \omega_{\id}. \qedhere
\end{align*} 
\end{proof}

\begin{remark}
The singular cohomology $H^d(\oPi_f,\C)$ is one-dimensional, where $d$ is the dimension of $\Pi_f$.  The canonical form $\omega_f$ spans this cohomology group.  The singular cohomology groups $H^i(\oPi_f,\C)$ for $i < d$ are also very interesting \cite{LamSpeyer}.
\end{remark}

\section{Relation space of a graph}\label{sec:relspace}

In this section, we describe a way to obtain a point $\Rel(N)$ in the Grassmannian from a bicolored network $N$ using only linear algebra.  This construction is closely related to the ``on-shell diagrams" in the physics literature; see \cite{ABCGPT,EH} and the references therein.  While it is certainly expected by experts, I could not find in the literature a description of the precise relationship between $\Rel(N)$ and the point $X(N)$ constructed by enumerating matchings.  Indeed, there are some subtle sign issues.

One advantage of this approach over the perfect matching approach is that one obtains a point in the Grassmannian for nonplanar bicolored networks, with no additional work.

This section does not play a big role in the rest of this article, and can be safely skipped on first reading.

\subsection{Definition of the relation space}
In this section we will work with the following version of bicolored networks.  Let $G$ be a bicolored graph with no isolated vertices.  Let $\F$ be a field.  A bicolored network $N$ associates to each oriented edge $e = (u,v)$ of $G$ a weight $w(u,v) \in \F^*$ satisfying the condition that
$$
w(u,v) w(v,u) = 1.
$$
Since $w(u,v)$ and $w(v,u)$ determine each other, we will often think of the two as a single ``edge weight".  Also it makes sense to say that an edge has weight $1$ or $-1$, without specifying an orientation.

Associate a formal variable $z_{(u,e)}$ to each half-edge $(u,e)$.  Abusing notation, when there are no multiple edges, we identify half-edges with oriented edges, so that if $e = (u,v)$ we have $z_{(u,v)}:=z_{(u,e)}$.  If $e = (u,v)$ is an edge, then we impose the condition that
\begin{equation}\label{eq:zue}
w(u,v) \, z_{(v,e)}= z_{(u,e)}.
\end{equation}
To each black vertex $v$ in $N$, we associate the equations
\begin{equation}\label{eq:zee'}
z_{(v,e)} = z_{(v,e')}
\end{equation}
for every pair of edges $e,e'$ incident to $v$.  
To each white vertex $u$ in $N$, we associate the equation
\begin{equation}\label{eq:zue2}
\sum_{e \text{ incident to } u} \;z_{(u,e)} = 0.
\end{equation}

Let $S(N)$ denote the system of all these linear equations in the variables $\{z_{u,e}\}$, as we consider all vertices of $N$.  For each boundary vertex $i$, let $z_i \coloneqq z_{(i,e_i)}$ where $e_i$ is the unique edge connected to $i$.  Define $\Rel(N)$ to the space of relations on $z_1,z_2,\ldots,z_n$ induced by $S(N)$.  More precisely, consider each equation in $S(N)$ to be a vector in $\F^{2|E(N)|}$, where $2|E(N)|$ is equal to the number of half edges in $N$.  Let $V \subset \F^{2|E(N)|}$ be the subspace where the only non-zero coordinates are the ones indexing the half-edges $(i,e_i)$.  Then we have $\Rel(N) \coloneqq \spn(S(N)) \cap V$ is the space of relations on $z_1,z_2,\ldots,z_n$ that do not mention the interior half-edges.  

Let us compute an estimate on the dimension of $\Rel(N)$.  There are two variables $z_{(v,e)}$ and $z_{(u,e)}$ for each edge $e$.  There is one relation \eqref{eq:zue} per edge.  There are $\deg(v)-1$ relations \eqref{eq:zee'} per black vertex.  There is one relation \eqref{eq:zue2} per white vertex.  Thus the expected dimension of $\Rel(N)$ is equal to
$k_N = \sum_{v \text{ black}} (\deg(v)-1) + \sum_{v \text{ white}} 1 - \#\text{interior edges}.
$
This can also be written in the more black-white symmetric form
\begin{equation}\label{eq:kG}
k_N \coloneqq \frac{1}{2} \left(n + \sum_{v \text{ black}} (\deg(v)-2) + \sum_{v \text{ white}} (2 -\deg(v)) \right)
\end{equation}
which has no mention of the number of interior edges.  This is identical to the formula \eqref{eq:kO}.  
We shall consider $\Rel(N)$ to be a point in the Grassmannian $\Gr(k_N, V) = \Gr(k_N, \F^n)$.  If $\dim \Rel(N) \neq k_N$, we shall instead declare $\Rel_N$ to be undefined.

\begin{example}
Consider a network $N$ with four boundary vertices $1,2,3,4$ and two interior vertices, one black $v$ and one white $u$.  Suppose we have the following edges: $1 - v$, $2-v$, $v-u$, $3-u$, $4-u$.  Then $\deg(v) = \deg(u) = 3$, and $k_N = 2$.  Let the variables at the boundary vertices be $z_1,z_2,z_3,z_4$, and set $z\coloneqq z_{(v,u)}$.  Then the two interior vertices $v$ and $u$ give the equations
\begin{align*}
\beta_1 z_1 = \beta_2 z_2 = z
\qquad \text{and} \qquad
\beta_3 z_3 + \beta_4 z_4 + \gamma z = 0
\end{align*}
respectively.  Here $\beta_i$ come from the weights of the boundary edges, and $\gamma = w(u,v)$.  Cancelling $z$, we obtain $\beta_1z_1 = \beta_2 z_2 = -(1/\gamma)(\beta_3z_3 + \beta_4 z_4)$, assuming $\gamma \neq 0$.  Thus $\Rel(N) \in \Gr(2,4)$ is represented by the matrix
$$
\begin{bmatrix}
\beta_1 & - \beta _2 & 0 & 0 \\
\beta_1 \gamma & 0 &-\beta_3 & -\beta_4\\
\end{bmatrix}.
$$
It is important to note that the construction does not depend on any planar embedding of $N$.  For non-zero values of $\beta_i$ and $\gamma$, we have $\Rel_N \in \Pi_{[3,5,4,6]}$.
\end{example}

\subsection{Moves preserving $\Rel(N)$}
\label{sec:relmoves}
We first discuss operations on a bicolored graph that do not change $\Rel(N)$.  It is helpful to compare this discussion to Postnikov's moves on plabic graphs \cite{Pos}.

\subsubsection{Gauge equivalence}
Fix an interior vertex $u$. Let $N'$ be obtained from $N$ by scaling $w(u,v)$ by a fixed $c \in \F$, for all $v$ adjacent to $u$. 

\subsubsection{Degree two vertex removal}
Suppose $u$ is an interior vertex of degree two, and let $e_1 = (u,v_1)$ and $e_2 = (u,v_2)$ be the two vertices adjacent to it.  Let $N'$ be obtained from $N$ by removing $u$, and replacing $e_1$ and $e_2$ with a single edge $(v_1,v_2)$ with weight $w(v_1,v_2) = \pm w(v_1,u) w(u,v_2)$, where we take the plus sign if $u$ is black and the minus sign if $u$ is white.

\subsubsection{Gluing and separting vertices of the same color}
Suppose $u$ and $v$ are interior vertices with the same color and are joined by an edge $(u,v)$.  By applying gauge equivalences we can assume that $w(u,v) = 1 = w(v,u)$.  Let $N'$ be obtained from $N$ by removing $(u,v)$ and identifying $u$ and $v$.  If $u$ and $v$ are white, in addition we multiply all edge weights of edges that were incident to $u$ by $-1$.  (By gauge equivalence we could also choose to multiply all edge weights of edges that were incident to $v$ by $-1$.)  

\subsubsection{Square move}
Suppose we have a square of two white $w_1,w_2$ and two black $b_1,b_2$ trivalent interior vertices as arranged in Figure \ref{fig:relsquare}.  Let the edge weights $w(w_j,b_i)$ be denoted $w_{ij}$.  Also write $z_{ij}\coloneqq z_{(b_i,w_j)}$ and $z_{b_i}$ (resp. $z_{w_i}$) for the formal variable associated to the external half-edge attached to $b_i$ (resp. $w_i$).

Then the four sets of equations are 
\begin{align*}
z_{b_1} = z_{12} &= z_{11} \\
z_{w_2} + w_{12}z_{12}+w_{22}z_{22} &= 0 \\
z_{b_2} = z_{21} &= z_{22} \\
z_{w_1} + w_{11}z_{11} + w_{21}z_{21} &= 0.
\end{align*}
Set $W = w_{11} w_{22}-w_{12} w_{21}$.  These equations induce the same relations on $z_{b_1},z_{w_2},z_{b_2},z_{w_1}$ as
\begin{align*}
z_{b_1}  + \frac{-w_{21}}{W}z'_{21} + \frac{w_{22}}{W} z'_{11} &= 0 \\
z_{w_2} = z'_{21} &= z'_{22}\\
z_{b_2}  + \frac{w_{11}}{W}z'_{22} + \frac{-w_{12}}{W} z'_{12} &= 0 \\
z_{w_1} = z'_{11} &= z'_{12} 
\end{align*}
Draw a new square with two white $w'_1,w'_2$ and two black $b'_1,b'_2$ vertices, so that $w'_i$ (resp. $b'_i$) is connected to the outside in the same way $b_i$ (resp. $w_i$ used to be).  Set the edge weights of the square by
$$
w'_{11}=  \frac{w_{22}}{W} \qquad w'_{12}=  -\frac{w_{12}}{W}  \qquad w'_{21}=  -\frac{w_{21}}{W}  \qquad w'_{22}=  \frac{w_{11}}{W} 
$$
where $w'_{ij}:=w(w'_j,b'_i)$.  Call this new bicolored graph $N'$.  Assuming that $W \neq 0$ (which always holds if the edge weights of $N$ are algebraically independent), we have $\Rel(N) = \Rel(N')$.  (In Figure \ref{fig:relsquare} we have drawn the graph as planar, but the planar embedding is not part of the data of a bicolored graph.)

\begin{figure}
\begin{center}
\begin{tikzpicture}[scale=1.5]
\begin{scope}[thick,decoration={
    markings,
    mark=at position 0.5 with {\arrow{>}}}
    ] 
\draw[postaction={decorate}] (1,0)-- node[above] {$w_{12}$} (0,0);
\draw[postaction={decorate}]  (1,0) -- node[right] {$w_{22}$} (1,-1);
\draw[postaction={decorate}] (0,-1)-- node[below] {$w_{21}$} (1,-1);
\draw[postaction={decorate}] (0,-1)-- node[left] {$w_{11}$} (0,0);
\draw (0,0) -- (-0.5,0.5);
\draw (1,0) -- (1.5,0.5);
\draw (1,-1) -- (1.5,-1.5);
\draw (0,-1) -- (-0.5,-1.5);

\node at (0,0.2) {$b_1$};
\node at (1,0.2) {$w_2$};
\node at (0,-1.2) {$w_1$};
\node at (1,-1.2) {$b_2$};

\filldraw[black] (0,0) circle (0.1cm);
\filldraw[black] (1,-1) circle (0.1cm);
\filldraw[white] (1,0) circle (0.1cm);
\draw (1,0) circle (0.1cm);
\filldraw[white] (0,-1) circle (0.1cm);
\draw (0,-1) circle (0.1cm);
\end{scope}

\begin{scope}[thick,decoration={
    markings,
    mark=at position 0.5 with {\arrow{<}}}
    ] 
\begin{scope}[shift={(4,0)}]
\draw[postaction={decorate}] (1,0)-- node[above] {$w'_{21}$} (0,0);
\draw[postaction={decorate}]  (1,0) -- node[right] {$w'_{22}$} (1,-1);
\draw[postaction={decorate}] (0,-1)-- node[below] {$w'_{12}$} (1,-1);
\draw[postaction={decorate}] (0,-1)-- node[left] {$w'_{11}$} (0,0);
\draw (0,0) -- (-0.5,0.5);
\draw (1,0) -- (1.5,0.5);
\draw (1,-1) -- (1.5,-1.5);
\draw (0,-1) -- (-0.5,-1.5);

\node at (0,0.2) {$w'_1$};
\node at (1,0.2) {$b'_2$};
\node at (0,-1.2) {$b'_1$};
\node at (1,-1.2) {$w'_2$};

\filldraw[black] (0,-1) circle (0.1cm);
\filldraw[black] (1,0) circle (0.1cm);
\filldraw[white] (0,0) circle (0.1cm);
\draw (0,0) circle (0.1cm);
\filldraw[white] (1,-1) circle (0.1cm);
\draw (1,-1) circle (0.1cm);
\end{scope}
\end{scope}

\end{tikzpicture}
\end{center}
\caption{}
\label{fig:relsquare}
\end{figure}

\subsubsection{Parallel edge reduction}
Suppose $u$ and $v$ are interior vertices of different colors connected by two edges $e_1$ and $e_2$, with weights $w_1$ and $w_2$ when oriented from white to black.  Assuming $w_1 + w_2 \neq 0$, let $N'$ be obtained from $N$ by replacing $e_1$ and $e_2$ by a single edge $e$ with weight $w_1+w_2$ when oriented from white to black.  

\subsubsection{Leaf removal}
Suppose $u$ is an interior leaf, joined to a vertex $v$ of the other color.  Suppose the other half-edges incident to $v$ are $(v,e_1),(v,e_2),\ldots,(v,e_r)$.  Let $N'$ be obtained from $N$ by removing $u$ and $v$, creating new vertices $x_1,x_2,\ldots,x_r$ with the same color as $u$, and replacing the half-edge $(v,e_i)$ by $(x_i,e_i)$.  (Note that each $(x_i,e_i)$ is itself a leaf, so by gauge equivalences, the weight of the incident edge does not matter.)


\subsubsection{Dipole removal}
Suppose $u$ and $v$ are interior degree one vertices joined be an edge, and they are of opposite colors.  Let $N'$ be obtained from $N$ by removing $u,v$ and the edge.

\subsubsection{Loop removal}
Suppose $e$ is a loop at the vertex $u$.  Assume that if $u$ is black the weight of $e$ is not 1, and if $u$ is white the weight of $e$ is not $-1$.
Then we can replace the edge $e$ by an edge $(u,v)$ for a new vertex $v$ which has color opposite to $u$.  Then we can apply leaf removal to obtain a new bicolored graph $N'$.

\begin{proposition}
For any of the above moves, we have $k_N = k_{N'}$ and assuming $\Rel(N)$ is well-defined, we have $\Rel(N) = \Rel(N')$.
\end{proposition}
\begin{proof}
Checked case-by-case.
\end{proof}

%
%

\subsection{Disjoint sum and gluing}
If $N$ and $N'$ are two bicolored networks with boundary vertex sets $S$ and $S'$, then $N \cup N'$ is a bicolored graph with boundary vertex set $S \cup S'$.  If $V \in \Gr(k,\F^S)$ and $V' \in \Gr(k',\F^{S'})$ then we have a natural point $V \boxplus V' \in \Gr(k+k',\F^{S \cup S'})$.

\begin{proposition}
Let $N$ and $N'$ be two bicolored networks with boundary vertex sets $S$ and $S'$ and relation spaces $\Rel(N) \in \Gr(k,\F^S)$ and $\Rel(N') \in \Gr(k',\F^{S'})$.    Then $\Rel(N \cup N') = \Rel(N) \boxplus \Rel(N')$.
\end{proposition}

Suppose $N$ is a bicolored network and $a,b \in S$ are two boundary vertices of $N$.  We suppose that the edges incident to $a$ and $b$ have weight $1$.  Let $N' = \Glue_{a,b}(N)$ be the bicolored network on boundary vertex set $S \setminus \{a,b\}$ obtained by gluing the two boundary edges incident to $a$ and $b$ together (removing $a$ and $b$ in the process), and giving the new edge weight $1$.  We shall describe $\Rel(N')$.  Let $S'=S-\{a,b\} \cup \{c\}$ and let $\phi: \F^S \to \F^{S'}$ be the linear map induced by the set map given by $a \mapsto c$ and $b \mapsto c$, and the identity on other elements.  Let $S \setminus \{a,b\} \simeq V_0 \subset \F^{S'}$ be the subspace of vectors where the coefficient in the $c$-direction is 0.  Then
\begin{equation}\label{eq:gluen}
\Rel(N') = \phi(\Rel(N)) \cap V_0.
\end{equation}
The operation $\Glue_{a,b}(N)$ does not change the degrees of interior vertices of $N$, so by \eqref{eq:kG}, we have $k_{N'} = k_N - 1$.  

For convenience, we assume that we have a total order on $S$ given by  $a < b < \text{rest}$.  Given such a total order, the notion that $\Rel(N)$ is TNN makes sense.

\begin{proposition}\label{prop:glue}
Suppose $\Rel(N)$ is totally nonnegative, or the edge weights are generic.  Then $\dim(\Rel(N')) = k_{N'}$ if and only if $\Delta_{I}(\Rel(N)) \neq 0$ for some $I \in \binom{S}{k}$ satisfying $|I \cap \{a,b\}| = 1$.  Furthermore, in this case $\Rel(N')$ is represented by Pl\"ucker coordinates $\Delta_J(\Rel(N')) =  \Delta_{aJ}(\Rel(N)) + \Delta_{bJ}(\Rel(N))$.
\end{proposition}
\begin{proof}
Let $k = k_N$ and $k' = k_{N'}$.  Suppose $\Rel(N)$ is represented by a $k \times n$ matrix $X$ with column vectors $\{v_s \in \F^k \mid s \in S\}$, where $n = |S|$.  Assuming $\dim\Rel(N') = k'$, we let $X'$ be a $k' \times n$ matrix representing $\Rel(N')$.
There is a torus $(\F^*)^S$ acting on the columns of $v_s$.  The genericity condition means that we only have to consider relations that are torus invariant.  For example, we do not need to consider the possibility that $v_a = - v_b$.


If $\Rel(N)$ is TNN or generic, the cases we have to consider are:
\begin{enumerate}
\item 
both $v_a$ and $v_b$ are equal to 0: then $\Rel(N')$ is simply the projection of $\Rel_G$ from $\F^S$ to $\F^{S \setminus \{a,b\}}$ by forgetting two of the coordinates, and $\dim \Rel(N') = k \neq k'$.  In this case, we have $\Delta_{aJ}(X) = 0 = \Delta_{bJ}(X)$ for all $J \in \binom{S \setminus\{a,b\}}{k-1}$.
\item 
$v_a = 0$ and $v_b \neq 0$ (resp. $v_b = 0$ and $v_a \neq 0$): then $\dim \Rel(N') = k'$ and $\Rel(N')$ is simply equal to $\Rel(N) \cap \F^{S \setminus \{a,b\}}$.  In this case, we have $\Delta_J(X') = \Delta_{bJ}(X) = \Delta_{bJ}(X)  + \Delta_{aJ}(X)$ (resp. $\Delta_J(X') = \Delta_{aJ}(X) = \Delta_{aJ}(X)  + \Delta_{bJ}(X)$).
\item
$v_a = \alpha v_b$ and both are non-zero: then $k' = k-1$ and $\Rel(N') = \Rel(N) \cap \F^{S \setminus \{a,b\}}$.  In this case $\Delta_{aJ}(X) = \alpha \Delta_{bJ}(X)$, and we can take $\Delta_J(X') = \Delta_{aJ}(X)  + \Delta_{bJ}(X)$.

\item
the vectors $v_a$ and $v_b$ are linearly independent and the vector $(1,-1,0,0,\ldots,0)$ belongs to $\Rel(N)$: under the TNN or genericity conditions this holds only if both $v_a$ and $v_b$ do not lie in the span of the rest of the columns.  In this case we have $\dim \Rel(N') = k-2$ 
and $\Delta_{aJ}(X) = 0 = \Delta_{bJ}(X)$ for all $J \in \binom{S \setminus\{a,b\}}{k-1}$.

\item
$v_a$ and $v_b$ are independent of each other: in this case, by a change of matrix we can make $v_a =  (1,0,0,\ldots,0)^T$ and $v_b = (0,1,0,\ldots,0)^T$.  Then we have vectors $(1,0,\alpha_1,\alpha_2,\ldots,\alpha_{n-2})$ and $(0,1,\beta_1,\beta_2,\ldots,\beta_{n-2})$ in $\Rel(N)$.  Hence $\Rel(N')$ is spanned by $(\alpha_1-\beta_1,\ldots,\alpha_r-\beta_r)$ together with $\Rel(N) \cap \F^{S \setminus \{a,b\}}$, and $\dim(\Rel(N) \cap \F^{S \setminus \{a,b\}}) = k-2$.  In this case, we calculate that $\Delta_J(X') = \Delta_{aJ}(X) + \Delta_{bJ}(X)$.
\end{enumerate}
\end{proof}
%
%

\subsection{Planarity and positivity}
Suppose $\tilde N$ is a usual planar bipartite network.  We obtain from $\tilde N$ a bicolored graph $N$ in the sense of this section by setting $w(u,v)$ to be equal to the weight of the edge $e = (u,v)$ in $N$ whenever $u$ is white or $v$ is black.  In the following we do not distinguish between $\tilde N$ and $N$.

Recall that there is a (positive) rotation map $\chi: \Gr(k,n) \to \Gr(k,n)$ which moves the last column of a $k \times n$ matrix to the front, with a sign of $(-1)^{k-1}$.  

\begin{lemma}\label{lem:relrotate}
We have $\chi(\Rel(N)) = \Rel(N')$, where $N'$ is obtained from $N$ by multiplying the weight of the half edge incident to boundary vertex $n$ by $(-1)^{k_N-1}$ and then relabeling the boundary vertices $1\to2,2 \to 3,\ldots,n \to 1$. 
\end{lemma}

Let $G$ and $G'$ be plabic graphs with boundary vertices labeled $\{1,2,\ldots,m\}$ and $\{m+1,\ldots,n\}$.  Then $G \cup G'$ is naturally a plabic graph with boundary vertices labeled $\{1,2,\ldots,n\}$.  For plabic graphs, we will always assume that disjoint unions $\cup$ are taken in a planar way.
%


Suppose $G$ is a plabic graph with edge set $E(G)$.  Let $\varepsilon \in \{+1,-1\}^{|E(G)|}$ be a choice of sign for each edge of $G$.  For $t = \{t_e \mid e \in E(G)\} \in \R_{>0}^{|E|}$, let $N(t,\varepsilon)$ be the plabic network with underlying graph $G$, and edge weights given by $\varepsilon_e \cdot t_e$.  In other words, $N$ has signed edge weights with signs given by $\varepsilon$.

\begin{theorem}\label{thm:posgraph}
Suppose $G$ is a planar bipartite graph with almost perfect matchings.  Then there exists $\varepsilon_G \in \{+1,-1\}^{|E(G)|}$ such that for any $t \in \R_{>0}^{|E|}$ we have
\begin{equation}\label{eq:posgraph}
\Rel(N(t,\varepsilon_G))= X(N(t,1)).
\end{equation}
%
\end{theorem}
\begin{proof}
Let us say $\varepsilon_G$ ``exists" if there is $\varepsilon_G \in \{+1,-1\}^{|E(G)|}$ such that \eqref{eq:posgraph} is satisfied.
We prove the result by a sequence of reductions.

%

If $G$ and $G'$ are planar bipartite graphs on boundary vertices $\{1,2,\ldots,m\}$ and $\{m+1,\ldots,n\}$, and $\varepsilon_G$ and $\varepsilon_{G'}$ both exist then it is easy to see that $\varepsilon_{G \cup G'}$ exists.

Now suppose $\varepsilon_G$ exists and $G' = \Glue_{1,2}(G)$ is bipartite, and that $G'$ has almost perfect matchings (which implies $G$ has almost perfect matchings).  We claim that $\varepsilon_{G'}$ also exists.  
The assumption that $G'$ has almost perfect matchings implies that $G$ has at least one almost perfect matching $\Pi$ such that $|I(\Pi) \cap \{1,2\}| = 1$.  We may now apply Proposition \ref{prop:glue}, using the assumption that $\Rel(N(t,\varepsilon_G)) = X(N(t))$ is TNN.  Note that we defined $\Glue_{1,2}(N)$ by insisting that the new edge has weight 1, but by applying gauge equivalences before and after the gluing operation, we see that there is $\varepsilon_{G'}$ such that $\Rel(N(t',\varepsilon_{G'}))$ is always TNN.  
To check that $\Rel(N'(t',\varepsilon_{G'})) = X(N'(t'))$, we do a matching computation.  For simplicity, assume that boundary vertex 1 in $G$ is connected to a black vertex $u$ and boundary vertex 2 is connected to a white vertex $v$.  In $G'$, an almost perfect matching either uses the edge $(u,v)$ or it does not, and these matchings correspond to $\Delta_{2J}$ and $\Delta_{1J}$ respectively, for some $J \in \binom{[n-1]}{k}$.  Thus $\Delta_J(N') = \Delta_{aJ}(N) + \Delta_{bJ}(N)$, agreeing with Proposition \ref{prop:glue}.

Next, using Lemma \ref{lem:relrotate} we see that if $\varepsilon_G$ exists then so does $\varepsilon_{G'}$, whenever $G'$ is a rotation of $G$.  In particular, we can apply $\Glue_{i,i+1}$ instead of just $\Glue_{1,2}$ and $\varepsilon$ will still exist.  But every planar bipartite graph can be built up from $m$-valent vertices (that is, the graph with a single interior vertex connected to $m$ boundary vertices) and gluing operations of adjacent boundary vertices.


Thus the theorem follows from checking that it holds for a single $m$-valent vertex (see Example \ref{ex:trivalent}).
\end{proof}

We suspect there is a simple explicit description of $\varepsilon$.

\begin{example}\label{ex:trivalent}
Consider the planar bipartite network
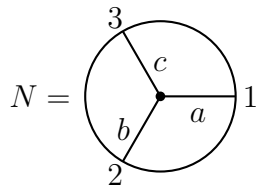
\begin{center}
\begin{tikzpicture}[scale=0.5]
\begin{scope}[thick,decoration={
    markings,
    mark=at position 0.5 with {\arrow{>}}}
    ] 
\node at (180:3.2) {$N = $};
\blackdot{(0,0)};
\draw (0,0) circle (2cm);
\draw (0:0) -- node[below] {$a$} (0:2);
\draw (0:0) -- node[right] {$c$}(120:2);
\draw (0:0) -- node[left] {$b$}(240:2);
\node at (0:2.4) {$1$};
\node at (240:2.4) {$2$};
\node at (120:2.4) {$3$};
\end{scope}
\end{tikzpicture}
\end{center}
Using matching enumeration, the Pl\"ucker coordinates are calculated to be $\Delta_{12}(N) = c$, $\Delta_{13}(N) = b$, $\Delta_{23}(N) = a$.  Now let us calculate $\Rel(N)$.  Let $z_1,z_2,z_3$ be the formal variables associated to the half-edges at the boundary vertices.  From the definitions, we obtain the relations
$$
\frac{1}{a} z_1 = \frac{1}{b} z_2 = \frac{1}{c} z_3.
$$
In coordinates, $\Rel(N)$ is the row span of the matrix
$$
\begin{bmatrix}
1/a & -1/b & 0 \\
1/a &  0 & -1/c
\end{bmatrix}
$$
which has Pl\"ucker coordinates $\Delta_{12} = 1/ab$, $\Delta_{13} = -1/ac$, $\Delta_{23} = 1/bc$.  This represents the same point as $X(N)$ if we set $b \mapsto -b$.  Thus we can choose $\varepsilon = (1,-1,1)$ when the edges are ordered $(a,b,c)$.  More generally, we can choose $\varepsilon = (1,-1,\ldots,(-1)^{m-1})$ for a $m$-valent vertex.

For a $m$-valent white vertex, no signs are required.
\end{example}

For a bicolored graph $G$, let $\F_G$ be the field of rational functions in a set of variables, one for each edge.  There is a natural bicolored network $N$ with edge weights that are variables in $\F_G$.

\begin{corollary}
Let $G$ be a planar bipartite graph and $N$ be the planar bipartite network with indeterminate edge weights in $\F_G$.  Then there exists a reduced planar bipartite network $\tilde N$, with edge weights taking values in $\F_G$, such that $\Rel(N) = \Rel(\tilde N)$.
\end{corollary}
\begin{proof}
Suppose $G$ is a planar bipartite graph, and $\tilde G$ is the reduced planar bipartite graph obtained just by reducing $G$ combinatorially (without considering edge weights).  We have to show that the local moves used to change $G$ into $\tilde G$ are well-defined when we start with indeterminate edge weights.  It is enough to show that the local moves are well-defined starting with a Zariski-dense subset of edge weights in $(\C^{*})^{|E(G)|}$.  Since $\R_{>0}$ and $-\R_{>0}$ are both Zariski-dense in $\C^*$, it is enough to show that the local moves are well-defined as each edge weight varies over either the positive or the negative reals.

We shall use Theorem \ref{thm:posgraph} to prove this last statement.  By a direct verification, we see that the choice of signs for edge weights in Theorem \ref{thm:posgraph} can be made compatible with all the local moves and reduction moves.   Suppose $G_1$ and $G_2$ are planar bipartite graphs with positive edge weights related by a local move or reduction move.  Let $\varepsilon_1$ and $\varepsilon_2$ be the corresponding signs in Theorem \ref{thm:posgraph}.  We have $\Rel(N_1(t_1,\varepsilon_1)) = X(N_1(t_1)) = X(N_2(t_2)) =\Rel(N_1(t_2,\varepsilon_2))$ for positive real $t_1,t_2$ related by some rational formulae.  It follows that there is a choice of edge weights for $\tilde N$ with values in $\F_G$ such that $\Rel(N) = \Rel({\tilde N})$ when all edge weights are specialized to be appropriately signed and real; thus the equality holds over $\F_G$ as well.
\end{proof}

\part{Grassmann polytopes}\label{part:two}

\section{Grassmann polytopes}\label{sec:poly}
We assume the reader is familiar with the usual theory of polytopes, for example as presented in \cite{Zie}.


\subsection{Grassmann polytopes}
For $k = 1$, the TNN Grassmannian $\Gr(1,n)_{\geq 0}$ is the subset of $\P^{n-1}$ consisting of points with nonnegative coordinates.  Thus $\Gr(1,n)_{\geq 0}$ can naturally be identified with the $(n-1)$-dimensional simplex $\Delta_{n-1}$.  Our plan is to take seriously the analogy
\begin{align*}
\text{simplex} & \longrightarrow \text{TNN Grassmannian}\\
\text{faces of the simplex} & \longrightarrow \text{positroid cells} \\
\text{boolean lattice} & \longrightarrow \hBound(k,n).
\end{align*}
Recall that $\hBound(k,n)$ is the poset of $(k,n)$-bounded affine permutations $\Bound(k,n)$ with an additional minimum adjoined.

Let $Z$ be a real $n \times r$ matrix with $r \leq n$.  Denote the rows of $Z$ by $z_1,z_2,\ldots,z_n \in \R^r$.  We may think of $Z: \R^n \to \R^r$ as a linear map.  It induces a rational map $Z_\Gr: \Gr(k,n)_\R \dashedrightarrow \Gr(k,r)_\R$ for any $1 \leq k \leq r$, sending a dimension $k$ subspace $V \subset \R^n$ to the image subspace $Z(V) \subset \R^r$.  The map $Z_\Gr$ is not defined if $\dim Z(V) < k$.

Call $Z$ \emph{positive} if its maximal ($r \times r$) minors are strictly positive.  Then Arkani-Hamed and Trnka \cite{AT} define the \defn{amplituhedron} to be the image of $\Gr(k,n)_{\geq 0}$ under the map $Z_\Gr$.  When $k=1$, the amplituhedron is a cyclic polytope \cite{Stu}.

Restricting $Z$ to have strictly positive maximal minors seems to be overly restrictive, so we introduce the following condition, in the style of Farkas' Lemma and its relatives:
\begin{equation}\label{eq:positive}
\mbox{There exists a $r \times k$ real matrix $M$ such that $Z \cdot M$ has positive $k \times k$ minors.}
\end{equation}

For $k = 1$, the condition \eqref{eq:positive} would guarantee that the cone spanned by the rows of $Z$ form a pointed polyhedral cone.  

The following definition is an analogue of the fact that every polytope is the image of a simplex.

\begin{definition}\label{def:main}
A \defn{Grassmann polytope} is the set
$$
P = Z(\Pi_{f,\geq 0}) \coloneqq \{Z_\Gr(X) \mid X \in \Pi_{f,\geq 0}\}
$$
for some $f \in \Bound(k,n)$ and $Z$ satisfying \eqref{eq:positive}.  Say $P$ is a full Grassmann polytope if $f = \id$, so that $P = Z(\Gr(k,n)_{\geq 0})$.
%

The \defn{dimension} $\dim(P)$ of a Grassmann polytope $P$ is the dimension of the Zariski closure $\overline{P} \subseteq \Gr(k,r)$.  
\end{definition}

Since $\Pi_{f, \geq 0}$ is Zariski-dense in $\Pi_f$, we have that $\overline{P}$ is equal to the variety $Z(\Pi_f)$ to be defined in Section \ref{sec:real}.

\begin{proposition}\label{prop:positive}
Suppose \eqref{eq:positive} holds.  Then the map $Z_\Gr$ is well-defined on $\Gr(k,n)_{\geq 0}$.  The Grassmann polytope $Z(\Pi_{f,\geq 0})$ is a closed connected subset of $\Gr(k,r)_{\geq 0}$.
\end{proposition}
\begin{proof}
Let $X$ be a $k \times n$ matrix representing a point in $\Gr(k,n)$, and suppose $Z_\Gr$ is not well-defined at this point.  Then the matrix $Y = X \cdot Z$ has rank less than $k$, and thus all Pl\"ucker coordinates $\Delta_I(Y)$ vanish.  

If $Z$ satisfies \eqref{eq:positive} and $X \in \Gr(k,n)_{\geq 0}$ has nonnegative minors (and at least one positive minor), then $Y \cdot M = X \cdot (Z \cdot M)$ is a $k \times k$ matrix whose determinant (by \eqref{eq:CB}) is given by $\sum_{J \in \binom{[n]}{k}} \Delta_J(X) \Delta_J(Z \cdot M) > 0$.  This implies that $Y$ itself must have rank $k$, and $Z_\Gr$ is well-defined at $X$.

The last statement follows from the fact that $Z(\Pi_{f,\geq 0})$ is a compact connected set (being a closed connected subset of $\Gr(k,n)_{\R}$), and $Z_\Gr$ is continuous when restricted to $\Gr(k,n)_\R \setminus E_Z$.  Here, $E_Z$ is the exceptional locus of the rational morphism $Z_\Gr: \Gr(k,n) \dashedrightarrow \Gr(k,r)$, to be discussed in further detail in Section \ref{sec:trunc}.
\end{proof}

We conjecture that every Grassmann polytope is contractible.

\begin{corollary}
The condition \eqref{eq:positive} implies that $\spn(z_{i_1},z_{i_2},\ldots,z_{i_k})$ has rank $k$ for any $I = \{i_1,\ldots,i_k\} \subset \binom{[n]}{k}$.
\end{corollary}
\begin{proof}
By Proposition \ref{prop:positive}, the map $Z_\Gr$ is well-defined at the torus fixed point $e_I \in \Gr(k,n)_{\geq 0}$.
\end{proof}

We say that $P$ and $P'$ are \defn{projectively equivalent} if $P = P' \cdot g$ where $g \in \GL(r)$ acts on $\Gr(k,r)$ by right multiplication.  If 
$Z$ and $Z'$ are related by $Z' = Z \cdot g$, then the Grassmann polytopes $P = Z(\Pi_{f,\geq 0})$ and $P' = Z'(\Pi_{f,\geq 0})$ are projectively equivalent.  Thus, up to projective equivalence, the Grassmann polytope $P$ only depends on the column space of $Z$.  Equivalently, $P$ only depends on the image of $Z$ in $\Gr(r,n)$, or again equivalently, $P$ only depends on the kernel of $Z$.

\begin{remark}\label{rem:cones}
Proposition \ref{prop:positive} can also be interpreted using cones.  Let $\hGr(k,n)$ be the cone over the Grassmannian (see Section \ref{sec:coord}) and let $\hGr(k,n)_{\geq 0} \subset \hGr(k,n)$ be the locus with nonnegative Pl\"ucker coordinates.  Similarly define $\hPi_{f,\geq 0}$.  Then the proof of Proposition \ref{prop:positive} shows that $Z(\hPi_{f,\geq 0})$ lies completely within the closed half-space 
$$
M^+ \coloneqq \{Y \in \hGr(k,r) \mid \det(Y \cdot M) \geq 0\}
$$
and the intersection of $Z(\hPi_{f,\geq 0})$ with $M_0 \coloneqq \{Y \in \hGr(k,r) \mid \det(Y \cdot M) = 0\}$ consists only of the origin.  Thus $Z(\hPi_{f,\geq 0})$ is a pointed cone.
\end{remark}
\begin{remark}
Definition \ref{def:main} is a Grassmann analogue of projective polytopes.  There is also an analogue of Euclidean polytopes.  To work with this, we fix the first $k$ rows of $M$ to be the $k \times k$ identity matrix.  Then the condition \eqref{eq:positive} is that the first $k$ columns of $Z$ give a point in $\Gr(k,n)_{>0}$.  (For $k = 1$, the condition is that the first column of $Z$ has positive entries, and these entries are usually fixed to equal $1$).
\end{remark}

We allowed Grassmann polytopes to be $Z(\Pi_{f, \geq 0})$ for arbitrary $f \in \Bound(k,n)$ in Definition \ref{def:main} rather than just $Z(\Gr(k,n)_{\geq 0})$.  This is because we would like the totally nonnegative strata $\Pi_{f,\geq 0}$ (the faces of $\Gr(k,n)_{\geq 0}$) to be Grassmann polytopes too.  But the strata $\Pi_{f,\geq 0}$ are inherently different to $\Gr(k,n)_{\geq 0}$; the totally nonnegative Grassmannian $\Gr(k,n)_{\geq 0}$ has dimension $k(n-k)$, but $\Pi_{f,\geq 0}$ can have any dimension.  Furthermore, if $Z$ has full rank then $Z(\Gr(k,n)_{\geq 0})$ will always be full-dimensional in $\Gr(k,r)_\R$, and thus have dimension $k(r-k)$.

We check that \eqref{eq:positive} is satisfied by positive $Z$.
\begin{lemma}\label{lem:posZ}
Suppose $Z$ has positive $r \times r$ minors.  Then $Z$ satisfies \eqref{eq:positive}.
\end{lemma}
\begin{proof}
Let $e_{[r]} = \spn(e_1,\ldots,e_r)$ be the 0-dimensional cell $\Pi_f$ in $\Gr(r,n)$, where $f = [1+n,2+n,\ldots,r+n,r+1,\ldots,n]$.  Let $w \in S_n$ be the permutation such that $fw = \id$.  Then $w = (r+1)(r+2)\cdots n 1 2 \cdots r$ in one-line notation.  Let $i_1 i_2 \cdots i_\ell$ be a reduced word for $w$.  Then by the proof of Theorem \ref{thm:main}, adding the bridges indexed by $i_1, i_2, \ldots, i_\ell$ to the lollipop graph of $e_{[r]}$ gives a planar bipartite graph that represents $G$ such that $M_G$ parametrizes $\Gr(k,n)_{>0}$.  Thus for any $X \in \Gr(k,n)_{>0}$, there are (unique) parameters $a_1,a_2,\ldots,a_\ell \in \R_{>0}$ such that the matrix $g = x_{i_1}(a_1) \cdots  x_{i_\ell}(a_\ell)$ satisfies $e_{[r]} \cdot g = X$. 

Now let $v = r(r-1)\cdots 1 n(n-1) \cdots (r+1)$ be the longest element in the parabolic subgroup $S_r \times S_{n-r}$, and let $j_1 \cdots j_p$ be some reduced word for $v$.  Let $g' = x_{j_1}(b_1) \cdots x_{j_p}(b_p)$ where $b_1,b_2,\ldots,b_p \in \R_{>0}$.  Then the product $g'g$ is in the ``top cell" of the totally nonnegative part of upper triangular matrices $U_{\geq 0} \subset \GL(n)_{\geq 0}$.  We have $e_{[r]} \cdot g'g = e_{[r]} \cdot g = X$ since $e_{[r]}$ is stabilized by $g'$.

The transpose matrix $Z^T$ represents a point in $\Gr(k,n)_{>0}$.  We can therefore find $g,g' \in \GL(n)$ as above, and $h \in \GL(k)$ so that
$$
h \cdot e_{[r]} \cdot g'g = Z^T
$$
as $r \times n$ \emph{matrices}, where $e_{[r]}$ is the $r \times n$ matrix equal to the identity in the first $r$ columns, and zero in the last $n-r$ columns.
Let $M$ be a $r \times k$ matrix.  Then 
$$
(Z \cdot M)^T = M^T \cdot Z^T = M^T \cdot h \cdot e_{[r]} \cdot g'g.
$$
Now, if $e_{[k],r}$ is the $k \times r$ matrix equal to the identity in the first $r$ columns, and zero in the remaining columns, then $e_{[k],r} e_{[r]} = e_{[k],n}$ is the $k \times n$ matrix with the same property.  Since $g'g$ is in the top cell of $U_{\geq 0}$, the same argument as above shows that $e_{[k],n} g'g$ represents a point in $\Gr(k,n)_{>0}$, and so must have strictly positive $k \times k$ minors.  
Thus $M = (e_{[k],r} \cdot h^{-1})^T$ shows that $Z$ satisfies \eqref{eq:positive}.
\end{proof}

We now define the notion of a dual Grassmann polytope.  Let $Z^*$ be a real $r \times n$ matrix, thought of as a linear map $Z^*:\R^r \to \R^n$.  We assume that $Z^*$ is full rank, so that we have an induced map $Z^*_\Gr: \Gr(k,r)_\R \to \Gr(k,n)_\R$ sending $V \subset \R^r$ to $Z(V) \subset \R^n$. 

\begin{definition}
A \defn{dual Grassmann polytope} is the set
$$
P =   (Z^*_\Gr)^{-1}(\Pi_{f,\geq 0}) \coloneqq (Z^*_\Gr)^{-1}(\Pi_{f,\geq 0} \cap Z^*_\Gr(\Gr(k,r)_\R)) \subset \Gr(k,r)_\R
$$
for some $f \in \Bound(k,n)$ and $Z^*$ such that $Z = (Z^*)^T$ satisfies \eqref{eq:positive}.  Say $P$ is a full dual Grassmann polytope if $f = \id$.
%
\end{definition}
In other words, a full dual Grassmann polytope is defined by pulling back the defining inequalities of $\Gr(k,n)_{\geq 0}$ to $\Gr(k,r)_\R$ via the injection $Z^*_\Gr$.  (Strictly speaking, the inequalities themselves do not make sense; only ratios of them do.)

\begin{proposition}
Suppose $Z^*$ is full rank and \eqref{eq:positive} holds.  Then the dual Grassmann polytope $P = (Z^*_\Gr)^{-1}(\Pi_{f,\geq 0})$ is a closed subset of $\Gr(k,r)_\R$, and the full dual Grassmann polytope $P = (Z^*_\Gr)^{-1}(\Gr(k,n)_{\geq 0})$ is in addition nonempty.
\end{proposition}
\begin{proof}
The map $Z^*_\Gr: \Gr(k,r)_\R \to \Gr(k,n)_\R$ embeds $\Gr(k,r)_\R$ as a closed submanifold of $\Gr(k,n)_\R$.  The condition \eqref{eq:positive} for $(Z^*)^T$ is exactly the condition that $Z^*_\Gr(\Gr(k,n)_\R)$ intersects $\Gr(k,n)_{>0}$.
\end{proof}

We will focus on Grassmann polytopes rather than dual Grassmann polytopes in this work.

\subsection{Some unusual behavior}
We begin with a list of warnings concerning the behavior of Grassmann polytopes.

\subsubsection{Facet inequalities do not cut out a Grassmann polytope}
Consider the TNN Grassmannian $\Gr(2,4)_{\geq 0}$.  There are four positroid cells of codimension one, indexed by the bounded affine permutations $[4356], [3546], [3465], [2457] \in \Bound(2,4)$.  These four ``facets'' are the intersections of $\Gr(k,n)_{\geq 0}$ with the four cyclic rotations of the Schubert variety $X_{12}$, and are cut out by the ``hyperplanes" $\Delta_{12} = 0$, $\Delta_{23} = 0$, $\Delta_{34}=0$, and $\Delta_{14} = 0$ respectively.  

However, the inequalities $\Delta_{12} \geq 0$, $\Delta_{23} \geq 0$, $\Delta_{34} \geq 0$, $\Delta_{14} \geq 0$ cut out the union of $\Gr(2,4)_{\geq 0}$ with the twisted totally nonnegative Grassmannian $\Gr(2,4)_{\geq 0,\tau}$ (see \eqref{eq:twisted}) where $\Delta_{12},\Delta_{23},\Delta_{34},\Delta_{14}$ are nonnegative, and $\Delta_{13},\Delta_{24}$ are nonpositive.  To see this, note that the right hand side of the Pl\"ucker relation $\Delta_{13}\Delta_{24} = \Delta_{12}\Delta_{34}+\Delta_{14}\Delta_{23}$ must be nonnegative, so $\Delta_{13}$ and $\Delta_{24}$ must have the same sign.

To cut out the totally nonnegative Grassmannian, we must include the additional inequalities $\Delta_{13} \geq 0$ and $\Delta_{24} \geq 0$.  The intersection of $\Delta_{13} = 0$ with $\Gr(k,n)_{\geq 0}$ is the union of the four codimension two positroid cells indexed by bounded affine permutations $[1467], [3564], [2358], [5346] \in \B(2,4)$.

\subsubsection{Grassmann polytopes can be cut out by higher degree equations}

The Grassmann polytope $Z(\Pi_{f,\geq 0})$ may be cut out by higher degree equations in Pl\"ucker coordinates.  In Section \ref{sec:deg2}, we will give an example of $Z(\Pi_{f,\geq 0})$ that is codimension one in the ambient Grassmannian, with Zariski closure (the analogue of the affine span) a degree two hypersurface.

Because of this, there seems to be no hope of a simple Grassmann analogue of Fourier-Motzkin elimination, which computes the defining inequalities of the projection of a polytope in terms of the defining inequalities of the original polytope.  

Since a dual Grassmann polytope is always defined by equations that are linear in Pl\"ucker coordinates, it follows that our notion of a Grassmann polytope is distinct from our notion of a dual Grassmann polytope.

\begin{problem}
Give a description of a Grassmann polytope by inequalities.
\end{problem}

\begin{problem}
Give a description of a dual Grassmann polytope as a projection.
\end{problem}

\subsubsection{Dimension-preserving maps can be many to one}
Suppose $P \subset \R^n$ is a polytope, not necessarily of full-dimension, and $Q = \phi(P)$ is the image of $P$ under an affine map $\phi:\R^n \to \R^r$.   Then if $Q$ and $P$ have equal dimensions, the map $\phi$ is a one-to-one map from $P$ to $Q$.  This follows from the corresponding statement for the affine span of $P$ mapping to the affine span of $Q$.

In Example \ref{ex:twotoone}, we give an example of a Grassmann polytope $Z(\Pi_f)$, where $Z$ is positive, such that $\Pi_{f,>0} \mapsto Z(\Pi_{f,>0})$ is not one-to-one, but the typical fiber can have two points.

This is the standard symptom of a Schubert calculus problem: a seemingly ``linear" problem turns out to have finitely many, but more than one, solutions.  Indeed, we will explain in Section \ref{sec:trunc} that these fibers are often intersections of a Schubert variety with a positroid variety.  More precisely, we will study in Section \ref{sec:trunc} the behavior of these fibers over complex points.  Understanding the fibers of $\Pi_{f,>0} \mapsto Z(\Pi_{f,>0})$ would require understanding reality, and even positive reality, issues in Schubert calculus.

\subsubsection{Dimension may be forced to collapse}
Suppose $P \subset \R^n$ is a polytope and $\phi:\R^n \to \R^r$ is an affine map.  If $\dim(P) = r$, then we expect that $\phi(P)$ has dimension $r$ as well, and this is the case if $\phi$ is a generic map.

In Section \ref{sec:trunc}, we give examples where $\dim Z(\Pi_f) < \dim(\Pi_f)$ for generic maps $Z$, even though $\dim(\Pi_f)$ is equal to the dimension of the image Grassmannian $\Gr(k,r)_\R$.

\subsubsection{Differences between the boolean lattice and $\Bound(k,n)$}

The partial order $\Bound(k,n)$ is neither self-dual nor a lattice.  In Section \ref{sec:facets} we will see that this is related to the phenomenon that facets of Grassmann polytopes are non-trivial unions of smaller Grassmann polytopes.

%

\section{Grassmann matroids}\label{sec:matroids}
We work over the field $\C$ in this section since we are thinking algebro-geometrically, but most of the discussion makes sense over $\R$ or another field.
\subsection{Matroids}
A \defn{matroid} of rank $k$ on $[n]$ is a non-empty collection $\M \subseteq \binom{[n]}{k}$ of $k$-element subsets satisfying the exchange axiom:
$$
\mbox{if $I, J \in \M$ and $i \in I$ then there exists $j \in J$ such that $I \setminus \{i\} \cup \{j\}$ belongs to $\M$.}
$$
A set $I \in \M$ is called a \defn{base} of $\M$.

Matroids can be characterized in many ways, including in terms of independent sets, circuits, flats, and rank functions.  If $\M$ is a matroid, we say that $I \subset [n]$ is an \defn{independent set} of $\M$ if $I \subset J$ for some $J \in \M$.  Write $\II(\M)$ for the collection of independent sets of $\M$.  The following axioms of independent sets give another axiomatization of a matroid:
\begin{enumerate}
\item
We have $\emptyset \in \II$.
\item
If $I \in \II$ and $J \subset I$ then  $J \in \II$.
\item
If $I, J \in \II$ and $|I| < |J|$ then there exists $j \in J$ such that $(I \cup \{j\}) \in \II$.
\end{enumerate}

The \defn{circuits} of $\M$ are the subsets $C \subset [n]$ with the property that $C \notin \II(\M)$ but $C' \in \II(\M)$ for any $C' \subsetneq C$.

The \defn{rank function} of $\M$ is the function $r: 2^{[n]} \to \Z_{\geq 0}$ given by 
$$
r(S) = \mbox{size of the largest independent set $T \subset S$}
$$
for $S \subset [n]$.  Rank functions of matroids are characterized by the axioms:
\begin{enumerate}
\item[Rank a)]
For any $S \subset [n]$, we have $r(S) \leq |S|$.
\item[Rank b)]
For any $S,T \subset [n]$, we have $r(S \cup T) + r(S \cap T) \leq r(S) + r(T)$.
\item[Rank c)]
For any $S \subset [n]$, and $i \in [n]$ we have $r(S) \leq r(S \cup \{i\}) \leq r(S) + 1$.
\end{enumerate}
The independent sets are recovered as those subsets $I \subseteq [n]$ satisfying $r(I) = |I|$.  

A subset $F \subset [n]$ is a \defn{flat} of $\M$ if for any $i \in [n] \setminus F$ we have $r(F \cup \{i\}) > r(F)$.  We denote the set of flats of $\M$ by $\F(M)$.  The flats of a matroid are characterized by the axioms:
\begin{enumerate}
\item
We have $[n] \in \F$.
\item
If $F, G \in \F$ then $(F \cap G) \in \F$.
\item
Suppose $F \in \F$.  We say that $G$ covers $F$ if $F \subsetneq G$ and there are no flats strictly between $F$ and $G$.  Then, as we vary $G$ over covers of $F$, the sets $G\setminus F$ partition $[n] \setminus F$.
\end{enumerate}
A subset $H \subset [n]$ is a \defn{hyperplane} if it is a flat with rank $k-1$.  A matroid is also completely determined by its \defn{hyperplanes}.  The \defn{cocircuits} of a matroid are the complements $[n]\setminus H$ of the hyperplanes.

\subsection{Realizable matroids}\label{sec:realmat}
Let $Z = (z_1,z_2,\ldots,z_n) \in \C^r$ be $n$ vectors in $\C^r$ that span $\C^r$.  We obtain a realizable matroid $\M$ as the collection of subsets $I = \{i_1,i_2,\ldots,i_r\}$ such that $z_{i_1},\ldots,z_{i_r}$ form a basis of $\C^r$.  We will recover the independent sets, circuits, rank function, and flats of the realizable matroid $\M_Z$ in a somewhat unorthodox fashion.  


Abusing notation, we also write $Z: \C^n \to \C^r$ for the linear map sending the basis vectors $e_1,e_2,\ldots,e_n$ to $z_1,z_2,\ldots,z_n$.  This induces a rational map $Z_\P: \P^{n-1} \dashedrightarrow \P^{r-1}$.  Thinking of $\P^{n-1}$ as the space of lines in $\C^n$, the rational map $Z_\P$ has exceptional locus $E_Z = \{L \subset \C^n \mid L \subseteq \ker(Z)\} \subset \P^{n-1}$.

Suppose $I = \{i_1,i_2,\ldots,i_s\}$.  Let $H_I \subseteq \P^{n-1}$ be the coordinate hyperspace given by the image of $\spn (e_{i_1},\ldots,e_{i_s})$ in $\P^{n-1}$.  Define the \defn{image} $Z(H_I)$ of $H_I$ under the map $Z_\P$ to be 
\begin{equation}\label{eq:ZHI}
Z(H_I) \coloneqq \overline{Z_\P(H_I \setminus E_Z)}
\end{equation}
where the closure here is taken in the Zariski topology.  The subvariety $Z(H_I)$ is simply a linear hyperspace in $\P^{r-1}$: it is the image of the linear space $Z(\spn (e_{i_1},\ldots,e_{i_s})) \subset \C^r$.  By definition, if $H_I \subseteq E_Z$, then $Z(H_I) \coloneqq \emptyset$.

Then we have the following dictionary:

\begin{enumerate}
\item
A subset $I \in \binom{[n]}{r}$ is a base of $\M_Z$ if $\dim(Z(H_I)) = \dim(H_I) = \dim \P^{r-1}$.
\item
A subset $I \in 2^{[n]}$ is an independent set of $\M_Z$ if and only if $\dim(Z(H_I)) = \dim(H_I)$.  
\item
A subset $C \in 2^{[n]}$ is a circuit of $\M_Z$ if $\dim(Z(H_C)) < \dim(H_C)$ and $C$ is minimal under inclusion amongst subsets satisfying this condition.  Equivalently, $C$ is a circuit if $E_Z \cap H_C \neq \emptyset$ and $E_Z \cap H_{C'} =\emptyset$ for all $C' \subsetneq C$.
\item
The rank function $r_Z: 2^{[n]} \to \Z_{\geq 0}$ is given by $r_Z(S) = \dim(H_S)+1$.
\item
A subset $F \subset [n]$ is a flat of $\M_Z$ if for every $i \notin F$, we have $Z(H_F) \subsetneq Z(H_{F \cup \{i\}})$.  Equivalently, define an equivalence relation on $2^{[n]}$ by $S \sim T$ if $Z(H_S) = Z(H_T)$.  Then $F$ is a flat if it is the unique maximal element in its equivalence class.
\end{enumerate}


\subsection{Realizable Grassmann matroids}\label{sec:real}
Now fix $1 \leq k \leq n$.  The linear map $Z: \C^n \to \C^r$ also induces a rational morphism
$$
Z_\Gr: \Gr(k,n) \dashrightarrow \Gr(k,r).
$$
The exceptional locus $E_Z \subset \Gr(k,n)$ is
$$
E_Z = \{X \in \Gr(k,n) \mid \dim(X \cap \ker(Z)) \geq 1\}
$$
with points in $\Gr(k,n)$ thought of as $k$-dimensional subspaces.
Motivated by the analogies in Section \ref{sec:poly}, we take the positroid varieties $\Pi_f$ as the Grassmannian-analogue of the coordinate subspaces $H_I$, and define
\begin{equation}\label{eq:ZPi}
Z(\Pi_f) \coloneqq \overline{Z_\Gr(\Pi_f \setminus E_Z)}.
\end{equation}
If $\Pi_f \subset E_Z$, we define $Z(\Pi_f) \coloneqq \emptyset$.  

Now, let us pretend that a \defn{Grassmann matroid} $\G_Z$ exists, and ask for the analogue of bases, independent sets, circuits, rank function, and flats.

\subsubsection{Bases}
A bounded affine permutation $f \in \hBound(k,n)$ is a \defn{base} of $\G_Z$ if $\dim(\Pi_f) = \dim(\Gr(k,r)) = \dim(\Gr(k,r))$.  Note that by the irreducibility of $\Gr(k,r)$ this implies that $Z(\Pi_f) = \Gr(k,r)$.

\subsubsection{Independent sets}
A bounded affine permutation $f \in \hBound(k,n)$ is an \defn{independent set} of $\G_Z$ if $\dim(\Pi_f) = \dim(Z(\Pi_f))$.
As it is shown in Example \ref{ex:ind}, if $g < f$ and $f$ is independent it may not be the case that $g$ is also independent.  It is therefore unlikely that the bases of $\G_Z$ capture all the information in the Grassmann matroid.

\subsubsection{Circuits}
A bounded affine permutation $f \in \hBound(k,n)$ is a \defn{circuit} of $\G_Z$ if $\dim(\Pi_f) > \dim(Z(\Pi_f))$ (that is, $f$ is not independent) and if $g < f$ then $\dim(\Pi_g) = \dim(Z(\Pi_g))$.  As for bases, it is unlikely that the circuits of $\G_Z$ capture all the information.

\begin{remark}
When $k = 1$, the conditions that $\dim(\Pi_f) > \dim(Z(\Pi_f))$ and $\Pi_f \cap E_Z \neq \emptyset$ are equivalent.  However, this is not the case for $k > 1$.
\end{remark}
%

\subsubsection{Rank and class function}
Define the \defn{rank function} of $\G_Z$ to be the function $r_Z:\hBound(k,n) \to \Z_{\geq 0}$ given by 
$$
r_Z(f) \coloneqq \dim(Z(\Pi_f)) + 1.
$$
The only natural invariant of a linear hyperspace $Y \subseteq \P^{r-1}$ is its dimension, but for the subvarieties $Z(\Pi_f) \subset \P^{r-1}$ there are other natural $\GL(r)$-invariants.  As we shall see in Section \ref{sec:sphericoid}, the subvarieties $Z(\Pi_f) \subseteq \Gr(k,r)$ can have degree greater than one; it would be reasonable to keep track of this too.  

We thus define the \defn{class function} $c_Z$ of $\G_Z$ by
$$
c_Z(f) \coloneqq [Z(\Pi_f)] \in H^*(\Gr(k,r)).
$$
When $k = 1$, we have $[Z(H_I)] =[H]^c$ where $[H] \in H^2(\P^{r-1})$ is the hyperplane class, and $c$ is the codimension of $Z(H_I)$.  Thus $c_Z$ and $r_Z$ contain the same information in this case.  For $k > 1$, the class function contains information such as the degree of $Z(\Pi_f)$. 

\subsubsection{Flats}
Define an equivalence relation on $\hBound(k,n)$ by $f \sim g$ if $Z(\Pi_g) = Z(\Pi_f)$.  We call each equivalence class a \defn{flat} of $\G_Z$.

\subsection{Axioms?}
We will not formally axiomatize a Grassmann matroid.  Here we will be content with an informal sketch of a heuristic that the class function might satisfy inequalities analogous to the axioms Rank a), Rank b), and Rank c) of a rank function of a matroid. 

\begin{enumerate}
\item
The analogue of Rank a) is that $\dim Z(\Pi_f) \leq \dim \Pi_f$, which in turn gives
$c_Z(f) \in H^d(\Gr(k,r))$ where $d \geq {\rm codim}(\Pi_f)$.  
\item
Let $f,g \in \Bound(k,n)$.  Then $\Pi_f$ and $\Pi_g$ typically do not intersect transversally in $\Gr(k,n)$.  However, $\Pi_f$ and $\Pi_g$ could potentially intersect transversally inside $\Pi_h$, where $h > f,g$ is minimal among elements greater than $f$ and $g$.  (When $\hBound(k,n)$ is the boolean lattice, we would take $h$ to be the join of $f$ and $g$.)  Assuming the pushforward and pullback maps between $H^*(\Pi_h)$ and $H^*(\Gr(k,n))$ behave well, we would then obtain an equality
$$
[\Pi_f]\cdot [\Pi_g] = [\Pi_h] \cdot [\Pi_f \cap \Pi_g] \text{ in $H^*(\Gr(k,n))$}
$$
via the projection formula and Theorem \ref{thm:Ful}.  Since $Z(\Pi_f \cap \Pi_g) \subseteq Z(\Pi_f) \cap Z(\Pi_g)$, this equality turns into an inequality for $c_Z$ similar to Rank b).
\item
The analogue of the inequality $r(S) \leq r(S \cup \{i\})$ in axiom Rank c) comes from the inclusion $Z(\Pi_f) \subseteq Z(\Pi_g)$ whenever $f \lessdot g$, from which one can deduce an appropriate cohomological identity for $c_Z$.   The analogue of $r(S \cup \{i\}) \leq r(S) + 1$, comes from the inequality $[\Pi_g] \cdot s_1 \geq [\Pi_f]$ whenever $f \lessdot g$.  Here, the Schur function $s_1$ is the class of the Schubert divisor, and ``$\geq$'' means that the difference is a nonnegative linear combination of Schubert classes.
\end{enumerate}

\subsection{Canonical basis matroid}
There is a Grassmann matroid $\G_Z$ for each value of $1 \leq k \leq r$.  It is not clear to us to what extent these Grassmann matroids determine each other.  In particular, we do not know which data in a Grassmann matroid depends only on the usual matroid $\M_Z$.

Theorem \ref{thm:Posideal} and the examples in the rest of the paper suggest that we should also consider the \defn{canonical basis matroid} $\B_Z$ of $Z$, defined to be
$$
\B_Z \coloneqq \bigcup_d \{T \in B(d\omega_r) \mid G(T)^*(Z) \neq 0\} \subseteq \bigsqcup_d B(d\omega_r),
$$
where we consider the $n \times r $ matrix $Z$ to represent a point in $\Gr(r,n)$.  The degree $d = 1$ part of $\B_Z$ is then the usual matroid $\M_Z$ of $Z$.  

\begin{problem}
Characterize canonical basis matroids.
\end{problem}

Note that Theorems \ref{thm:canbasis} and \ref{thm:Posideal} characterize the canonical basis matroids of $Z$ that give a point in $\Gr(r,n)_{\geq 0}$.  That is, we have a classification of canonical basis positroids.  In particular, if $Z \in \Gr(r,n)_{>0} = \Pi_{\id,>0}$ is in the totally positive part of the Grassmannian, then $G(T)^*(Z) > 0$ for all $T \in B(d\omega_r)$.  It follows that every totally positive point has the uniform canonical basis matroid.

We suspect the knowledge of the vanishing or non-vanishing of a finite subset of the canonical basis will completely govern the behavior of Grassmann matroids.
%
%
%

\section{The uniform Grassmann matroid}\label{sec:trunc}
In this section, we consider the case that $Z$ is a generic matrix, which corresponds to the \defn{uniform Grassmann matroid}.  The meaning of ``generic" in this section is made precise in the paper \cite{Lamtrunc}. 
 In particular, a Zariski-open subset of real matrices $Z$ are generic.  In the case that $Z$ is a generic, we call the varieties $Z(\Pi_f)$ \defn{amplituhedron varieties} and denote them by $Y_f$.

\subsection{Combinatorial criterion for intersection with the exceptional locus}
Recall that the exceptional locus is $E_Z = \{X \in \Gr(k,n) \mid \dim(X \cap \ker(Z)) \geq 1\}$.  We have $\dim(\ker(Z)) = n - r$, so $E_Z$ is a Schubert variety with codimension $\codim(E_Z) = r-k+1 = m+1$.  The cohomology class is $[E_Z] = s_{m+1}$.

\begin{proposition}\label{prop:EZ}
For generic $Z$, the positroid variety $\Pi_f$ will intersect $E_Z$ if and only if $\tF_f \cdot s_{m+1} \neq 0$ in $H^*(\Gr(k,n))$.
\end{proposition}
\begin{proof}
This follows from combining Theorem \ref{thm:KLS} with Theorem \ref{thm:Ful}.
\end{proof}

\subsection{Truncations of affine Stanley symmetric functions}
  For $\mu \subseteq (m)^k$ we let $\mu^{+\ell} \subseteq (n-k)^k$ be the partition obtained from $\mu$ by adding $\ell$ columns of height $k$ to the left of $\mu$.  For example, with $\ell = 2$ and $k = 4$, we may have
$$
\mu = \tabll{&&&\\&&\\ &\bl \\ & \bl} \qquad \qquad \qquad
\mu^{+\ell} = \tabll{&&&&&\\&&&&\\&&\\&&} 
$$

Given $f = \sum_{\lambda \subset (n-k)^k} c_\lambda s_\lambda$ representing a cohomology class in $H^*(\Gr(k,n))$, we define the {\it truncation} $\tau_{r}(f) \in H^*(\Gr(k,r))$ by
$$
\tau_{r}(f) = \sum_{\mu \subseteq (m)^k} c_{\mu^{+(n-r)}} s_\mu.
$$
If $\dim(\Pi_f) = \dim(Y_f)$ (that is, $f$ is independent), we let $d_f$ denote the degree of the map $Z_\Gr|_{\Pi_f}: \Pi_f \to Y_f$.

\begin{theorem}[\cite{Lamtrunc}]\label{thm:trunc}
Suppose $Z$ is generic and $f \in \Bound(k,n)$.  
\begin{enumerate}
\item
If $\tau_r(f) = 0$, then $\dim(Y_f) < \dim(\Pi_f)$.  In other words, $f$ is not independent.
\item
If $\tau_r(f) = 0$, then $\dim(Y_f) = \dim(\Pi_f)$.  Thus $f$ is independent.   In this case, the cohomology class $[Y_f]$ of the amplituhedron variety $Y_f$ is equal to $\frac{1}{d_f} \tau_{r}(\tF_f)$.
\end{enumerate}
\end{theorem}

Note that Theorem \ref{thm:trunc} only says something about the rank function in the case that $f$ is independent.  Furthermore, we do not yet have a combinatorial formula for the degree $d_f$.  

Here is a sketch of the proof of Theorem \ref{thm:trunc}.  To compute $[Y_f]$, it is enough to compute the number of intersection points in $Y_f \cap Y_J(F_\bullet)$ for a Schubert variety $Y_J(F_\bullet) \subset \Gr(k,r)$ intersecting $Y_f$ transversally in a finite number of points (see Theorem \ref{thm:Ful}).  The inverse image $\overline{Z_\Gr^{-1}(Y_J(F_\bullet))} \subset \Gr(k,n)$ is itself a Schubert variety, and assuming the intersection is transverse, the number of intersection points of $\Pi_f \cap \overline{Z_\Gr^{-1}(Y_J(F_\bullet))}$ can be computed using Theorem \ref{thm:KLS} and Theorem \ref{thm:Ful}.

It would be interesting to extend Theorem \ref{thm:trunc} to the case where $Z$ is not generic.  This would presumably involve understanding non-transverse intersections between positroid varieties and Schubert varieties.

\begin{remark}
In \cite{Lamtrunc}, we used the terminology ``$f$ has kinematical support" instead of ``$f$ is independent".  When $m = 4$ and $\dim(\Pi_f) = 4k$, this agrees with the notion of kinematical support in physics \cite{ABCGPT}.  
\end{remark}

\begin{example}
Let $k = 2$, $r = 5$, and $n = 8$.  Suppose $f = [2,3,4,8,6,7,12] \in \B(2,7)$.  It is given by the rank conditions $r(1,4) \leq 1$ and $r(5,7) \leq 1$.  Then by Proposition \ref{prop:k2}, we have $\tF_f = h_3 h_2 \equiv s_{32} + s_{41} + s_{5}$ in $H^*(\Gr(2,7))$.  Thus $f$ is independent and $[Y_f] = s_1$ in $H^*(\Gr(2,5))$.

If instead we have $r = 6$, then $f$ is still independent and $[Y_f] = s_{3} + s_{21}$ in $H^*(\Gr(2,6))$.
\end{example}

\begin{example}\label{ex:twotoone}
Let $k=2$, $r = 6$, and $n = 8$. Suppose $f = [4,3,6,5,8,7,10,9] \in \Bound(2,8)$.  Then by Proposition \ref{prop:k2}, we have $\tF_f=  h_1^4 \equiv s_{4} + 3s_{31} + 2 s_{22}$ in $H^*(\Gr(2,8))$.  
The coefficient of $s_{22}$ in $\tF_f$ is equal to $2$.  So $f$ is independent and the map $Z_\Gr: (\Pi_f \setminus E_Z) \to Y_f$ has degree $2$.  In a similar manner we can easily produce maps $Z_f$ of arbitrarily high finite degree.

The map $Z_\Gr$ can have fibers of cardinality greater than one even when restricted to $\Pi_{f,>0}$.  An explicit example is to take 
$$
X = \left[
\begin{array}{cccccccc}
 1 & 10 & 40 & 10 & 11 & 0 & 0 & -9 \\
 0 & 1 & 4 & 10 & 11 & 6 & 11 & 0 \\
\end{array}
\right] \qquad Z = \left[
\begin{array}{cccccc}
 1 & 0 & 0 & 0 & 0 & 0 \\
 13 & 1 & 0 & 0 & 0 & 0 \\
 24 & 18 & 1 & 0 & 0 & 0 \\
 0 & 40 & 9 & 1 & 0 & 0 \\
 0 & 0 & 50 & 12 & 1 & 0 \\
 0 & 0 & 0 & 6 & 8 & 1 \\
 0 & 0 & 0 & 0 & 20 & 4 \\
 0 & 0 & 0 & 0 & 0 & 2 \\
\end{array}
\right].
$$
Then the fiber over $Z_\Gr(X)$ has another point in $\Pi_{f,>0}$.
\end{example}

\begin{example}\label{ex:ind}
Let $f = [4,3,6,5,7] \in \B(2,5)$.  Then by Proposition \ref{prop:k2}, we have that $\tF_f = h_1^2 = s_2 + s_{11}$ in $H^*(\Gr(2,5))$.  We have $g = [6,3,4,5,7] \lessdot f$ (recall that we are using the opposite of Bruhat order).  Then by Proposition \ref{prop:k2}, we have $\tF_g = h_3 = s_3$.  Suppose $r = 4$.  Then $f$ is independent with $\tau_4(\tF_f) \equiv 1$ in $H^*(\Gr(2,4))$, but $g$ is not independent.  This shows that independent sets do \emph{not} form an order ideal in $\hBound(k,n)$.
\end{example}

\subsection{An example of the geometry of $Z_\Gr$}\label{sec:exgeom}
We work through the geometry of a few examples explicitly.
Take $k = 2, r= 3, n = 4$.  Inside $\Gr(2,4)$ we consider the three subvarieties:
\begin{enumerate}
\item
The Schubert variety $A = \{V \in \Gr(2,4) \mid V \subset E\}$, where $E \subset \C^4$ is a three-dimensional subspace.
\item
The Schubert variety $B = \{V \in \Gr(2,4) \mid \dim(V \cap F) \geq 1\}$, where $F \subset \C^4$ is a one-dimensional subspace.
\item
The positroid variety $C = \Pi_f$ where $f = [2547] \in \B(2,4)$.  By \eqref{eq:fX}, this is the closure of the locus $\{V \in \Gr(2,4) \mid \dim(V \cap \spn({e_1,e_2})) =1 \text{ and } \dim(V \cap \spn({e_3,e_4})) =1\}$.  
\end{enumerate}
All three varieties $A,B,C$ are two-dimensional.  We study their behavior under $Z_\Gr$ for a generic $Z$.  Identify $\Gr(2,4)$ with the space of lines in complex projective three-space $\CP^3$.  Then the map $Z_\Gr$ is identified with the projection from the point $p \in \CP^3$ which corresponds to the kernel of the map $Z$, to some hyperplane $H_0 \simeq \CP^2 \subset \CP^3$.  The exceptional locus $E_Z$ of $Z_\Gr$ is then identified with the subvariety of lines passing through the point $p$.

\subsubsection{The Schubert variety $A$}
The variety $A$ is isomorphic to $\Gr(2,3)$.  For a generic $Z$, it will not intersect the exceptional locus of $Z_\Gr$ by Proposition \ref{prop:EZ}.  In this case, $Z_\Gr$ maps $A$ isomorphically onto $\Gr(2,3)$.  We picture this geometrically as follows: $A$ is identified with the space of lines contained inside a two-plane $H \subset \CP^3$ (the image of the three-dimensional subspace $E$).  If $p \notin H$, then $E_Z$ does not intersect $A$, and the projection maps the space of lines inside $H$ isomorphically to the space of lines inside $H_0$.  

\subsubsection{The Schubert variety $B$}

By Theorem \ref{thm:trunc}, we have $\dim(Z(B)) < \dim(B)$.  The variety $B$ can be identified with the space of lines that pass through a point $q$ (the image in $\CP^3$ of the one-dimensional subspace $F \subset \C^4$).  Generically, $p \neq q$.  The line joining $p$ and $q$ lies in the exceptional locus $B \cap E_Z$.  Let $r$ be the intersection of this line with $H_0$.

Let $L_0 \subset H_0$ be a line.  If the plane spanned by $p$ and $L_0$ does not intersect $p$, then $L_0$ is not in the image of $Z_\Gr(B \setminus E_Z)$.  Otherwise, there is a one-dimensional family of lines in that plane that pass through $p$ and project to $L_0$.  To summarize, the image $Z_\Gr(B \setminus E_Z)$ is the $\P^1$-of lines passing through $r$.  The fiber of $B \setminus E_Z$ over a point in this image is an $\A^1$-of lines. Note that the line joining $p$ and $q$ does not belong to $B \setminus E_Z$, which is why we have an $\A^1$ instead of $\P^1$. 

\subsubsection{The positroid variety $C$}

We now consider the variety $C$.  Let $L_{12}$ (resp. $L_{34}$) be the image of $\spn(e_1,e_2)$ (resp. $\spn(e_3,e_4)$) in $\CP^3$.  The variety $C$ is the space of lines that intersect both $L_{12}$ and $L_{34}$.  Generically, $p$ does not lie on either $L_{12}$ or $L_{34}$.  Projecting the two lines to $H_0$ we get $L'_{12}$ and $L'_{34}$.  Let $x_0$ be the intersection of $L'_{12}$ and $L'_{34}$.  Now let $L_0 \subset H_0$ be a line.  If $L_0$ does not pass through $x_0$, then it intersects $L'_{12}$ and $L'_{34}$ at $y_0$ and $z_0$.  Let $y \in L_{12}$ (resp. $z \in L_{34}$) be the intersection of $L_{12}$ (resp. $L_{34}$) with the line passing through $y_0$ (resp. $z_0$) and $p$.  Then the line passing through $y$ and $z$ is the unique point of $C$ that maps to $L_0$ under $Z_\Gr$.

Now suppose $L_0$ passes through $x_0$.  If $L_0 = L'_{12}$ or $L_0 = L'_{34}$, then there is a $\P^1$-worth of lines that map to it.  If $L_0$ is any other line passing through $x_0$ then no line in $C \setminus E_Z$ will map to it.  To summarize, the map $Z_\Gr: (C \setminus E_Z) \to \Gr(2,3)$ is one-to-one over a dense open subset $\Gr(2,3) \setminus \P^1 \simeq \C^2$.  On the $\P^1$ there are two distinguished points which lie in the image of $Z_\Gr$, and each has a fiber isomorphic to $\A^1$.

\section{The ideal of an amplituhedron variety}\label{sec:ampliideal}
Linear subspaces of $\P^{r-1}$ are cut out by linear equations, and linear algebra computes the equations that cut out the varieties $Z(H_I)$ of Section \ref{sec:realmat}. 
In this section, we discuss the computation of the ideal $I(Y_f)$ of an amplituhedron variety $Y_f$.  The main idea is to relate the geometry of the rational map $Z_\Gr: \Gr(k,n) \dashedrightarrow \Gr(k,r)$ to the geometry of the direct sum rational map
$$
\bigoplus: \Gr(k,n) \times \Gr(\ell,n) \longdashedrightarrow \Gr(k+\ell,n)
$$
that takes a $k$-plane $X$ and a $\ell$-plane $K$ to the $k+\ell$-plane $\spn(X,K)$.  The map $\bigoplus$ is induced by projection maps $V_{d\omega_k} \otimes V_{d\omega_\ell} \to V_{d\omega_{k+\ell}}$ of highest weight representations.
 
The material of this section relies heavily on the material in Section \ref{sec:coord}.  We will also use some terminology from geometric invariant theory \cite{Mum}.

\subsection{The universal amplituhedron variety}
Fix $1 \leq k \leq n$, and $r \in [k+1,n]$.  Set $\ell \coloneqq n-r$ and $m \coloneqq k-r$.  We will sometimes work with the cone $\hGr(k,n)$ over the Grassmannian in this section, as the language with coordinate rings becomes simpler.  There is a distinguished cone point $0 \in \hGr(k,n)$.

We have a map
$$
\mu: \hGr(k,n) \times \Mat(n,r) \to \hGr(k,r)
$$
given on the level of matrices by
$$
(X,Z) \mapsto X\cdot Z
$$
where $X$ denotes a $k \times n$ matrix representing a point in $\hGr(k,n)$.  Note that if $X \in E_Z$, we have $\mu(X,Z) = 0$.  Using the Cauchy-Binet formula \eqref{eq:CB}, the Pl\"ucker coordinates of $X \cdot Z$ can be written explicitly in terms of the Pl\"ucker coordinates of $X$ and the matrix entries of $Z$.

Let $\id: \Mat(n,r) \to \Mat(n,r)$ be the identity map.  Let $\mu \times \id: \hGr(k,n) \times \Mat(n,r) \to \hGr(k,r) \times \Mat(n,r)$ be the map $(X,Z) \mapsto (X \cdot Z,Z)$.  We define \defn{the universal amplituhedron variety} to be 
$$
\Y_f \coloneqq \overline{(\mu \times \id)(\hPi_f \times \Mat(n,r))}.
$$
There is a natural projection map $p: \hPi_f \times \Mat(n,r) \to \Mat(n,r)$.  For a generic $Z \in \Mat(n,r)$, the (affine cone over the) amplituhedron variety $Y_f$ is the fiber $p|_{\Y_f}^{-1}(Z)$ of the universal amplituhedron variety.

\subsection{$\GL(r)$ action on $\Y_f$}\label{sec:quotient}
Both $\hGr(k,r)$ and $\Mat(n,r)$ have right actions of the group $\GL(r)$.  Thus $\GL(r)$ acts on $\hGr(k,n) \times \Mat(n,r)$ by acting on the second factor, and acts on $\hGr(k,r) \times \Mat(n,r)$ by acting simultaneously on both factors.  Furthermore, the map $\mu \times \id$ commutes with these two actions:
$$
(\mu\times \id)(X, Z \cdot g) = (X \cdot Z \cdot g, Z \cdot g) = (X \cdot Z,Z) \cdot g.
$$
Since $\hPi_f \times \Mat(n,r)$ is preserved by this action, we deduce that $\Y_f$ is a $\GL(r)$-invariant subvariety of $\hGr(k,r) \times \Mat(n,r)$.

Let $A(k,r,n) \coloneqq \C[\hGr(k,r) \times \Mat(n,r)]$ denote the coordinate ring of $\hGr(k,r) \times \Mat(n,r)$.  It is generated by the Pl\"ucker coordinates $\Delta_I(Y)$ of $\hGr(k,r)$ and the matrix entry coordinates of $\Mat(n,r)$.  Define the functions
$$
(Y,Z) \mapsto \Delta_I(Z)
$$
for $I \in \binom{[n]}{r}$ and 
$$
(Y,Z) \mapsto \Delta(Y,Z_J)
$$
for $J \in \binom{[n]}{m}$, where $\Delta(Y,Z_J)$ is the determinant of the $r \times r$ matrix, whose first $k$ rows are given by $Y$ and last $m$ rows are given by the rows of $Z$ labeled by $J$.  

\begin{theorem}[\cite{Lam+}] \label{thm:invariants}
The $\SL(r)$-invariants $A^{\SL(r)}$ are generated by $\Delta(Y,Z_J)$ for $J \in \binom{[n]}{r-k}$ and $\Delta_I(Z)$ for $I \in \binom{[n]}{r}$.
\end{theorem}

Note that the functions $\Delta(Y,Z_J)$ satisfy the Pl\"ucker relations for $\hGr(m,n)$.  For notational convenience, we actually identify $\{\Delta(Y,Z_J) \mid J \in \binom{[n]}{m}\}$ with a point in $\hGr(n-m,n) = \hGr(k+\ell,n)$ under the isomorphism $\hGr(m,n) \simeq \hGr(n-m,n)$ that sends the Pl\"ucker coordinate $\Delta_J$ to the Pl\"ucker coordinate $\Delta_{[n] \setminus J}$.

\begin{remark}
Theorem \ref{thm:invariants} generalizes Weyl's first fundamental theorem for $\SL(r)$ invariants of polynomial functions on matrices.  Indeed, for $k = 0$, we have Weyl's classical result.
\end{remark}

\begin{corollary}[\cite{Lam+}] The GIT-quotient $\Gr(k,r) \times \Mat(r,n) \sslash \GL(r)$ can be identified with the projective subvariety $\AA$ of $\Gr(k+\ell,n) \times \Gr(r,n)$ with homogeneous coordinate ring $A^{\SL(r)}$.  The GIT-quotient $\AA_f \coloneqq \Y_f \sslash \GL(k+m)$ is a closed subvariety of $\AA$.  
\end{corollary}
In fact, $\AA$ can be identified with a partial flag variety.

The ideal $I(\Y_f) \subset A$ of the universal amplituhedron variety is generated by the ideal $I(\AA_f) = I(Y_f)^{\SL(r)} \subset A^{\SL(r)}$.  Let $\pi: \Gr(k+\ell,n) \times \Gr(r,n) \to \Gr(k+\ell,n)$ be the projection to the first factor.  Then the map $\AA_f \to \pi(\AA_f)$ is a fiber bundle with fiber $\Gr(k,k+\ell)$.  The ideal $I(\AA_f)$, and hence also $I(\Y_f)$ is generated by the pullback of the ideal $I(\pi(\AA_f))$.  

\subsection{The direct sum map}
Let us now describe $\pi(\AA_f)$ more explicitly.

The $\GL(r)$-equivariant map $\mu \times \id: \hGr(k,n) \times \Mat(n,r) \to \hGr(r,n) \times \Mat(n,r)$ induces a (rational) map
\begin{equation}\label{eq:mu}
\Gr(k,n) \times \Mat(n,r)\sslash \GL(r) \to \AA \to \Gr(k+\ell,n).
\end{equation}
We have $\Mat(n,r)\sslash \GL(r) \simeq \Gr(r,n)$.  Let $\ker: \Gr(r,n) \to \Gr(\ell,n)$ be given by $Z \mapsto K$, where 
$$
\Delta_I(K) = (-1)^{\inv(I,[n]\setminus I)}\Delta_{[n]\setminus I}(Z)
$$
for any $I \in \binom{[n]}{\ell}$.  Here $\inv(A,B) = \#\{a \in A, b\in B \mid a > b\}$ denotes the inversion number.  The notation is explained by the following result.

\begin{lemma}
Suppose $\Delta_I(X)$ are the Pl\"ucker coordinates of a point $X \in \Gr(k,n)$.  Then the kernel $\ker(X) \in \Gr(n-k,n)$ of $X$ is represented by the point with Pl\"ucker coordinates $\Delta_J(\ker(X)) = (-1)^{\inv(J,[n]\setminus J)}\Delta_{[n]\setminus J}(X)$ for $J \in \binom{[n]}{n-k}$.
\end{lemma}
\begin{remark}\label{rem:kernsign}
We have $(-1)^{\inv(I,[n]\setminus I)} = (-1)^{o(I)+\lceil k/2 \rceil}$, where $o(I)$ denotes the number of odd elements in $I$ and $k = |I|$.  From this it is easy to see that the automorphism $\Delta_I \mapsto (-1)^{\inv(I,[n]\setminus I)}\Delta_I$ acts as a sign in every weight space.  That is, the sign associated to a monomial $\Delta_{I_1}\Delta_{I_2} \cdots \Delta_{I_d}$ depends only on the multiset $I_1 \cup I_2 \cup \cdots \cup I_d$.
\end{remark}

Let $\theta: \Gr(k+\ell,n) \to \Gr(k+\ell,n)$ be the involution given by $\Delta_I \mapsto (-1)^{\lceil k/2 \rceil + \lceil \ell/2 \rceil + o(I)}\Delta_I$, for $I \in \binom{[n]}{k+\ell}$.

\begin{proposition}
Composing the map \eqref{eq:mu} with the isomorphism $\ker^{-1}:\Gr(\ell,n) \to \Gr(r,n)$, and the isomorphism $\theta: \Gr(k+\ell,n) \to \Gr(k+\ell,n)$, we obtain the \defn{direct sum} rational morphism
$$
\bigoplus: \Gr(k,n) \times \Gr(\ell,n) \longrightarrow \Gr(k+\ell,n)
$$
given by 
$$
(X,K) \longmapsto X+K = \spn(X,K).
$$
\end{proposition}
Note that the direct sum map is only a rational map because the sum $X+K$ may have dimension less than $k+ \ell$.  On the level of homogeneous coordinate rings, the map $\bigoplus$ is dual to the ring homomorphism
$
\phi_{k,\ell}:R(k+\ell,n) \to R(k,n) \otimes R(\ell,n)
$
where
\begin{equation}\label{eq:factor}
\phi_{k,\ell}(\Delta_I) = 
\sum_{J  \subset I}(-1)^{\inv(J,I \setminus J)} \Delta_J(X) \Delta_{I \setminus J}(K).
\end{equation}


Let us also note that the isomorphism $\ker: \Gr(r,n) \to \Gr(\ell,n)$ takes positroid varieties to positroid varieties, but it takes $\Gr(r,n)_{\geq 0}$ to the \defn{twisted positive part}
\begin{equation}\label{eq:twisted}
\Gr(\ell,n)_{\geq 0,\tau} \coloneqq \{K \in \Gr(\ell,n) \mid (-1)^{\inv(I,[n]\setminus I)}\Delta_I(K) \geq 0\}.
\end{equation}
The subvariety $\pi(\AA_f) \subseteq \Gr(k+\ell,n)$ is then identified with $\overline{\bigoplus(\Pi_f \times \Gr(\ell,n))}$.  


%

\section{Sphericoid varieties}\label{sec:sphericoid}
\subsection{Ideals and cohomology classes of sphericoid varieties}
Let $f \in \Bound(k,n)$ and $f' \in \Bound(\ell,n)$.  Define the \defn{sphericoid variety} $\Pi_{f,f'}$ to be
$$
\overline{\bigoplus(\Pi_f \times \Pi_{f'})} \subseteq \Gr(k+\ell,n).
$$
Then $\Pi_{f,\id}$ is the variety $\pi(\AA_f)$ of Section \ref{sec:ampliideal}.

\begin{remark}
There is a formula for the cohomology class $[\Pi_{f,\id}] \in H^*(\Gr(k+\ell,n))$ similar to Theorem \ref{thm:trunc}.  It is an interesting problem to compute the more general cohomology classes $[\Pi_{f,f'}] \in H^*(\Gr(k+\ell,n))$.
\end{remark}

Let $R(\Pi_{f,f'})$ denote the homogeneous coordinate ring of the sphericoid variety $\Pi_{f,f'}$, and let $I(\Pi_{f,f'}) \subset R(k+\ell,n)$ be its homogeneous ideal.  For a fixed $Z$, define $\psi: R(k+\ell,n) \to R(k,r)$ by 
$$
\Delta_J \longmapsto \Delta(Y,Z_{[n]\setminus J})
$$ 
for $J \in \binom{[n]}{k+\ell}$.  Then the discussion of Section \ref{sec:ampliideal} can be summarized as:

\begin{proposition}\label{prop:YPi}
Suppose $Z$ is generic.  Then $\psi(I(\Pi_{f,\id})) = I(Y_f)$.
\end{proposition}

Thus calculating the ideal of a sphericoid variety also computes the ideal of an amplituhedron variety.  We now give a representation theoretic description of the former ideal.  Let $\kappa^d_{k,\ell}: V_{d\omega_k} \otimes V_{d\omega_\ell} \to V_{d\omega_{k+\ell}}$ be the $\GL(n)$-projection to the direct summand $V_{d\omega_{k+\ell}} \subset V_{d\omega_k} \otimes V_{d\omega_\ell}$, which appears with multiplicity one.

\begin{theorem}[\cite{Lam+}]\label{thm:sphericoid}
The $d$-th degree component of $I(\Pi_{f,f'})$ is given by
$$
I(\Pi_{f,f'})_d =\left( \kappa^d_{k,\ell}\left(V_f(d\omega_k) \otimes V(d\omega_\ell) \cap V(d\omega_k) \otimes V_{f'}(d\omega_\ell)\right) \right)^\perp.
$$
In particular, when $f' = \id$, we have
$$
I(\Pi_{f,\id})_d = \kappa^d_{k,\ell}\left(V_f(d\omega_k) \otimes V(d\omega_\ell) \right) ^\perp.
$$
Equivalently, 
$$
I(\Pi_{f,\id})_d = \{p \in R(k,\ell)_d \mid p \in I(\Pi_f)_d \otimes R(\ell,n)_d\}.
$$
\end{theorem}

%

\begin{remark}\label{rem:torus}
The torus $(\C^*)^n \subset \GL(n)$ acts on $\Gr(k,n),\Gr(\ell,n)$, and $\Gr(k+\ell,n)$.  Since positroid varieties are torus-invariant, the sphericoid variety $\Pi_{f,\id}$ is also torus-invariant.  In particular, $I(\Pi_{f,\id})_d \subset V(d\omega_{k+\ell})^*$ is spanned by weight vectors.  
\end{remark}

We give some examples explaining how to compute with Theorem \ref{thm:sphericoid}.

\subsection{When $\Pi_{f,\id}$ is a linear hypersurface}\label{sec:hyper}
Let $J \in \binom{[n]}{k+\ell}$.  By Theorem \ref{thm:sphericoid}, we have $\Delta_J \in I(\Pi_{f,\id})$ if 
$$
\phi_{k,\ell}(\Delta_J) = \sum_{I \subset J} (-1)^{\inv(I, J \setminus I)} \Delta_I \otimes \Delta_{J \setminus I} \in I(\Pi_f)_1 \otimes R(\ell,n)_1,
$$
using \eqref{eq:factor}.  By Theorem \ref{thm:Posideal}, $I(\Pi_f)_1$ has as basis the Pl\"ucker coordinates $\{\Delta_I \mid I \in \M(f)\}$, so  $\Delta_J \in I(\Pi_{f,\id})$ if and only if 
$$
\left \{I \in \binom{[n]}{k} \mid I \subset J \right\} \subset \M(f).
$$
For example, take $k = 2$.  We classified $\Bound(2,n)$ in Section \ref{sec:k2}.  Suppose $\Pi_f$ is given by the conditions $\rank(\spn(v_a,v_{a+1},\ldots,v_b)) \leq 1$ for cyclic intervals $[a_i,b_i]$ (and no rank conditions of the form $v_c = 0$).  Then $\Delta_J \in I(\Pi_{f,\id})$ if $J \subset [a_i,b_i]$ for some $i$.  If all the cyclic intervals $[a_i,b_i]$ have cardinality less than $k+\ell$, then no Pl\"ucker coordinate $\Delta_J$ vanishes on $\Pi_{f,\id}$.  By Remark \ref{rem:torus}, $I(\Pi_{f,\id})_1 = 0$ in this case.  

\subsection{Degree-two examples}\label{sec:deg2}
Let $k = 2$, and suppose $f \in \Bound(2,n)$ is given by the conditions $\rank(\spn(v_a,v_{a+1},\ldots,v_b)) \leq 1$ for cyclic intervals $[a_1,b_1], \ldots, [a_s,b_s]$ (and no rank conditions of the form $v_c = 0$).  Let $\beta_i = |[a_i,b_i]|$.

By Proposition \ref{prop:k2}, we have $\tF_f \equiv [\Pi_f] = \prod_{i=1}^s h_{\beta_i -1} \in H^*(\Gr(2,n))$.

Suppose $k = 2,n,r$ are fixed and $f$ is chosen so that $\dim(\Pi_f) = \dim(\Gr(2,r)) - 1 = 2m-1$, where $m = r-2$.  But $\codim(\Pi_f) = \sum_{i=1}^s (\beta_i - 1) = n-s$, so we have $s = n-2\ell-1$, where $\ell = n-r = n-m-2$.  Applying Theorem \ref{thm:trunc}, we have the following cases:
\begin{enumerate}
\item
If $\max(\beta_i) > \ell+2$ then $[s_{(\ell+1,\ell)}]\tF_f = 0$.  In this case $f \in \Bound(2,n)$ is not independent.  Thus $Y_f$ has codimension two or more.
\item
If $\max(\beta_i) = \ell+2$ then $[s_{(\ell+1,\ell)}]\tF_f = 1$.  In this case, by Section \ref{sec:hyper}, $\Pi_{f,\id}$ is cut out of $\Gr(\ell+2,n)$ by the linear equation $\Delta_{J} = 0$ where $J = [a_t,b_t]$ is the (necessarily unique) cyclic interval satisfying $|[a_t,b_t]| = \ell+2$.  By Proposition \ref{prop:YPi}, $Y_f$ is cut of $\Gr(2,r)$ by the equation $\Delta_{Y,Z_{[n] \setminus J}} = 0$, so it is a linear hypersurface.
\item
If $\max(\beta_i) < \ell+2$ and 
$$
\# \{ i \mid \beta_i \geq 2 \} \geq 4,
$$
then $[\Pi_f]$ is the product of at least four (non-identity) homogeneous symmetric functions.  In this case, $[s_{(\ell+1,\ell)}]\tF_f \geq 3$.  We expect that in general $Y_f$ is cut out by an equation with degree three or higher (though Theorem \ref{thm:trunc} only guarantees that it is cut out by an equation with degree at most $[s_{(\ell+1,\ell)}]\tF_f$).  We will give an example of such an equation in Section \ref{sec:deg3ex}.
\end{enumerate}

There is a fourth case, where we expect $Y_f$ to be a codimension one hypersurface cut out by a quadratic equation.

\begin{proposition}[\cite{Lam+}]\label{prop:quad}
Suppose $\max(\beta_i) < \ell+2$ and 
$$
\# \{ i \mid \beta_i \geq 2 \} = 3.
$$
Then $\tF_f \equiv h_{a-1}h_{b-1}h_{c-1} \in H^*(\Gr(2,n))$ for some $a,b,c \geq 2$ and $[s_{(\ell+1,\ell)}]\tF_f = 2$.  There is a $(\ell+2,n)$-partial non-crossing matching $(\tau,\emptyset)$ with $(a+b+c)/2$ strands such that $F_{(\tau,\emptyset)} \in I(\Pi_{f,\id})_2$, and this element generates the ideal $I(\Pi_{f,\id})$.
\end{proposition}
The non-crossing matching $\tau$ of Proposition \ref{prop:quad} is illustrated in the following picture.  Here, $f \in \Bound(2,n)$ is given by rank conditions for the cyclic intervals $[1,4]$, $[5,7]$, and $[8,10]$, and we have $\beta_1 = 4$, $\beta_2 = 3$, and $\beta_3 = 3$.
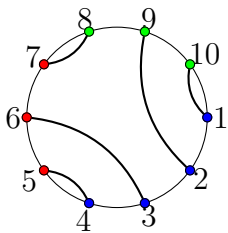
\begin{center}
\begin{tikzpicture}[scale=0.6,baseline=-0.5ex]
\coordinate (a1) at (0:2);
\coordinate (a2) at (-36:2);
\coordinate (a3) at (-72:2);
\coordinate (a4) at (-108:2);
\coordinate (a5) at (-144:2);
\coordinate (a6) at (180:2);
\coordinate (a7) at (144:2);
\coordinate (a8) at (108:2);
\coordinate (a9) at (72:2);
\coordinate (a10) at (36:2);

\node at (0:2.3) {$1$};
\node at (-36:2.3) {$2$};
\node at (-72:2.3) {$3$};
\node at (-108:2.4) {$4$};
\node at (-144:2.4) {$5$};
\node at (180:2.3) {$6$};
\node at (144:2.3) {$7$};
\node at (108:2.3) {$8$};
\node at (72:2.3) {$9$};
\node at (36:2.4) {$10$};

\draw (0,0) circle (2);

\draw[thick] (a1) to [bend left] (a10);
\draw[thick] (a2) to [bend left] (a9);
\draw[thick] (a3) to [bend right] (a6);
\draw[thick] (a4) to [bend right] (a5);
\draw[thick] (a7) to [bend right] (a8);

\filldraw[blue] (a1) circle (0.1cm);
\draw (a1) circle (0.1cm);
\filldraw[blue] (a2) circle (0.1cm);
\draw (a2) circle (0.1cm);
\filldraw[blue] (a4) circle (0.1cm);
\draw (a4) circle (0.1cm);
\filldraw[blue] (a3) circle (0.1cm);
\draw (a3) circle (0.1cm);
\filldraw[red] (a5) circle (0.1cm);
\draw (a5) circle (0.1cm);
\filldraw[red] (a6) circle (0.1cm);
\draw (a6) circle (0.1cm);
\filldraw[red] (a7) circle (0.1cm);
\draw (a7) circle (0.1cm);
\filldraw[green] (a8) circle (0.1cm);
\draw (a8) circle (0.1cm);
\filldraw[green] (a9) circle (0.1cm);
\draw (a9) circle (0.1cm);
\filldraw[green] (a10) circle (0.1cm);
\draw (a10) circle (0.1cm);
\end{tikzpicture}

\end{center}

Proposition \ref{prop:quad} is proven by a general formula that expresses $\phi_{k,\ell}(F_{(\tau,T)})$ as a linear combination of $F_{(\eta,T')} \otimes F_{(\nu,T'')}$ where $(\eta,T') \in \AA_{2,n}$ is a $(2,n)$-partial non-crossing matching and $(\nu,T'') \in \AA_{\ell,n}$ is a $(\ell,n)$-partial non-crossing matching.  

\begin{example}\label{ex:deg2}
Suppose $n = 6$, and $f$ is given by the cyclic intervals $[1,2],[3,4],[5,6]$.  Then $\dim(\Pi_f) = 5$, so with $r = 5$ and $\ell = 1$, we have that $Y_f \subset \Gr(2,5)$ is codimension one.  In this case $Y_f$ is cut out by the single equation $\psi(F_{(\tau,\emptyset)})$, where $\tau = \{(1,6),(2,3),(4,5)\}$.  One calculates using Theorem \ref{thm:TL} that
$$
F_{(\tau,\emptyset)} = \Delta_{124}\Delta_{356} - \Delta_{123}\Delta_{456}.
$$
\end{example}

\begin{example}
Suppose $n = 8$, and $f$ is given by the cyclic intervals $[1,3], [4,6],[7,8]$.  Then $\dim(\Pi_f) = 7$, so with $r = 6$ and $\ell = 2$, we have that $Y_f \subset \Gr(2,6)$ is codimension one.  In this case $Y_f$ is cut out by the single equation $\psi(F_{(\tau,\emptyset)})$, where $\tau = \{(1,8),(2,5),(3,4),(6,7)\}$.
\end{example}

\subsection{A degree-three example}
\label{sec:deg3ex}
Let $k = 2, m = 5, n = 9$.  Consider the bounded affine permutation $f = [2,3,6,5,8,7,10,9,13] \in \Bound(2,9)$.  Then $\Pi_f$ is cut out by the conditions
$$
\dim \spn(v_1,v_2,v_3) \leq 1, \;\; \dim \spn(v_4,v_5) \leq 1,\;\; \dim \spn(v_6,v_7) \leq 1,\;\; \dim \spn(v_8,v_9) \leq 1.
$$
By Proposition \ref{prop:k2}, or using the reduced factorization $f = \id s_1s_0s_7s_5s_3$, we obtain
$$
\tF_f = h_2h_1^3 = 3s_{3,2} + \text{other terms}.
$$
According to Theorem \ref{thm:trunc},  $f$ is independent.  Since $\dim(\Pi_f) = 9$, $Y_f$ is a hypersurface in $\Gr(2,7)$ and $[Y_f] \in H^*(\Gr(2,7))$ is either equal to $s_1$ or $3s_1$.  

By Proposition \ref{prop:YPi}, we can check that $Y_f$ is not a linear hypersurface by checking that none of the Pl\"ucker coordinates $\Delta_J$ vanish identically on $\Pi_{f,\id}$.  Note that $\Pi_{f,\id}$ is a torus invariant subvariety of $\Gr(k+\ell,n)$ (for the torus $(\C^*)^n \subset \GL(n)$), so $I(\Pi_{f,\id})$ is spanned by weight vectors, and in particular $I(\Pi_{f,\id})_1$ is spanned by Pl\"ucker coordinates.  It follows that we must have $[Y_f] = 3s_1$.  (We can also check this by numerically computing that $Z_\Gr: \Pi_f \dashedrightarrow Y_f$ is a birational map.)

We shall confirm that $[Y_f] = 3s_1$ by finding a section in $R(2,7)_3 = \Gamma(\Gr(2,7),\O(3))$ that cuts out $Y_f$.  Indeed, in terms of Pl\"ucker coordinates, we have that $\Pi_{f,\id}$ is cut out by 
\begin{align}\label{eq:g}g &= \Delta_{1, 2, 3, 5}\Delta_{1, 2, 3, 7}\Delta_{4, 6, 8, 9} - \Delta_{1, 2, 3, 4}\Delta_{1, 2, 3, 7}\Delta_{5, 6, 8, 9}
   \\ \nonumber
&
 -\Delta_{1, 2, 3, 5}\Delta_{1, 2, 3, 6}\Delta_{4, 7, 8, 9}+
 \Delta_{1, 2, 3, 4}\Delta_{1, 2, 3, 6}\Delta_{5, 7, 8, 9} .
\end{align}
The reader is invited to check that $Y_f$ is cut out by the cubic
\begin{align*}&\Delta(Y,Z_{\overline{1, 2, 3, 5}})\Delta(Y,Z_{\overline{1, 2, 3, 7}})\Delta(Y,Z_{\overline{4, 6, 8, 9}}) -
 \Delta(Y,Z_{\overline{1, 2, 3, 4}})\Delta(Y,Z_{\overline{1, 2, 3, 7}})\Delta(Y,Z_{\overline{5, 6, 8, 9}}) 
   \\
-& \Delta(Y,Z_{\overline{1, 2, 3, 5}})\Delta(Y,Z_{\overline{1, 2, 3, 6}})\Delta(Y,Z_{\overline{4, 7, 8, 9}})+\Delta(Y,Z_{\overline{1, 2, 3, 4}})\Delta(Y,Z_{\overline{1, 2, 3, 6}})\Delta(Y,Z_{\overline{5, 7, 8, 9}})
\end{align*}
where $\overline{I} \coloneqq [9] \setminus I$, agreeing with Proposition \ref{prop:YPi}. 
The cubic $g$ is in fact an instance of a \emph{web immanant} introduced in \cite{Lamweb}.  It is indexed by the following \emph{web}:

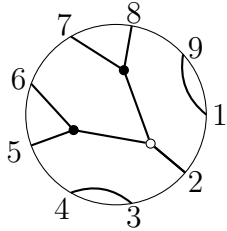
\begin{center}
\begin{tikzpicture}[scale=0.6,baseline=-0.5ex]
\coordinate (a1) at (0:2);
\coordinate (a2) at (-40:2);
\coordinate (a3) at (-80:2);
\coordinate (a4) at (-120:2);
\coordinate (a5) at (-160:2);
\coordinate (a6) at (160:2);
\coordinate (a7) at (120:2);
\coordinate (a8) at (80:2);
\coordinate (a9) at (42:2);

\node at (0:2.3) {$1$};
\node at (-40:2.3) {$2$};
\node at (-80:2.3) {$3$};
\node at (-120:2.4) {$4$};
\node at (-160:2.4) {$5$};
\node at (160:2.3) {$6$};
\node at (120:2.3) {$7$};
\node at (80:2.3) {$8$};
\node at (40:2.3) {$9$};

\draw (0,0) circle (2);

\coordinate (w1) at (-40:1);
\coordinate (b1) at (-160:1);
\coordinate (b2) at (80:1);

\draw[thick] (a1) to [bend left] (a9);
\draw[thick] (a2) to (w1);
\draw[thick] (a3) to [bend right] (a4);
\draw[thick] (w1) to (b1);
\draw[thick] (w1) to (b2);
\draw[thick] (w1) to (a2);
\draw[thick] (b1) to (a5);
\draw[thick] (b1) to (a6);
\draw[thick] (b2) to (a7);
\draw[thick] (b2) to (a8);

\whitedot{(w1)};
\blackdot{(b2)};
\blackdot{(b1)};
\end{tikzpicture}

\end{center}

The calculation of this web immanant is obtained by combining the results of \cite{Lamweb} with \cite{KhKu}.  We sketch the calculation assuming the reader is familiar with both works.  Consider the following 5 tableaux
\begin{equation} \label{eq:5tab}
\tableau[sbY]{1&1&4\\2&2&6\\3&3&8\\5&7&9} \qquad
\tableau[sbY]{1&1&4\\2&2&7\\3&3&8\\5&6&9} \qquad
\tableau[sbY]{1&1&5\\2&2&6\\3&3&8\\4&7&9} \qquad
\tableau[sbY]{1&1&5\\2&2&7\\3&3&8\\4&6&9} \qquad
\tableau[sbY]{1&1&6\\2&2&7\\3&3&8\\4&5&9}.
\end{equation}
The growth algorithm of \cite{KhKu} gives a bijection between these 5 tableaux and the following 5 webs:
$$
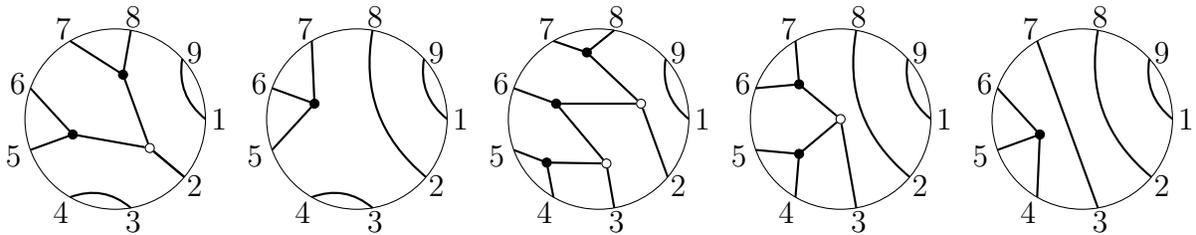
\begin{tikzpicture}[scale=0.6,baseline=-0.5ex]
\coordinate (a1) at (0:2);
\coordinate (a2) at (-40:2);
\coordinate (a3) at (-80:2);
\coordinate (a4) at (-120:2);
\coordinate (a5) at (-160:2);
\coordinate (a6) at (160:2);
\coordinate (a7) at (120:2);
\coordinate (a8) at (80:2);
\coordinate (a9) at (42:2);

\node at (0:2.3) {$1$};
\node at (-40:2.3) {$2$};
\node at (-80:2.3) {$3$};
\node at (-120:2.4) {$4$};
\node at (-160:2.4) {$5$};
\node at (160:2.3) {$6$};
\node at (120:2.3) {$7$};
\node at (80:2.3) {$8$};
\node at (40:2.3) {$9$};

\draw (0,0) circle (2);

\coordinate (w1) at (-40:1);
\coordinate (b1) at (-160:1);
\coordinate (b2) at (80:1);

\draw[thick] (a1) to [bend left] (a9);
\draw[thick] (a2) to (w1);
\draw[thick] (a3) to [bend right] (a4);
\draw[thick] (w1) to (b1);
\draw[thick] (w1) to (b2);
\draw[thick] (w1) to (a2);
\draw[thick] (b1) to (a5);
\draw[thick] (b1) to (a6);
\draw[thick] (b2) to (a7);
\draw[thick] (b2) to (a8);

\whitedot{(w1)};
\blackdot{(b2)};
\blackdot{(b1)};
\end{tikzpicture}
\begin{tikzpicture}[scale=0.6,baseline=-0.5ex]
\coordinate (a1) at (0:2);
\coordinate (a2) at (-40:2);
\coordinate (a3) at (-80:2);
\coordinate (a4) at (-120:2);
\coordinate (a5) at (-160:2);
\coordinate (a6) at (160:2);
\coordinate (a7) at (120:2);
\coordinate (a8) at (80:2);
\coordinate (a9) at (42:2);

\node at (0:2.3) {$1$};
\node at (-40:2.3) {$2$};
\node at (-80:2.3) {$3$};
\node at (-120:2.4) {$4$};
\node at (-160:2.4) {$5$};
\node at (160:2.3) {$6$};
\node at (120:2.3) {$7$};
\node at (80:2.3) {$8$};
\node at (40:2.3) {$9$};

\draw (0,0) circle (2);

\coordinate (b2) at (160:1);

\draw[thick] (a1) to [bend left] (a9);
\draw[thick] (a2) to [bend left] (a8);
\draw[thick] (a3) to [bend right] (a4);
\draw[thick] (b2) to (a5);
\draw[thick] (b2) to (a6);
\draw[thick] (b2) to (a7);

\blackdot{(b2)}
\end{tikzpicture}
\begin{tikzpicture}[scale=0.6,baseline=-0.5ex]
\coordinate (a1) at (0:2);
\coordinate (a2) at (-40:2);
\coordinate (a3) at (-80:2);
\coordinate (a4) at (-120:2);
\coordinate (a5) at (-160:2);
\coordinate (a6) at (160:2);
\coordinate (a7) at (120:2);
\coordinate (a8) at (80:2);
\coordinate (a9) at (42:2);

\node at (0:2.3) {$1$};
\node at (-40:2.3) {$2$};
\node at (-80:2.3) {$3$};
\node at (-120:2.4) {$4$};
\node at (-160:2.4) {$5$};
\node at (160:2.3) {$6$};
\node at (120:2.3) {$7$};
\node at (80:2.3) {$8$};
\node at (40:2.3) {$9$};

\draw (0,0) circle (2);

\coordinate (w1) at (20:1);
\coordinate (w2) at (-80:1);
\coordinate (b1) at (160:1);
\coordinate (b2) at (-140:1.5);
\coordinate (b3) at (100:1.5);

\draw[thick] (a1) to [bend left] (a9);
\draw[thick] (w1) to (b1);
\draw[thick] (w1) to (b3);
\draw[thick] (w1) to (a2);
\draw[thick] (w2) to (b1);
\draw[thick] (w2) to (b2);
\draw[thick] (w2) to (a3);
\draw[thick] (b2) to (a4);
\draw[thick] (b2) to (a5);
\draw[thick] (b1) to (a6);
\draw[thick] (b3) to (a7);
\draw[thick] (b3) to (a8);

\whitedot{(w1)};
\whitedot{(w2)};
\blackdot{(b3)};
\blackdot{(b2)};
\blackdot{(b1)};
\end{tikzpicture}
\begin{tikzpicture}[scale=0.6,baseline=-0.5ex]
\coordinate (a1) at (0:2);
\coordinate (a2) at (-40:2);
\coordinate (a3) at (-80:2);
\coordinate (a4) at (-120:2);
\coordinate (a5) at (-160:2);
\coordinate (a6) at (160:2);
\coordinate (a7) at (120:2);
\coordinate (a8) at (80:2);
\coordinate (a9) at (42:2);

\node at (0:2.3) {$1$};
\node at (-40:2.3) {$2$};
\node at (-80:2.3) {$3$};
\node at (-120:2.4) {$4$};
\node at (-160:2.4) {$5$};
\node at (160:2.3) {$6$};
\node at (120:2.3) {$7$};
\node at (80:2.3) {$8$};
\node at (40:2.3) {$9$};

\draw (0,0) circle (2);

\coordinate (w1) at (-40:0);
\coordinate (b1) at (-140:1.2);
\coordinate (b2) at (140:1.2);

\draw[thick] (a1) to [bend left] (a9);
\draw[thick] (a2) to [bend left] (a8);
\draw[thick] (b1) to (a4);
\draw[thick] (b1) to (a5);
\draw[thick] (b2) to (a6);
\draw[thick] (b2) to (a7);
\draw[thick] (w1) to (a3);
\draw[thick] (w1) to (b1);
\draw[thick] (w1) to (b2);

\whitedot{(w1)};
\blackdot{(b1)};
\blackdot{(b2)};
\end{tikzpicture}
\begin{tikzpicture}[scale=0.6,baseline=-0.5ex]
\coordinate (a1) at (0:2);
\coordinate (a2) at (-40:2);
\coordinate (a3) at (-80:2);
\coordinate (a4) at (-120:2);
\coordinate (a5) at (-160:2);
\coordinate (a6) at (160:2);
\coordinate (a7) at (120:2);
\coordinate (a8) at (80:2);
\coordinate (a9) at (42:2);

\node at (0:2.3) {$1$};
\node at (-40:2.3) {$2$};
\node at (-80:2.3) {$3$};
\node at (-120:2.4) {$4$};
\node at (-160:2.4) {$5$};
\node at (160:2.3) {$6$};
\node at (120:2.3) {$7$};
\node at (80:2.3) {$8$};
\node at (40:2.3) {$9$};

\draw (0,0) circle (2);

\coordinate (b1) at (-160:1);

\draw[thick] (a1) to [bend left] (a9);
\draw[thick] (a2) to [bend left] (a8);
\draw[thick] (a3) to (a7);
\draw[thick] (b1) to (a5);
\draw[thick] (b1) to (a6);
\draw[thick] (b1) to (a4);
\blackdot{(b1)};
\end{tikzpicture}
$$
The expansion of the standard monomials labeled by these 5 tableaux in terms of the corresponding web immanant (see \cite[Theorem 4.13]{Lamweb}) is given by the $5 \times 5$ matrix
$$
\begin{bmatrix} 1&1&1&1&0 \\
0&1&0&1&1 \\
0&0&1&1&0 \\
0&0&0&1&1\\
0&0&0&0&1
\end{bmatrix}.
$$
The first row of the inverse of this matrix has entries $(1,-1,-1,1,0)$.  These are the coefficients of the standard monomials in the web immanant $g$.  In fact, the cubic $g$ also belongs to the dual canonical basis, and is indexed by the leftmost tableau in \eqref{eq:5tab}.

\subsection{A conjecture}
The above examples give evidence for the following conjecture.
\begin{conjecture}
The homogeneous ideal $I(\Pi_{f,\id})$ is generated by elements of the dual canonical basis of $R(k+\ell,n)$.
\end{conjecture}
%
%

\subsection{Sphebic graphs}
\label{sec:sphebic}
In Part \ref{part:one} of this work, we constructed points in $\Gr(k,n)_{\geq 0}$ by enumerating perfect matchings in a planar bipartite network.  We now discuss a construction of points in a sphericoid variety $\Pi_{f,f'}$ using bipartite graphs on a sphere.  The more general spherical bicolored graphs, might be called ``sphebic" graphs, following Postnikov's terminology.  

Let $S^2$ be the two-sphere and $H \subset S^2$ be the equator of the sphere, $H^+$ the upper hemisphere and $H^-$ the lower hemisphere, so that $H^+ \cap H^- = H$.  Both $H^+$ and $H^-$ are closed disks.
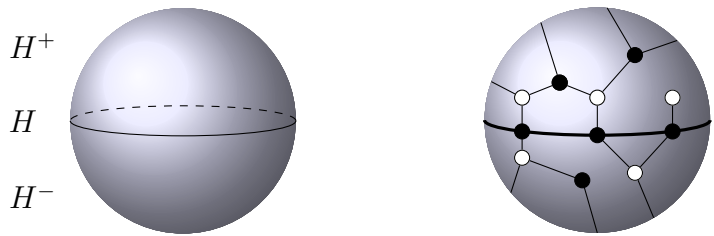
\begin{equation*}
\begin{tikzpicture}
\fill[ball color=blue!10] (0,0) circle (1.5 cm);
\newcommand\latitude[1]{%
  \draw (#1:1.5) arc (0:-180:{1.5*cos(#1)} and {0.2*cos(#1)});
  \draw[dashed] (#1:1.5) arc (0:180:{1.5*cos(#1)} and {0.2*cos(#1)});
}
\node at (-2,1) {$H^+$};
\node at (-2.13,0) {$H$};
\node at (-2,-1) {$H^-$};
\latitude{0};
\end{tikzpicture}
\qquad \qquad \qquad
\begin{tikzpicture}
\fill[ball color=blue!10] (0,0) circle (1.5 cm);
\newcommand\frontlatitude[1]{%
  \draw[very thick] (#1:1.5) arc (0:-180:{1.5*cos(#1)} and {0.2*cos(#1)});
}
\frontlatitude{0};
\draw (1,-0.15)--(1,0.3);
\draw (0,-0.2)--(0,0.3);
\draw (-1,-0.2)--(-1,0.3);
\draw (-1,0.3)--(-0.5,0.5);
\draw (-1,0.3)--(160:1.5);
\draw (-0.5,0.5)--(0,0.3);
\draw (-0.5,0.5)--(120:1.5);
\draw (0,0.3)--(60:1);
\draw (60:1)--(45:1.5);
\draw (60:1)--(85:1.5);

\draw (0,-0.2)--(0.5,-0.7)--(1,-0.15);
\draw (0.5,-0.7)--(300:1.5);
\draw (-1,-0.15)--(-1,-0.5)--(220:1.5);
\draw (-1,-0.5)--(-0.2,-0.8)--(270:1.5);

\blackdot{(-0.2,-0.8)}
\blackdot{(0,-0.2)}
\blackdot{(-1,-0.15)}
\blackdot{(1,-0.15)}
\whitedot{(-1,0.3)}
\whitedot{(-1,-0.5)}
\whitedot{(0.5,-0.7)}
\whitedot{(1,0.3)}
\whitedot{(-0.5,0.5)}
\whitedot{(0,0.3)}
\blackdot{(-0.5,0.5)}
\blackdot{(60:1)}
\end{tikzpicture}
\end{equation*}

A \defn{spherical bipartite network} is a weighted bipartite graph $N$ embedded into $S^2$ with distinguished vertices $1,2,\ldots,n$ arranged in order on $H$ such that $N \cap H$ consists only of these distinguished equatorial vertices, and both $N \cap H^+$ and $N \cap H^-$ are planar bipartite networks.

We now define the equatorial measurements of a spherical bipartite network $N$.  For simplicity, we will make the following assumption: all the vertices $1,2,\ldots,n$ on $H$ are black.  We can always add two-valent vertices (move (M2)) to arrange a boundary vertex to have the desired color, so we lose no generality.

An \defn{almost perfect matching} $\Pi$ of $N$ is a collection of edges using all vertices of $N \setminus H$ exactly once each, and using each of the vertices $1,2,\ldots,n$ either once or not at all.  The equatorial subset $I(\Pi)$ is the set of equatorial vertices that are used.  Let $I^+(\Pi)$ (resp. $I^-(\Pi)$) be the set of equatorial vertices connected to an edge in $H^+$ (resp. $H^-$), so we have $I(\Pi) = I^+(\Pi) \sqcup I^-(\Pi)$.  

Let 
$$
k = \#\{\text{white vertices in $N \setminus H$}\} - \#\{\text{black vertices in $N \setminus H$}\}.
$$
Then $|I(\Pi)| = k$ for any almost perfect matching $\Pi$ in $N$.  Define the {\it equatorial measurement}
$$
\Delta_I(N) \coloneqq \sum_{I(\Pi) = \Pi} (-1)^{\inv(I^+(\Pi),I^-(\Pi))} \wt(\Pi).
$$

\begin{theorem}\label{thm:sphebic}
Suppose $\Delta_I(N) \neq 0$ for some $I \in \binom{[n]}{k})$.  Then the vector $X(N) \coloneqq (\Delta_I(N) \mid I \in \binom{[n]}{k})$ defines a point in the Grassmannian $\Gr(k,n)$.
\end{theorem}
Here we say ``affine cone" because it is possible that all $\Delta_I(G)$ are zero.  Theorem \ref{thm:sphebic} follows from the next result.

\begin{theorem}\label{thm:equatorial}
Let $X^+ \coloneqq X(N\cap H^+) \in \Gr(k_1,n)$ and $X^- \coloneqq X(N\cap H^-) \in \Gr(k_2,n)$ be the points represented by the upper and lower planar bipartite networks.  Then $X(N) = \bigoplus(X^+,X^-) \in\Gr(k,n) \cup \{0\}$, and $k = k_1+k_2$.
\end{theorem}
Here $X(N) = 0$ means that all $\Delta_I(N) = 0$ for all $I$.
\begin{proof}
Follows from the definition of equatorial measurements and equation \eqref{eq:factor}.
\end{proof}

Note that if $V^+ + V^-$ has dimension smaller than $k = k_1 + k_2$, then $X(N) = 0$.  The equatorial measurement vector $X(N)$ is usually not nonnegative, even when all the weights are nonnegative.  

Let $f \in \Bound(k,n)$ and $f' \in \Bound(\ell,n)$.
Let $N^+(a_1,a_2,\ldots,a_d)$ be a planar bipartite network (with edge weights $a_i$ varying over $\R_{>0}^d$ or $\C^d$) that ``parametrizes" $(\Pi_f)_{>0}$ or $\Pi_f$.  Let $N^-(b_1,b_2,\ldots,b_{d'})$ be a planar bipartite network that parametrizes $(\Pi_{f'})_{>0}$ or $\Pi_{f'}$.  Let $N(a_1,\ldots,a_d,b_1,\ldots,b_{d'})$ be the spherical bipartite network obtained from $N^+$ and $N^-$ by gluing them at the boundary vertices.  

\begin{proposition}
The Zariski closure of $\{X(N(a_1,a_2,\ldots,a_d,b_1,\ldots,b_{d'}))\} \subset \Gr(k+\ell,n)$ is $\Pi_{f,f'}$.
\end{proposition}
%

\section{Facets of Grassmann polytopes}\label{sec:facets}
In this section, we give taste of the facial structure of Grassmann polytopes.  We will not attempt a full development of the theory; instead we will illustrate some definitions with examples computed using the techniques of Sections \ref{sec:trunc} and \ref{sec:ampliideal}.  

\subsection{Definition of facet}

Let $P = Z(\Pi_{h,\geq 0}) \subset \Gr(k,r)$ be a Grassmann polytope.   
Let $p$ be a homogeneous element of the homogeneous coordinate ring $R(k,r)$ of $\Gr(k,r)$.  We call the set
$$
F \coloneqq \{p = 0\} \cap P
$$
a \defn{global geometric facet} of $P$ if

\noindent
(1) $p$ takes a constant sign on $P$, and

\noindent
(2)
$F$ contains $Z(\Pi_{g,\geq 0})$ for some $g < h \in \hBound(k,n)$ satisfying $\dim(Z(\Pi_{g,\geq 0})) = \dim(P) -1$.

To make sense of condition (1) precisely, we use the corresponding cones in $\hGr(k,r)$, as in Remark \ref{rem:cones}.  

By definition, a global geometric facet $F$ contains at least one Grassmann polytope $Z(\Pi_{g,\geq 0})$.  In fact, $F$ is typically a (non-disjoint) union of many Grassmann polytopes, as we'll illustrate.

\begin{remark}
Here we only define the notion of a global facet.  There are other sets on the boundary of $P$ that may be considered facets and do not satisfy these conditions.
\end{remark}

\subsection{Facets of the amplituhedron}
In this section, we assume that $Z$ is positive and $P = Z(\Gr(k,n)_{\geq 0})$ is the amplituhedron.  Write $X, K, S$ for typical points in the Grassmannians $\Gr(k,n), \Gr(\ell,n)$, and $\Gr(k+\ell,n)$.  If $p$ is a torus-invariant polynomial (that is, a weight vector in $R(k+\ell,n)_d$ for some $d$) in the Pl\"ucker coordinates $\Delta_J(S)$, write $\psi(p)$ for the polynomial obtained via the substitution $\Delta_J(S) \mapsto \Delta(Y,Z_{[n]\setminus J})$.  By Proposition \ref{prop:YPi} if  $p$ vanishes on $\Pi_{f,\id}$ then $\psi(p)$ vanishes on $Z(\Pi_f)$.

We state the positive version of this result.  Recall that the twisted totally nonnegative Grassmannian $\Gr(\ell,n)_{\geq 0, \tau})$ is defined in \eqref{eq:twisted}.
\begin{lemma}\label{lem:transl}
Suppose $Z$ is positive.  If $p$ has a fixed sign on $\bigoplus(\Pi_{f,\geq 0}, \Gr(\ell,n)_{\geq 0, \tau})$ then $\psi(p)$ has a fixed sign on $Z(\Pi_{f,\geq 0})$.
\end{lemma}
Here and henceforth, ``fixed sign" is made sense of by using cones (see Remark \ref{rem:cones}).

Say that $I \in \binom{[n]}{m}$ satisfies the \defn{evenness condition} if for every $i,i' \notin I$, the number of elements in $I$ between $i$ and $i'$ is even.  For example, $\{2,3,6,7\} \subset [8]$ satisfies the evenness condition.  The following result is well known for cyclic polytopes \cite{Zie}, and is discussed in \cite{AT} for amplituhedra with even $m$.

\begin{proposition}\label{prop:AT}
Suppose $Z$ is positive.  Let $I$ satisfy the evenness condition.  Then $\Delta(Y,Z_I)$ takes a fixed sign on $P = Z(\Gr(k,n)_{\geq 0})$.
\end{proposition}
\begin{proof}
Let $J = [n] \setminus I$.  Then 
\begin{equation}\label{eq:exp}
\phi(\Delta_J(S)) = \sum_{L \subset J} (-1)^{\inv(L,J\setminus L)} \Delta_L(X) \Delta_{J \setminus L}(K).
\end{equation}
  Now, if $X \in \Gr(k,n)_{\geq 0}$ and $K \in \Gr(\ell,n)_{\geq 0, \tau}$, then the term indexed by $L \in \binom{[n]}{k}$ will have sign (by Remark \ref{rem:kernsign})
$$
(-1)^{\inv(L,J\setminus L)} (-1)^{o(J \setminus L) + \lceil k/2 \rceil}
$$
where $o(T)$ denotes the number of odd elements in $T$.  When $I$ satisfies the evenness condition, the parity of $\inv(L,J\setminus L) + o(J\setminus L)$ does not depend on $L$.  Thus $\Delta_J(S)$ has a fixed sign on $\bigoplus(\Gr(k,n)_{\geq 0}, \Gr(\ell,n)_{\geq 0, \tau})$.
By Lemma \ref{lem:transl}, $\Delta(Y,Z_I)$ takes a fixed sign on $P$.
\end{proof}

Let us investigate the intersection $P \cap \Delta(Y,Z_I)$ for $I$ satisfying the evenness condition.  
Set $J = [n] \setminus I$.   Then $\Delta(Y,Z_I)$ vanishes at $Y = Z_\Gr(X)$ only if all the monomials in \eqref{eq:exp} vanish, and this happens exactly when $\Delta_L(X) = 0$ vanishes for every $L \in \binom{[J]}{k}$.  The ideal generated by $\{\Delta_L \mid L \in \binom{[J]}{k}\} \subset R(k,n)$ is the Schubert variety $A \subset \Gr(k,n)$ given by
\begin{equation}\label{eq:A}
A = \{X \in \Gr(k,n) \mid \dim(X \cap \spn( e_i \mid i \in I)) \geq 1\}.
\end{equation}
Thus when $Z$ is positive, the geometric facet cut out by $\{\Delta(Y,Z_I)=0\}$ is given by 
$$
P \cap \{\Delta(Y,Z_I)=0\} = Z(A_{\geq 0})
$$
where $A_{\geq 0} \coloneqq A \cap \Gr(k,n)_{\geq 0}$.  Since we have a disjoint union $\Gr(k,n)_{\geq 0} = \bigsqcup \Pi_{f,>0}$, and every $X \in \Pi_{f,>0}$ has matroid equal to $\M(f)$, we see that 
\begin{equation}\label{eq:Ageq}
A_{\geq 0} = \bigsqcup_{\Pi_g \subseteq A} \Pi_{g, \geq 0}.
\end{equation}
In other words, $A_{\geq 0}$ only contains points in $\Pi_{g,\geq 0}$ when the whole of $\Pi_g$ is contained in $A$.  Note that the equality $A =  \bigsqcup_{\Pi_g \subseteq A} \Pi_{g}$ is certainly not true.

We can now explain one of the motivations for our study of canonical bases in Sections \ref{sec:coord} and \ref{sec:ampliideal}.  By Theorem \ref{thm:Posideal}(6), we have complete control of the vanishing and non-vanishing of canonical basis elements on $\Gr(k,n)_{\geq 0}$ (generalizing the fact that we have classified all positroids in Section \ref{sec:positroids}).  If the variety $A$ were to be cut out not simply by minors, but by higher degree elements of the canonical basis, then we would obtain a union analogous to \eqref{eq:Ageq} by using Theorem \ref{thm:Posideal}.

%

\medskip

Let us investigate the union \eqref{eq:Ageq} further.
%

\subsubsection{Suppose $m = 2$.}  Without loss of generality, we can pick $I = \{1,2\}$, and so $J = \{3,4,\ldots,n\}$.  Then $A$ is given by rank conditions on cyclically consecutive intervals, so it is itself a positroid variety.  
For example, if $k = 2$, it is the positroid variety indexed by the bounded affine permutation $g = [3,1+n,4,5,6,\ldots,n,2]$.  Applying Theorem \ref{thm:trunc}, it is not hard to see that $A$ contains a positroid variety (indeed, a Schubert variety) $\Pi_{g'}$ such that $g' \in \II(Z)$ and $\dim(\Pi_{g'}) = 2k-1 = \dim(P) -1$.  In particular, $Z(A_{\geq 0})$ itself has dimension $\dim(P)-1$.  Thus $Z(A_{\geq 0})$ is a global geometric facet of $P$ which is itself a single Grassmann polytope. 

\subsubsection{Suppose $m = 4$.} For simplicity assume that $k = 2$ and $n = 8$.

(a) We first consider $I = \{1,2,3,4\}$.  Then $J = \{5,6,7,8\}$ and $A$ is again itself a positroid variety.  The facet is simply $Z(A_{\geq 0})$, a single Grassmann polytope.

(b) Now suppose $I = \{1,2,4,5\}$.  Then $J = \{3,6,7,8\}$.  In this case $A$ is not a positroid variety.  
The maximal positroid varieties $A_1,A_2,A_3,A_4$ contained in $A$ are given by the rank conditions
\begin{align*}
A_1 &\coloneqq \rank(\{3,4,5,6,7,8\}) \leq 1, \\
A_2 &\coloneqq \rank(\{1,2,3,6,7,8\}) \leq 1, \\
A_3 &\coloneqq \rank(\{6,7,8\}) \leq 1, \;\;\rank(\{3\}) = 0, \\
A_4 &\coloneqq \rank(\{6\}) = \rank(\{7\}) =  \rank(\{8\}) = 0.
\end{align*}
One deduces from Theorem \ref{thm:trunc} and Proposition \ref{prop:k2} that $\dim(Z(A_{3,\geq 0})) = \dim(P) - 1$, but $\dim(Z(A_{4,\geq 0})) = \dim(P) - 2$ and $\dim(Z(A_{1,\geq 0})) = \dim(Z(A_{2,\geq 0})) \leq \dim(P) - 2$.  On the other hand, this facet is not equal to $Z(A_{3,\geq 0})$ itself.
%

To see this, note that on $Z(A_3)$, we have
$$
\det(Y,Z_{1,2,4,8}) = \Delta_{56}(X) \Delta_{5,6,1,2,4,8}(Z) + \Delta_{57}(X) \Delta_{5,7,1,2,4,8}(Z) 
$$
since $\Delta_{35}(X) = \Delta_{36}(X) = \Delta_{37}(X) = \Delta_{67}(X) = 0$ when $X \in A_3$.  Both terms are positive when $Z$ is positive and $X \in A_{3,> 0}$.  Thus the function $\det(Y,Z_{1,2,4,8})$ takes a fixed sign on $Z(A_{3,\geq 0})$.  However, on $Z(A_2)$, we have
$$
\det(Y,Z_{1,2,4,8}) = \Delta_{35}(X) \Delta_{3,5,1,2,4,8}(Z) + \Delta_{56}(X) \Delta_{5,6,1,2,4,8}(Z)  +\Delta_{57}(X) \Delta_{5,7,1,2,4,8}(Z)
$$
and there are terms of both signs.  So the function $\det(Y,Z_{1,2,4,8})$ takes both positive and negative values on $Z(A_{2,\geq 0})$.  Thus $Z(A_{2,\geq 0})$ is not contained in $Z(A_{3,\geq 0})$.  So in this case the facet is a non-trivial union of Grassmann polytopes of different dimensions.  There is, however, a unique ``component'' which has dimension $\dim(P)-1$, in this case.

%

(c) Now suppose $I = \{1,2,5,6\}$ and $n = 8$.  Then $J = \{3,4,7,8\}$.  In this case $A$ is not a positroid variety.  
The maximal positroid varieties $A_1,A_2,A_3,A_4$ contained in $A$ are given by the rank conditions
\begin{align*}
A_1 &\coloneqq \rank(\{3,4,5,6,7,8\}) \leq 1, \\
A_2 &\coloneqq \rank(\{1,2,3,4,5,6\}) \leq 1, \\
A_3 &\coloneqq \rank(\{7,8\}) \leq 1, \;\;\rank(\{3\}) =\rank(\{4\})= 0, \\
A_4 &\coloneqq \rank(\{3,4\}) \leq 1, \;\;\rank(\{7\}) =\rank(\{8\})= 0.
\end{align*}
By Proposition \ref{prop:k2}, we have
\begin{align*}
[A_1] = s_{\tabll{&&&&}},\qquad
[A_2] = s_{\tabll{&&&&}}, \qquad
[A_3] = s_{\tabll{&&\\&}}, \qquad
[A_4] = s_{\tabll{&&\\&}}.
\end{align*}
Thus $\dim(Z(A_{3,\geq 0})) = \dim(Z(A_{4,\geq 0})) = \dim(P) - 1$ are codimension one.  We shall show that neither $Z(A_{3,\geq 0})$ or $Z(A_{4,\geq 0})$ contains the other.
Let us consider the function $\det(Y,Z_{2,6,7,8})$ on $Z(A_3)$ and $Z(A_4)$.  On $Z(A_{3,\geq 0})$, we have
$$
\det(Y,Z_{2,6,7,8}) = \Delta_{15}(X)\Delta_{1,5,2,6,7,8}(Z) \leq 0,
$$
since $\Delta_{13}(X)=\Delta_{14}(X)=\Delta_{34}(X)=\Delta_{35}(X)=\Delta_{45}(X) = 0$ on $A_3$.  On $Z(A_{4,\geq 0})$, we have
\begin{align*}
\det(Y,Z_{2,6,7,8}) = &\Delta_{13}(X) \Delta_{1,3,2,6,7,8}(Z) + \Delta_{14}(X) \Delta_{1,4,2,6,7,8}(Z) + \Delta_{15}(X) \Delta_{1,5,2,6,7,8}(Z) \\ + & \Delta_{35}(X) \Delta_{3,5,2,6,7,8}(Z) + \Delta_{45}(X) \Delta_{4,5,2,6,7,8}(Z).
\end{align*}
There are terms of both signs, and this function takes both positive and negative values on $Z(A_{4,\geq 0})$.  Similarly, there are functions that take a fixed sign on $Z(A_{4,\geq 0})$, but take both positive and negative values on $Z(A_{3,\geq 0})$.  It follows that neither $Z(A_{4,\geq 0})$ or $Z(A_{3,\geq 0})$ contains the other.

However, $Z(A_{\geq 0})$ is not contained in $Z(A_{3,\geq 0}) \cup Z(A_{4,\geq 0})$.  To see this, consider the function $p = \det(Y,Z_{1,5,6,7})$.  Then $p$ is equal to
\begin{align*}
&\Delta_{23}(X) \Delta_{2,3,1,5,6,7}(Z) + \Delta_{24}(X) \Delta_{2,4,1,5,6,7}(Z) + \Delta_{28}(X) \Delta_{2,8,1,5,6,7}(Z) & \mbox{on $Z(A_{1,\geq 0})$,} \\
&\Delta_{28}(X) \Delta_{2,8,1,5,6,7}(Z) & \mbox{on $Z(A_{3,\geq 0})$,} \\
&\Delta_{23}(X) \Delta_{2,3,1,5,6,7}(Z) + \Delta_{24}(X) \Delta_{2,4,1,5,6,7}(Z)  & \mbox{on $Z(A_{4,\geq 0})$.}
\end{align*}
Thus $p$ is positive (or zero) on $Z(A_{4,\geq 0})$, negative (or zero) on $Z(A_{3,\geq 0})$, and takes both signs on $Z(A_{1,\geq 0})$.  Similarly, $q = \det(Y,Z_{2,4,5,6})$ is equal to
\begin{align*}
&\Delta_{13}(X) \Delta_{1,3,2,4,5,6}(Z) + \Delta_{17}(X) \Delta_{1,7,2,4,5,6}(Z) + \Delta_{18}(X) \Delta_{1,8,2,4,5,6}(Z) & \mbox{on $Z(A_{1,\geq 0})$,} \\
&\Delta_{17}(X) \Delta_{1,7,2,4,5,6}(Z) + \Delta_{18}(X) \Delta_{1,8,2,4,5,6}(Z) & \mbox{on $Z(A_{3,\geq 0})$,} \\
&\Delta_{13}(X) \Delta_{1,3,2,4,5,6}(Z)   & \mbox{on $Z(A_{4,\geq 0})$.}
\end{align*}
Thus $q$ is positive (or zero) on $Z(A_{3,\geq 0})$, negative (or zero) on $Z(A_{4,\geq 0})$, and takes both signs on $Z(A_{1,\geq 0})$.  So $p$ and $q$ takes opposite signs on $Z(A_{3,\geq 0}) \cup Z(A_{4,\geq 0})$.  However, one can check from the above formulae that $p$ and $q$ can take the same sign at certain points of $Z(A_{1,\geq 0})$.  Thus $Z(A_{1,\geq 0})$ is not contained in $Z(A_{3,\geq 0}) \cup Z(A_{4,\geq 0})$.

So, in this case the facet is a union of two Grassmann polytopes of dimension $\dim(P)-1$, together with some lower dimensional Grassmann polytopes.

(d) Now suppose still that $m = 4$ and $k \geq 2$ is arbitrary.  Let $J = [n] \setminus I = J_1 \sqcup J_2$, where $J_1$ and $J_2$ are disjoint cyclic intervals that we assume to be nonempty.  For each pair $(k_1,k_2)$ of nonnegative integers satisfying $k_i \leq |J_i|$ and $k_1+k_2 = k-1$, we have a positroid variety 
$\Pi_{f_{k_1,k_2}}$
given by the rank conditions
$$
\rank(J_1) \leq k_1 \qquad \text{and} \qquad \rank(J_2) \leq k_2.
$$
One can show that each $\Pi_{f_{k_1,k_2}}$ is maximal amongst positroid varieties contained inside the Schubert variety $A$ of \eqref{eq:A}.  These are the positroid varieties corresponding to the ``factorization" of scattering amplitudes discussed in \cite[Section 11]{AT}.  As the previous examples illustrate, the Grassmann polytopes $Z(\Pi_{f_{k_1,k_2},\geq 0})$ are sometimes of lower dimension, and in general there are additional lower-dimensional components in the facets of the amplituhedron.

\subsubsection{Suppose $m \geq 2$ is even.}
Suppose that $I$ satisfies the evenness condition and $Z$ is positive.  Suppose $[n] \setminus I = J_1 \sqcup J_2 \sqcup \cdots \sqcup J_t$ is a decomposition into disjoint cyclic intervals.  For each $t$-tuple $(k_1,k_2,\ldots,k_t)$ of nonnegative integers satisfying $k_i \leq |J_i|$ and $k_1+k_2 + \cdots +k_t = k-1$, we have a positroid variety $\Pi_{f_{(k_1,k_2,\ldots,k_t)}}$ satisfying the rank conditions $\rank(J_i) \leq k_i$.  Clearly $\Pi_{f_{(k_1,k_2,\ldots,k_t)}} \subset A$, where $A$ is given by \eqref{eq:A}.

We conjecture that the geometric facet $P \cap \{\Delta(Y,Z_I) =0\}$ is the union of the Grassmann polytopes $Z(\Pi_{f_{(k_1,k_2,\ldots,k_t)},\geq 0})$ together with lower-dimensional Grassmann polytopes.  

\subsection{A degree-two facet}
The facets of Proposition \ref{prop:AT} are all linear facets.  However, Grassmann polytopes can have higher degree facets.  This is not surprising since in Section \ref{sec:sphericoid} we already gave many examples of amplituhedron varieties which were cut out by higher degree polynomials.

Take $k = 2, r = 5, n = 6$, and consider $P = Z(\Pi_{f,\geq 0})$ where $f = [3, 5, 4, 7, 6, 8] \in \Bound(2,6)$.  Consider $p = F_{(\tau,\emptyset)} \in R(3,6)_2$, where $\tau = \{(1,6),(2,3),(4,5)\}$.  We claim that $\psi(p)$ is a global geometric facet of $P$.  By Example \ref{ex:deg2} and Proposition \ref{prop:YPi}, we know that $\{\psi(p) = 0\}$ contains $Z(\Pi_{g,\geq 0})$ where $g = [2, 5, 4, 7, 6, 9] \in \Bound(2,6)$.  Furthermore $\dim(Z(\Pi_{g,\geq 0})) = \dim(P)-1$.  

We then check that $p$ has a fixed sign on the following set of matrices:
$$
\begin{bmatrix}
 1 & \alpha_1+\alpha_6 & \alpha_3\alpha_5 & \alpha_3 \alpha_4 & 0 & 0 \\
 0 & 1 & \alpha_5 & \alpha_4& \alpha_2 & \alpha_1 \\
 \beta_1 & -\beta_2 & \beta_3 & -\beta_4 & \beta_5 & -\beta_6 \\
\end{bmatrix},
$$
for positive $\alpha$ and $\beta$.  The top two rows parametrize $\Pi_{f,>0}$, and the bottom row runs through $\Gr(1,6)_{\geq 0, \tau}$.  By Lemma \ref{lem:transl}, it follows that $\psi(p)$ takes a fixed sign on $Z(\Pi_{f,\geq 0})$.

\section{Canonical form}\label{sec:form2}
In Section \ref{sec:form}, we defined a canonical rational differential form $\omega_f$ of top degree on $\Pi_f$.  In this section, we define the canonical form $\omega_{Z(\Pi_f)}$.

\subsection{Traces}
Suppose $f: X \to Y$ is a proper, surjective morphism of smooth complex algebraic varieties of the same dimension.  We want to define the \defn{trace}, or pushforward, $f_*\omega$ of a rational differential form $\omega$ on $X$.

We first describe the construction complex analytically.  Away from a hypersurface $D \subset Y$, the map $f$ is a finite unramified covering map.  For sufficiently small neighborhoods $U \subset Y \setminus D$, we have $f^{-1}(U) = V_1 \sqcup V_2 \sqcup \cdots \sqcup V_d$ is a disjoint union, and $f: V_i \to U$ is a holomorphic map with holomorphic inverse $g_i: U \to V_i$.  We then define
$$
f_* \omega|_U \coloneqq g_1^* \omega + g_2^* \omega + \cdots + g_d^* \omega.
$$
This defines $f_*\omega$ on $Y \setminus D$, and the form extends to a meromorphic form on $Y$.  

The algebraic version of the construction is as follows.  The map $f: X \to Y$ restricts to an \'etale morphism $f^{-1}(W) \to W$ for a Zariski-open subset $W \subset Y$.  We assume that $W = \Spec(A)$ and $f^{-1}(W) = \Spec(B)$ are affine, and the map $\phi: A \to B$ is \'etale.
Let $K = \Frac(A)$ and $L = \Frac(B)$.  The inclusion $K \subset L$ is a finite field extension, and has a well-defined trace map $\Tr: L \to K$.  Let $\Omega^p_{B/\C} \otimes_B L$ (resp. $\Omega^p_{A/\C} \otimes_A K$) be the module of rational K\"ahler differential $p$-forms on $B$ (resp. $A$).  By the definition of \'etale morphism, we have $\Omega^p_{B/\C} \simeq \Omega^p_{A/\C} \otimes_A B$, and we obtain a map
$$
\Tr: \Omega^p_{B/\C} \otimes_B L \simeq \Omega^p_{A/\C} \otimes_A L \to \Omega^p_{A/\C} \otimes_A K
$$
by using the trace map $\Tr:L \to K$.  In this way, a rational differential form on $X$ gives a rational differential form on $W$, and hence also on $Y$.  

We shall need the following result \cite[Proposition 2.5]{KR} saying that pushforward commutes with residues.

\begin{proposition}\label{prop:KR}
Let $f:X \to Y$ be a proper, surjective morphism of complex algebraic varieties of the same dimension $n$.  Let $\omega$ be a rational differential form with only poles of the first order along a smooth hypersurface $V$ in $X$.  Suppose $V_o = f(V)$ is a smooth hypersurface in $Y$.  Then $f_*\omega$ has first order poles on $V_o$ and
$$
\Res_{V_o}(f_*\omega) = \bar f_* \Res_{V}(\omega),
$$
where $\bar f: V \to V_o$ is the restriction of $f$.
\end{proposition}

\subsection{Canonical form}
We now define a canonical form $\omega_{Z(\Pi_f)}$ on $Z(\Pi_f)$.  Let $Z: \Pi_f \dashedrightarrow Z(\Pi_f)$ be the rational map defining $Z(\Pi_f)$.  If $\dim(Z(\Pi_f)) < \dim \Pi_f$, we declare $\omega_{Z(\Pi_f)} = 0$.  

Otherwise, $\dim(Z(\Pi_f)) = \dim \Pi_f$ and in particular $\Pi_f \setminus E_Z$ is Zariski-open and dense in $\Pi_f$.  We define $\omega_{Z(\Pi_f)}$ using the graph construction, as follows.  Let 
$$
\overline{\Pi}_f \coloneqq \overline{\{(X,Z_\Gr(X)) \mid X \in \Pi_f \setminus E_Z\}} \subseteq \Gr(k,n) \times \Gr(k,r)
$$
be the closure of the graph of $Z_\Gr: \Pi_f \setminus E_Z \to \Gr(k,r)$.  We have a natural birational morphism $\Pi_f  \dashedrightarrow \overline{\Pi}_f$ allowing us to pullback $\omega_{f}$ to a rational differential form $\bar \omega_{f}$ on $\overline{\Pi}_f$.  The morphism $\overline{Z}: \overline{\Pi}_f \to Z(\Pi_f)$ is induced by the projection $\Gr(k,n) \times \Gr(k,r)$ and is thus proper.  We can restrict $\overline{Z}$ to a proper surjective morphism $\overline{Z}|_U:U \to W$ where both $U \subset \overline{\Pi}_f$ and $W \subset Z(\Pi_f)$ are smooth.  We then define $\omega_{Z(\Pi_f)}$ to be the pushforward of $\bar \omega_f$ (extended to $Z(\Pi_f)$ under the map $\overline{Z}|_U: U \to W$.

\subsection{Poles and zeroes of the canonical form}
We would like to investigate the poles and zeroes of $\omega_{Z(\Pi_f)}$.  To simplify the discussion, we shall assume that $Z(\Pi_f)$ is a normal variety.  In fact, we make the following conjecture.

\begin{conj}\label{conj:normal}
Suppose $Z$ is generic and $\dim(Z(\Pi_f)) = \dim \Pi_f$.  Then the amplituhedron variety $Y_f = Z(\Pi_f)$ is projectively normal.
\end{conj}


By Theorem \ref{thm:normal}, Conjecture \ref{conj:normal} holds for positroid varieties themselves.  Also Conjecture \ref{conj:normal} obviously holds when $Y_f = \Gr(k,r)$, which is the most important case in the construction of the amplituhedron form (see Section \ref{sec:triangulations}).

We assume $\dim(Z(\Pi_f)) = \dim \Pi_f$ and $Z(\Pi_f)$ is a normal variety from now on.  Suppose that $g \gtrdot f$ so that $\Pi_g$ is a codimension one subvariety in $\Pi_f$.  Suppose also that $\dim(Z(\Pi_g)) = \dim \Pi_g$.  Combining Proposition \ref{prop:KR} and the fact that $\Res_{\Pi_g} \omega_f = \omega_g$ (Theorem \ref{thm:Resform}) we see that
\begin{equation}\label{eq:Res}
\Res_{Z(\Pi_g)}(\omega_{Z(\Pi_f)}) = \omega_{Z(\Pi_g)}.
\end{equation}
We know that $\omega_f$ has no zeroes and only poles along the $\Pi_g$.  Equation \eqref{eq:Res} says that $\omega_{Z(\Pi_f)}$ has simple poles along each of the codimension one subvarieties $Z(\Pi_g)$.  So we are led to the question: what are the other poles and zeroes of $\omega_{Z(\Pi_f)}$?  

It is easy to see that we should expect that in general $\omega_{Z(\Pi_f)}$ does have zeroes.  For example, the anticanonical divisor of $\Gr(k,r)$ is $r$ times the hyperplane class.  However, there are some $f \in \Bound(k,n)$ such that $Y_f = \Gr(k,r)$ where $\Pi_f$ has poles along more than $r$ divisors $\Pi_g$ (and the corresponding $Y_g$ do produce poles for $\omega_{Y_f}$).  So $\omega_{Y_f}$ must have zeroes to compensate for this.

Let $d$ be the degree of the map $\Pi_f \dashedrightarrow Z(\Pi_f)$.  We assume that $d = 1$.  Then the map $Z_\Gr: \Pi_f \setminus E_Z \to Z(\Pi_f)$ is a morphism that is birational.  Let $W \subset Z(\Pi_f)$ denote the image $Z_\Gr(\Pi_f \setminus E_Z)$.  In this case, we suspect (possibly requiring a genericity condition on $Z$) that the poles and zeroes of $\omega_{Z(\Pi_f)}$ are supported on $Z(\Pi_f) \setminus W$.  Furthermore, in simple cases, $\omega_{Z(\Pi_f)}$ only has poles along $Z(\Pi_g)$'s and the only zeroes are supported on $Z(\Pi_f) \setminus W$.  We make the following rather speculative conjecture.

\begin{conjecture}\label{conj:polesandzeroes}
Suppose $Z$ is generic, $d = 1$, and $\dim(Z(\Pi_f)) = \dim(\Pi_f)$.  Then $\omega_{Z(\Pi_f)}$ has (simple) poles only along the codimension one subvarieties $Z(\Pi_g)$, and all the zeroes of $\omega_{Z(\Pi_f)}$ lie in $Z(\Pi_f) \setminus Z_\Gr(\Pi_f \setminus E_Z)$.
\end{conjecture}

We expect that the zeroes along $Z(\Pi_f) \setminus W$ roughly correspond to zeroes acquired when pulling back $\omega_f$ under a blowup of $E_Z \cap \Pi_f \subset \Pi_f$.

When $d > 1$, we may have to further consider the behavior along the ramification locus.


\subsection{An example}
We explicitly compute the canonical form $\omega_{Y_f}$ for $f = [2547] \in \B(2,4)$.  This example continues the study of the variety $C$ in Section \ref{sec:exgeom}. 
The boundary $\partial \Pi_f$ consists of four codimension one positroid varieties.  Let us describe these positroid varieties in terms of lines in three-space.  Let $q_1,q_2,q_3,q_4$ be the images of $e_1,e_2,e_3,e_4$.  Let $\Pi_i$ be the locus of lines passing through $q_i$ and the line $L_{12}$ if $i \notin \{1,2\}$ or $L_{13}$ if $i \notin \{3,4\}$.  Then $\Pi_1,\Pi_2,\Pi_3,\Pi_4$ are our four boundary positroid varieties, and each one is isomorphic to $\P^1$.

Let $q'_i$ be the projection of $q_i$ onto $H_0$.  We assume that the $q_i$ are distinct from $p$, that $q'_i$ are distinct from each other, and that they are distinct from $x_0$ as well.  The image of $Z_\Gr(\Pi_i) = Y_i$ is the locus of lines in $H_0$ passing through $q'_i$.  In particular, each $\Pi_i$ is mapped isomorphically onto $Y_i$.  Since $\omega_{\Pi_f}$ had simple poles along $\Pi_i$, the meromorphic form $\omega_{Y_f}$ also has simple poles along $Y_i$.  (Note that a dense open subset of $Y_i$ belongs to the open subset of $\Gr(2,3)$ over which the map $Z_\Gr:(C \setminus E_Z) \to \Gr(2,3)$ is an isomorphism.)

The anticanonical divisor of $\Gr(2,3)$ is three times the hyperplane class (and each $Y_i$ represents such a class).  Thus $\omega_{Y_f}$ must have a zero somewhere.  We claim that it has a simple zero along the divisor $D \subset \Gr(2,3)$ corresponding to the locus of lines that pass through $x_0$.  This divisor $D$ is the complement $Z(\Pi_f) \setminus Z_\Gr(\Pi_f \setminus E_Z)$ appearing in Conjecture \ref{conj:polesandzeroes}.  Indeed, since $Z_\Gr$ is an isomorphism away from $D$, and $\omega_{\Pi_f}$ has no zeroes, the any extra poles and zeroes of $\omega_{Y_f}$ must be supported on $D$.  Considering the class of the canonical divisor of $\Gr(2,3)$ we see that it must have a simple zero along $D$.  This confirms Conjecture \ref{conj:polesandzeroes} in this case.

Let us check this directly using local coordinates.  Parametrize a dense subset of $C$ as the space 
of matrices of the form
$$
\begin{bmatrix} 1 & a & 0 & 0 \\ 0 & 0 & 1 & b 
\end{bmatrix}.
$$
Then we have $\omega_{\Pi_f} = \dlog a \wedge \dlog b$.
Then $Y=X\cdot Z$ is represented by the matrix
$$
\begin{bmatrix} 1 & 0 & u  \\ 0 & 1 & v 
\end{bmatrix}
$$
where 
\begin{align*} a = -\frac{\det(Y,Z_1)}{\det(Y,Z_2)} = \frac{-u z_{1,1} - v z_{1, 2} + z_{1,3}}{u z_{2, 1} + v z_{2, 2} - z_{2, 3}},\qquad b = -\frac{\det(Y,Z_3)}{\det(Y,Z_4)} = \frac{-u z_{3,1} - v z_{3, 2} + z_{3,3}}{u z_{4, 1} + v z_{4, 2} - z_{4, 3}}.
\end{align*}
We have $da \wedge db = J(u,v) du \wedge dv$ where 
$$
J(u,v) = \det \begin{bmatrix} \dfrac{\partial a}{\partial u} & \dfrac{\partial a}{\partial v}\\[2ex]
\dfrac{\partial b}{\partial u}&\dfrac{\partial b}{\partial v} \end{bmatrix}.
$$
Substituting, we get
$$
\omega_{Y_f} = \frac{f(u,v)}{\det(Y,Z_1)\det(Y,Z_2)\det(Y,Z_3)\det(Y,Z_4)} du \wedge dv
$$
where $f(u,v)$ is of the form $\alpha u + \beta v - \gamma$.  Note that $\det(Y,Z_i) = 0$ is exactly the equation that cuts out the locus $Y_i \subset \Gr(2,4)$ of lines that pass through $q'_i$.  A brute force calculation shows that the vector $x = (\alpha,\beta,\gamma)$ satisfies $\det(Z_1,Z_2,x) = \det(Z_3,Z_4,x) = 0$.  In other words, the condition $f(u,v) = 0$ cuts out the locus of lines $L \subset H_0$ that pass through the intersection of $\overline{q'_1 q'_2}$ and $\overline{q'_3 q'_4}$.  This intersection is the divisor $D$ of lines passing through $x_0$.  So $\omega_{Y_f}$ has a simple zero along $D$, as claimed.

Finally, we can check that there are no poles or zeroes at infinity (in the $u, v$ coordinates).  For this, we just note that the degree of $\det(Y,Z_1)\det(Y,Z_2)\det(Y,Z_3)\det(Y,Z_4)$ is three more than the degree of $f(u,v)$ as polynomials in either $u$, or $v$.

\section{Triangulations of Grassmann polytopes}\label{sec:triangulations}
In this section, we return to the situation that $Z$ is a real matrix.  We aim to make contact with the motivating work \cite{AT} by informally discussing triangulations of Grassmann polytopes.
%
%

\subsection{The canonical form of a Grassmann polytope}\label{sec:conjform}
In the following conjecture, ``triangulation" and ``facets" are in quotation marks because we have not given a complete definition of either notion.  Let $P = Z(\Pi_{f, \geq 0})$ be a Grassmann polytope.

\begin{conjecture}\label{conj:form}
There is a canonical rational differential top form $\omega_P$ on $Z(\Pi_h) = \overline{P}$, uniquely defined up to sign, with the following properties:
\begin{enumerate}
\item
$\omega_P$ has simple poles along the Zariski-closure $\overline{F}$ of each of its ``facets" $F$, and no other poles;
\item
for any ``facet" $F$ of $P$, with $F = \bigcup_i P_i$ a union of Grassmann polytopes, we have $\Res_{\overline{F}}(\omega_P) = \sum_i \pm \omega_{P_i}$;
\item
if $\T$ is a "triangulation" of $P$, then $\omega_P = \sum_{f \in \T} \pm \omega_{Z(\Pi_f)}$.
\end{enumerate}
\end{conjecture}
This conjecture is the natural extension to Grassmann polytopes of the conjecture of \cite{AT} for the amplituhedron.  In the case that $P = Z(\Pi_{h,\geq 0})$ and $h \in \G_Z$ is a base, part (1) is closely related to Conjecture \ref{conj:polesandzeroes}.

\begin{remark}
Conjecture \ref{conj:form} reflects two philosophies common in the theory of scattering amplitudes (see Section \ref{sec:amplitudes}). (1) The amplitude $\AA^{{\rm tree}}_{k,n}$ is uniquely determined by its poles, together with the factorization properties at these poles; this is analogous to a polytope being determined by its facets.  (2) The amplitude can be written as a sum of certain simpler functions in many different ways; this is analogous to a polytope having many triangulations.
\end{remark}

Let us give an argument that $\omega_P$ should be unique once all $\omega_{P_i}$ and all signs have been fixed.  Suppose $\omega_P$ and $\omega'_P$ only have simple poles along its facets, and $\Res_{\overline{F}}(\omega_P) = \Res_{\overline{F}}(\omega'_P)$ for each facet $F$.  Then the difference $\omega_P - \omega'_P$ has no poles anywhere, because the only possible poles are simple poles along facets $F$, and the residue along each facet $F$ is 0.  Now, assuming that it is a normal variety (see Conjecture \ref{conj:normal}), $Z(\Pi_h)$ is a unirational projective variety, and it does not have holomorphic canonical sections, so we conclude that $\omega_P - \omega'_P = 0$.

%

\subsection{The polytope form}
A rational differential form $\omega_P$ satisfying Conjecture \ref{conj:form} does indeed exist for a polytope $P$, with the usual notions of triangulation and facets.   I like to think of this rational differential form as the Laplace transform of the characteristic function of the dual polytope.  This form has been studied by physicists in \cite{AT,AHT}, and in a different language by mathematicians for example in \cite{BT,Fill}.  

It is simpler to first construct a rational differential form on $\R^r$.  Let $C \subset \R^r$ be the pointed polyhedral cone spanned by the rows of $Z$, and let $x_1,\ldots,x_r$ be coordinates on $\R^r$.  Let $C^* \subset (\R^r)^*$ be the dual (or polar) cone and let $y_1,\ldots,y_r$ be coordinates on $(\R^r)^*$.  Define
$$
\omega_C(x_1,\ldots,x_r) \coloneqq \left(\int_{C^*} e^{-\langle x, y \rangle} dy_1 \wedge dy_2 \wedge \cdots dy_r\right) dx_1 \wedge dx_2 \cdots \wedge dx_r,
$$
for $(x_1,x_2,\ldots,x_r)$ in the interior of $C$, and extend to a rational differential form on $\R^r$.  That this is naturally a differential form rather than a rational function reflects the fact that the Laplace transform depends on a choice of measure on $(\R^r)^*$.  Indeed, there is a general theory of Laplace transforms of piece-wise linear functions giving rational forms; see Brion and Vergne \cite{BV} for a discussion of these ideas.

Let $P \subset \P^{r-1}$ be the projective polytope that is the image of $C \subset \R^r$.  

The differential form $\omega_C$ can be written as $q(x_1,x_2,\ldots,x_r) \frac{dx_1}{x_1} \wedge \cdots \wedge \frac{dx_r}{x_r}$ for a rational function $q$ homogeneous of degree 0.  It gives a differential form $\omega_P$ on $\P^{r-1}$: on the affine chart where $x_1 = 1$, we have $\omega_P = q(1,x_2,\ldots,x_r)\frac{dx_2}{x_2} \wedge \cdots \wedge \frac{dx_r}{x_r}$, and this does not depend on the choice of chart.

\begin{theorem}
With the usual notion of facets and triangulations, Conjecture \ref{conj:form} holds for polytopes with this canonical form $\omega_P$.
\end{theorem}

The canonical polytope form $\omega_P$ will be discussed in some detail in upcoming joint work with Arkani-Hamed and Bai \cite{ABL}, where many further properties of the form will be given.

We remark that Part (1) of Conjecture \ref{conj:form} for polytopes follows from the results of \cite{BV}, and Part (3) of Conjecture \ref{conj:form} is essentially \cite[Theorem 1]{Fill}.  Also, the rational form $\omega_C(x_1,\ldots,x_r)$ is called a $\mathcal{X}$-function in \cite{BT}, where it is defined as a rational function instead.  

%
%

\subsection{Geometric Triangulations}
We now discuss a number of possible notions of triangulations of Grassmann polytopes, starting with the analogue of the most familiar notion of triangulation for point sets.

Let $P= Z(\Pi_{h, \geq 0})$ be a Grassmann polytope.  When we discuss triangulations of $P$, the matrix $Z$ is itself part of the data.  For example, when $k=1$, some of the vectors $z_i$ may be in the interior of $P$, but can be used in a triangulation.  Let $\G_Z$ denote the Grassmann matroid of $Z$.

A \defn{maximal cell} (of $P$) is an independent set $f \in \T$ such that $f \leq h$ and $\dim(Z(\Pi_{f,\geq 0}))=\dim(P)$.  

\begin{conjecture}\label{conj:cell}
Let $P$ be a nonempty Grassmann polytope.  Then $P$ has a maximal cell. 
\end{conjecture}
I expect Conjecture \ref{conj:cell} is easy to see when $Z$ is generic.


Let $f \in \II(\G_Z)$ be an independent set.  A \defn{face} of $f$ is an independent set $g \in \II(\G_Z)$ such that $g \leq f$ in $\hBound(k,n)$.  
We say that two independent sets $f,f'$ \defn{intersect properly} if the intersection $Z(\Pi_{f,\geq 0}) \cap Z(\Pi_{f',\geq 0})$ is a union $\bigcup_{g \in \II(\G_Z)} Z(\Pi_{g,\geq 0})$ where each $g$ is a face of both $f$ and $f'$.  When $k = 1$, a convex union of faces will always be a face and this is closely related to the fact that the boolean poset is a lattice.

A \defn{geometric simplicial triangulation} of $P$ is a collection $\T \subset \hBound(k,n)$ of maximal cells of $\G_Z$ satisfying the following conditions:
\begin{enumerate}
\item
We have $P = \bigcup_{f \in \T} Z(\Pi_{f,\geq 0})$. 
\item
For distinct $f,f' \in \T$, the two maximal cells $f,f'$ intersect properly.
\end{enumerate}
This is the usual definition of a triangulation of a point configuration.  Unfortunately, the methods that we have developed in Section \ref{sec:trunc} and Section \ref{sec:ampliideal} do not give us a consistent way to check either condition.  

\subsection{Combinatorial triangulations}
It is desirable to have a more combinatorial, and less geometric/semi-algebraic way to check if $\T$ is a triangulation.  For triangulations of point configurations, there are a number of such possibilities, many of which are formulated using the language of oriented matroids.  We use some of that language below, but will not develop a Grassmann analogue of oriented matroids \cite{OM}.  We refer the reader to the book \cite{dLRS} for background on triangulations.

An independent set $f \in \II(\M_Z)$ with $\dim(Z(\Pi_{f,\geq 0})) = \dim(P)-1$ is called a \defn{facet cell} of $P$ if $Z(\Pi_{f,> 0})$ is not contained in the interior of $P$.

\subsubsection{Pseudomanifold property}

We say that a collection $\T$ of maximal cells satisfies the \defn{pseudo-manifold property} if for every facet cell $g$ of a maximal cell $f \in \T$ that is not a facet cell of $P$, there is another maximal cell $f' \in \T$ such that $g$ is a facet cell of $f'$.  

\subsubsection{Signed circuits}

We would now like to define a signed circuit of $Z$.  Let us recall the usual notion.  For a circuit $C = \{c_1,c_2,\ldots,c_t\} \subset [n]$ of $Z$, there is a unique up to scalar equality
$$
a_1 z_{c_1} + a_2 z_{c_2} + \cdots a_t z_{c_t} = 0
$$
where $a_1,a_2,\ldots,a_t$ are all nonzero real numbers.  We then obtain a signed circuit $(C_+,C_-)$, where $C_+ \subset C$ (resp. $C_- \subset C$) is the set of $c_i$ where $a_i > 0$ (resp. $a_i < 0$).  So $C = C_+ \sqcup C_-$.  Consider the subsets $F \subset C$ with size $|F| = |C| - 1$.  There are two kinds of these subsets: the ones that contain the whole of $C_+$ (and all but one element of $C_-$), and the ones that contain the whole of $C_-$ (and all but one element of $C_+$). 

Let us try to find such a decomposition for Grassmann polytopes.
Recall that $\Pi_f$ has a canonical form $\pm \omega_f$ defined uniquely up to sign.  An \defn{orientation} of $\Pi_f$ is a choice of one of the two signs for this canonical form.  An orientation $\omega_f$ of $\Pi_f$ determines an orientation for each $\Pi_g$ where $g \lessdot f$: we choose the orientation $\Res_{\Pi_g} {\omega_f}$ of $\Pi_g$.  Note that taking residues does not involve any choices.

Now suppose $f$ is a circuit of $\G_Z$.  Thus $\dim(Z(\Pi_f)) = \dim(\Pi_f) - 1$, and $f$ is minimal with respect to this property.  Fix an orientation $\omega_f$ of $\Pi_f$, and thus obtain orientations $\Res_{\Pi_g} {\omega_f}$ of each $\Pi_g$.  We also obtain a pushforward orientation $\omega_{Z(\Pi_g)} = (Z_\Gr)_*(\Res_{\Pi_g} {\omega_f})$.  By our assumptions $P$ lies inside a submanifold of $Z(\Pi_h)_\R$ that is orientable (Remark \ref{rem:cones}).  Let us choose an orientation top-form $\omega(P)$.  So the set $\{g \lessdot f \} \subset \hBound(k,n)$ can be split into two subsets
\begin{align*}
g \in \begin{cases}
D_+ & \mbox{if $\omega_{Z(\Pi_g)}$ is a positive multiple of $\omega(P)$ at a point of $Z(\Pi_{g,> 0})$,} \\
D_- & \mbox{if $\omega_{Z(\Pi_g)}$ is a negative multiple of $\omega(P)$ at a point of $Z(\Pi_{g,> 0})$.} 
\end{cases}
\end{align*}
(It is not clear to me if this is well-defined in general, for example, when $\Pi_g \dashedrightarrow Z(\Pi_g)$ has degree greater than one.  But let us proceed as if it were defined.)
We can then define the \defn{signed circuit} of $f$ as follows: $C_+$ is the collection of maximal elements of $\{g' \mid g' \leq g \text{ and } g \in D_+\} \subset \Bound(k,n)$, and similarly for $C_-$.  When $k = 1$, $\Bound(1,n)$ is a lattice, so $C_{\pm}$ are just single elements of $\Bound(1,n)$.

\begin{remark}
I expect that the notion of a bistellar flip of a triangulation $\T$ is to replace a set of maximal cells $D_+$ by a set of maximal cells $D_-$, or vice versa.  This is closely related to the homological identities of \cite{ABCGPT}.
\end{remark}

\subsubsection{Definition of combinatorial triangulation}
A \defn{combinatorial triangulation} of $P$ is a collection $\T$ of maximal cells with the properties:
\begin{enumerate}
\item
$\T$ satisfies the pseudomanifold property, and
\item
there do not exist $f, f' \in \T$ so that $C_+ \leq f$ and $C_- \leq f'$ for some signed circuit $(C_+,C_-)$.
\end{enumerate}
For $k = 1$, a combinatorial triangulation in this sense is equivalent to a geometric simplicial triangulation \cite[Chapter 4]{dLRS}.  Here, we say $C_+ \leq f$ if all elements of $C_+$ are less than $f$.  (Though, I must admit I do not have convincing evidence that this is the correct definition.)

\subsection{Homological triangulations}
Conjecture \ref{conj:form}(3) suggests that it may also be interesting to study a purely homological notion of triangulation.

A \defn{homological triangulation} of $P$ is a collection $\T \subset \II(\G_Z)$ of maximal cells, together with an orientation $\omega_f$ for each cell $f \in \T$, such that
\begin{enumerate}
\item
for each facet cell $g$ of $P$, there is a unique maximal cell $f \in \T$ so that $g$ is a facet of $f$, and
\item
for each facet $g$ of a maximal cell $f \in \T$ that is not a facet cell of $P$, there is a unique other maximal cell $f' \in \T$ with $g$ as a facet, and the orientations on $Z(\Pi_{g})$ induced by $\omega_f$ and $\omega_{f'}$ are negatives of each other.
\end{enumerate}
This is a simpler notion of triangulation, which is certainly not equivalent to the usual notion.  It is, however, much easier to check.

\subsection{Momentum-twistor BCFW recursion}
Let us suppose now that $Z$ is positive, $r = k+4$, and $P = Z(\Gr(k,n)_{\geq 0})$ is the amplituhedron.  There is a recursive formula (in fact, many) for the amplituhedron form $\omega_P$ as a sum of forms $\omega_{Z(\Pi_f)}$, and it is the conjecture of Arkani-Hamed and Trnka that these recursive formulae give rise to ``triangulations".  We describe the version of this recursion due to Bai and He \cite{BH}.

Suppose $n \geq k +4$.  We shall recursively define collections $C(k,n) \subset \Bound(k,n)$ of bounded affine permutations $f$ satisfying $\dim(\Pi_f) = 4k$, as follows.  First, if $n = r = k +4$, then we have $C(k,k+4) = \{f_\id\}$ consists of the single bounded affine permutation indexing the top cell of $\Gr(k,k+4)_{\geq 0}$.  Suppose $C(k',n')$ has been computed for all $(k',n')$ where either $k'<k$ and $n' \leq n$ or $k'\leq k$ and $n'<n$.  Then $C(k,n)$ is the union of all bounded affine permutations $f \in \Bound(k,n)$ where either (1) $f = f_G$ is the bounded affine permutation of the planar bipartite graph $G$ obtained by adding a black lollipop at $n$ to a reduced planar bipartite graph $G(f')$ representing some $f' \in C(k,n-1)$, 
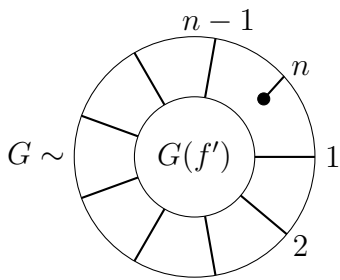
\begin{equation*}
G \sim
\begin{tikzpicture}[baseline=-0.5ex,scale=0.8]
\coordinate (a1) at (0:2);
\coordinate (a2) at (-40:2);
\coordinate (a3) at (-80:2);
\coordinate (a4) at (-120:2);
\coordinate (a5) at (-160:2);
\coordinate (a6) at (160:2);
\coordinate (a7) at (120:2);
\coordinate (a8) at (80:2);
\coordinate (a9) at (42:2);

\node at (0:2.3) {$1$};
\node at (-40:2.3) {$2$};
\node at (80:2.3) {$n-1$};
\node at (40:2.3) {$n$};

\draw (0,0) circle (2);
\draw (0,0) circle (1);

\node at (0:0) {$G(f')$};

\draw[thick] (a1) to ($(0:1)$);
\draw[thick] (a2) to ($(-40:1)$);
\draw[thick] (a3) to ($(-80:1)$);
\draw[thick] (a4) to ($(-120:1)$);
\draw[thick] (a5) to ($(-160:1)$);
\draw[thick] (a6) to ($(-200:1)$);
\draw[thick] (a7) to ($(-240:1)$);
\draw[thick] (a8) to ($(-280:1)$);
\draw[thick] (a9) to (-320:1.5);

\blackdot {(-320:1.5)};
\end{tikzpicture}
\end{equation*}
or (2) $f = f_G$ is the bounded affine permutation of
a planar bipartite graph $G$ of the form
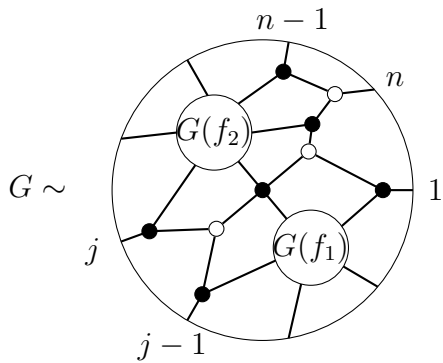
\begin{equation*}
G \sim
\begin{tikzpicture}[baseline=-0.5ex]
\coordinate (a1) at (0:2);
\coordinate (a2) at (-40:2);
\coordinate (a3) at (-80:2);
\coordinate (a4) at (-120:2);
\coordinate (a5) at (-160:2);
\coordinate (a6) at (160:2);
\coordinate (a7) at (120:2);
\coordinate (a8) at (80:2);
\coordinate (a9) at (42:2);

\node at (0:2.3) {$1$};
\node at (-120:2.4) {$j-1$};
\node at (-160:2.4) {$j$};
\node at (80:2.3) {$n-1$};
\node at (40:2.3) {$n$};

\draw (0,0) circle (2);
\draw (130:1) circle (0.5);
\node at (130:1) {$G(f_2)$};
\draw (-50:1) circle (0.5);
\node at (-50:1) {$G(f_1)$};

\coordinate (b1) at (0:1.6);
\coordinate (b2) at (53:1.1);
\coordinate (b3) at (80:1.6);
\coordinate (b4) at (0:0);

\coordinate (w1) at (40:0.8);
\coordinate (w2) at (53:1.6);
\coordinate (w3) at (-140:0.8);
\coordinate (b5) at (-120:1.6);
\coordinate (b6) at (-160:1.6);

\draw [thick] (b5) to (a4);
\draw [thick] (b5) to (w3);
\draw [thick] (b6) to (a5);
\draw [thick] (b6) to (w3);
\draw [thick] (b4) to (w3);
\draw [thick] (b4) to (w1);
\draw [thick] (b1) to (w1);
\draw [thick] (b1) to (a1);
\draw [thick] (b2) to (w1);
\draw [thick] (b2) to (w2);
\draw [thick] (b3) to (a8);
\draw [thick] (b3) to (w2);
\draw [thick] (w2) to (a9);

\draw[thick] (b1) to ($(-50:1)+(35:0.5)$);
\draw[thick] (a2) to ($(-50:1)+(-30:0.5)$);
\draw[thick] (a3) to ($(-50:1)+(-105:0.5)$);
\draw[thick] (b5) to ($(-50:1)+(-160:0.5)$);
\draw[thick] (b4) to ($(-50:1)+(130:0.5)$);

\draw[thick] (a6) to ($(130:1)+(180:0.5)$);
\draw[thick] (a7) to ($(130:1)+(100:0.5)$);
\draw[thick] (b3) to ($(130:1)+(50:0.5)$);
\draw[thick] (b2) to ($(130:1)+(0:0.5)$);
\draw[thick] (b4) to ($(130:1)+(-50:0.5)$);
\draw[thick] (b6) to ($(130:1)+(-120:0.5)$);

\blackdot{(b1)};
\blackdot{(b2)};
\blackdot{(b3)};
\blackdot{(b4)};
\blackdot{(b5)};
\blackdot{(b6)};
\whitedot{(w1)}
\whitedot{(w2)}
\whitedot{(w3)}
\end{tikzpicture}
\end{equation*}
where $G(f_1)$ (resp. $G(f_2)$) is a reduced planar bipartite graph representing $f_1 \in C(k_1,j)$ (resp. $f_2 \in C(k_2,n-j+2)$) for any $j \in [3,n-2]$ and nonnegative integers $k_1,k_2$ satisfying $k_1+k_2 = k - 1$.  Here, we may have to insert two-valent white vertices on the boundary edges of $G(f_1)$ and $G(f_2)$ to ensure that the resulting graph is bipartite.  Also, if the graph $G$ is not reduced, then (by a face-count argument) it will represent a positroid cell with dimension less than $4k$, and should be ignored.  This defines a collection $C(k,n) \subset \Bound(k,n)$ for all $n \geq k+4$.

A basic conjecture of \cite{AT,BH} is that the set $C(k,n)$ gives a ``triangulation" of the amplituhedron.  


\begin{example}
Suppose $k = 0$.  Then $C(0,n) = \Bound(0,n)$ consists of just one element.  In the recursion, construction (2) is never used.
\end{example}

\begin{example}
Suppose $k = 1$.  Then in the recursion we always have $k_1 = k_2 = 0$, so for the bounded affine permutations arising from construction (2), the only choice is the index $j \in [3,n-2]$.  The corresponding planar bipartite graphs $G$ consist of a single interior white vertex connected to the boundary vertices $\{1,j-1,j,n-1,n\}$ and all other boundary vertices are connected to black lollipops.

This recursion gives rise to a triangulation of the four-dimensional cyclic polytope (inside $\Gr(1,5) = \P^4$) with $n$ vertices.  Namely, the triangulation given by $C(1,n)$ uses the simplices $\{1,i-1,i,j-1,j\}$ for all $i,j$ satisfying $2 < i < j-1 <n$.  One can verify that this is a triangulation, for example, via the work of Rambau \cite{Ram}. 
\end{example}

\section{Scattering amplitudes}\label{sec:amplitudes}
In this section, we give an informal discussion (intended to be complementary to the discussion in Section \ref{sec:intro}) of the theory of scattering amplitudes intended for someone like myself who has no background in physics.  For a general introduction to scattering amplitudes, the reader is referred to the books \cite{EH,HP}.  For the relation between amplitudes and the totally nonnegative Grassmannian, see the very extensive work \cite{ABCGPT}.  However, we must warn the reader that most of this work is set in ``momentum space", while the amplituhedron only appears to exist in ``momentum-twistor space".  For a discussion of the relation of the two settings see \cite{EHKLORS,ACCK,MaSk}.

There are many other connections of scattering amplitudes with mathematics (see for example \cite{DHP,GGSVV}), but we will only discuss the story directly related to the tree amplituhedron.  For recent work on loop amplituhedra, see \cite{AT2, BH, FGMT}.

Scattering amplitudes in particle physics are used to compute the probability that certain particle interactions occur.  One starts by picking a quantum field theory, which is usually fixed by writing down a Lagrangian.  This choice amounts to choosing the types of particles that will be studied and the basic rules for their interaction.  Scattering amplitudes correspond to particle creation/annihilation experiments that occur in an isolated part of the universe.  To define an amplitude, one first decides on the list of say $n$ particles that will be involved (for example, one photon and one electron incoming, and one photon and one electron outgoing).  The \defn{scattering amplitude} for this scattering process is then a function $A(p_1,p_2,\ldots,p_n)$ of the momenta of the $n$ particles (other data like polarization vectors are often also involved).  

There is a formal expression for the function $A_n = A(p_1,p_2,\ldots,p_n)$ as an infinite sum of integrals of rational functions.  The sum is over an infinite list of increasingly complicated Feynman diagrams, which are graphs decorated with some additional data.  The integrals are over additional variables called internal propagators, and the integrand is a function of the momenta $p_i$ and the propagators.  There is a formal (but infinite) procedure for writing down such an expression for the amplitude once one is given the Lagrangian that defines the quantum field theory.  It is a notoriously difficult problem to make sense (for example, ``renormalization") of these formal expressions to compute the finite probabilities in particle physics experiments, or other areas of physics where quantum field theories are used.

The particular quantum field theory relevant to the story of the totally nonnegative Grassmannian and the amplituhedron is called ``four-dimensional super Yang-Mills".  In this theory, the particles to be considered are light-like particles: the momenta $p_i$ are vectors in four-dimensional space-time that have zero length with respect to the Lorentzian metric.  There is a choice of a gauge group for this theory, which is chosen to be the group $\SU(N)$; this symmetry group corresponds to internal symmetries of the particles.  

We now make two simplifications.  There is an expansion
$$
A_n = A_n^{\text{tree}} + A_n^{\text{1-loop}} + A_n^{\text{2-loop}} + \cdots
$$
where $A_n^{\text{tree}}$ consists of the terms indexed by finitely many Feynman diagrams that are trees, and these diagrams contribute terms that have no integrals.  We only consider $A_n^{\text{tree}}$.  Next, there is a trick called color-ordering that gives a formula of the form
$$
A_n^{\text{tree}} = (\text{group theory factor}) \AA_n^{\text{tree}},
$$
so that the answer $\AA_n^{\text{tree}}(p_1,p_2,\ldots,p_n)$ depends only on the kinematical data (the momentum vectors) and not on the choice of gauge group.  The group theory factor is, roughly speaking, a sum over traces $\Tr(\xi_{a_1} \xi_{a_2} \cdots \xi_{a_n})$ of elements $\xi \in {\mathfrak {su}}(N)$.  Because of the cyclicity of the trace of a product of matrices, the $n$ momenta in $\AA_n^{\text{tree}}(p_1,p_2,\ldots,p_n)$ acquire a cyclic-ordering; and the answer is cyclically symmetric.  (Strictly speaking, the function $\AA_n^{\text{tree}}(p_1,p_2,\ldots,p_n)$ that we shall discuss is the amplitude in the \emph{planar sector}.)

\medskip

For me, this cyclicity is the simplest explanation for the mathematical structures that arise.  Planar bipartite graphs have rotational symmetry; Grassmannians have an action of a cyclic group; rectangular shaped Young-tableaux have a promotion operator; affine permutations have rotational symmetry.  I expect that the cyclic symmetry for the type $A$ affine Lie algebra will be playing a role. 

\medskip

It turns out that the formula for $\AA_n^{\text{tree}}(p_1,p_2,\ldots,p_n)$ is simplest not as a rational function in the space-time momentum variables $p_i$, but in terms of something called spinor-helicity formalism.  In these variables, the answer exhibits an additional symmetry, called \emph{dual superconformal symmetry}, that was previously hidden; furthermore, superconformal symmetry and dual super conformal symmetry glue together to give a \emph{Yangian} algebra of infinitesimal symmetries.

When written in ``super-momentum-twistor" coordinates, the answer $\AA_n^{\text{tree}}$ is a function of four bosonic variables $z_1,z_2,z_3,z_4$ (really, a $n \times 4$ matrix) and $k$ fermionic variables $\eta_1,\eta_2,\ldots,\eta_k$.  Here, the fermionic variables are present because of the choice of the maximally supersymmetric version of Yang-Mills theory.  The supersymmetry gives rise to additional types of particles in the quantum field theory; the fermionic variables act as variables in a generating function over possible particle types.  The additional parameter $k$ corresponds (with a shift!) to the total ``helicity" of the particles involved: many different collections of particles have the same total helicity, and all amplitudes for such experiments are encoded in a single $\AA_{k,n}^{\text{tree}}(z_1,z_2,z_3,z_4,\eta_1,\eta_2,\ldots,\eta_k)$.

The differential form $\omega_P$ of Section \ref{sec:conjform} for the case that $P$ is the amplituhedron is a rational form $\omega_{SYM}(Y,Z)$ where $Y \in \Gr(k,k+4)$ and $Z$ is a $n \times (k+4)$ matrix.  The rational form is invariant under the simultaneous action of $\GL(k+4)$ on $\Gr(k,k+4)$ and $\Mat(n,k+4)$.  The matrix $Z$ is some basis of the span of the $k+4$ vectors 
$\{z_1,z_2,z_3,z_4,\eta_1,\eta_2,\ldots,\eta_k\}$, and the point $Y \in \Gr(k,k+4)$ keeps track of the $k$-dimensional subspace spanned by the fermionic vectors.  In the amplituhedron form, the entire $n \times (k+4)$ matrix is considered bosonic.  Furthermore, while the original momentum variables $p_i$ are real vectors, the form $\omega_{SYM}(Y,Z)$ should be considered a complex analytic object.  To recover $\AA_{k,n}^{\text{tree}}$ from $\omega_{SYM}$ one performs an integral on $\omega_{SYM}$ involving delta functions and fermionic variables.  This amounts to a formal, algebraic procedure that produces an expression in the variables $\{z_1,z_2,z_3,z_4,\eta_1,\eta_2,\ldots,\eta_k\}$.

Entirely new considerations come into play when discussing the higher-loop contributions to the amplitude -- non-trivial integrals must be performed.  Due to  my own unfamiliarity with this part of the subject, I will not attempt to discuss it.


\begin{thebibliography}{xxx}


\bibitem[ARW]{ARW} F. Ardila, F. Rincon, and L.~Williams, Positively oriented matroids are realizable, preprint, 2013; {\tt arXiv:1310.4159}.


\bibitem[ABL]{ABL} N. Arkani-Hamed, Y.~Bai, and T.~Lam, in preparation.

\bibitem[ABCGPT]{ABCGPT} N. Arkani-Hamed, J. L. Bourjaily, F. Cachazo, A. B. Goncharov, A. Postnikov, and J. Trnka, Scattering Amplitudes and the Positive Grassmannian, preprint, 2012; {\tt arXiv:1212.5605}.

\bibitem[ABCHT]{ABCHT} N. Arkani-Hamed, J. L. Bourjaily, F. Cachazo, A. Hodges, and J. Trnka, A note on polytopes for scattering amplitudes. JHEP 1204, 081 (2012).

\bibitem[ACCK]{ACCK} N. Arkani-Hamed, F. Cachazo, C. Cheung, and J. Kaplan, A duality for the S matrix. JHEP 03 (2010), 020.

\bibitem[AHT]{AHT} N. Arkani-Hamed, A.~Hodges, and J.~Trnka, 
Positive Amplitudes In The Amplituhedron, preprint, 2014; {\tt arXiv:1412.8478}.

\bibitem[ArTr13a]{AT} N. Arkani-Hamed and J. Trnka,
The Amplituhedron, preprint, 2013; {\tt arXiv:1312.2007}.

\bibitem[ArTr13b]{AT2} N. Arkani-Hamed and J. Trnka,
Into the Amplituhedron, preprint, 2013; {\tt arXiv:1312.7878}.

\bibitem[BaHe]{BH} Y.~Bai and S.~He, The Amplituhedron from Momentum Twistor Diagrams, preprint, 2014; {\tt arXiv: 1408.2459}.

\bibitem[BaTs]{BT} V. Batyrev and Y. Tschinkel,  Manin's conjecture for toric varieties. J. Algebraic Geom. 7 (1998), no. 1, 15--53.

\bibitem[BiCo]{BC} S.~Billey and I.~Coskun, Singularities of generalized Richardson varieties. Comm. Algebra 40 (2012), no. 4, 1466--1495. 

\bibitem[Bjo]{Bjo} A. Bj\"orner, Posets, regular CW complexes and Bruhat order. Europ. J. Combin. 5 (1984), 7--16.

\bibitem[BjBr]{BB} A.~Bj\"orner and F.~Brenti, Combinatorics of Coxeter groups. Graduate Texts in Mathematics, 231. Springer, New York, 2005. xiv+363 pp.

\bibitem[BLSWZ]{OM} A.~Bj\"orner, M.~Las Vergnas, B.~Sturmfels, N.~White, and G.M.~Ziegler, Oriented matroids. Second edition. Encyclopedia of Mathematics and its Applications, 46. Cambridge University Press, Cambridge, 1999. xii+548 pp.

\bibitem[Blu]{Blu} S.~Blum, Base-sortable Matroids and Koszulness of Semigroup Rings. Europ. J. Combin. 22 (2001), 937--951.

\bibitem[BrVe]{BV} M. Brion and M. Vergne, Arrangement of hyperplanes. I. Rational functions and Jeffrey-Kirwan residue. Ann. Sci. \'Ecole Norm. Sup. (4) 32 (1999), no. 5, 715--741. 

\bibitem[BCMP]{BCMP} A.S. Buch, P.-E. Chaput, L.C. Mihalcea, and N. Perrin, Projected Gromov-Witten varieties in cominuscule spaces, preprint, 2013; {\tt arXiv:1312.2468}.

\bibitem[DHP]{DHP} J. M. Drummond, J. M. Henn, and J. Plefka, Yangian Symmetry of Scattering
Amplitudes in $N = 4$ Super Yang-Mills Theory. JHEP 05 (2009), 046.

\bibitem[ElHu]{EH} H.~Elvang and Y.-t.~Huang, Scattering Amplitudes in Gauge Theory and Gravity. Cambridge University Press, 2015.  

\bibitem[EHKLORS]{EHKLORS} H.~Elvang, Y.-t.~Huang, C.~Keeler, T.~Lam,
T.M. Olson, S.B. Roland, and D.E.~Speyer,
Grassmannians for scattering amplitudes in $4d$ SYM and $3d$ ABJM. JHEP 12 (2014), 181.

\bibitem[FaPo]{FP} M.~Farber and A.~Postnikov, Arrangements of equal minors in the positive Grassmannian, preprint, 2015; {\tt arXiv:1502.01434}.

\bibitem[Fil]{Fill} P. Filliman, The volume of duals and sections of polytopes.
Mathematika 39 (1992), no. 1, 67--80. 

\bibitem[FZ]{FZ:tests} S. Fomin and A. Zelevinsky, Total positivity: tests and parametrizations. Math. Intelligencer 22 (2000), no. 1, 23--33.

\bibitem[FGMT]{FGMT} S.~Franco, D.~Galloni, A.~Mariotti, and J.~Trnka, Anatomy of the Amplituhedron, preprint, 2014; {\tt arXiv:1408.3410}.

\bibitem[Ful]{Ful} W.~Fulton, Young tableaux.
With applications to representation theory and geometry. London Mathematical Society Student Texts, 35. Cambridge University Press, Cambridge, 1997. x+260 pp.

\bibitem[GaKr]{GaKr} F.P. Gantmacher and M.G. Krein, Oscillation matrices and kernels and small vibrations of mechanical systems, AMS Chelsea Publishing, Providence, RI, 2002.

\bibitem[GSV09]{GSV} M.~Gekhtman, M.~Shapiro, and A.~Vainshtein, 
Poisson geometry of directed networks in a disk. 
Selecta Math. (N.S.) 15 (2009), no. 1, 61--103. 

\bibitem[GSV12]{GSVsurface} M.~Gekhtman, M.~Shapiro, and A.~Vainshtein, Poisson geometry of directed networks in an annulus. J. Eur. Math. Soc. 14 (2012), no. 2, 541--570.


\bibitem[GGSVV]{GGSVV} J.~Golden, A.B.~Goncharov, M.~Spradlin, C.~Vergu, and A.~Volovich, Motivic Amplitudes and Cluster Coordinates. JHEP 01 (2014), 91.

\bibitem[GoKe]{GK} A. Goncharov and R. Kenyon, Dimers and cluster integrable systems. Ann. Sci. \'Ecole Norm. Sup. (4) 46 (2013), no. 5, 747--813. 

\bibitem[GoYa]{GY} K.R.~Goodearl and M.~Yakimov, Poisson structures on affine spaces and flag varieties. II. Trans. Amer. Math. Soc. 361 (2009), no. 11, 5753--5780.

\bibitem[GrHa]{GH} P.~Griffiths and J.~Harris, Principles of algebraic geometry. Pure and Applied Mathematics. Wiley-Interscience [John Wiley \& Sons], New York, 1978. xii+813 pp.

\bibitem[HeLa]{HL}  X. He and T. Lam, Projected Richardson varieties and affine Schubert varieties. Annales de l'Institut Fourier, to appear. {\tt arXiv:1106.2586}.

\bibitem[HePl]{HP} J.M.~Henn and J.C.~Plefka, Scattering amplitudes in gauge theories. Lecture Notes in Physics, 883. Springer, Heidelberg, 2014. xvi+195 pp.

\bibitem[Her]{Hersh} P.~Hersh, Regular cell complexes in total positivity. 
Invent. Math. 197 (2014), no. 1, 57--114. 

\bibitem[Hod]{Hod} A. Hodges, Eliminating spurious poles from gauge-theoretic amplitudes. JHEP 1305 (2013), 135.

\bibitem[HLZ]{HLZ} A.~Huang, B.~Lian, and X.~Zhu, Period Integrals and the Riemann-Hilbert Correspondence, preprint, 2013; {\tt arXiv:1303.2560}.

\bibitem[Kar]{Kar} R.~Karpman, Bridge graphs and Deodhar parametrizations for positroid varieties, preprint, 2014; {\tt arXiv:1411.2997}.

\bibitem[Kas91]{Kascan} M.~Kashiwara, On crystal bases of the $Q$-analogue of universal enveloping algebras. Duke Math. J. 63 (1991), no. 2, 465--516.

\bibitem[Kas93]{Kasdem}  M.~Kashiwara, The crystal base and Littelmann's refined Demazure character formula. Duke Math. J. 71 (1993), no. 3, 839--858. 

\bibitem[KeWi]{KW} R.W.~Kenyon and D.B.~Wilson, Boundary partitions in trees and dimers. Trans. Amer. Math. Soc. 363 (2011), no. 3, 1325--1364.

\bibitem[KeRo]{KR} B. Khesin and A. Rosly, Polar Homology.  Canad. J. Math. 55 (2003), 1100--1120.

\bibitem[KhKu]{KhKu} M.~Khovanov and G.~Kuperberg, Web bases for $sl(3)$ are not dual canonical. Pacific J. Math. 188 (1999), no. 1, 129--153.

\bibitem[Knu]{Knu} A. Knutson, Puzzles, positroid varieties, and equivariant K-theory of Grassmannians, preprint, 2010; {\tt arXiv:1008.4302}.

\bibitem[KLS13]{KLS} A. Knutson, T. Lam, and D.E. Speyer, Positroid varieties: juggling and geometry. Compos. Math. 149 (2013), no. 10, 1710--1752.

\bibitem[KLS14]{KLS2} A. Knutson, T. Lam, and D.E. Speyer, Projections of Richardson varieties. J. Reine Angew. Math. 687 (2014), 133--157.

\bibitem[KoWi]{KoWi} Y.~Kodama and L.~Williams, KP solitons and total positivity on the Grassmannian. Inventiones Mathematicae 198 (2014), 637--699.

\bibitem[Kuo]{Kuo} E.H. Kuo, Applications of graphical condensation for enumerating matchings and tilings. Theoret. Comput. Sci. 319 (2004), no. 1-3, 29--57. 

\bibitem[LaLi]{LaLi} V.~Lakshmibai and P.~Littelmann, Richardson varieties and equivariant K-theory. Special issue celebrating the 80th birthday of Robert Steinberg. J. Algebra 260 (2003), no. 1, 230--260.

\bibitem[Lam06]{LamAffineStanley} T.~Lam, Affine Stanley symmetric functions. Amer. J. Math. 128 (2006), no. 6, 1553--1586. 

\bibitem[Lam08]{LamJAMS} T.~Lam, Schubert polynomials for the affine Grassmannian. J. Amer. Math. Soc. 21 (2008), no. 1, 259--281.

\bibitem[Lam13a]{LamWhittaker} T.~Lam, Whittaker functions, geometric crystals, and quantum Schubert calculus, Proceedings of MSJ-SI Schubert calculus, to appear; {\tt arXiv:1308.5451}.
\bibitem[Lam13b]{Lamnotes} T.~Lam, Notes on the totally nonnegative Grassmannian,  \url{http://www.math.lsa.umich.edu/~tfylam/Math665a/positroidnotes.pdf}.

\bibitem[Lam14a]{Lamweb} T.~Lam, Dimers, webs, and positroids, preprint, 2014; {\tt arXiv:1404.3317}.

\bibitem[Lam14b]{Lamtrunc} T.~Lam, Amplituhedron cells and Stanley symmetric functions, preprint, 2014; {\tt arXiv:1408.5531}.
\bibitem[Lam+]{Lam+} T.~Lam, in preparation.

\bibitem[LLMSSZ]{book} T.~Lam, L.~Lapointe, J.~Morse, A.~Schilling, M.~Shimozono, and M.~Zabrocki, $k$-Schur Functions and Affine Schubert Calculus, Fields Institute Monographs 33, Springer 2014.

\bibitem[LaPo07]{LPalcove} T.~Lam and A.~Postnikov, Alcoved polytopes. I. Discrete Comput. Geom. 38 (2007), no. 3, 453--478.

\bibitem[LaPo+]{LPmembrane} T.~Lam and A.~Postnikov, in preparation.

\bibitem[LPP]{LPP} T.~Lam, A.~Postnikov, and P.~Pylyavskyy, Schur positivity and Schur log-concavity. Amer. J. Math. 129 (2007), no. 6, 1611--1622.

\bibitem[LaPy12]{LPasha1} T.~Lam and P.~Pylyavskyy, Total positivity in loop groups, I: Whirls and curls. Adv. Math. 230 (2012), no. 3, 1222--1271. 

\bibitem[LaPy13a]{LPasha2} T.~Lam and P.~Pylyavskyy, Total positivity for loop groups II: Chevalley generators. Transform. Groups 18 (2013), no. 1, 179--231.

\bibitem[LaPy13b]{LPashasurface} T.~Lam and P.~Pylyavskyy, Crystals and total positivity on orientable surfaces. Selecta Math. (N.S.) 19 (2013), no. 1, 173--235.

\bibitem[LaSp]{LamSpeyer} T.~Lam and D.~Speyer, DeRham cohomology of cluster algebras, in preparation.

\bibitem[LaLe09]{LL} S.~Launois and T.H.~Lenagan, From totally nonnegative matrices to quantum matrices and back, via Poisson geometry, preprint, 2009; {\tt arXiv:0911.2990}.

\bibitem [LaLe11]{LLtwist} S.~Launois and T.H.~Lenagan, Twisting the quantum Grassmannian. Proc. Amer. Math. Soc. 139 (2011), no. 1, 99--110.

\bibitem[Lec]{Lec} B.~Leclerc, Cluster structures on strata of flag varieties, preprint, 2014; {\tt arXiv:1402.4435}.

\bibitem[LeRi]{LR} T.H.~Lenagan and L.~Rigal, Quantum analogues of Schubert varieties in the Grassmannian. Glasg. Math. J. 50 (2008), no. 1, 55--70.

\bibitem[LeYa]{LY} T.H.~Lenagan and M.~Yakimov, Prime factors of quantum Schubert cell algebras and clusters for quantum Richardson varieties, preprint, 2015; {\tt arXiv:1503.06297}.

\bibitem[Lin]{Lin} B.~Lindstr\"om, On the vector representations of induced matroids. Bull. London Math. Soc. 5 (1973), 85--90.

\bibitem[dLRS]{dLRS} J.A.~De Loera, J.~Rambau, and F.~Santos, Triangulations. Structures for algorithms and applications. Algorithms and Computation in Mathematics, 25. Springer-Verlag, Berlin, 2010. xiv+535 pp.

\bibitem[Loe]{Loe} C. Loewner, On totally positive matrices. Math. Z. 63 (1955), 338--340.

\bibitem[Lus93]{Luscan} G.~Lusztig, Introduction to quantum groups. Progress in Mathematics, 110. Birkh\"auser Boston, Inc., Boston, MA, 1993. xii+341 pp.

\bibitem[Lus94]{LusTP} G.~Lusztig, Total positivity in reductive groups. Lie theory and geometry, 531--568, Progr. Math., 123, Birkh\"auser Boston, Boston, MA, 1994.

\bibitem[Lus98a]{Luspartial} G.~Lusztig, Total positivity in partial flag manifolds. Represent. Th. 2 (1998), 70--78 (electronic).

\bibitem[Lus98b]{Lusintro} G.~Lusztig, Introduction to total positivity, in ``Positivity in Lie theory: open problems" ed. J.Hilgert et. al., de Gruyter 1998, 133--145.

\bibitem[Mac]{Mac} I.G.~Macdonald, Symmetric functions and Hall polynomials. Second edition. With contributions by A. Zelevinsky. Oxford Mathematical Monographs. Oxford Science Publications. The Clarendon Press, Oxford University Press, New York, 1995. x+475 pp.

\bibitem[MaSk]{MaSk} L.J. Mason and D. Skinner, Dual superconformal invariance, momentum twistors and Grassmannians. JHEP 11 (2009), 045.

\bibitem[M\'eCa]{MC} A. M\'eriaux and G. Cauchon, Admissible diagrams in $U_q^w({\mathfrak{g}})$ and combinatoric properties of
Weyl groups, Represent. Theory 14 (2010), 645--687.

\bibitem[MS14]{MS} G. Muller and D.E.Speyer, Cluster Algebras of Grassmannians are Locally Acyclic, preprint, 2014; {\tt arXiv:1401.5137}.


\bibitem[Mum]{Mum} D.~Mumford, Geometric invariant theory. Ergebnisse der Mathematik und ihrer Grenzgebiete, Neue Folge, Band 34 Springer-Verlag, Berlin-New York 1965 vi+145 pp. 

\bibitem[Oh]{Oh} S. Oh, Positroids and Schubert matroids. J. Combin. Theory Ser. A 118 (2011), no. 8, 2426--2435.

\bibitem[OPS]{OPS} S.~Oh, A.~Postnikov, and D.E.~Speyer, Weak separation and plabic graphs. Proc. Lond. Math. Soc. (3) 110 (2015), no. 3, 721--754.

\bibitem[OhSp]{OS} S.~Oh and D.E.~Speyer, Links in the complex of weakly separated collections, preprint, 2014; {\tt arXiv:1405.5191}.


\bibitem[PPR]{PPR} T.K.~Petersen, P.~Pylyavskyy, and B.~Rhoades, 
Promotion and cyclic sieving via webs. J. Algebraic Combin. 30 (2009), no. 1, 19--41. 

\bibitem[Pos]{Pos} A. Postnikov, Total positivity, Grassmannians, and networks,

 {\tt http://www-math.mit.edu/~apost/papers/tpgrass.pdf}.

\bibitem[PSW]{PSW} 
A.~Postnikov, D.~Speyer, and L. Williams, Matching polytopes, toric geometry, and the totally non-negative Grassmannian. J. Algebraic Combin. 30 (2009), no. 2, 173--191.

\bibitem[Pro]{Propp} J. Propp, Enumeration of Matchings: Problems and Progress. New perspectives in algebraic combinatorics,
255--291, Math. Sci. Res. Inst. Publ., 38, Cambridge Univ. Press, Cambridge, 1999.


\bibitem[Ram]{Ram} J.~Rambau, Triangulations of cyclic polytopes and higher Bruhat orders. Mathematika 44 (1997), no. 1, 162--194.

\bibitem[Rho]{Rho} B.~Rhoades, Cyclic sieving, promotion, and representation theory. J. Combin. Theory Ser. A 117 (2010), no. 1, 38--76.

\bibitem[RhSk]{RS} B.~Rhoades and M.~Skandera, Temperley-Lieb immanants. Ann. Comb. 9 (2005), no. 4, 451--494.

\bibitem[Rie]{Rie} K.~Rietsch, An algebraic cell decomposition of the nonnegative part of a flag variety. J. Algebra 213 (1999), no. 1, 144--154.

\bibitem[RiWi]{RW} K.~Rietsch and L.~Williams, Discrete Morse theory for totally non-negative flag varieties. Adv. Math. 223 (2010), no. 6, 1855--1884.

\bibitem[Sch]{Sch} I.J. Schoenberg: Selected papers, Vol. 1; Contemporary Mathematicians, Birkh\"auser Boston, Inc., Boston, MA, 1988.

\bibitem[Sco]{Sco} J.S.~Scott, Grassmannians and cluster algebras. Proc. London Math. Soc. (3) 92 (2006), no. 2, 345--380.

\bibitem[Shi]{Shi} M.~Shimozono, Affine type A crystal structure on tensor products of rectangles, Demazure characters, and nilpotent varieties. J. Algebraic Combin. 15 (2002), no. 2, 151--187.

\bibitem[Sni]{Sni} M. Snider, Affine Patches on Positroid Varieties and Affine Pipe Dreams, Ph.D Thesis 2011, Cornell University.

\bibitem[Sta84]{Sta} R.P.~Stanley, On the number of reduced decompositions of elements of Coxeter groups. Europ. J. Combin. 5 (1984), 359--372.

\bibitem[Sta99]{EC2} R.P.~Stanley, Enumerative combinatorics. Vol. 2. With a foreword by Gian-Carlo Rota and appendix 1 by Sergey Fomin. Cambridge Studies in Advanced Mathematics, 62. Cambridge University Press, Cambridge, 1999. xii+581 pp.

\bibitem[Stu]{Stu} B.~Sturmfels, Totally positive matrices and cyclic polytopes. Proceedings of the Victoria Conference on Combinatorial Matrix Analysis (Victoria, BC, 1987). 
Linear Algebra Appl. 107 (1988), 275--281.

\bibitem[Tal]{Tal} K. Talaska, A formula for Pl\"ucker coordinates associated with a planar network. Int. Math. Res. Not. IMRN 2008, Art. ID rnn 081, 19 pp.

\bibitem[Vin]{Vin}  E.B.~Vinberg, The Theory of Convex Homogeneous Cones. Trans.of Moscow Math. Soc. 12, (1963), 340--403.

\bibitem[Whi]{Whi} A.M. Whitney, A reduction theorem for totally positive matrices. J. d’Analyse Math. 2 (1952), 88--92.
%

\bibitem[Wil]{Wil} L.~Williams, Shelling totally nonnegative flag varieties. J. Reine Angew. Math. 609 (2007), 1--21. 

\bibitem[Yak]{Y} M. Yakimov, Invariant prime ideals in quantizations of nilpotent Lie algebras. Proc. London
Math. Soc. (3) 101 (2010), no. 2, 454--476.

\bibitem[Zie]{Zie} G.M.~Ziegler, Lectures on polytopes. Graduate Texts in Mathematics, 152. Springer-Verlag, New York, 1995. x+370 pp.

\end{thebibliography}
\end{document}